\documentclass[12pt,PhD]{muthesis}
\usepackage{amsmath,amssymb}
\let\proof\relax \let\endproof\relax
\usepackage{amsthm} 
\usepackage{hyperref}
\hypersetup{linkcolor=blue,
    filecolor=magenta,      
    urlcolor=cyan,
}

\usepackage{xpatch}
\xpatchcmd{\qed}{\hfill}{}{}{}
\usepackage{mathrsfs}
\usepackage{theoremref}
\usepackage{enumerate}
\usepackage{cite}
\usepackage[numbers]{natbib}
\usepackage{nameref,cleveref}
\usepackage{mathrsfs}

\usepackage{xspace}
\let\OldS\S
\renewcommand{\S}{\OldS\xspace}

\newcommand{\A}{\mathbb{A}}
\newcommand{\C}{\mathbb{C}}

\newcommand{\F}{\mathbb{F}}

\newcommand{\N}{\mathbb{N}}

\newcommand{\Z}{\mathbb{Z}}

\newcommand{\ccD}{\mathcal{D}}
\newcommand{\sD}{\mathscr{D}}

\newcommand{\cL}{\mathcal{L}}

\newcommand{\cN}{\mathcal{N}}
\newcommand{\cO}{\mathcal{O}}

\newcommand{\cZ}{\mathcal{Z}}

\newcommand{\fc}{\mathfrak{c}}

\newcommand{\fg}{\mathfrak{g}}
\newcommand{\fgl}{\mathfrak{gl}}
\newcommand{\fh}{\mathfrak{h}}

\newcommand{\fL}{\mathfrak{L}}
\newcommand{\fm}{\mathfrak{m}}

\newcommand{\fM}{\mathfrak{M}}
\newcommand{\fn}{\mathfrak{n}}

\newcommand{\fsl}{\mathfrak{sl}}
\newcommand{\ft}{\mathfrak{t}}

\newcommand{\fz}{\mathfrak{z}}

\newcommand{\T}{\mathrm{T}}

\DeclareMathAlphabet{\mathpzc}{OT1}{pzc}{m}{it}
\newcommand{\dd}{\mathpzc}

\DeclareMathOperator{\ad}{ad}

\DeclareMathOperator{\Aut}{Aut}

\DeclareMathOperator{\chari}{char}

\DeclareMathOperator{\Der}{Der}

\DeclareMathOperator{\diver}{div}

\DeclareMathOperator{\GL}{GL}

\DeclareMathOperator{\gr}{gr}

\DeclareMathOperator{\Hom}{Hom}

\DeclareMathOperator{\Ima}{Im}

\DeclareMathOperator{\Id}{Id}

\DeclareMathOperator{\Jac}{Jac}

\DeclareMathOperator{\Ker}{Ker}
\DeclareMathOperator{\Lie}{Lie}
\DeclareMathOperator{\Mat}{Mat}
\DeclareMathOperator{\MT}{MT}

\DeclareMathOperator{\modd}{mod}

\DeclareMathOperator{\Rad}{Rad}

\DeclareMathOperator{\rk}{rk}

\DeclareMathOperator{\Soc}{Soc}
\DeclareMathOperator{\spn}{span}

\DeclareMathOperator{\del}{\partial}

\DeclareMathOperator{\Card}{Card}
\makeatletter

\renewenvironment{proof}[1][\proofname]{%
   \par\pushQED{\qed}\normalfont%
   \topsep6\p@\@plus6\p@\relax
   \trivlist\item[\hskip\labelsep\bfseries#1\@addpunct{.}]%
   \ignorespaces
   }{%
   \popQED\endtrivlist\@endpefalse
}

\newenvironment{myproof}[1][\proofname]{\proof[#1]\mbox{}\\*}{\endproof}
\newenvironment{mythm}[1][\thmname]{\thm[#1]\mbox{}\\*}{\endthm}

\renewenvironment{myproof}[1][\proofname]{%
   \par\pushQED{\qed}\normalfont%
   \topsep6\p@\@plus6\p@\relax
   \trivlist\item[\hskip\labelsep\bfseries#1\@addpunct{.}]%
   \ignorespaces
   }{%
   \popQED\endtrivlist\@endpefalse
}

\newtheorem{thm}{Theorem}[section]
\newtheorem{thm*}{Theorem}
\newtheorem{mythm*}{Theorem}
\newtheorem{pro}{Proposition}[section]
\newtheorem{lem}{Lemma}[section]
\newtheorem{mylem*}{Lemma}
\newtheorem{mypro*}{Proposition}

\newtheorem{cor}{Corollary}[section]

\newtheorem{defn}{Definition}[section]

\newtheorem{conj}{Conjecture}
\newtheorem{rmk}{Remark}[section]

\newtheorem{no}{Notation}[section]
\newtheorem{ex}{Example}[section]
\newtheorem{ex*}{Example}

\begin{document}
\title{Nilpotent varieties of some finite dimensional restricted Lie algebras}
\author{Cong Chen}
\school{Natural Sciences}
\department{Mathematics}
\faculty{Science and Engineering}
\def\wordcount{23,388}

\tablespagefalse

\figurespagefalse


\beforeabstract
In the late 1980s, A.~Premet conjectured that the variety of nilpotent elements of any finite dimensional restricted Lie algebra over an algebraically closed field of characteristic $p>0$ is irreducible. This conjecture remains open, but it is known to hold for a large class of simple restricted Lie algebras, e.g. for Lie algebras of connected algebraic groups, and for Cartan series $W, S$ and $H$.

In this thesis we start by proving that Premet's conjecture can be reduced to the semisimple case. The proof is straightforward. However, the reduction of the semisimple case to the simple case is very non-trivial in prime characteristic as semisimple Lie algebras are not always direct sums of simple ideals. Then we consider some semisimple restricted Lie algebras. Under the assumption that $p>2$, we prove that Premet's conjecture holds for the semisimple restricted Lie algebra whose socle involves the special linear Lie algebra $\fsl_2$ tensored by the truncated polynomial ring $k[X]/(X^{p})$. Then we extend this example to the semisimple restricted Lie algebra whose socle involves $S\otimes \cO(m; \underline{1})$, where $S$ is any simple restricted Lie algebra such that $\ad S=\Der S$ and its nilpotent variety $\cN(S)$ is irreducible, and $\cO(m; \underline{1})=k[X_1, \dots, X_m]/(X_1^p, \dots, X_m^p)$ is the truncated polynomial ring in $m\geq 2$ variables.

In the final chapter we assume that $p>3$. We confirm Premet's conjecture for the minimal $p$-envelope $W(1; n)_{p}$ of the Zassenhaus algebra $W(1; n)$ for all $n\in \N_{\geq 2}$. This is the main result of the research paper \cite{C19} which was published in the Journal of Algebra and Its Applications.

\afterabstract

\prefacesection{Acknowledgements}
I would like to thank Professor Alexander Premet for his knowledge and guidance during my research. I feel lucky to have such a supervisor who always promptly answers my questions, generously shares ideas with me, and patiently explains maths in great detail to me. Thanks for giving me the freedom to choose a project that I like. Thanks for helping me to prepare a research paper and guiding me through the publication process. Overall, it has been my great pleasure to be his student and my PhD life is enjoyable.

I would also like to thank Dr Rudolf Tange and Professor Mike Prest for their extremely useful comments to help me improving an earlier version of the thesis.

Thanks to Professor Helmut Strade for giving me a copy of the textbook \textit{Simple Lie Algebras over Fields of Positive Characteristic}. 

Thanks to the University of Manchester for awarding me a scholarship to support my research. Thanks to the School of Mathematics for providing me a nice working environment. Thanks for the support from all friendly staff and students at Alan Turing Building.

To my parents, thank you for providing me the opportunity to study in the UK. Without your constant support, I would not have been gone this far.

\afterpreface


\chapter{Introduction}
The theory of modular Lie algebras begins with E.~Witt who discovered a new nonclassical simple Lie algebra $W(1; 1)$ sometime before 1937. It is now called the \textit{Witt algebra}. Then H. Zassenhaus generalized the Witt algebra and got a new Lie algebra $W(1; n)$, called the \textit{Zassenhaus algebra}. In 1937, N.~Jacobson introduced the concept of \textit{``restricted Lie algebras''}. Later, more nonclassical Lie algebras were constructed.

In this introductory chapter we first review some basic concepts in the theory of modular Lie algebras. Then we explain the construction process of a class of nonclassical Lie algebras, namely the Lie algebras of Cartan type. In the end, we introduce Premet's conjecture on the variety of nilpotent elements of any finite dimensional restricted Lie algebra over an algebraically closed field of characteristic $p>0$. We shall discuss what have been done so far. 

Throughout the thesis we assume that all Lie algebras are finite dimensional, and $k$ is an algebraically closed field of characteristic $p>0$ (unless otherwise specified). We denote by $k^{*}$ the multiplicative group of $k$. We denote by $\F_p$ the finite field with $p$ elements. We denote by $\N_0$ the set of all nonnegative integers, i.e. $\N_{0}=\{0, 1, 2, \dots\}$.

\section{Restricted Lie algebras}\label{reLiealgebras}
Let us first introduce the notion of restricted Lie algebras.

\begin{defn}\citetext{\citealp[Definition 4, Sec.~7, Chap.~V]{J79}; \citealp[Sec.~2.1, Chap.~2]{S88}}\thlabel{defres}
Let $\fg$ be a Lie algebra over $k$. A mapping $[p]: \fg\to \fg, x \mapsto x^{[p]}$ is called a \textnormal{$[p]$-th power map} if it satisfies
\begin{enumerate}
\item $\ad x^{[p]} = (\ad x)^{p}$,
\item $(\lambda x)^{[p]} = \lambda^{p}x^{[p]}$,
\item $(x+y)^{[p]}=x^{[p]}+y^{[p]}+\sum_{i=1}^{p-1}s_{i}(x, y)$, where the terms $s_{i}(x, y)\in \fg$ are such that
\begin{equation*}
(\ad(tx+y))^{p-1}(x)=\sum_{i=1}^{p-1}is_{i}(x, y)t^{i-1}
\end{equation*} 
for all $x, y\in \fg$, $\lambda \in k$ and $t$ a variable. The pair ($\fg, [p]$) is referred to as a \textnormal{restricted Lie algebra}.
\end{enumerate}
\end{defn}

The third condition in the definition is known as \textit{Jacobson's formula for $p$-th powers}. A more general form of this formula is given by

\begin{lem}\cite[Lemma 1.1.1]{S04} \thlabel{generalJacobF}
Let ($\fg, [p]$) be a restricted Lie algebra over $k$. Then for all $x_1, \dots, x_m \in \fg$, the following holds:
\begin{align*}
\bigg(\sum_{i=1}^{m}x_i\bigg)^{[p]^{n}}=\sum_{i=1}^{m}x_i^{[p]^{n}}+ \sum_{l=0}^{n-1}v_{l}^{[p]^{l}},
\end{align*}
where $v_l$ is a linear combination of commutators in $x_i, 1\leq i\leq m$. By Jacobi identity, we can rearrange each $v_l$ so that $v_l$ is in the span of $[w_t,[w_{t-1},[\dots,[w_2,[w_1, w_0]\dots]$, where $t=p^{n-l}-1$ and each $w_j, 0\leq j\leq t$, is equal to some $x_i$, $1\leq i\leq m$.
\end{lem}

Let us give some examples of restricted Lie algebras.

\begin{ex}\thlabel{exres}
\begin{enumerate}
\item Any associative algebra $A$ over $k$ is a Lie algebra with the Lie bracket given by $[x,y]:=xy-yx$ for all $x, y\in A$. Denote this Lie algebra by $A^{(-)}$. Then $A^{(-)}$ is a restricted Lie algebra with the $[p]$-th power map given by $x \mapsto x^{p}$ for all $x\in A$; see \cite[Sec.~2.1, Chap.~2]{S88} for details. In particular, if $A=\Mat_{n}(k)$, the algebra of $n\times n$ matrices with entries in $k$, then the general linear Lie algebra $\fgl_n(k):=\Mat_{n}(k)^{(-)}$ is a restricted Lie algebra.
\item Let $A$ be any algebra over $k$ (not necessarily associative). The derivation algebra $\Der A$ is a restricted Lie algebra with the $[p]$-th power map given by $D\mapsto D^{p}$ for all $D\in \Der A$; see \cite[Sec.~7, Chap.~V]{J79} for details. As an example, the Lie algebra $\fg$ of an algebraic $k$-group $G$ is a restricted Lie algebra as $\fg$ is identified with the subalgebra of left invariant derivations of $k[G]$; see \cite[Sec.~9.1, Chap.~III]{H95}.
\end{enumerate}
\end{ex}

If a $[p]$-th power map exists on a Lie algebra $\fg$, we may ask how many different $[p]$-th power maps there are. It turns out that

\begin{lem}\cite[Proposition 2.1, Sec.~2.2, Chap.~2]{S88}
In a restricted Lie algebra $(\fg, [p])$, every $[p]$-th power map is of the form $x\mapsto x^{[p]}+f(x)$, where $f$ is a map of $\fg$ into its centre $\fz(\fg)$ satisfying $f(\lambda x+y)=\lambda^pf(x)+f(y)$ for all $x, y\in\fg, \lambda \in k$.
\end{lem}

An immediate corollary is that
\begin{cor}\cite[Corollary 2.2, Sec.~2.2, Chap.~2]{S88}\thlabel{centrelesslem}
If a Lie algebra $\fg$ is centreless, then $\fg$ has at most one $[p]$-th power map.
\end{cor}

\begin{defn}\cite[Sec.~2.1, Chap.~2]{S88}
Let ($\fg, [p]$) be a restricted Lie algebra over $k$. A subalgebra (respectively an ideal) $S$ of $\fg$ is called a \textnormal{$p$-subalgebra} (respectively  a \textnormal{$p$-ideal}) if $x^{[p]} \in S$ for all $x \in S$.
\end{defn}

Examples of $p$-ideals include the centre $\fz(\fg)$ and the radical $\Rad \fg$. Moreover, if $I$ and $J$ are $p$-ideals of $\fg$, then so are $I+J$, $I\cap J$, and $(I+J)/J \cong I/(I\cap J)$. 

Note that if $I$ is a $p$-ideal of $\fg$, then the quotient Lie algebra $\fg/I$ carries a natural $[p]$-th power map given by $(x+I)^{[p]}:=x^{[p]}+I$ for all $x \in \fg$; see \cite[Proposition 1.4, Sec.~2.1, Chap.~2]{S88}.

\begin{defn}\cite[p.~65]{S88}\thlabel{Spdefn}
Let $S$ be a subset of a restricted Lie algebra ($\fg, [p]$).
\begin{enumerate}[\upshape(i)]
\item The intersection of all $p$-subalgebras of $\fg$ containing $S$, denoted $S_p$, is a $p$-subalgebra of $\fg$ and is referred to as the \textnormal{$p$-subalgebra generated by $S$ in $\fg$}. Note that $S_p$ is the smallest $p$-subalgebra of $\fg$ containing $S$. 
\item Let $i\in \N_{0}$. The image of $S$ under the iterated application of the $[p]$-th power map, denoted $S^{[p]^{i}}$, is defined by  
\[
S^{[p]^i}:=\{x^{[p]^{i}}\,|\, x\in S\}.
\]
\end{enumerate}
\end{defn} 
Let us give an explicit characterization of $S_p$ in the following case.

\begin{lem}\cite[Proposition 1.3(1), Sec.~2.1, Chap.~2]{S88}\thlabel{computeSp}
Let ($\fg, [p]$) be a restricted Lie algebra over $k$, and let $H$ be a subalgebra of $\fg$ with basis $\{e_j\, |\, j\in J\}$. Then the $p$-subalgebra of $\fg$ generated by $H$ is given by
\[
H_p=\sum_{i\geq 0}\langle H^{[p]^{i}}\rangle=\sum_{j\in J, i\geq 0} ke_j^{[p]^{i}}.
\]
\end{lem}

\begin{defn}\cite[Sec.~2.3, Chap.~2]{S88}
Let ($\fg, [p]$) be a restricted Lie algebra over $k$. An element $x \in \fg$ is called \textnormal{semisimple} (or \textnormal{$p$-semisimple}) if $x\in (kx^{[p]})_p=\sum_{i\geq 1}kx^{[p]^{i}}$. If $x^{[p]}=x$, then $x$ is called \textnormal {toral}.
\end{defn}

\begin{lem}\cite[Proposition 3.3, Sec.~2.3, Chap.~2]{S88}\thlabel{sselemt}
Let ($\fg, [p]$) be a restricted Lie algebra over $k$. Then the following statements hold:
\begin{enumerate}[\upshape(i)]
\item If $x$ is toral, then $x$ is semisimple.
\item If $x$ and $y$ are semisimple and $[x, y]=0$, then $x+y$ is semisimple.
\item If $x$ is semisimple, then $x^{[p]^{i}}$ is semisimple for every $i\in \N$. Moreover, $y$ is semisimple for every $y \in (kx)_p$.
\end{enumerate}
\end{lem}

\begin{defn}\cite[Sec.~2.4, Chap.~2]{S88}
Let ($\fg, [p]$) be a restricted Lie algebra over $k$. A subalgebra $\ft\subset \fg$ is called a \textnormal {torus} if $\ft$ is an abelian $p$-subalgebra of $\fg$ consisting of semisimple elements.
\end{defn}

\begin{lem}\cite[Theorem 3.6(1), Sec.~2.3, Chap.~2]{S88}\thlabel{toralbasis}
Let ($\fg, [p]$) be a restricted Lie algebra over $k$. Then any torus in $\fg$ has a basis consisting of toral elements.
\end{lem}


\begin{defn}\cite[Notation 1.2.5]{S04}
Let ($\fg, [p]$) be a restricted Lie algebra over $k$. Set
\[
\MT(\fg):=\max \{\dim \ft\,|\, \text{$\ft$ is a torus of $\fg$}\},
\]
the \textnormal{maximal dimension of tori} in $\fg$.
\end{defn}

\begin{lem}\cite[Lemma 1.2.6(2)]{S04}\thlabel{MTlem}
Let ($\fg, [p]$) be a restricted Lie algebra over $k$ and let $I$ be a $p$-ideal of $\fg$. Then the following holds:
\begin{align*}
\MT(\fg)=\MT(\fg/I)+\MT(I).
\end{align*}
\end{lem}

Note that in a restricted Lie algebra ($\fg, [p]$), Cartan subalgebras can be described by maximal tori.

\begin{defn}\cite[Sec.~1.3, Chap.~1]{S88}
Let $L$ be a Lie algebra over $k$ (not necessarily restricted). The \textnormal{lower central series} of $L$ is the sequence of ideals of $L$ defined as follows: $L^1=L, L^{i+1}=[L, L^i]$ for $i\geq 1$. Then 
\[
L=L^1\supseteq L^2\supseteq\dots\supseteq L^{n}\supseteq\dots
\]
We say that $L$ is \textnormal{nilpotent} if $L^{m}=0$ for some $m\in \N$. 
\end{defn}

\begin{defn}\cite[Sec.~1.4, Chap.~1]{S88}
Let $L$ be a Lie algebra over $k$ (not necessarily restricted). A subalgebra $\fh$ of $L$ is called a \textnormal{Cartan subalgebra} if it is nilpotent and equal to its own normalizer, i.e.
\[
\fh=\text{N}_{L}(\fh)=:\{x\in L\,|\,[x, h]\in \fh\, \text{for all $h\in \fh$}\}.
\]
\end{defn}

\begin{thm}\cite[Theorem 4.1, Sec.~2.4, Chap.~2]{S88}\thlabel{CSAmaxitori}
Let ($\fg, [p]$) be a restricted Lie algebra over $k$. Let $\fh$ be a subalgebra of $\fg$. The following statements are equivalent:
\begin{enumerate}[\upshape(i)]
\item $\fh$ is a Cartan subalgebra.
\item There exists a maximal torus $\ft$ in $\fg$ such that $\fh=\fc_{\fg}(\ft)$, the centralizer of $\ft$ in $\fg$.
\end{enumerate}
\end{thm}

\begin{defn}\cite[Sec.~2.1, Chap.~2]{S88}
Let ($\fg, [p]$) be a restricted Lie algebra over $k$. An element $x \in \fg$ is called \textnormal{nilpotent} (or \textnormal{$p$-nilpotent}) if there is $n\in \N$ such that $x^{[p]^{n}}=0$ .
\end{defn}

We denote by $\cN(\fg)$ the variety of all nilpotent elements in $\fg$. It is well known that $\cN(\fg)$ is a Zariski closed, conical subset of $\fg$. 

\begin{mythm}[Jordan-Chevalley Decomposition \citetext{\citealp[Theorem 3.5, Chap.~2]{S88}}] \thlabel{JCthm} 
Let ($\fg, [p]$) be a restricted Lie algebra over $k$. For any $x\in\fg$, there exist a unique semisimple element $x_s\in \fg$ and a unique nilpotent element $x_n\in \fg$ such that $x=x_s+x_n$ and $[x_s, x_n]=0$.
\end{mythm}

It follows from the above theorem that $\fg=\cN(\fg)$ if and only if $\MT(\fg)=0$. 

\section{$p$-envelopes}
Let $L$ be a Lie algebra over $k$. It is useful to embed $L$ into a restricted Lie algebra. 

\begin{defn}\cite[Sec.~2.5, Chap.~2]{S88}
Let $L$ be a Lie algebra over $k$. A triple $(\cL, [p], i)$ consisting of a restricted Lie algebra $(\cL, [p])$ and a Lie algebra homomorphism $i: L\to \cL$ is called a \textnormal{$p$-envelope} of $L$ if $i$ is injective and the $p$-subalgebra generated by $i(L)$, denoted $(i(L))_p$, coincides with $\cL$.
\end{defn}

We often identify $L$ with $i(L) \subset \cL$. Let us review some properties of $p$-envelopes.

\begin{thm}\cite[Theorem 1.1.7]{S04}
Let $(\cL_1, [p]_1, i_1)$ and $(\cL_2, [p]_2, i_2)$ be two \\$p$-envelopes of $L$. Then there exists an isomorphism $\psi$ of restricted Lie algebras 
\[
\psi: \cL_1/\fz(\cL_1)\xrightarrow{\sim} \cL_2/\fz(\cL_2)
\]
such that $\psi \circ \pi_1 \circ i_1=\pi_2 \circ i_2$, where $\pi_1: \cL_1 \to \cL_1/\fz(\cL_1)$ and $\pi_2: \cL_2 \to \cL_2/\fz(\cL_2)$ are the canonical homomorphisms of restricted Lie algebras.
\end{thm}

\begin{defn}\cite[Sec.~2.5, Chap.~2]{S88}
A $p$-envelope $(\cL, [p], i)$ of $L$ is called \textnormal{minimal} if $\fz(\cL)\subset \fz(i(L))$.
\end{defn}

\begin{thm}\citetext{\citealp[Theorem 1.1.6 and Corollary 1.1.8]{S04}; \citealp[Theorem 5.8, Sec.~2.5, Chap.~2]{S88}}\thlabel{sspenvelope}
\begin{enumerate}[\upshape(i)]
\item If $(\cL, [p], i)$ is a $p$-envelope of $L$, then there exists a minimal $p$-envelope $(H, [p]_1, i_1)$ of $L$ and an ideal $J \subset \fz(\cL)$ such that $\cL=H\oplus J$ and $i_1=i$ (i.e. $H\subset \cL$).
\item Any two minimal $p$-envelopes of $L$ are isomorphic as ordinary Lie algebras.
\item Suppose $L$ is semisimple. Then every minimal $p$-envelope of $L$ is semisimple, and all minimal $p$-envelopes of $L$ are isomorphic as restricted Lie algebras.
\end{enumerate}
\end{thm}

\begin{rmk}\citetext{\citealp[p. 22]{S04}; \citealp[p. 97]{S88}}\thlabel{ssminpenvelope}
If $L$ is semisimple, then we can easily describe its minimal $p$-envelope. Since $L$ is semisimple, there is an embedding $L\cong \ad L \hookrightarrow \Der L$ via the adjoint representation. Then the minimal $p$-envelope of $L$ is the $p$-subalgebra of $\Der L$ generated by $\ad L$, i.e. $(\ad L)_p$. To compute $(\ad L)_p$, we often identify $L$ with $\ad L$.
\end{rmk}

\section{Gradations and standard filtrations}
\begin{defn}\cite[Sec.~3.2, Chap.~3]{S88}
Let $L$ be a Lie algebra over $k$. A \textnormal{$\Z$-grading} of $L$ is a collection of subspaces $(L_i)_{i\in \Z}$ such that 
\begin{enumerate}[\upshape(i)]
\item $L =\bigoplus_{i\in \Z} L_i$ and 
\item $[L_i, L_j] \subset L_{i+j}$ for all $i, j \in \Z$.
\end{enumerate}
If there exist $r, s\in \Z$ such that $L =\bigoplus_{-r}^{s} L_i$, then $r$ (respectively $s$) is called the \textnormal{depth} (respectively \textnormal{height}) of this gradation. 
\end{defn}

Note that $L_0$ is a Lie subalgebra of $L$ and each subspace $L_i$ obtains an $L_0$-module structure via the adjoint representation.

\begin{defn}\cite[Sec.~3.2, Chap.~3]{S88}
Let $(\fg, [p])$ be a restricted Lie algebra over $k$. A gradation $(\fg_i)_{i\in\Z}$ of $\fg$ is called \textnormal{restricted} if $\fg_{i}^{[p]}\subset \fg_{pi}$ for all $i\in\Z$.
\end{defn}

\begin{defn}\cite[Sec.~1.9, Chap.~1]{S88}\thlabel{fildefn}
Let $L$ be a Lie algebra over $k$. A \textnormal{descending filtration} of $L$ is a collection of subspaces $(L_{(i)})_{i\in\Z}$  such that 
\begin{enumerate}[\upshape(i)]
\item $L_{(i)}\supset L_{(j)}$ if $i\leq j$.
\item $[L_{(i)}, L_{(j)}]\subset L_{(i+j)}$ for all $i, j\in\Z$.
\end{enumerate}
A filtration is called \textnormal{separating} if $\cap_{i\in\Z} L_{(i)}=\{0\}$ and \textnormal{exhaustive} if $\cup_{i\in \Z} L_{(i)}=L$. The notion of \textnormal{ascending filtration} is defined similarly.
\end{defn}

It is common to use descending filtrations for Lie algebras. Note that $L_{(0)}$ is a Lie subalgebra of $L$ and each subspace $L_{(i)}$ obtains an $L_{(0)}$-module structure via the adjoint representation. If the filtration is exhaustive, then $\cap_{i\in\Z} L_{(i)}$ is an ideal of $L$. If in addition that $L$ is simple, then either $L_{(i)}=L$ for all $i$, or the filtration is separating; see \cite[p. 100]{S88}. By a result of Weisfeiler \cite{Wei68}, we can define standard filtrations.

\begin{defn}\citetext{\citealp[Sec.~2.4]{PS2006}; \citealp[Definition 3.5.1]{S04}}\thlabel{sfiltrations}
Let $L$ be a Lie algebra over $k$ and $L_{(0)}$ be a maximal subalgebra of $L$. Let $L_{(-1)}$ be an $L_{(0)}$-invariant subspace of $L$ which contains $L_{(0)}$. Moreover, assume that $L_{(-1)}/L_{(0)}$ is an irreducible $L_{(0)}$-module.
Set
\begin{align*}
L_{(i+1)}&:=\{x\in L_{(i)}\,|\,[x, L_{(-1)}]\subset L_{(i)}\}, \quad i\geq 0, \\
L_{(-i-1)}&:=[L_{(-i)}, L_{(-1)}]+L_{(-i)}, \quad \quad \quad \quad i\geq 1.
\end{align*}
The sequence of subspaces $(L_{(i)})_{i\in\Z}$ defines a \textnormal{standard filtration} on $L$.
\end{defn}

Since $L_{(0)}$ is a maximal subalgebra of $L$ this filtration is exhaustive. If $L$ is simple, then this filtration is separating. So there are $s_1>0$ and $s_2\geq0$ such that 
\begin{equation}\label{sfnotation}
L=L_{(-s_1)}\supset \dots\supset L_{(0)}\supset \dots \supset L_{(s_{2}+1)}=(0).
\end{equation}

\begin{thm}\cite[Theorem 1.3, Sec.~3.1, Chap.~3]{S88}
Let $L$ be a simple Lie algebra over an algebraically closed field of characteristic $p>3$. If there is $x\in L$ such that $(\ad x)^{p-1}=0$, then there exists a standard filtration as above \eqref{sfnotation}.
\end{thm}

It was proved by A.~Premet that such $x\neq 0$ with $(\ad x)^{p-1}=0$ always exists in simple Lie algebras over algebraically closed fields of characteristic $p>3$; see \cite[Theorem 1]{P86} . Hence they admit a standard filtration.

We can define restricted filtrations in a similar way.

\begin{defn}\cite[Sec.~3.1, Chap.~3]{S88}
Let $(\fg, [p])$ be a restricted Lie algebra over $k$. A filtration $(\fg_{(i)})_{i\in\Z}$ of $\fg$ is called \textnormal{restricted} if $\fg_{(i)}^{[p]}\subset \fg_{(pi)}$ for all $i\in\Z$.
\end{defn}

It is useful to note the interrelation between gradations and filtrations; see \cite[Sec.~3.3, Chap.~3]{S88}. Given any $\Z$-graded Lie algebra $L=\bigoplus_{i\in \Z} L_i$, set $L_{(j)}:= \bigoplus _{i\geq j} L_i$. Then this $\Z$-grading induces a filtration on $L$. Conversely, suppose $L$ has a descending filtration $(L_{(i)})_{i\in \Z}$. We can define the graded Lie algebra $\gr L$ associated with $L$. Put $L_j:=L_{(j)}/L_{(j+1)}$ for all $j\in\Z$. Then $\gr L:=\bigoplus_{j\in \Z} L_j$. The Lie bracket in $\gr L$ is given by 
\[
[x+L_{(j+1)}, y+L_{(l+1)}]:=[x, y]+L_{(j+l+1)}
\]
for all $x\in L_{(j)}$ and $y\in L_{(l)}$.

\section{Graded Lie algebras of Cartan type}\label{Cartan}
In a series of papers, A.~Premet and H.~Strade have completed the classification of finite dimensional simple modular Lie algebras and proved the following: 
\begin{mythm}[Classification Theorem \citetext{\citealp[Theorem 1.1]{PS08}}] \label{classthm}
Any finite dimensional simple Lie algebra over an algebraically closed field of characteristic $p>3$ is of classical, Cartan or Melikian type.
\end{mythm}
The \textit{classical} simple modular Lie algebras include both classical simple (modulo its centre for $\fsl_{mp}$) and exceptional Lie algebras over $\C$. They were constructed using a Chevalley basis by reduction modulo $p$ \cite{C56}. The \textit{Melikian} algebras $\mathcal{M}(m, n)$ depend on two parameters $m, n \in \N$, and they only occur in characteristic $5$ \cite{M80}. The Lie algebras of \textit{Cartan} type provide a large class of nonclassical simple Lie algebras. They are finite dimensional modular analogues of the four families Witt, special, Hamiltonian, contact of infinite dimensional complex Lie algebras. Their construction was motivated by Cartan's work on pseudogroups. The formal power series algebras over $\C$ were replaced by divided power algebras over $k$; see \cite{KS66} and \cite{KS69}. Our work relates to Lie algebras of Cartan type, particularly the general Cartan type Lie algebras. So let us give a detailed description of these Lie algebras. All definitions and theorems can be found in \cite{BGP05}, \cite{S04} and \cite{S88}.

\begin{no}\citetext{\citealp[Sec.~2.1, Chap.~2]{S04}; \citealp[Sec.~3.5, Chap.~3]{S88}}
Let $\N_0^{m}$ denote the set of all $m$-tuples of nonnegative integers. For $a=(a_1, \dots, a_m), b=(b_1, \dots, b_m)\in \N_0^{m}$, we write
\begin{align*}
x_i^{(a_i)}&:=\frac{1}{a_i!} x_i^{a_i}, \,\quad\quad\quad\quad x^{(a)}:=\prod_{i=1}^{m} x_i^{(a_i)},\\
{a+b}\choose{b}&:=\prod_{i=1}^{m}{{a_i+b_i}\choose{b_i}},\quad\quad a!:=\prod_{i=1}^{m}a_i!, \\
|a|&:=\sum_{i=1}^{m}a_i.
\end{align*}
\end{no}

\begin{defn}\cite[Sec.~2.1, Chap.~2]{S04}
Let $\cO(m)$ denote the commutative associative algebra with unit element over $k$ defined by generators $x_i^{(r)}, 1\leq i\leq m, r\geq 0$, and relations
\begin{align*}
x_i^{(0)}=1, \quad x_i^{(r)}x_i^{(s)}={{r+s}\choose {r}}x_i^{(r+s)}, \quad 1\leq i\leq m, \quad r, s\geq 0.
\end{align*}
Then $\{x^{(a)}\,|\, a\in \N_0^{m}\}$ forms a basis of $\cO(m)$, and $\cO(m)$ is called the \textnormal{divided power algebra}.
\end{defn}

For simplicity, we write $x_i^{(1)}$ as $x_i$. Some results on binomial coefficients may be useful.
\begin{lem}\cite[Lemma 2.1.2(1)]{S04}\thlabel{bcmodp}
For $a,b\in \N$, let $a=\sum_{i\geq 0}a_ip^{i}$, $b=\sum_{i\geq 0}b_ip^{i}$, $0\leq a_i, b_i\leq p-1$, be the $p$-adic expansions of $a$ and $b$. Then the following congruence holds:
\begin{align*}
{{a}\choose {b}}\equiv \prod_{i\geq 0} {{a_i}\choose {b_i}}\quad (\modd p).
\end{align*}
\end{lem}
\begin{myproof}[Sketch of proof]
Let $a$ and $b$ be as in the lemma. Let $Y$ be a variable. Note that for any integer $n$ such that $1\leq n\leq p-1$, 
\[
{{p}\choose {n}}\equiv 0\quad(\modd p).
\]
Hence 
\[
(1+Y)^{p}\equiv 1+Y^p\quad (\modd p).
\]
In general, one can show by induction that for any $i\geq 1$, 
\[
(1+Y)^{p^{i}}\equiv 1+Y^{p^{i}}\quad (\modd p).
\]
Consider $(1+Y)^{a}$ and expand it over $\Z$, we get
\[
(1+Y)^{a}=\prod_{i\geq 0}(1+Y)^{a_ip^{i}}\equiv \prod_{i\geq 0}(1+Y^{p^{i}})^{a_i}\quad (\modd p).
\]
Compare the coefficients of $Y^{b}$ on both sides, we get
\begin{align*}
{{a}\choose {b}}\equiv \prod_{i\geq 0} {{a_i}\choose {b_i}}\quad (\modd p).
\end{align*} 
This completes the sketch of proof.
\end{myproof}

Applying the above result, we can prove that
\begin{cor}\thlabel{bcmodpapp}
Let $r\in \N$ be such that $r\geq 2$ and $1\leq s\leq r-1$. Then for any $0\leq i\leq p^{s}-1$, the following congruence holds:
\begin{align*}
{{p^r-p^s+i}\choose {p^r-p^s}}\equiv 1 \quad (\modd p).
\end{align*}
\end{cor}
\begin{proof}
Let $r\in \N$ be such that $r\geq 2$ and $1\leq s\leq r-1$. Then the $p$-adic expansion of $p^r-p^s$ is  
\begin{equation}\label{paeofrs1}
p^r-p^s=\sum_{j=s}^{r-1}(p-1)p^j.
\end{equation}
Since $0\leq i\leq p^{s}-1$ and the $p$-adic expansion of $p^s-1$ is $p^s-1=\sum_{j=0}^{s-1}(p-1)p^j$, it follows that the $p$-adic expansion of $i$ is
\begin{equation}\label{paeofrs2}
i=\sum_{j=0}^{s-1}a_{j}p^j,
\end{equation}
where $0\leq a_j\leq p-1$ and $0\leq \sum_{j=0}^{s-1}a_j\leq s(p-1)$. By \eqref{paeofrs1} and \eqref{paeofrs2}, the $p$-adic expansion of $p^r-p^s+i$ is 
\begin{align*}
p^r-p^s+i=\sum_{j=0}^{s-1}a_{j}p^j+\sum_{j=s}^{r-1}(p-1)p^j.
\end{align*}
It follows from \thref{bcmodp} that 
\begin{align*}
{{p^r-p^s+i}\choose {p^r-p^s}}\equiv \prod_{j=0}^{s-1}{{a_j}\choose {0}}\cdot \prod_{j=s}^{r-1}{{p-1}\choose {p-1}}\quad (\modd p)= 1\quad (\modd p).
\end{align*}
This completes the proof.
\end{proof}

Note that there is a $\Z$-grading on $\cO(m)$ given by $\cO(m)_i:=\spn\{x^{(a)}\,|\, |a|=i\}$. Hence $\cO(m)=\bigoplus_{i=0}^{\infty}\cO(m)_i$. Put $\cO(m)_{(j)}:=\bigoplus_{i\geq j} \cO(m)_i$. Then this $\Z$-grading induces a descending filtration on $\cO(m)$, called the {\it standard filtration}.

\begin{defn}\cite[Definition 2.1.1]{S04}\thlabel{dpower}
A system of \textnormal{divided powers} on $\cO(m)_{(1)}$ is a sequence of maps 
\[
\gamma_{r}: \cO(m)_{(1)}\to \cO(m),\quad f\mapsto f^{(r)}\in \cO(m),
\]
where $r\geq 0$, satisfying 
\begin{enumerate}[\upshape(i)]
\item \text{$f^{(0)}=1, \,\,f^{(r)}\in \cO(m)_{(1)}$  \,\,\,\quad\quad for all $f \in \cO(m)_{(1)},\,\, r>0$},
\item \text{$f^{(1)}=f$ \,\,\,\,\,\,\,\quad\quad\qquad\qquad\qquad for all $f \in \cO(m)_{(1)}$},
\item \text{$f^{(r)}f^{(s)}=\frac{(r+s)!}{r!s!}f^{(r+s)}$ \,\,\,\,\quad \quad\quad for all $f \in \cO(m)_{(1)}$,\,\, $r, s\geq 0$},
\item \text{$(f+g)^{(r)}=\sum_{l=0}^{r}f^{(l)}g^{(r-l)}$  \,\quad for all $f, g\in \cO(m)_{(1)}, \,\,r\geq 0$},
\item \text{$(fg)^{(r)}=f^{r}g^{(r)}$  \qquad\qquad\qquad for all $f\in \cO(m), g\in \cO(m)_{(1)}, \,\,r\geq 0$}, 
\item \text{$(f^{(s)})^{(r)}=\frac{(rs)!}{r!(s!)^{r}}f^{(rs)}$ \qquad\qquad for all $f\in \cO(m)_{(1)}, \,\,r\geq 0,\,\, s>0$}.
\end{enumerate}
\end{defn}

\begin{defn}\cite[Definition 2.1.1(2)]{S04}\thlabel{specialderivation}
A derivation $D$ of $\cO(m)$ is called \textnormal{special} if $D (f^{(r)})=f^{(r-1)}D(f)$ for all $f \in \cO(m)_{(1)}$ and $r>0$.
\end{defn}
For $1\leq i\leq m$, set $\epsilon_i=(\delta_{i1}, \dots, \delta_{im})$. Let $\del_i$ denote the \textnormal{$i$th partial derivative} defined by $\del_i(x^{(a)})=x^{(a-\epsilon_i)}$ if $a_i>0$ and $0$ otherwise; see \cite[p. 132]{S88}. We denote by $W(m)$ the set of all special derivations of $\cO(m)$. This is a Lie subalgebra of $\Der \cO(m)$ and it obtains an $\cO(m)$-module structure via $(fD)(g):=fD(g)$ for all $f, g\in \cO(m)$ and $D\in W(m)$. Since each $D\in W(m)$ is uniquely determined by its effects on $x_1, \dots, x_m$, the Lie algebra $W(m)$ is a free $\cO(m)$-module of rank $m$ generated by the partial derivatives $\del_1, \dots, \del_m$; see \cite[Proposition 2.1.4]{S04}. By \cite[Lemma 2.1(1), Sec.~4.2, Chap.~4]{S88}, we know that for any $f, g\in \cO(m)$ and $D, E\in W(m)$, 
\begin{equation}\label{LbracketWm}
[fD, gE]=fD(g)E-gE(f)D+fg[D, E].
\end{equation}
Since $\{x^{(a)}\del_i\,|\, a\in \N_0^{m}, 1\leq i\leq m\}$ forms a basis for $W(m)$ and $[\del_i, \del_j]=0$ for any $1\leq i, j\leq m$, it follows from \eqref{LbracketWm} that the Lie bracket in $W(m)$ is given by 
\begin{equation}\label{Lbracketinbasis}
[x^{(a)}\del_i, x^{(b)}\del_j]={{a+b-\epsilon_i}\choose{a}}x^{(a+b-\epsilon_i)}\del_j-{{a+b-\epsilon_j}\choose{b}}x^{(a+b-\epsilon_j)}\del_i.
\end{equation}
Note that $W(m)$ inherits a grading and descending filtration from $\cO(m)$:
\begin{align*}
W(m)_i:=\bigoplus_{j=1}^{m} \cO(m)_{i+1}\del_j, \quad W(m)_{(i)}:=\bigoplus_{j=1}^{m} \cO(m)_{(i+1)}\del_j
\end{align*}
for $i\geq -1$. Both are called {\it standard}.

For any $m$-tuple $\underline{n}:=(n_1, \dots, n_m) \in \N^{m}$, define 
\[
\cO(m;\underline{n}):=\spn\{x^{(a)}\,|\, 0\leq a_i <p^{n_i}\}.
\] 
It is easy to see that $\cO(m;\underline{n})$ is a subalgebra of $\cO(m)$ invariant under $\del_i$ for all $1\leq i\leq m$. Moreover, $\dim \cO(m;\underline{n})=p^{|\underline{n}|}$. The {\it general} Cartan type Lie algebra $W(m; \underline{n})$ is the Lie subalgebra of $W(m)$ which normalizes $\cO(m; \underline{n})$. Since $\cO(m; \underline{n})$ is a subalgebra of $\cO(m)$, the grading and filtration on $\cO(m; \underline{n})$ induce a grading and filtration on $W(m;\underline{n})$:
\begin{equation}\label{wmnfiltration}
W(m; \underline{n})_i:=\bigoplus_{j=1}^{m} \cO(m; \underline{n})_{i+1}\del_j, \quad W(m; \underline{n})_{(i)}:=\bigoplus_{j=1}^{m} \cO(m; \underline{n})_{(i+1)}\del_j
\end{equation} 
for $i\geq -1$. In particular, 
\begin{align*}
W(m; \underline{n})=\bigoplus_{i=-1}^{s}W(m; \underline{n})_i,
\end{align*}
where $s=(\sum_{i=1}^{m} p^{n_i})-m-1$; see \cite[Proposition 2.2(3), Sec.~4.2, Chap.~4]{S88}.

\begin{thm}\cite[Proposition 5.9, Sec.~3.5, Chap.~3; Proposition 2.2 and Theorem 2.4, Sec.~4.2, Chap.~4]{S88}\thlabel{wittthm}
\begin{enumerate}[\upshape(i)]
\item $W(m; \underline{n})$ is a free $\cO(m; \underline{n})$-module with basis $\{\del_1, \dots, \del_m\}$. 
\item The set $\{x^{(a)}\del_i\,|\, 0\leq a_i <p^{n_i}, 1\leq i\leq m\}$ forms a basis for $W(m; \underline{n})$. Hence $\dim W(m; \underline{n})=mp^{|\underline{n}|}$.
\item $W(m; \underline{n})$ is simple unless $m=1$ and $p=2$.
\item $W(m; \underline{n})$ is a subalgebra of the restricted Lie algebra $\Der \cO(m;\underline{n})$.
\item $W(m; \underline{n})$ is restricted if and only if $\underline{n}=(1, \dots, 1)$, and in that case $D^{[p]}=D^p$ for all $D \in W(m; \underline{n})$ and the gradation is restricted.
\end{enumerate}
\end{thm}

We refer to the Lie algebras $W(m)$ or $W(m; \underline{n})$ as Lie algebras of {\it Witt} type. In Chapter 3 we will spell out $W(m; \underline{1})$ in more details. By \cite[Lemma 2.1(3), Sec.~4.2, Chap.~4]{S88}, we know that $\cO(m; \underline{1})$ is isomorphic to the truncated polynomial ring $k[X_1, \dots, X_m]/(X_1^{p}, \dots, X_{m}^{p})$ in $m$ variables. Hence $W(m; \underline{1})\cong\Der \cO(m; \underline{1})$, and it is called the \textit{$m$th Jacobson-Witt algebra}. In Chapter 4 we will study the \textit{Zassenhaus algebra} $W(1; n)$. Note that if $\chari k=p>2$ and $n=1$, then $W(1; n)$ coincides with the \textit{Witt algebra} $W(1; 1)$, a simple and  restricted Lie algebra. If $\chari k=p>2$ and $n\geq 2$, then $W(1; n)$ provides the first example of a simple, non-restricted Lie algebra. In this case it is useful to consider its minimal $p$-envelope. 

Let us determine the minimal $p$-envelope of the simple, non-restricted Witt algebra $W(m;\underline{n})$. Since $W(m; \underline{n})$ is simple, it follows from \thref{sspenvelope}(iii) that all its minimal $p$-envelopes are isomorphic as restricted Lie algebras. Moreover, there is an embedding $W(m;\underline{n})\cong \ad W(m;\underline{n}) \hookrightarrow \Der W(m;\underline{n})$ via the adjoint representation. By \thref{ssminpenvelope}, the minimal $p$-envelope of $W(m;\underline{n})$, denoted $W(m; \underline{n})_{[p]}$, is the \\$p$-subalgebra $(\ad W(m;\underline{n}))_p$ of $\Der W(m;\underline{n})$ generated by $\ad W(m;\underline{n})$, i.e.
\begin{align*}
W(m;\underline{n})\cong \ad W(m;\underline{n})\hookrightarrow  W(m; \underline{n})_{[p]}=(\ad W(m;\underline{n}))_p\hookrightarrow \Der W(m;\underline{n});
\end{align*}
see \thref{Spdefn} and \thref{computeSp} for notations. In \cite[Sec.~7.1 and 7.2]{S04}, H.~Strade computed $W(m; \underline{n})_{[p]}$. He first proved the following:
\begin{thm}\cite[Theorems 7.1.2(1)]{S04}\thlabel{DerW(m;n)ans}
\begin{align*}
\Der W(m; \underline{n})\cong W(m; \underline{n})+\sum_{i=1}^{m}\sum_{0<j_i<n_i}k\del_i^{p^{j_i}}.
\end{align*}
The isomorphism is given by the adjoint representation, $W(m;\underline{n})\cong \ad W(m;\underline{n}) \hookrightarrow \Der W(m;\underline{n})$.
\end{thm}
Then H.~Strade computed $W(m; \underline{n})_{[p]}=(\ad W(m;\underline{n}))_p$. He identified $W(m;\underline{n})$ with $\ad W(m;\underline{n})$. By \thref{wittthm}(iv), we know that $W(m;\underline{n})$ is a subalgebra of the restricted Lie algebra $\Der \cO(m;\underline{n})$. So instead of computing the $p$-subalgebra $(\ad W(m;\underline{n}))_p$ of $\Der W(m;\underline{n})$ generated by $\ad W(m;\underline{n})$, H.~Strade computed the $p$-subalgebra $(W(m;\underline{n}))_p$ of $\Der \cO(m;\underline{n})$ generated by $W(m;\underline{n})$. By \thref{Spdefn} and \thref{computeSp}, 
\[
(W(m;\underline{n}))_p=\sum_{i\geq 0}\langle W(m;\underline{n})^{p^{i}}\rangle,
\]
where $W(m;\underline{n})^{p^{i}}:=\{D^{p^{i}}\,|\, D\in W(m;\underline{n})\}$ is the image of $W(m;\underline{n})$ under the iterated application of the $[p]$-th power map of $\Der \cO(m;\underline{n})$. H.~Strade first observed that 

\begin{lem}\cite[Lemma 7.1.1(3)]{S04}\thlabel{W(m;n)_(0)restricted}
$W(m;\underline{n})_{(0)}$ is a restricted Lie subalgebra of $\Der \cO(m;\underline{n})$.
\end{lem}
It follows that $W(m; \underline{n})_{[p]}=(\ad W(m;\underline{n}))_p$ contains $W(m; \underline{n})$ and all iterated $p$-th powers of the partial derivatives $\del_1, \dots, \del_m$. Applying \thref{DerW(m;n)ans}, we get
\begin{thm}\cite[Theorem 7.2.2(1)]{S04}\thlabel{wittpenvelope}
The minimal $p$-envelope $ W(m; \underline{n})_{[p]}$ of \\$W(m; \underline{n})$ in $\Der W(m; \underline{n})$ is given by 
\begin{align*}
W(m; \underline{n})_{[p]}=W(m; \underline{n})+\sum_{i=1}^{m}\sum_{0<j_i<n_i}k\del_i^{p^{j_i}}.
\end{align*}
\end{thm}
Since $W(m; \underline{n})$ is the Lie subalgebra of $W(m)$, the Lie bracket in $W(m; \underline{n})$ is given by \eqref{Lbracketinbasis}. By \eqref{LbracketWm}, we have that for any $1\leq i, j\leq m$, $0<r<n_i$ and $0\leq a_i<p^{n_i}$, the brackets $[\del_i^{p^{r}}, x^{(a)}\del_j]=x^{(a-p^{r}\epsilon_i)}\del_j$ if $a_i\geq p^r$ and $0$ otherwise.

It remains to describe special, Hamiltonian and contact Lie algebras of Cartan type. Consider the divergence map 
\begin{align*}
\diver: W(m; \underline{n}) &\to \cO(m; \underline{n})\\
\sum_{i=1}^{m}f_i\del_i &\mapsto \sum_{i=1}^{m}\del_i(f_i).
\end{align*} 
It is easy to check that $\diver ([D, E])=D(\diver(E))-E(\diver (D))$ for all $D, E\in W(m; \underline{n})$. As a result, the set 
\begin{equation}\label{Sdiv0}
S(m; \underline{n}):=\big\{D \in W(m; \underline{n})\,|\, \diver(D)=0\big\}
\end{equation}
is a Lie subalgebra of $W(m; \underline{n})$ \cite[Lemma 3.1, Sec.~4.3, Chap.~4]{S88}. It is not simple. But its derived subalgebra $S(m; \underline{n})^{(1)}$ is simple. We refer to the Lie algebra $S(m; \underline{n})^{(1)}$ as the simple \textit{special} Lie algebra of Cartan type. More generally, $S(m; \underline{n})$ or $S(m; \underline{n})^{(1)}$ is referred to as a special Cartan type Lie algebra. 

Let us describe the structure of $S(m; \underline{n})^{(1)}$ in more detail. Define 
\begin{align*}
D_{i, j}: \cO(m; \underline{n}) &\to W(m; \underline{n})\\
f&\mapsto \del_j(f)\del_i-\del_i(f)\del_j
\end{align*}

\begin{thm}\cite[Lemma 3.2, Proposition 3.3, Theorems 3.5 and 3.7, Sec.~4.3, Chap.~4]{S88}
Suppose $m\geq 3$.
\begin{enumerate}[\upshape(i)]
\item $D_{i, j}$ is a linear map of degree $-2$ satisfying $D_{i,i}=0$ and $D_{i,j}=-D_{j, i}$ for all $1\leq i, j\leq m$.
\item $D_{i, j}(\cO(m; \underline{n})) \subset S(m; \underline{n})$ for all $1\leq i, j\leq m$.
\item $S(m; \underline{n})^{(1)}$ is the subalgebra of $S(m; \underline{n})$ generated by 
\[
\text{\big\{$D_{i, j}(x^{(a)})\,|\, 0\leq a_l < p^{n_l}$ for $1\leq l \leq m$ and $1\leq i<j\leq m$\big\}.}
\]
\item $S(m; \underline{n})^{(1)}$ is a simple Lie algebra of dimension $(m-1)(p^{\sum_{i=1}^{m}n_i}-1)$.
\item $S(m; \underline{n})^{(1)}$ is a graded subalgebra of $W(m; \underline{n})$, i.e.
\[
S(m; \underline{n})^{(1)}=\bigoplus_{i=-1}^{s_1}(S(m; \underline{n})^{(1)})_i,
\]
where $s_1=(\sum_{i=1}^{m}p^{n_i})-m-2$ and $(S(m; \underline{n})^{(1)})_i= S(m; \underline{n})^{(1)}\cap W(m; \underline{n})_{i}$.
\item $S(m; \underline{n})^{(1)}$ is restricted if and only if $\underline{n}=(1, \dots, 1)$, and in that case $S(m; \underline{n})^{(1)}$ is a $p$-subalgebra of $W(m; \underline{n})$ with restricted gradation. 
\end{enumerate}
\end{thm}

Alternatively, we can define special Lie algebras of Cartan type using differential forms on $\cO(m; \underline{n})$; see \cite[Sec.~3.2]{PS2006} and \cite[Sec.~4.2]{S04}. Set 
\begin{align*}
\Omega^{0}(m; \underline{n}):= \cO (m; \underline{n}), \quad \Omega^{1}(m; \underline{n}):= \Hom_{\cO(m; \underline{n})}(W (m; \underline{n}), \cO (m; \underline{n})).
\end{align*}
Then $\Omega^{1}(m; \underline{n})$ admits an $\cO(m; \underline{n})$-module structure via 
\[
\text{$(f\alpha)(D):=f\alpha(D)$ for all $f \in \cO(m; \underline{n}), \alpha \in \Omega^{1}(m; \underline{n}), D \in W (m; \underline{n})$},
\]
and a $W(m; \underline{n})$-module structure via 
\[
\text{$(D\alpha)(E)=D(\alpha(E))-\alpha([D, E])$ for all $D, E \in W(m; \underline{n}), \alpha \in \Omega^{1}(m; \underline{n})$}.
\]
Since $W(m; \underline{n})$ is a free $\cO(m; \underline{n})$-module with basis $\del_1, \dots, \del_m$, every $\alpha\in \Omega^{1}(m; \underline{n})$ is determined by its effects on $\del_1, \dots, \del_m$. It is easy to check that $\alpha=\sum_{i=1}^{m} \alpha(\del_i)d x_i$. This implies that $\Omega^{1}(m; \underline{n})$ is a free $\cO(m; \underline{n})$-module with basis $dx_1, \dots, dx_m$. 

Define $d: \Omega^{0}(m; \underline{n}) \to \Omega^{1}(m; \underline{n})$ by $df(D)=D(f)$ for all $f\in \cO(m; \underline{n}), \\D\in W(m; \underline{n})$. Then $d$ is a homomorphism of $W(m; \underline{n})$-modules. Set 
\[
\Omega^{r}(m; \underline{n}):=\bigwedge^{r} \Omega^{1}(m; \underline{n}),
\]
the $r$-fold exterior power algebra over $\cO(m; \underline{n})$. It is a free $\cO(m; \underline{n})$-module with basis $\{dx_{i_1}\wedge \dots\wedge dx_{i_r}\,|\, 1\leq i_1<\dots <i_r\leq m\}$. Let 
\[
\Omega(m; \underline{n}):=\bigoplus_{1\leq r\leq m} \Omega^{r}(m; \underline{n}).
\]
Then elements of $\Omega(m; \underline{n})$ are called \textit{differential forms} on $\cO(m; \underline{n})$. We can extend the above linear operator $d$ to $\Omega(m; \underline{n})$ by setting 
\[
d(\alpha_1 \wedge \alpha_2):= d(\alpha_1)\wedge\alpha_2 +(-1)^{\deg(\alpha_1)}\alpha_1 \wedge d(\alpha_2)
\] 
for all $\alpha_1, \alpha_2 \in \Omega(m; \underline{n})$. Then $d$ is a linear operator of degree $1$ satisfying 
\[
d^2(\alpha)=0,\, D(d\alpha)=dD(\alpha),\, d(f\alpha)=(df)\wedge\alpha+fd(\alpha), \,D(df)=dD(f)
\]
for all $f\in\cO(m; \underline{n}), D \in W(m; \underline{n}), \alpha\in \Omega(m; \underline{n})$. Note that $D(f\alpha)=(Df)\alpha+fD(\alpha)$ for every $D \in W(m; \underline{n})$. Hence $D$ extends to a derivation of $\Omega(m; \underline{n})$.

Recall the volume form
\[
\omega_S:=dx_1 \wedge\dots \wedge dx_m, \quad m\geq 3.
\]
Then 
\begin{align*}
\big\{D\in W(m; \underline{n})\,|\, D(\omega_S)=0\big\}
\end{align*}
coincides with $S(m; \underline{n})$ \eqref{Sdiv0}; see \cite[p. 161]{S88}. The simple special Lie algebra of Cartan type is the derived subalgebra of $S(m; \underline{n})$.

Suppose now $\chari k=p>2$ and $m=2r\geq 2$. The Hamiltonian form
\begin{align*}
\omega_H:=\sum_{i=1}^{r} dx_i\wedge dx_{i+r}, \quad m=2r\geq 2
\end{align*}
gives rise to a Lie subalgebra 
\begin{equation}\label{Hform}
H(2r; \underline{n}):=\big\{D \in W(2r; \underline{n})\,|\, D(\omega_H)=0\big\}
\end{equation}
of $W(2r; \underline{n})$. The second derived subalgebra $H(2r; \underline{n})^{(2)}$ is simple. We refer to $H(2r; \underline{n})^{(2)}$ as the simple \textit{Hamiltonian} Lie algebra of Cartan type. More generally, $H(2r; \underline{n})$ or $H(2r; \underline{n})^{(2)}$ is a Hamiltonian Cartan type Lie algebra.

Alternatively, we can define $H(2r; \underline{n})$ using a linear map. Let us first introduce some notations. Set
\begin{align}\label{sigmaj}
\sigma(j)&:=
\begin{cases}
    1   &\quad \quad\text{if $1\leq j\leq r$},\\
    -1  & \quad \quad\text{if $r<j\leq 2r$},\\
  \end{cases}\\
j'&:=
\begin{cases}\label{j'}
    j+r  & \quad\, \text{if $1\leq j\leq r$},\\
    j-r  & \quad \,\text{if $r<j\leq 2r$}.
  \end{cases}
\end{align}
Consider the set
\[
\bigg\{D=\sum_{i=1}^{2r}f_i\del_i\in W(2r; \underline{n})\,|\, \sigma(i)\del_{j'}(f_{i})=\sigma(j)\del_{i'}(f_{j}), 1\leq i, j\leq 2r\bigg\}.
\]
One can check that this is equivalent to $H(2r;\underline{n})$ \eqref{Hform}. To describe $H(2r; \underline{n})^{(2)}$, we define
\begin{align*}
D_{H}: \cO(2r; \underline{n})&\to W(2r; \underline{n})\\
f&\mapsto \sum_{i=1}^{2r}\sigma(i)\del_i(f)\del_{i'}.
\end{align*}
Denote the image of $D_{H}$ by $\tilde{H}(2r; \underline{n})$. Note that $\tilde{H}(2r; \underline{n})$ is a proper subset of $H(2r; \underline{n})$. Indeed derivations $x_{j}^{(p^{n_j}-1)}\del_{j'}$ for $1\leq j\leq 2r$ lie in $H(2r; \underline{n})$, but do not lie in $\tilde{H}(2r; \underline{n})$; see \cite[p. 163]{S88}.

\begin{thm}\cite[Lemma 4.1, Proposition 4.4 and Theorem 4.5, Sec.~4.4, Chap.~4]{S88}\thlabel{Hthm}
\begin{enumerate}[\upshape(i)]
\item $D_{H}$ is a linear map of degree $-2$ with $\Ker D_{H}=k$.
\item $[D_{H}(f), D_{H}(g)]=D_{H}(D_{H}(f)(g))$ for all $f, g\in \cO(2r; \underline{n})$.
\item $H(2r; \underline{n})^{(1)}\subseteq \tilde{H}(2r; \underline{n})$.
\item $H(2r; \underline{n})^{(2)}$ is a simple Lie algebra with basis
\[
\big\{D_{H}(x^{(a)})\,|\, (0, \dots, 0)<a <(p^{n_1}-1, \dots, p^{n_{2r}}-1)\big\}.
\]
Hence $\dim H(2r; \underline{n})^{(2)}= p^{\sum_{i=1}^{2r}n_i}-2$.
\item $H(2r; \underline{n})^{(2)}$ is a graded subalgebra of $W(2r; \underline{n})$, i.e.
\begin{align*}
H(2r; \underline{n})^{(2)}=\bigoplus_{i=-1}^{s_{2}} (H(2r; \underline{n})^{(2)})_{i},
\end{align*}
where $s_2=(\sum_{i=1}^{2r}p^{n_i})-2r-3$ and $(H(2r; \underline{n})^{(2)})_{i}=H(2r; \underline{n})^{(2)} \cap W(2r; \underline{n})_{i}$.
\item $H(2r; \underline{n})^{(2)}$ is restricted if and only if $\underline{n}=(1, \dots, 1)$, and in that case $H(2r; \underline{n})^{(2)}$ is a $p$-subalgebra of $W(2r; \underline{n})$ with restricted gradation.
\end{enumerate}
\end{thm}

\begin{rmk}\cite[p. 168]{S88}\thlabel{Pbracket}
Note that $D_{H}$ defines a Lie bracket on $\cO(2r; \underline{n})$. For any $f, g \in \cO(2r; \underline{n})$, define 
\begin{align*}
\{f, g\}:=\sum_{i=1}^{2r}\sigma(i)\del_i(f)\del_{i'}(g)=D_{H}(f)(g).
\end{align*}
It follows from \thref{Hthm} that $(\cO(2r; \underline{n}), \{,\})$ is a Lie algebra with centre $k$, and $(\cO(2r; \underline{n})/k)^{(1)}\cong H(2r; \underline{n})^{(2)}$; see \cite[p. 54]{BGP05}. The Lie bracket $\{,\}$ is referred to as the \textnormal{Poisson bracket}.
\end{rmk}

Suppose $\chari k=p>2$ and $m=2r+1\geq 3$. Consider the contact form
\begin{align*}
\omega_{K}:=dx_m+\sum_{i=1}^{r}(x_i dx_{i+r}-x_{i+r}dx_{i}), \quad m=2r+1 \geq 3.
\end{align*}
Set 
\begin{equation}\label{Kform}
\big\{D \in W(2r+1; \underline{n})\,|\, D(\omega_{K})\in \cO(2r+1; \underline{n})\omega_{K}\big\}.
\end{equation}
This gives a Lie subalgebra of  $W(2r+1; \underline{n})$, denoted $K(2r+1; \underline{n})$, called the \textit{contact} Cartan type Lie algebra. The derived subalgebra $K(2r+1; \underline{n})^{(1)}$ is simple.

Similarly, we can describe the structure of $K(2r+1; \underline{n})$ using a linear map. Let $\sigma(j)$ and $j'$ be as in \eqref{sigmaj} and \eqref{j'}. Define 
\begin{align*}
D_{K}: \cO(2r+1; \underline{n}) &\to W(2r+1; \underline{n})\\
f&\mapsto \sum_{j=1}^{2r}\big(\sigma(j)\del_{j}(f)+x_{j'}\del_{2r+1}(f)\big)\del_{j'}\\
&\quad \quad +\big(2f-\sum_{j=1}^{2r}x_j\del_j(f)\big)\del_{2r+1}.
\end{align*}
Then the image $D_{K}(\cO(2r+1; \underline{n}))$ gives a Lie subalgebra of $W(2r+1; \underline{n})$ which coincides with $K(2r+1; \underline{n})$ \eqref{Kform}; see \cite[p. 189]{S04}.

\begin{thm}\citetext{\citealp[Theorem 2.6.2]{BGP05}; \citealp[Proposition 5.3, Theorems 5.5 and 5.6, Sec.~4.5, Chap.~4]{S88}}
\begin{enumerate}[\upshape(i)]
\item $D_{K}$ is an injective linear map of degree $-2$.
\item $K(2r+1; \underline{n})$ is graded, i.e. 
\begin{align*}
K(2r+1; \underline{n})=\bigoplus_{i=-2}^{s_3} K(2r+1; \underline{n})_i,
\end{align*}
where $s_3=(\sum_{i=1}^{2r}p^{n_{i}})+2p^{n_{2r+1}}-2r-3$ and 
\[
K(2r+1; \underline{n})_i=\spn\bigg\{D_{K}(x^{(a)})\,|\, \sum_{i=1}^{2r+1}a_i+a_{2r+1}-2=i\bigg\}.
\]
Note that this grading has depth $2$.
\item $K(2r+1; \underline{n})^{(1)}$ is simple, and
\begin{align*}
K(2r+1; \underline{n})^{(1)}=&
\begin{cases}
    K(2r+1; \underline{n})   & \text{if $2r+4\not \equiv 0\,(\modd p)$},\\
    \spn\big\{D_{K}(x^{(a)})\,|\, 0\leq a< \tau(\underline{n})\big\}  & \text{if $2r+4\equiv 0\,(\modd p)$},
  \end{cases}\\
\text{where} \,\tau(\underline{n}):=(p^{n_1}-1, &\dots, p^{n_{2r+1}}-1). \,\text{Then}\\
\dim K(2r+1; \underline{n})^{(1)}=&
\begin{cases}
  p^{\sum_{i=1}^{2r+1} n_i}    &\quad \text{if $2r+4\not \equiv 0\,(\modd p)$},\\
  p^{\sum_{i=1}^{2r+1} n_i}-1    & \quad \text{if $2r+4\equiv 0\,(\modd p)$}.
  \end{cases}
\end{align*}
\item $K(2r+1; \underline{n})^{(1)}$ is restricted if and only if $\underline{n}=(1, \dots, 1)$, and in that case \\$K(2r+1; \underline{n})^{(1)}$ is a $p$-subalgebra of $W(2r+1; \underline{n})$ with restricted gradation. 
\end{enumerate}
\end{thm}

\begin{rmk}\citetext{\citealp[p.~191(4.2.17)]{S04}; \citealp[p.~172]{S88}}\label{contactb}
As in the case of Hamiltonian Lie algebras, the linear map $D_{K}$ also defines a Lie bracket on $\cO(2r+1; \underline{n})$. For any $f, g \in\cO(2r+1; \underline{n})$, define
\begin{align*}
\langle f, g\rangle:=D_{K}(f)(g)-2g\del_{2r+1}(f).
\end{align*} 
By \cite[Proposition 5.2, Sec.~4.5, Chap.~4]{S88}, $D_{K}(\langle f, g\rangle)=[D_{K}(f), D_{K}(g)]$. Moreover, $D_{K}$ is injective. Hence $(\cO(2r+1; \underline{n}), \langle,\rangle)$ is a Lie algebra over $k$, and $\cO(2r+1; \underline{n})\cong K(2r+1; \underline{n})$. The Lie bracket $\langle,\rangle$ is referred to as the \textnormal{contact bracket}. 
\end{rmk}
These are the four families of Lie algebras of Cartan type. We finish this section by emphasizing that 
\begin{thm}\cite[Corollary 7.2.3]{S04}
The simple restricted Lie algebras of Cartan type are $W(m; \underline{1}), m\geq 1, S(m; \underline{1})^{(1)}, m\geq 3, H(2r; \underline{1})^{(2)}, r\geq 1$ and $K(2r+1; \underline{1})^{(1)}, r\geq 1$.
\end{thm}
\section{Nilpotent varieties}\label{nilpotentvarietyresultsWn}
In this section we review Premet's results on nilpotent variety of any finite dimensional restricted Lie algebra over $k$. Then we introduce Premet's conjecture and discuss what have been done so far. 

Let $(\fg, [p])$ be a finite dimensional restricted Lie algebra over $k$. By Jacobson's formula (see \thref{defres}(3)), we see that the $[p]$-th power map is a morphism given by homogeneous polynomial functions on $\fg$ of degree $p$. Recall the nilpotent variety $\cN(\fg)$ which is the set of all $x\in\fg$ such that $x^{[p]^{N}}=0$ for $N\gg 0$. It is well known that $\cN(\fg)$ is a Zariski closed, conical subset of $\fg$. Let $\fg_{ss}$ denote the set of all semisimple elements of $\fg$. By \thref{JCthm}, it is straightforward to see that $\fg^{[p]^{N}}=\fg_{ss}$ for $N\gg0$. Define $e=e(\fg)$ to be the smallest nonnegative integer such that $W^{[p]^{e}}\subseteq\fg_{ss}$ for some nonempty Zariski open subset $W$ of $\fg$. By \cite[Theorem 2]{P87}, we know that $e=0$ if and only if $\fg$ possesses a toral Cartan subalgebra. Let $s=\MT(\fg)$ denote the maximal dimension of tori in $\fg$. In Sec.~\ref{reLiealgebras}, we already observed that $s=0$ if and only if $\fg$ coincides with $\cN(\fg)$. We present the theorem on $\cN(\fg)$ which was proved by A.~Premet.

\begin{thm}\citetext{\citealp[Theorem 2 and Corollary 2]{P90}; \citealp[Theorem 4.2]{P03}}\thlabel{nvarietythm}
There exist homogeneous polynomials $\psi_0, \dots, \psi_{s-1} \in k[\fg]$ with $\deg\psi_i=p^{e+s}-p^{e+i}$ such that for any $x \in \fg$, $x^{[p]^{e+s}}=\sum_{i=0}^{s-1}\psi_i(x)x^{[p]^{e+i}}$. Moreover, the following are true:
\begin{enumerate}[\upshape(i)]
\item $\psi_i(x^{[p]})=\psi_i(x)^{p}$ for all $x\in\fg$ and $i\leq s-1$.
\item $\cN(\fg)=\cZ (\psi_0, \dots, \psi_{s-1})$, the set of all common zeros of $\psi_0, \dots, \psi_{s-1}$ in $\fg$.
\item All irreducible components of $\cN(\fg)$ have dimension $\dim\fg-s$, i.e. $\cN(\fg)$ is equidimensional.
\item For any $x\in \cN(\fg)$, $x^{[p]^{e+s}}=0$.
\end{enumerate}
\end{thm} 

The above theorem gives useful information on $\cN(\fg)$ and plays an important role in Chapters 3 and 4. It is well known that the nilpotent variety of any finite dimensional algebraic Lie algebra over $\C$ is irreducible. Then Premet conjectured that

\begin{conj}\cite[p. 563]{P90}\thlabel{pconj}
For any finite dimensional restricted Lie algebra $\fg$ over $k$ the variety $\cN(\fg)$ is irreducible.
\end{conj}

This conjecture is still open, but it is known that if $\fg$ is the Lie algebra of a connected algebraic group $G'$, and $\fn$ is the set of nilpotent elements in a Borel subalgebra of $\fg$, then $\cN(\fg) =\{g.n \,|\, g\in G', n \in \fn\}$. Since $G'$ is connected and $\fn$ is irreducible, the variety $\cN(\fg)$ is irreducible; see \cite[p.~64]{J04}. Moreover, this conjecture holds for the Jacobson-Witt algebras $W(n; \underline{1})$ \cite[Lemma 6]{P91}, for the Special Lie algebras $S(n; \underline{1})$ and $S(n; \underline{1})^{(1)}$ \cite[Theorem A]{WCL13}, for the Poisson Lie algebras $(\cO(2n; \underline{1}), \{,\})$ \cite[Theorem 6.4]{S2002} and for the Hamiltonian Lie algebras $H(2n; \underline{1})^{(2)}$ \cite[Theorem A]{W14}. There are no known counterexamples.

A good understanding of the proof in the case $W(n; \underline{1})$ is important. This is because some of the results will be used in Sec.~\ref{JacobWittinfo}, Chapter 3. Moreover, a similar idea works for the minimal $p$-envelope $W(1; n)_p$ of the Zassenhaus algebra $W(1; n)$; see Chapter 4. So let us first introduce some notations and state a few preliminary results. Then we show that $\cN(W(n; \underline{1}))$ is irreducible.

Let $k$ be an algebraically closed field of characteristic $p>2$. Let $\cO(n; \underline{1})=k[X_1, \dots, X_n]/(X_1^{p}, \dots, X_{n}^{p})$ denote the truncated polynomial ring in $n$ variables. Note that $\cO(n; \underline{1})$ is a local ring with its unique maximal ideal denoted $\fm$. Let $W(n; \underline{1})$ denote the derivation algebra of $\cO(n; \underline{1})$. This is a simple restricted Lie algebra. Moreover, $W(n; \underline{1})$ obtains an $\cO(n; \underline{1})$-module structure via $(fD)(g)=fD(g)$ for all $f, g\in \cO(n; \underline{1})$ and $D \in W(n; \underline{1})$. Since each $D\in W(n; \underline{1})$ is uniquely determined by its effects on $x_1, \dots, x_n$, it is easy to see that $W(n; \underline{1})$ is a free $\cO(n; \underline{1})$-module of rank $n$ generated by the partial derivatives $\del_1, \dots, \del_n$ such that $\del_i(x_j)=\delta_{ij}$ for all $1\leq i, j\leq n$. Hence $\dim W(n; \underline{1})=np^n$. Note that there is a standard filtration $\{W(n; \underline{1})_{(i)}\}_{-1\leq i \leq m(p-1)-1}$ defined on $W(n; \underline{1})$; see \eqref{wmnfiltration}. In particular, the subalgebra 
\[
W(n; \underline{1})_{(0)}:=\bigg\{\sum_{i=1}^{n}f_i\del_i\,|\, \text{$f_i\in \fm$ for all $i$}\bigg\}
\]
is referred to as the \textit{standard maximal subalgebra} of $W(n; \underline{1})$. Note that for any \\$D_1\in W(n; \underline{1})_{(0)}$, the following holds: $D_1(\cO(n; \underline{1}))\subseteq \fm$.

Let $G$ denote the automorphism group of $\cO(n; \underline{1})$. Each $\sigma\in G$ is uniquely determined by its effects on $x_1, \dots, x_n$. An assignment $\sigma(x_i)=f_i $ extends to an automorphism of $\cO(n; \underline{1})$ if and only if $f_i\in \fm$, and the Jacobian $\Jac(f_1, \dots, f_n):=\big|\big(\frac{\del f_i}{\del x_j}\big)_{1\leq i, j\leq n}\big|\notin \fm$. It follows that $G$ is a connected algebraic group of dimension $\dim W(n; \underline{1})-n$. It is well known that any automorphism of $W(n; \underline{1})$ is induced by an automorphism of $\cO(n; \underline{1})$ via the rule $D^{\sigma}=\sigma \circ D\circ \sigma^{-1}$ for all $\sigma \in G$ and $D \in W(n; \underline{1})$. So we can identify $G$ with the automorphism group of $W(n; \underline{1})$. Note that for any $f\in \cO(n; \underline{1})$, $D\in W(n; \underline{1})$ and $\sigma \in G$,
\begin{equation}\label{autconj1}
(fD)^{\sigma}=f^{\sigma}D^{\sigma}, 
\end{equation}
where $f^{\sigma}=f(\sigma(x_1), \dots, \sigma(x_n))$; see \cite[p.~154]{P91}. It follows from \eqref{autconj1} that if \\$D_1=\sum_{i=1}^{n}g_i\del_i$ is an element of $W(n; \underline{1})$, then
\begin{equation}\label{autoconjugationrule2}
D_1^{\sigma}=\sum_{i,\, j=1}^{n}g_i^{\sigma}\bigg(\del_i\big(\sigma^{-1}(x_j)\big)\bigg)^{\sigma}\del_j;
\end{equation}
see \cite[Sec.~3]{D70}. Note also that $G$ respects the standard filtration of $W(n; \underline{1})$. 

It is known that $\MT(W(n; \underline{1}))=n$ and $e(W(n; \underline{1}))=0$. By \thref{nvarietythm}, we know that $\dim \cN(W(n; \underline{1}))=\dim W(n; \underline{1})-n$ and the nilpotency index of any element in $\cN(W(n; \underline{1}))$ is at most $p^n$. 

We want to show that $\cN(W(n; \underline{1}))$ is irreducible. Let us start with the following result:

\begin{lem}\cite[Lemma 3]{P91}\thlabel{WnDresultslem3}
Let $\sD=\del_1+x_1^{p-1}\del_2+\dots+x_1^{p-1}\cdots x_{n-1}^{p-1}\del_n$. Then 
\begin{enumerate}[\upshape(i)]
\item 
\[
\sD^{p^{l}}=(-1)^{l}(\del_{l+1}+x_{l+1}^{p-1}\del_{l+2}+\dots+x_{l+1}^{p-1}\cdots x_{n-1}^{p-1}\del_n)
\] for all $0\leq l\leq n-1$ and $\sD^{p^{n}}=0$.
\item $\sD, \sD^{p}, \dots, \sD^{p^{n-1}}$ forms a basis of the $\cO(n; \underline{1})$-module $W(n; \underline{1})$.
\item $\sD^{p^{n}-1}(x_1^{p-1}\cdots x_n^{p-1})=(-1)^{n}$. Hence the matrix of the endomorphism \\$\sD : \cO(n; \underline{1})\to \cO(n; \underline{1})$ is similar to a Jordan block of size $p^{n}$ with zeros on the main diagonal.
\item The stabilizer of $\sD$ in $G$ is trivial.
\end{enumerate}
\end{lem}
\begin{myproof}[Sketch of proof]
(i) We prove (i) by induction on $l$. For $l=0$, the result is clear. For $l=1$, let $\sD'=x_1^{p-1}\del_2+\dots+x_1^{p-1}\cdots x_{n-1}^{p-1}\del_n$. Recall Jacobson's formula, 
\[
(a+b)^{p}=a^{p}+b^{p}+\sum_{i=1}^{p-1}s_{i}(a, b),
\]
where the terms $s_{i}(a, b)\in W(n; \underline{1})$ are such that
\[
(\ad(ta+b))^{p-1}(a)=\sum_{i=1}^{p-1}is_{i}(a, b)t^{i-1}.
\] 
Set $a=\sD'$ and $b=\del_1$. Then for any $s\leq p-2$, 
\begin{align*}
[a, (\ad b)^{s}(a)]=(-1)^ss![x_1^{p-1}\del_2+&\dots+x_1^{p-1}\cdots x_{n-1}^{p-1}\del_n,\\
&x_1^{p-1-s}\del_2+\dots+x_1^{p-1-s}x_2^{p-1}\cdots x_{n-1}^{p-1}\del_n]=0.
\end{align*}
Hence 
\begin{align*}
\sD^{p}&=\del_1^{p}+(\sD')^{p}+\sum_{i=1}^{p-1}s_i(\sD', \del_1)=s_1(\sD', \del_1)=(\ad \del_1)^{p-1}(\sD')\\
&=-(\del_{2}+x_{2}^{p-1}\del_{3}+\dots+x_{2}^{p-1}\cdots x_{n-1}^{p-1}\del_n).
\end{align*}
So the result holds for $l=1$. Continuing in this way and by induction, one can show that 
\[
\sD^{p^{l}}=(-1)^{l}(\del_{l+1}+x_{l+1}^{p-1}\del_{l+2}+\dots+x_{l+1}^{p-1}\cdots x_{n-1}^{p-1}\del_n)\]
for all $0\leq l\leq n-1$. In particular, $\sD^{p^{n-1}}=(-1)^{n-1}\del_{n}$. Then $\sD^{p^{n}}=0$. This proves (i).

(ii) Let $M: W(n; \underline{1})\to W(n; \underline{1})$ be the endomorphism such that $M(\del_i)=\sD^{p^{i-1}}$ for $1\leq i\leq n$. By (i), it is easy to check that the matrix of $M$ with respect to $\del_1, \dots, \del_n$ is  an upper triangular matrix with $1$ or $-1$ on the diagonal. Hence $M$ is invertible. Therefore, $\sD, \sD^{p}, \dots, \sD^{p^{n-1}}$ forms a basis of the $\cO(n; \underline{1})$-module $W(n; \underline{1})$. This proves (ii).

(iii) By (i), it is easy to check that for $0\leq l\leq n-2$, 
\begin{align*}
&(\sD^{p^{l}})^{p-1}(x_{l+1}^{p-1}\cdots x_{n}^{p-1})\\
=&\big((-1)^{l}(\del_{l+1}+x_{l+1}^{p-1}\del_{l+2}+\dots+x_{l+1}^{p-1}\cdots x_{n-1}^{p-1}\del_n)\big)^{p-1}(x_{l+1}^{p-1}\cdots x_{n}^{p-1})\\
=&(-1)^{l(p-1)}(-1)x_{l+2}^{p-1}\cdots x_{n}^{p-1}\\
=&-x_{l+2}^{p-1}\cdots x_{n}^{p-1}.
\end{align*}
For $l=n-1$, 
\begin{align*}
(\sD^{p^{n-1}})^{p-1}(x_{n}^{p-1})&=\big((-1)^{n-1}\del_n\big)^{p-1}(x_{n}^{p-1})=(-1)^{(n-1)(p-1)}(-1)=-1.
\end{align*}
Then 
\begin{align*}
\sD^{p^{n}-1}(x_{1}^{p-1}\cdots x_{n}^{p-1})=&\sD^{p^{n}-p}\sD^{p-1}(x_{1}^{p-1}\cdots x_{n}^{p-1})=-\sD^{p^{n}-p}(x_{2}^{p-1}\cdots x_{n}^{p-1})\\
=&-(\sD^{p})^{p^{n-1}-1}(x_{2}^{p-1}\cdots x_{n}^{p-1})\\
=& -(\sD^{p})^{p^{n-1}-p}(\sD^{p})^{p-1}(x_{2}^{p-1}\cdots x_{n}^{p-1})\\
=&(-1)^{2}(\sD^{p})^{p^{n-1}-p}(x_{3}^{p-1}\cdots x_{n}^{p-1})\\
=&(-1)^{2}(\sD^{p^{2}})^{p^{n-2}-1}(x_{3}^{p-1}\cdots x_{n}^{p-1})\\
=&\dots=(-1)^{n-1}(\sD^{p^{n-1}})^{p-1}(x_n^{p-1})=(-1)^{n-1}(-1)=(-1)^{n}.
\end{align*}
Hence $\sD^{p^{n}-1}\neq 0$. Since $\sD^{p^{n}}=0$, this implies that the matrix of $\sD$ is similar to a Jordan block of size $p^{n}$ with zeros on the diagonal. This proves (iii).

(iv) Let $C$ be the stabilizer of $\sD$ in $G$ and $c\in C$. By (iii), we know that $\cO(n; \underline{1})$ is an indecomposable $k\sD$-module, i.e.
\begin{equation}\label{ON1ksdindecomposoble}
\cO(n;\underline{1})=\bigoplus_{i=0}^{p^{n}-1}k\sD^{i}(x_{1}^{p-1}\cdots x_{n}^{p-1}).
\end{equation}
Since $G$ acts on the straight line $kx_{1}^{p-1}\cdots x_{n}^{p-1}$, we have that 
\[
c(x_{1}^{p-1}\cdots x_{n}^{p-1})=\lambda_c x_{1}^{p-1}\cdots x_{n}^{p-1}
\]
for some $\lambda_c \in k^{*}$. Moreover, $c$ commutes with $\sD$. Applying \eqref{ON1ksdindecomposoble}, we get
\begin{align*}
c\big(\sD^{i}(x_{1}^{p-1}\cdots x_{n}^{p-1})\big)&=\sD^{i} \big(c(x_{1}^{p-1}\cdots x_{n}^{p-1})\big)=\lambda_c\sD^{i}(x_{1}^{p-1}\cdots x_{n}^{p-1}).
\end{align*}
Hence $c=\lambda_c\Id$. Since $c(1)=1$, this implies that $\lambda_c=1$. So $c=\Id$. This proves (iv).
\end{myproof}

Set $O:=G.\sD$. Since $G$ is connected, it follows from the last result that the Zariski closure of $O$ is an irreducible variety of dimension $\dim G=\dim \cN(W(n; \underline{1}))$. Since $\sD$ is nilpotent, it follows that $\overline{O}$ is an irreducible component of $\cN(W(n; \underline{1}))$. We want to describe $O$ explicitly. For that, we need the following result:

\begin{lem}\cite[Lemma 6]{D70}\thlabel{D70Lem6part}
Let $z$ be an element of $W(n; \underline{1})$ such that $z\notin W(n; \underline{1})_{(0)}$. Then $z$ is conjugate under $G$ to $z_1=\del_1+x_1^{p-1}\sum_{i=1}^{n}\varphi_i\del_i$, where \\$\varphi_i\in k[X_2,\dots, X_{n}]/(X_2^{p}, \dots, X_{n}^{p})$ for all $i$. Moreover, $z_1^{p}=-(1+x_1^{p-1}\varphi_1)\sum_{i=1}^{n}\varphi_i\del_i$.
\end{lem}
\begin{myproof}[Sketch of proof]
In \cite[Sec.~3]{D70}, S.~P.~Demushkin used the convention that for any $D\in W(n; \underline{1})$ and $\Phi \in G$, $D^{\Phi}=\Phi^{-1}\circ D\circ\Phi$. Our convention is that $D^{\Phi}=\Phi\circ D\circ\Phi^{-1}$. So we need to slightly modify his proof.

Let $z=\sum_{i=1}^{n}f_i\del_i$ be an element of $W(n; \underline{1})$ such that $z\notin W(n;\underline{1})_{(0)}$. Then $f_\mu(0, \dots, 0)\neq 0$ for some $1\leq \mu\leq n$. If $\mu\neq 1$, then we show that $z$ is conjugate under $G$ to $\sum_{i=1}^{n}f_i\del_i$ with $f_1(0, \dots, 0)\neq 0$. Let $\Phi\in G$ be such that $\Phi(x_1)=x_\mu, \Phi(x_\mu)=x_1$ and $\Phi(x_j)=x_j$ for $j\neq 1, \mu$. Then $\Phi^{-1}(x_1)=x_\mu, \Phi^{-1}(x_\mu)=x_1$ and $\Phi^{-1}(x_j)=x_j$ for $j\neq 1, \mu$. By \eqref{autoconjugationrule2}, 
\begin{align*}
z^{\Phi}=\sum_{i,\, j=1}^{n}f_i^{\Phi}\bigg(\del_i\big(\Phi^{-1}(x_j)\big)\bigg)^{\Phi}\del_j=f_{\mu}^{\Phi}\del_1+f_1^{\Phi}\del_\mu+\sum_{i\neq 1,\,\mu}f_{i}^{\Phi}\del_i.
\end{align*}
Since $f_\mu^{\Phi}(0, \dots, 0)=f_\mu(0, \dots, 0)\neq 0$, we may assume from the beginning that \\$z=\sum_{i=1}^{n}f_i\del_i$ with $f_1(0, \dots, 0)\neq 0$. 

Let $\Phi_2\in G$ be such that $\Phi_2(x_1)=f_1(0, \dots, 0)x_1$ and $\Phi_2(x_j)=x_j$ for $2\leq j\leq n$. Then $\Phi_2^{-1}(x_1)=f_1^{-1}(0, \dots, 0)x_1$ and $\Phi_2^{-1}(x_j)=x_j$ for $2\leq j\leq n$. Applying $\Phi_2$ to $z$, we get 
\[
z^{\Phi_2}=\sum_{i,\, j=1}^{n}f_i^{\Phi_2}\bigg(\del_i\big(\Phi_2^{-1}(x_j)\big)\bigg)^{\Phi_2}\del_j=f_1^{\Phi_2}f_1^{-1}(0, \dots, 0)\del_1+\sum_{i=2}^{n}f_i^{\Phi_2}\del_i.
\]
Since $f_1^{\Phi_2}(0, \dots, 0)=f_1(0, \dots, 0)$, we may assume from the beginning that \\$z=\sum_{i=1}^{n}f_i\del_i$ with $f_1(0, \dots, 0)=1$. 

Let $\Phi_3\in G$ be such that $\Phi_3(x_1)=x_1$ and $\Phi_3(x_j)=x_j+\alpha_jx_1$, where $\alpha_j=f_j(0\, \dots, 0)$ for $2\leq j\leq n$. Then $\Phi_3^{-1}(x_1)=x_1$ and $\Phi_3^{-1}(x_j)=x_j-\alpha_jx_1$ for $2\leq j\leq n$. Applying $\Phi_3$ to $z$, we get 
\[
z^{\Phi_3}=\sum_{i,\, j=1}^{n}f_i^{\Phi_3}\bigg(\del_i\big(\Phi_3^{-1}(x_j)\big)\bigg)^{\Phi_3}\del_j=f_1^{\Phi_3}\del_1+\sum_{i=2}^{n}(f_i^{\Phi_3}-\alpha_if_1^{\Phi_3})\del_i.
\]
Hence we may assume from the beginning that $z=\sum_{i=1}^{n}f_i\del_i$ with $f_1(0, \dots, 0)=1$ and $f_i(0, \dots, 0)=0$ for all $2\leq i\leq n$. Then we can write
\begin{equation}\label{f1fi}
f_1=1+\sum_{l=1}^{m(p-1)}f_{1, l}\quad \text{and}\quad f_i=\sum_{l=1}^{m(p-1)}f_{i, l}
\end{equation}
for some $f_{1, l},f_{i, l}\in \fm$ with $\deg f_{1, l}=\deg f_{i, l}=l$. Let $\Phi_4\in G$ be such that $\Phi_4^{-1}(x_j)=x_j+g_j$ for $1\leq j\leq n$, where $g_j$ are elements of $\fm$ with the same degree $\nu\geq 2$. Applying $\Phi_4$ to $z$, we get
\begin{equation}\label{f1fiusing}
z^{\Phi_4}=\sum_{i,\, j=1}^{n}f_i^{\Phi_4}\bigg(\del_i\big(\Phi_4^{-1}(x_j)\big)\bigg)^{\Phi_4}\del_j=\sum_{i,\,j=1}^{n}f_i^{\Phi_4}\bigg( \del_i\big(x_j+g_j\big)\bigg)^{\Phi_4}\del_j.
\end{equation}
Substituting \eqref{f1fi} into \eqref{f1fiusing} and expanding out, we get 
\[
z^{\Phi_4}\equiv \sum_{i=1}^{n}\big(f_i+\del_1(g_i)\big)\del_i \quad (\modd W(n; \underline{1})_{(\nu-1)}).
\]
Since $f_i(0, \dots, 0)=\delta_{i1}$, we can rewrite the above as
\[
z^{\Phi_4}\equiv \del_1+\big(f_1-1+\del_1(g_1)\big)\del_1+\sum_{i=2}^{n}\big(f_i+\del_1(g_i)\big)\del_i \quad (\modd W(n; \underline{1})_{(\nu-1)}),
\]
where $f_1-1+\del_1(g_1), f_i+\del_1(g_i)\in \fm$ with degrees strictly less than $\nu$. It follows that $z$ is conjugate under $G$ to 
\[
z_1=\del_1+x_1^{p-1}\sum_{i=1}^{n}\varphi_i\del_i,
\]
where $\varphi_i\in k[X_2,\dots, X_{n}]/(X_2^{p}, \dots, X_{n}^{p})$.

It remains to show that $z_1^{p}=-(1+x_1^{p-1}\varphi_1)\sum_{i=1}^{n}\varphi_i\del_i$. Note that 
\begin{align*}
z_1(x_1)&=1+x_1^{p-1}\varphi_1,\\
z_1^{2}(x_1)&=(p-1)x_1^{p-2}\varphi_1,\\
\dots\\
z_1^{\eta}(x_1)&=(p-1)\cdots (p-\eta+1)x_1^{p-\eta}\varphi_1 \,\text{for $2\leq \eta\leq p-1$, and}\\
z_1^{p}(x_1)&=(p-1)!(1+x_1^{p-1}\varphi_1)\varphi_1=-(1+x_1^{p-1}\varphi_1)\varphi_1.
\end{align*}
Similarly, one can show that $z_1^{p}(x_i)=-(1+x_1^{p-1}\varphi_1)\varphi_i$ for $2\leq i\leq n$. Hence 
\[
z_1^{p}=-(1+x_1^{p-1}\varphi_1)\sum_{i=1}^{n}\varphi_i\del_i.
\] This completes the sketch of proof.
\end{myproof}

Recall that $\sD=\del_1+x_1^{p-1}\del_2+\dots+x_1^{p-1}\cdots x_{n-1}^{p-1}\del_n$. We can now describe $O=G.\sD$ as:
\begin{lem}\cite[Lemma 4]{P91}\thlabel{WnOresultslem4}
\[
O=\big\{D\in \cN(W(n; \underline{1}))\,|\, D^{p^{n-1}} \not \in W(n; \underline{1})_{(0)}\big\}.
\]
\end{lem}
\begin{myproof}[Sketch of proof]
Set $S=\big\{D\in \cN(W(n; \underline{1}))\,|\, D^{p^{n-1}} \not \in W(n; \underline{1})_{(0)}\big\}$. It follows from \thref{WnDresultslem3}(i) that $\sD^{p^{n}}=0$ and $\sD^{p^{n-1}}=(-1)^{n-1}\del_n\notin W(n; \underline{1})_{(0)}$. Hence $O\subseteq S$. To show that $S\subseteq O$, we proceed by induction on $n$. 
Let $z$ be an element of $W(n; \underline{1})$ such that $z\notin W(n; \underline{1})_{(0)}$. By \thref{D70Lem6part}, $z$ is conjugate under $G$ to $\del_1+x_1^{p-1}\sum_{i=1}^{n}\varphi_i\del_i$, where $\varphi_i\in k[X_2,\dots, X_{n}]/(X_2^{p}, \dots, X_{n}^{p})$. Moreover, 
\[
\big(\del_1+x_1^{p-1}\sum_{i=1}^{n}\varphi_i\del_i\big)^{p}=-(1+x_1^{p-1}\varphi_1)\sum_{i=1}^{n}\varphi_i\del_i.
\]
If $z^p=0$, then $(1+x_1^{p-1}\varphi_1)\varphi_i=0$ for all $1\leq i\leq n$. Since $1+x_1^{p-1}\varphi_1$ is invertible in $\cO(n;\underline{1})$, we have that $\varphi_i=0$ for all $1\leq i\leq n$. Hence any $D\in W(n; \underline{1})$ with $D \not \in W(n; \underline{1})_{(0)}$ and $D^{p}=0$ is conjugate under $G$ to $\del_1$. Thus, the result is true for $n=1$. 

Suppose $n>1$. Let $y \in \cN(W(n; \underline{1}))$ be such that $y^{p^{n-1}} \not \in W(n; \underline{1})_{(0)}$. Set $z=y^{p^{n-1}}$. Since $z \not \in W(n; \underline{1})_{(0)}$ and $z^{p}=0$, there exists $\Phi\in G$ such that $\Phi z \Phi^{-1}=(\Phi y \Phi^{-1})^{p^{n-1}}=\del_n$. So we may assume that $z=y^{p^{n-1}}=\del_n$. Since $[y, y^{p^{n-1}}]=0$, we have that $y \in \fc_{W(n; \underline{1})}(\del_n)$. One can check that the centralizer $\fc_{W(n; \underline{1})}(\del_n)$ is isomorphic to the semidirect product of $W(n-1; \underline{1})$ and an abelian ideal $J=\{f_n\del_n\,|\, f_n\in \cO(n-1; \underline{1})\}$. Hence  
\[
y=f_1\del_1+\dots+f_{n-1}\del_{n-1}+f_n\del_n
\]
for some $f_i\in \cO(n-1; \underline{1})$. Set $y_1=f_1\del_1+\dots+f_{n-1}\del_{n-1}$ and $y_2=f_n\del_n$. We show that $y_1$ is conjugate under $G$ to $\sD_0=\del_1+x_1^{p-1}\del_2+\dots+x_1^{p-1}\cdots x_{n-2}^{p-1}\del_{n-1}$. Since $y_1\in \fc_{W(n; \underline{1})}(\del_n)$, it follows from \eqref{LbracketWm} and Jacobson's formula that for any $j\geq 1$, $y^{p^{j}}=y_1^{p^{j}}+y_1^{p^{j}-1}(f_n)\del_n$. In particular,
\[
y^{p^{n-1}}=y_1^{p^{n-1}}+y_1^{p^{n-1}-1}(f_n)\del_n.
\]
Since $y^{p^{n}}=0$, this implies that $y_1\in \cN(W(n-1; \underline{1}))$. Then $y_1^{p^{n-1}}=0$ and $y^{p^{n-1}}=y_1^{p^{n-1}-1}(f_n)\del_n$. By our assumption, $y^{p^{n-1}} \not \in W(n; \underline{1})_{(0)}$. Hence $y_1^{p^{n-1}-1}(f_n)$ is invertible in $\cO(n-1; \underline{1})$. Note that if $D_1\in W(n; \underline{1})_{(0)}$, then $D_1(\cO(n; \underline{1}))\subseteq \fm$. It follows that $y_1^{p^{n-2}}\notin W(n-1; \underline{1})_{(0)}$. Indeed, if $y_1^{p^{n-2}}\in W(n-1; \underline{1})_{(0)}$, then $y_1^{p^{n-1}-1}(f_n)=(y_1)^{p^{n-2}}(f')\in \fm$, a contradiction. So $y_1^{p^{n-2}}\notin W(n-1; \underline{1})_{(0)}$. 

By the induction hypothesis, $y_1$ is conjugate under $\Aut(W(n-1; \underline{1}))$ to $\sD_0=\del_1+x_1^{p-1}\del_2+\dots+x_1^{p-1}\cdots x_{n-2}^{p-1}\del_{n-1}$. Since $W(n-1; \underline{1})$ is a Lie subalgebra of $W(n; \underline{1})$, we may identify $\Aut(W(n-1; \underline{1}))$ with a subgroup of $\Aut(W(n; \underline{1}))$ by letting $\sigma(x_n)=x_n$ for all $\sigma \in \Aut(W(n-1; \underline{1}))$. Hence we may assume that $y=\sD_0+\psi\del_n$ for some $\psi\in \cO(n-1; \underline{1})$ with $\sD_0^{p^{n-1}-1}(\psi)$ invertible in $\cO(n-1; \underline{1})$. Since $\sD_0^{p^{n-1}}=0$, we have that $\Ker(\sD_0^{p^{n-1}-1}|_{\cO(n-1; \underline{1})})=\sD_0(\cO(n-1; \underline{1}))$. By \thref{WnDresultslem3}(iii), we have that 
\[
\text{$\sD_0^{p^{n-1}-1}(\cO(n-1; \underline{1}))=k$ and $\sD_0^{p^{n-1}-1}(x_1^{p-1}\cdots x_{n-1}^{p-1})=(-1)^{n-1}$}.
\] 
Hence there exist $\varphi \in \cO(n-1; \underline{1}) \cap \fm$ and $\alpha \in k^{*}$ such that 
\[
\psi=\sD_0(\varphi)+\alpha x_1^{p-1}\cdots x_{n-1}^{p-1}.
\] 

Next we show that $y=\sD_0+\psi\del_n$ is conjugate under $G$ to $\sD$.
Let $\Phi_1 \in G$ be such that $\Phi_1|_{\cO(n-1; \underline{1})}=\Id$ and $\Phi_1(x_n)=x_n+\varphi$. It is easy to check that $\Phi_1^{-1}(x_i)=x_i$ for $1\leq i\leq n-1$ and $\Phi_1^{-1}(x_n)=x_n-\varphi$. Then for $1\leq i\leq n-1$,
\[
\Phi_1 y\Phi_1^{-1}(x_i)=\Phi_1(\sD_0+\psi\del_n)(x_i)=\sD_0(x_i),
\]
and 
\[
\Phi_1 y\Phi_1^{-1}(x_n)=\Phi_1(\sD_0+\psi\del_n)(x_n-\varphi)=\Phi_1\big(-\sD_0(\varphi)+\psi\big)=\alpha x_1^{p-1}\cdots x_{n-1}^{p-1}.
\]
So $\Phi_1 y\Phi_1^{-1}=\sD_0+\alpha x_1^{p-1}\cdots x_{n-1}^{p-1}\del_n$. Let $\Phi_2\in G$ be such that $\Phi_2|_{\cO(n-1; \underline{1})}=\Id$ and $\Phi_2(x_n)=\alpha x_n$. Then one can show that $\Phi_2 (\Phi_1 y\Phi_1^{-1})\Phi_2^{-1}=\sD$. It follows that $y\in O$. Therefore, we proved by induction that $S \subseteq O$. As a result, $O=S$. This completes the sketch of proof.
\end{myproof}

Now we consider the complement of $O$ in $\cN(W(n; \underline{1}))$ and show that this complement has dimension strictly less than $\dim\cN(W(n; \underline{1}))=\dim W(n; \underline{1})-n$; see \thref{WnN0resultslem5}. For that, we need to introduce the notion of admissible $n$-tuples and define a lexicographic ordering on them; see also \thref{lexordering2.2.2} in Sec.~\ref{nileleminND}.

\begin{defn}\cite[Lemma 5]{P91}\thlabel{admissibletupleorderingdefn}
\begin{enumerate}[\upshape(i)]
\item We say that an $n$-tuple $A=(a_1, \dots, a_n)\in \N_0^{n}$ is \textnormal{admissible} if $0\leq a_i\leq p-1$ for any $1\leq i\leq n$. We write $\boldsymbol{x}^{A}:=x_1^{a_1}\cdots x_n^{a_n}$.
\item We define a \textnormal{lexicographic ordering} $\prec_{\text{lex}}$ on the set of admissible $n$-tuples by extending the ordering $(1, 0, \dots, 0)\prec_{\text{lex}} (0, 1, \dots, 0)\prec_{\text{lex}}\dots\prec_{\text{lex}}(0, 0, \dots, 1)$.
Explicitly, for any two non-equal admissible $n$-tuples $A=(a_1, \dots, a_n)$ and $A'=(a'_1, \dots, a'_n)$, 
\[
\text{$A=(a_1, \dots, a_n)\prec_{\text{lex}} (a'_1, \dots, a'_n)=A'$ if and only if $a_i<a'_i$},
\]
where $i$ is the largest number in $\{1, \dots, n\}$ for which $a_i\neq a'_i$.
\end{enumerate}
\end{defn}

\begin{lem}\cite[Lemma 5]{P91}\thlabel{Dcompatiblewithordering}
The action of $\sD=\del_1+x_1^{p-1}\del_2+\dots+x_1^{p-1}\cdots x_{n-1}^{p-1}\del_n$ on $\cO(n; \underline{1})$ is compatible with the lexicographic ordering $\prec_{\text{lex}}$ defined in \thref{admissibletupleorderingdefn}. Explicitly, for any admissible $n$-tuple $A\neq (0, \dots, 0)$, we have that $\sD(\boldsymbol{x}^{A})=\lambda_{A}\boldsymbol{x}^{A'}$, where $\lambda_{A}\in k^{*}$ and $A'\prec_{\text{lex}} A$. It follows from the explicit formulas for $A'$ in the proof that if $B$ and $C$ are admissible $n$-tuples such that $B, C\neq (0, \dots, 0)$ and $B\prec_{\text{lex}} C$, then $B'\prec_{\text{lex}} C'$.
\end{lem}
\begin{myproof}[Sketch of proof]
Let $A=(a_1, \dots, a_n)\neq (0, \dots, 0)$ be any admissible $n$-tuple. We first show that $\sD(\boldsymbol{x}^A)=\lambda_{A}\boldsymbol{x}^{A'}$, where $\lambda_{A}\in k^{*}$ and $A'\prec_{\text{lex}} A$. If $a_1\geq 1$, then 
\begin{equation*}\label{sDaction1}
\sD(\boldsymbol{x}^{A})=a_1x_1^{a_1-1}x_2^{a_2}\cdots x_n^{a_n}=a_1\boldsymbol{x}^{A'},
\end{equation*}
where $A'=(a_1-1, a_2, \dots, a_n)$. It is easy to see that 
\[A'=(a_1-1, a_2, \dots, a_n)\prec_{\text{lex}} (a_1, \dots, a_n)=A.\]
If $a_1=\dots=a_{s-1}=0$ and $a_s\geq 1$, then 
\begin{equation*}\label{sDaction2}
\begin{aligned}
\sD(\boldsymbol{x}^{A})&=a_s x_1^{p-1}x_2^{p-1}\cdots x_{s-1}^{p-1}x_s^{a_s-1}x_{s+1}^{a_{s+1}}\cdots x_n^{a_n}=a_s\boldsymbol{x}^{A'},
\end{aligned}
\end{equation*}
where $A'=(p-1, \dots, p-1, a_s-1, a_{s+1}, \dots, a_n)$. It is easy to see that 
\[
A'=(p-1,  \dots, p-1, a_s-1, a_{s+1}, \dots, a_n)\prec_{\text{lex}}(0,  \dots, 0, a_s, a_{s+1}, \dots, a_n)= A.
\]
Hence $\sD(\boldsymbol{x}^{A})=\lambda_{A}\boldsymbol{x}^{A'}$ for some $\lambda_{A}\in k^{*}$ and $A'\prec_{\text{lex}} A$. If $A'\neq (0, \dots, 0)$, then applying $\sD$ to $\boldsymbol{x}^{A'}$, we get $\sD(\boldsymbol{x}^{A'})=\lambda_{A'}\boldsymbol{x}^{A''}$, where $\lambda_{A'}\in k^{*}$ and $A''\prec_{\text{lex}} A'$. It follows from the explicit formulas for $A'$ that if $B$ and $C$ are admissible $n$-tuples such that 
$B, C\neq (0, \dots, 0)$ and $B\prec_{\text{lex}} C$, then $B'\prec_{\text{lex}} C'$. Hence the action of $\sD$ is compatible with the lexicographic ordering $\prec_{\text{lex}}$. This completes the sketch of proof.
\end{myproof}

\begin{lem}\cite[Lemma 5]{P91}\thlabel{WnN0resultslem5}
Set
\[
\cN_0:=\big\{D\in \cN(W(n; \underline{1}))\,|\, D^{p^{n-1}} \in W(n; \underline{1})_{(0)}\big\}.
\]
Then
\[
\dim \cN_0<\dim W(n; \underline{1})-n.
\]
\end{lem}
\begin{myproof}[Sketch of proof]
By \thref{ADthm}, it is enough to construct an $(n+1)$-dimensional subspace $V$ in $W(n; \underline{1})$ such that $V \cap \cN_0=\{0\}$. Set $V:=T_n\oplus k\sD$, where $T_n$ is a maximal torus in $W(n; \underline{1})$ with basis $\{x_1\del_1, \dots, x_n\del_n\}$, and $\sD$ is the nilpotent element in \thref{WnDresultslem3}. We show that $V \cap \cN_0=\{0\}$. Suppose the contrary, i.e. $V \cap \cN_0\neq \{0\}$. Then $t+\sD\in \cN_0$ for some $0\neq t\in T_n$. Note that for any admissible $n$-tuple $A=(a_1, \dots, a_n)$, the straight line $k\boldsymbol{x}^{A}$ is invariant under $T_n$ and corresponds to the weight $\bar{c}_1\theta_1+\dots+\bar{c}_n\theta_n$, where $\{\theta_1, \dots, \theta_n\}$ is a basis of $T_n^{*}$ dual to the basis $\{x_1\del_1, \dots, x_n\del_n\}$ and $\bar{c_i}\equiv c_i\,(\modd p)$. By \thref{Dcompatiblewithordering}, we know the action of $\sD$ on $\cO(n; \underline{1})$ is compatible with the lexicographic ordering $\prec_{\text{lex}}$. Hence for the admissible $n$-tuple $\delta=(p-1, \dots, p-1)$, we have that 
\[
\text{$(t+\sD)^{p^{n}-1}(\boldsymbol{x}^\delta)=\sD^{p^{n}-1}(\boldsymbol{x}^\delta)+\sum_{A\succ_{\text{lex}} (0, \dots, 0)}\lambda_{A}\boldsymbol{x}^{A}$ for some $\lambda_{A}\in k^{*}$}.
\]
By \thref{WnDresultslem3}(iii), $\sD^{p^{n}-1}(\boldsymbol{x}^\delta)=(-1)^{n}$. Hence $(t+\sD)^{p^{n}-1}(\boldsymbol{x}^\delta)$ is invertible in $\cO(n; \underline{1})$. But $t+\sD\in \cN_0$ by our assumption. Hence $(t+\sD)^{p^{n}-1}\in W(n; \underline{1})_{(0)}$. Note that if $D_1\in W(n; \underline{1})_{(0)}$, then $D_1(\cO(n; \underline{1}))\subseteq \fm$. In particular, $(t+\sD)^{p^{n}-1}(\boldsymbol{x}^\delta) \in \fm$, which is not invertible. This is a contradiction. Hence $V \cap \cN_0=\{0\}$. Applying \thref{ADthm}, we get the desired result. This completes the sketch of proof.
\end{myproof}

We are now ready to prove that 
\begin{thm}\cite[Lemma 6]{P91}\thlabel{Wnproof}
The variety $\cN(W(n; \underline{1}))$ is irreducible.
\end{thm}
\begin{myproof}[Sketch of proof] 
By \thref{nvarietythm}, we know that $\cN(W(n; \underline{1}))$ is equidimensional of dimension $\dim W(n; \underline{1})-n$. By \thref{WnDresultslem3}, we know that $\overline{O}=\overline{G.\sD}$ is an irreducible component of $\cN(W(n; \underline{1}))$. Let $Z_1, \dots, Z_t$ be pairwise distinct irreducible components of $\cN(W(n; \underline{1}))$, and set $Z_1=\overline{O}$. Suppose $t\geq 2$. Observe that $Z_2 \cap O=\emptyset$. Hence $Z_2 \subseteq \cN(W(n; \underline{1}))\setminus O=\cN_0$; see \thref{WnOresultslem4}. Then 
\[
\dim W(n; \underline{1})-n=\dim Z_2\leq \dim \cN_0<\dim W(n; \underline{1})-n
\] 
by \thref{WnN0resultslem5}. This is a contradiction. Hence $t=1$ and $\cN(W(n; \underline{1}))$ is irreducible. This completes the sketch of proof.
\end{myproof}

By a similar argument, the nilpotent variety of $S(n; \underline{1})$ (respectively $S(n; \underline{1})^{(1)}$) is proved to be irreducible. By \thref{Pbracket}, we see that the Poisson Lie algebra $(\cO(2n; \underline{1}), \{,\})$ is closely related to $H(2n; \underline{1})^{(2)}$. In fact, $(\cO(2n; \underline{1})/k)^{(1)}\cong H(2n; \underline{1})^{(2)}$. So the proof for $H(2n; \underline{1})^{(2)}$ relies heavily on Skryabin's work for $(\cO(2n; \underline{1}), \{,\})$; see \cite[Lemma 1.5 and Theorem 6.4]{S2002}.

\section{Some useful theorems}
We present three Block's theorems which will be used in Chapter 3. The first one is useful for part (b) of the sketch proof of \thref{regularderithm} which characterizes all regular elements of the $m$th Jacobson Witt algebra $W(m; \underline{1})$. Let us begin with some definitions.

\begin{defn}\cite[p. 433]{B68}\thlabel{defndfiffsimple}
Let $B$ be a ring (not necessarily associative or has a unit element). A \textnormal{derivation} of $B$ is an additive mapping $d: B\to B$ such that $d(ab)=d(a)b+a d(b)$ for all $a, b\in B$. Let $D$ be a set of derivations of $B$. By a \textnormal{$D$-ideal} of $B$ we mean an ideal of $B$ which is invariant under $D$. The ring $B$ is called \textnormal{$D$-simple} (\textnormal{$d$-simple} if $D$ consists of a single derivation $d$) if $B^2\neq 0$ and if $B$ has no proper $D$-ideals. Also $B$ is called \textnormal{differentiably simple} if it is $D$-simple for some set of derivations $D$ of $B$, and hence for the set of \textnormal{all} derivations of $B$.
\end{defn}
Note that the above definitions are also used for algebras over a ring $C$ and the derivations are assumed to be $C$-linear. Suppose $B$ is a differentiably simple commutative associative ring. At characteristic $0$, $B$ is an integral domain. In particular, if $B$ has a minimal ideal then $B$ is a field; see \cite[Sec.~4, p. 441]{B68}. Now suppose $B$ has prime characteristic. The following theorem determines $B$:

\begin{thm}\cite[Theorem 4.1]{B68}\thlabel{Blockthmdsimple}
Let $B$ be a differentiably simple commutative associative ring of prime characteristic $p$, and let $R=\{x\in B\,|\, x^p=0\}$. If $Rx=0$ for some $x\neq 0$ in $B$ (this will hold, e.g. if $B$ has a minimal ideal), then there is a subfield $E$ of $B$ and an $r\geq 0$ such that $B\cong \cO(r; \underline{1})$ as $E$-algebras. Here $E$ may be taken to be any maximal subfield of $B$ containing the subfield $F$ of differential constants (i.e. elements of $B$ which are annihilated by all derivations).
\end{thm}

The second theorem describes the derivation algebra of a Lie algebra in the following form:

\begin{thm}\citetext{\citealp[Corollary 3.3.4]{S04}}\thlabel{blockthmder}
Let $S$ be a finite dimensional simple algebra such that $S^2\neq (0)$. Then 
\begin{align*}
\Der\big(S\otimes\cO(m;\underline{n})\big)&=\big((\Der S)\otimes\cO(m;\underline{n})\big) \rtimes\big(\Id_{S}\otimes \Der \cO(m; \underline{n})\big)\\
&\cong\big((\Der S)\otimes\cO(m';\underline{1})\big) \rtimes\big(\Id_{S}\otimes W (m'; \underline{1})\big)
\end{align*}
for some $m'\in \N$.
\end{thm}

The last line follows from the fact that considered just as an algebra, $\cO(m; \underline{n})$ is isomorphic to the truncated polynomial ring $\cO(m';\underline{1})$ in $m'=n_1+\dots+n_m$ variables; see \cite[p. 64] {S04}. Hence $\Der \cO(m; \underline{n}) \cong \Der\cO(m';\underline{1})=W (m'; \underline{1})$.

The third theorem describes the structure of any finite dimensional semisimple Lie algebras. Recall the \textit{socle} of a finite dimensional semisimple Lie algebra $L$, denoted $\Soc(L)$, is the direct sum $\oplus_{j} I_{j}$ of minimal ideals $I_j$ of $L$. In particular, these ideals $I_j$ are irreducible $L$-modules.

\begin{thm}\citetext{\citealp[Theorem 9.3]{B68}\thlabel{blockthm}; \citealp[Corollary 3.3.6]{S04}} Let $L$ be a finite dimensional semisimple Lie algebra. Then there are simple Lie algebras $S_{i}$ and truncated polynomial rings $\cO(m_{i}; \underline{1})$ such that $\Soc(L)=\bigoplus _{i=1}^{t} S_{i}\otimes \cO(m_{i}; \underline{1})$ and $L$ acts faithfully on it. Moreover,
\[
\bigoplus _{i=1}^{t} S_{i}\otimes \cO(m_{i}; \underline{1}) \subset L \subset \bigoplus _{i=1}^{t}\big((\Der S_{i})\otimes \cO(m_{i}; \underline{1})\big)\rtimes \big(\Id_{S_{i}}\otimes W(m_{i}; \underline{1})\big).
\]
\end{thm}

We will see in Chapter 2 that Premet's conjecture can be reduced to the case where the Lie algebra is semisimple. The above theorem leads us to the study of nilpotent varieties for a particular class of semisimple restricted Lie algebras, namely the ones that are sandwiched between 
\[
\bigoplus _{i=1}^{t} S_{i}\otimes \cO(m_{i} ; \underline{1})
\]
and 
\[
\bigoplus _{i=1}^{t} \big((\Der S_i)\otimes \cO(m_{i}; \underline{1})\big)\rtimes \big(\Id_{S_{i}}\otimes W(m_{i}; \underline{1})\big).
\]

We finish this section with some useful theorems from algebraic geometry. 
\begin{defn}\cite[p. 91, 3.7]{H77}
A morphism $\psi: V \to W$ of affine varieties is called \textnormal{dominant} if the image $\psi(V)$ is dense in $W$, i.e. $\overline{\psi(V)}=W$.
\end{defn}
Given any morphism of irreducible affine varieties, proving directly its dominance may be difficult. However, we can show its differential map is surjective. More precisely,

\begin{mythm}[Differential Criterion for Dominance \citetext{\citealp[Proposition 1.4.15]{G13}}]\thlabel{dcdominant}
Let $\psi: V \to W$ be a morphism of irreducible affine varieties. Let $v$ be a smooth point in $V$ such that $\psi(v)$ is a smooth point in $W$. If the differential of $\psi$ at $v$, $d_{v}\psi: \T_v V \to \T_{\psi(v)}W$, is surjective, then the morphism $\psi$ is dominant. 
\end{mythm}

Once we know that a morphism is dominant, we can get a nonempty open set from the image of the morphism. Specifically,
 
\begin{thm}\cite[Corollary 2.2.8]{G13}\thlabel{zopen}
Let  $\psi: V \to W$ be a dominant morphism of irreducible affine varieties. Then the image of any nonempty open subset $U\subseteq V$ contains a nonempty open subset of $W$.
\end{thm}

Finally, we present some theorems on dimensions. The first two relate to the dimension of fibres.

\begin{thm}\citetext{\citealp[Sec.~4.4, Chap.~2, II]{DS94}; \citealp[Theorem 1.25]{S94}}\thlabel{fdimthm}
Let $\psi: V \to W$ be a dominant morphism of irreducible varieties. Suppose that $\dim V=m$ and $\dim W=n$. Then $m\geq n$, and 
\begin{enumerate}[\upshape(i)]
\item $\dim F\geq m-n$ for any $w\in W$ and for any component $F$ of the fibre $\psi^{-1}(w)$;
\item there exists a nonempty open subset $U\subset W$ such that $\dim \psi^{-1}(u)=m-n$ for all $u\in U$.
\end{enumerate}
\end{thm}

\begin{mythm}[Chevalley's Semi-continuity Theorem \citetext{\citealp[Sec.~4.5, Chap.~2, II]{DS94}}]\thlabel{chethm}
Let $\psi: V\to W$ be a morphism of affine varieties. Then for every $r\in \N_0$, the set
\[
V_{r}=\big\{v\in V\,|\, \dim \psi^{-1}(\psi(v))\geq r\big\}
\]
is Zariski closed in $V$.
\end{mythm}

The last theorem in this section relates to the dimension of intersections in $\A^{n}$.

\begin{mythm}[Affine Dimension Theorem \citetext{\citealp[Proposition~7.1, Chap.~I]{H77}}]\thlabel{ADthm}
Let $V, W$ be varieties of dimensions $r, s$ in $\A^{n}$. Then every irreducible component $U$ of $V\cap W$ has dimension $\geq r+s-n$.
\end{mythm}
\section{Overview of results}

\textit{Chapter 2}. Let $k$ be an algebraically closed field of characteristic $p>0$. We start with Premet's conjecture which states that the nilpotent variety of any finite dimensional restricted Lie algebra over $k$ is irreducible. We prove that this conjecture can be reduced to the semisimple case.

\begin{mythm*}[see \thref{reductionthm}]
Let $(\fg, [p])$ be a finite dimensional restricted Lie algebra over $k$. Let $\Rad \fg$ denote the radical of $\fg$. Then $\cN(\fg)$ is irreducible if and only if $\cN(\fg/\Rad \fg)$ is irreducible.
\end{mythm*}

The proof is done by induction on $\dim \fg$ and it relies on a result from algebraic geometry (see \thref{l:irr}). Since semisimple Lie algebras are not always direct sums of simple ideals in prime characteristic, the reduction of Premet's conjecture to the simple case is very non-trivial. 

\textit{Chapter 3}. We start to look at a particular class of semisimple restricted Lie algebras and verify Premet's conjecture in that case. 

By \thref{blockthm}, we know that any finite dimensional semisimple Lie algebra is sandwiched between 
\[\bigoplus _{i=1}^{t} S_i\otimes \cO(m_{i} ; \underline{1})
\]
and 
\[\bigoplus _{i=1}^{t} \big((\Der S_i)\otimes \cO(m_{i}; \underline{1})\big)\rtimes \big(\Id_{S_{i}}\otimes W(m_{i}; \underline{1})\big)
\]
for some simple Lie algebras $S_i$ and truncated polynomial rings $\cO(m_{i} ; \underline{1})$ with\newline $\Der \cO(m_{i} ; \underline{1})=W(m_{i}; \underline{1})$. Thus, to verify Premet's conjecture, we begin with the simplest example, $\fg=(\fsl_2\otimes \cO(1; 1)) \rtimes (\Id_{\fsl_2}\otimes k \del)$, where $\fsl_2$ is the special linear Lie algebra, $\cO(1; 1)$ is the truncated polynomial ring $k[X]/(X^{p})$, and $\del=\frac{d}{dx}$ which acts on $\fsl_2\otimes \cO(1; 1)$ in the natural way. We assume further that $\chari k=p>2$. So $\fsl_2$ is a simple restricted Lie algebra over $k$ with all its derivations inner. We prove that the maximal dimension of tori in $\fg$ is $1$, and the nilpotency index of any element in $\cN(\fg)$ is at most $p^2$. It follows from \thref{nvarietythm} that $\dim \cN(\fg)=3p$. After gathering these pieces of information, we are ready to prove that 

\begin{mythm*}[see \thref{sl2thm}]
The variety $\cN(\fg)$ is irreducible.
\end{mythm*}

We will see that the argument in the proof is quite general. It works if we replace $\fsl_2$ by any Lie algebra $\fg_2=\Lie(G_2)$, where $G_2$ is a reductive algebraic group. Then we extend the example $\fg$ to the semisimple restricted Lie algebra 
\[
\cL:=(S\otimes \cO(m; \underline{1}))\rtimes (\Id_S\otimes \dd D),
\]
where $S$ is a simple restricted Lie algebra over $k$ such that $\ad S=\Der S$ and $\cN(S)$ is irreducible, $\cO(m; \underline{1})=k[X_1, \dots, X_m]/(X_1^{p}, \dots, X_{m}^{p})$ is the truncated polynomial ring in $m\geq 2$ variables, and $\dd{D}$ is a restricted transitive subalgebra of $W(m; \underline{1})=\Der \cO(m; \underline{1})$ such that $\cN(\dd{D})$ is irreducible. We split our study on $\cN(\cL)$ into three sections. In the first section, we study nilpotent elements of $\dd D$. Then we study nilpotent elements of $\cL$ and carry out some calculations using \thref{generalJacobF}. We finally prove that 

\begin{mythm*}[see \thref{NLirr}]
The variety $\cN(\cL)$ is irreducible.
\end{mythm*}

As a remark, we see from the above that Premet's conjecture holds for 
\[
\bigoplus _{i=1}^{t}(S_{i}\otimes \cO(m_{i}; \underline{1}))\rtimes (\Id_{S_{i}}\otimes \dd D_{i}),
\]
where each $S_{i}$ is a simple restricted Lie algebra over $k$ such that $\ad S_{i}=\Der S_{i}$ and $\cN(S_{i})$ is irreducible, $\cO(m_{i}; \underline{1})$ are truncated polynomial rings, and each $\dd D_{i}$ is a restricted transitive subalgebra of $W(m_i; \underline{1})=\Der \cO(m_{i}; \underline{1}) $ such that $\cN(\dd D_{i})$ is irreducible.

\textit{Chapter 4}. This is our final chapter. It corresponds to a paper \cite{C19} of the author which was published in the Journal of Algebra and Its Applications. We assume that $\chari k=p>3$ and $n\in\N_{\geq 2}$. Then the Zassenhaus algebra, denoted $\fL=W(1; n)$, provides the first example of a simple, non-restricted Lie algebra. We can embed $\fL$ into its minimal $p$-envelope $\fL_p=W(1;n)_p$. This restricted Lie algebra is semisimple. By \cite[Theorem 4.8(i)]{YS16}, the variety $\cN(\fL):=\cN(\fL_p)\cap \fL$ is reducible. So investigating the variety $\cN(\fL_p)$ becomes critical. 

Let $\cN$ denote the nilpotent variety of $\fL_p$. We split our study on $\cN$ into three sections. In the first section, we focus on nilpotent elements of $\fL_p$ and carry out some calculations using Jacobson's formula. This work enables us to identify an irreducible component $\cN_\text{reg}$ of $\cN$. Moreover, we can explicitly describe it as:
\begin{mypro*}[see \thref{l:l14} and \thref{l:l15}]
\[
\cN_\text{reg}=G.(\del+k\del^p+\dots+k\del^{p^{n-1}}),
\]
where $G=\Aut(\fL_p)$.
\end{mypro*}
In the final section, we show that the complement of $\cN_\text{reg}$ in $\cN$, denoted $\cN_\text{sing}$, has $\dim \cN_\text{sing}<\dim \cN$. The proof is similar to Premet's proof for the Jacobson-Witt algebra $W(n; \underline{1})$; see \thref{WnN0resultslem5}. But we have to construct a new subspace $V$ of $\fL_p$ such that $\dim V=n+1$ and $V\cap \cN_\text{sing}=\{0\}$; see \thref{l:l16}. Combining all these results, we are able to prove the last theorem in the thesis:

\begin{mythm*}[see \thref{main}]
The variety $\cN=\cN(\fL_p)$ coincides with the Zariski closure of
\[
\cN_\text{reg}=G.(\del+k\del^p+\dots+k\del^{p^{n-1}})
\] and hence is irreducible.
\end{mythm*}

\chapter{Reduction of Premet's conjecture to the semisimple case}
Let $(\fg, [p])$ be a finite dimensional restricted Lie algebra over $k$. Recall Premet's conjecture which states that the variety $\cN(\fg)=\{x\in\fg\, |\, x^{[p]^{N}}=0 \,\text{for $N\gg 0$}\}$ is irreducible. In this short chapter we show that this conjecture can be reduced to the case where $\fg$ is semisimple. For that, we need the following result from algebraic geometry. It gives a criterion for a variety to be irreducible.

\begin{lem}\thlabel{l:irr}
Let $\psi: X \to Y$ be a surjective morphism of algebraic varieties such that
\begin{enumerate}[\upshape(i)]
\item Y is irreducible,
\item all fibres of $\psi$ are irreducible and have the same dimension $d$, and
\item X is equidimensional.
\end{enumerate}
Then $X$ is irreducible.
\end{lem}

\begin{proof}
Let $X=X_{1}\cup \dots\cup X_{t}$ be the decomposition of $X$ into pairwise distinct irreducible components $X_i$. Suppose $t \geq 2$. Then for any $y \in Y$, 
\[
\psi^{-1}(y)= \bigcup_{i=1}^{t} \big(\psi^{-1}(y) \cap X_{i}\big).
\]
Since $\psi^{-1}(y)$ is irreducible, we have that $\psi^{-1}(y)=\psi^{-1}(y) \cap X_{i}$ for some $i$.

For every $i$, define $O_{i}:=X_{i}\setminus \bigcup_{j\neq i} (X_i\cap X_{j})$. Then $O_{i}$ is a nonempty open subset of $X_{i}$. If $y\in \psi(O_{i})$, then the fibre $\psi^{-1}(y)$ is not contained in $\psi^{-1}(y)\cap X_{j}$ for every $j \neq i$. Hence for any $y \in \psi(O_{i})$,
\begin{equation}\label{e:e1}
\text{$\psi^{-1}(y) =\psi^{-1}(y) \cap X_{i}$}.
\end{equation}

Next we show that there is a unique irreducible component $X_{j}$ of $X$ such that $Y=\overline{\psi(X_{j})}$. Since $\psi$ is surjective, we have that $Y=\bigcup_{i=1}^{t}\overline{\psi(X_i)}$. But $Y$ is irreducible, so that $Y= \overline{\psi(X_{j})}$ for some $j$. By definition of $O_j$, we also have that $Y= \overline{\psi(O_{j})}$, i.e. $\psi(O_{j})$ is dense in $Y$. Thus, $\psi(O_{j})$ contains a nonempty open subset $Z$ of $Y$. By~\eqref{e:e1}, we have that for any $y\in Z$, $\psi^{-1}(y) =\psi^{-1}(y) \cap X_{j}$. This means that if $i\neq j$, then $X_i\setminus (X_i\cap X_j) \subseteq  \psi^{-1}(Y\setminus Z)$. As a result, $X_i=\overline{X_{i}\setminus (X_i\cap X_{j})}$ is contained in $\psi^{-1}(Y\setminus Z)$. This implies that $\overline{\psi(X_{i})}\neq Y$ for every $i\neq j$. So there is a unique irreducible component $X_{j}$ of $X$ such that $Y=\overline{\psi(X_{j})}$.

We now show that for every $i$, $\dim X_{i}=\dim \overline{\psi(X_i)}+d$. Consider the restriction of $\psi$ to $X_i$, $\psi: X_{i}\to \overline{\psi(X_{i})}$. By \thref{fdimthm}, there exists a nonempty open subset $U\subset \overline{\psi(X_i)}$ such that for every $y \in U$,
\begin{equation*}\label{e:ereduce3}
\dim X_{i} = \dim \overline{\psi(X_{i})}+\dim (\psi^{-1}(y) \cap X_{i}).
\end{equation*}
Since $U\cap \psi (O_i)\neq \emptyset$, then \eqref{e:e1} implies that for every $i$,
\begin{equation}\label{e:dimxi}
\text{$\dim X_{i}= \dim \overline{\psi(X_{i})}+d$}.
\end{equation}

By the previous argument, we know that there is a unique $X_{j}$ that dominates $Y$. So if $i\neq j$, then $\overline{\psi(X_i)}$ is a proper subset of $Y$. This and \eqref{e:dimxi} imply that
\[
\dim X_{i}= \dim \overline{\psi(X_{i})}+d < \dim Y+d =\dim \overline{\psi(X_{j})}+d= \dim X_{j}.
\]
But this contradicts the fact that $X$ is equidimensional. Hence $t=1$ and $X$ is irreducible. This completes the proof.
\end{proof}

\begin{thm}\thlabel{reductionthm}
Let $(\fg, [p])$ be a finite dimensional restricted Lie algebra over $k$. We denote by $\Rad \fg$ the radical of $\fg$. Then \thref{pconj} can be reduced to the case where $\fg$ is semisimple. Explicitly, $\cN(\fg)$ is irreducible if and only if $\cN(\fg/\Rad \fg)$ is irreducible.
\end{thm}

\begin{proof}
We proceed by induction on $\dim \fg$. If $\dim \fg =0$ or $1$, the result is trivially true. Suppose $\dim\fg>1$ and ``$\cN(\fg)$ is irreducible if and only if $\cN(\fg/\Rad \fg)$ is irreducible'' holds for all $\fg$ of dimension less than $m$. Let $\dim \fg =m$. If $\fg$ is semisimple, then there is nothing to prove. If $\fg$ is not semisimple, i.e. $\Rad \fg \neq 0$, then there exists a nonzero abelian ideal. Let $R$ be such a $p$-ideal. 

\textbf{\textit{Case 1}}. If $R$ contains a non-nilpotent element, say $y$, then it follows from\thref{JCthm} that there exist a unique semisimple element $y_s\in \fg$ and a unique nilpotent element $y_n\in \fg$ such that $y=y_s+y_n$ and $[y_s,y_n]=0$. Since $y$ is non-nilpotent, then $y_s\neq 0$. Replace $y$ by its semisimple part $y_s$. We show that $y^{[p]^{N}}$ lies in the centre $\fz(\fg)$ for $N\gg 0$. Indeed, since $R$ is a $p$-ideal we have that $[y^{[p]^{N}}, \fg] \subseteq R$. Moreover, $R$ is abelian, then $[y^{[p]^{N}}, [y^{[p]^{N}}, \fg]]=0$. Since $y$ is semisimple, so is $y^{[p]^{N}}$ by \thref{sselemt}(iii). This implies that $[y^{[p]^{N}}, \fg]=0$. Hence $y^{[p]^{N}} \in \fz(\fg)$ and $\fz(\fg)\neq 0$. 

Let $\fz(\fg)=\fz_{s}\bigoplus\fz_{n}$, where $\fz_{s}$ is a subalgebra with a one-to-one $[p]$-th power map and $\fz_{n}$ is the subspace in $\fz(\fg)$ consisting of nilpotent elements of $\fz(\fg)$. By above, we already know that $y^{[p]^{N}}\in\fz_{s}$ and so $\fz_{s} \neq 0$. We prove that the canonical homomorphism $\psi: \fg \to \fg/\fz_{s}, x\mapsto x+\fz_{s}$ induces a bijective morphism $\tilde{\psi}: \cN(\fg)\to \cN(\fg/\fz_{s})$. If $x+\fz_{s} \in \cN(\fg/\fz_{s})$, then $x^{[p]^{N}}\in \fz_{s}$ for $N\gg0$. Since the $[p]$-th power map on $\fz_{s}$ is one-to-one, there exists $z\in \fz_{s}$ such that $x^{[p]^{N}}=z^{[p]^{N}}$. Since $z$ is an element of $\fz_{s}$, it commutes with $x$. As a result, $(x-z)^{[p]^{N}}=0$. This implies that $x-z \in \cN(\fg)$. So the morphism $\tilde{\psi}$ is surjective. For injectivity, suppose $\tilde{\psi}(x)=\tilde{\psi}(y)$ for some $x, y \in\cN(\fg)$. Then $y=x+z_1$ for some $z_1\in \fz_{s}$. Since $x$ and $y$ are nilpotent, then for $N\gg 0$ we get
\[
0=y^{[p]^{N}}=(x+z_1)^{[p]^{N}}=x^{[p]^{N}}+z_{1}^{[p]^{N}}=z_{1}^{[p]^{N}}.
\]
Since the $[p]$-th power map on $\fz_{s}$ is one-to-one, we have that $z_1=0$ and so $y=x$. Therefore, the morphism $\tilde{\psi}$ is bijective. Next we claim that $\cN(\fg)$ is irreducible if and only if $\cN(\fg/\fz_{s})$ is irreducible. The ``if'' part is trivial because surjective morphism preserves irreducibility. For the ``only if'' part, suppose $\cN(\fg/\fz_{s})$ is irreducible. By \thref{nvarietythm}(iii), we know that $\cN(\fg)$ is equidimensional. Since $\tilde{\psi}$ is bijective, it follows that all fibres of $\tilde{\psi}$ are single points. Applying \thref{l:irr} with $d=0$ we get $\cN(\fg)$ is irreducible.

\textbf{\textit{Case 2}}. If $R$ is contained in $\cN(\fg)$, then we may assume that $R^{[p]^{r}}=0$ for some $r\in \N$. Note that the canonical homomorphism $\psi: \fg \to \fg/R,\, x \mapsto \overline{x}=x+R$ induces a surjective mapping $\tilde{\psi}: \cN(\fg)\to \cN(\fg/R)$. We show that $\cN(\fg)$ is irreducible if and only if $\cN(\fg/R)$ is irreducible. The ``if'' part is trivial. For the ``only if'' part, suppose $\cN(\fg/R)$ is irreducible. We show that $\cN(\fg)=\tilde{\psi}^{-1}\big(\cN(\fg/R)\big)$, i.e. $x$ is nilpotent in $\fg$ if and only if $x+R$ is nilpotent in $\fg/R$. By Jacobson's formula and that $R$ is a $p$-ideal of $\fg$, we have that for any $v\in R$,
\[
\text{$(x+v)^{[p]}=x^{[p]}+v_1$ for some $v_1 \in R$}.
\]By Jacobson's formula again, we have that 
\[
\text{$(x+v)^{[p]^{N}}=x^{[p]^{N}}+v_{N}$ for some $v_{N} \in R$}.
\]
If $x\in \cN(\fg)$, then we can choose $N$ sufficiently large so that $(x+v)^{[p]^{N}}=\overline{0}$. Hence $x+v \in \cN(\fg/R)$. Conversely, if $x+R\in \cN(\fg/R)$, then for $N\gg 0$, $x^{[p]^{N}}\in R$. By our assumption, $R^{[p]^{r}}=0$. Hence $x^{[p]^{N+r}}=0$ and so $x \in \cN(\fg)$. Therefore, for any $\overline{x} \in \cN(\fg/R)$, the fibre $\tilde{\psi}^{-1}(\overline{x})$ consists of elements of the form $x+R$. But $x+R\cong R$ as an affine space, we conclude that all fibres of $\tilde{\psi}$ are irreducible and have the same dimension. It follows from \thref{nvarietythm}(iii) and \thref{l:irr} that $\cN(\fg)$ is irreducible.

To sum up, if $\Rad \fg \neq 0$, then there exists an abelian $p$-ideal $A$, i.e. $A=\fz_{s}$ or $A=R$ with $R^{[p]^{r}}=0$ for some $r\in \N$. We proved that $\cN(\fg)$ is irreducible if and only if $\cN(\fg/A)$ is irreducible. Since $\dim (\fg/A) < \dim \fg$, the induction hypothesis implies that $\cN(\fg/A)$ is irreducible if and only if $\cN\big((\fg/A)/\Rad (\fg/A)\big)\cong\cN(\fg/\Rad \fg)$ is irreducible. Therefore, we proved that $\cN(\fg)$ is irreducible if and only if $\cN(\fg/\Rad \fg)$ is irreducible. This completes the proof.
\end{proof}
It is natural to ask if we could reduce Premet's conjecture to the simple case. Unfortunately, this is very non-trivial in prime characteristic as semisimple Lie algebras are not always direct sums of simple ideals. Hence to prove Premet's conjecture, we shall focus on the semisimple case.

\chapter{Nipotent varieties of some semisimple restricted Lie algebras}\label{ssexamples}
In this chapter, we assume that $\chari k=p>2$. We prove Premet's conjecture holds for a class of semisimple restricted Lie algebras whose socle involves any simple restricted Lie algebra with all its derivations inner. We start with the semisimple restricted Lie algebra whose socle involves $\fsl_2$ tensored by the truncated polynomial ring $k[X]/(X^{p})$. Then we extend it to the semisimple restricted Lie algebra whose socle involves $S\otimes \cO(m; \underline{1})$, where $S$ is any simple restricted Lie algebra such that $\ad S=\Der S$ and $\cN(S)$ is irreducible, and $\cO(m; \underline{1})=k[X_1, \dots, X_m]/(X_1^p, \dots, X_m^p)$ is the truncated polynomial ring in $m\geq 2$ variables.

\section{Socle involves $\fsl_2$}\label{sl2section2.2.1}
By Block's theorem (\thref{blockthm}) on finite dimensional semisimple Lie algebras, the first example (perhaps the easiest one) we shall consider is $(\fsl_2\otimes \cO(1; 1)) \rtimes (\Id_{\fsl_2}\otimes k \del)$\footnote{This problem was introduced by A.~Premet in the workshop \textit{Lie Theory and Representation Theory} in Pisa, Italy, January-February 2015. In \cite{PD17}, A.~Premet and D.~I.~Stewart studied semisimple restricted Lie subalgebras of this type (referred to as \textit{exotic semidirect products}) in the Lie algebra $\fg_1=\Lie(G_1)$, where $G_1$ is a simple algebraic group of exceptional type over $k$.}. Let $k$ be an algebraically closed field of characteristic $p>2$. Recall the special linear Lie algebra $\fsl_2$ of $2\times 2$ matrices with trace $0$, and its standard basis
\[
e= \left(\begin{matrix}
 0 & 1  \\
0 & 0
\end{matrix}\right) , \quad f= \left(\begin{matrix}
 0 &0  \\
1 & 0
\end{matrix}\right), \quad h= \left(\begin{matrix}
 1 & 0  \\
0 & -1
\end{matrix}\right)
\]such that
\[
[e, f]=h, \quad [h, e]=2e, \quad [h, f]=-2f.
\]
Since $p>2$, we know that $\fsl_2$ is simple and all its derivations are inner \cite[Sec.~5, Chap.~V]{S67}. By~\thref{exres} and \thref{centrelesslem}, we know that $\fsl_2$ is a restricted Lie algebra with the unique $[p]$-th power map given by $p$-th power of matrices. An easy calculation shows that $e^{p}=f^{p}=0$ and $h^{p}=h$. Hence $e, f$ are nilpotent and $h$ is toral.

Let $\cO(1; 1)$ denote the $p$-dimensional truncated polynomial ring $k[X]/(X^{p})$. We write $x$ for the image of $X$ in $\cO(1; 1)$. Note that $\cO(1; 1)$ is a local ring. Let $\fm$ denote its unique maximal ideal. For every $g\in \cO(1; 1)$, $g^p=g(0)^p$. Then $g^{p}=0$ for all $g\in\fm$.

Consider the algebra $\fsl_2\otimes\cO(1; 1)$. The restricted Lie algebra structure on $\fsl_2$ induces a restricted Lie algebra structure on $\fsl_2\otimes\cO(1; 1)$. Explicitly, the Lie bracket is given by
\[
[y\otimes g_{1}, z\otimes g_{2}]:=[y, z]\otimes  g_{1} g_{2},
\] and the $[p]$-th power map is given by
\[
(y\otimes g_1)^{[p]}:=y^{p}\otimes g_1^{p}
\]
for all $y\otimes g_1, z\otimes g_2 \in \fsl_{2}\otimes\cO(1; 1)$. Note that $\dim (\fsl_2\otimes\cO(1; 1))=3p$. 

Let $W(1; 1)$ denote the derivation algebra of $\cO(1; 1)$. It is known as the \textit{Witt algebra}. By \thref{wittthm} and \thref{exres}, we know that $W(1; 1)$ is a simple restricted Lie algebra with the $[p]$-th power map given by $D^{[p]}=D^p$ for all $D \in W(1; 1)$. Moreover, $W(1; 1)$ is a free $\cO(1; 1)$-module of rank $1$ generated by $\del=\frac{d}{dx}$. Note that $\del$ acts on $\fsl_2\otimes\cO(1; 1)$ by differentiating truncated polynomials in $\cO(1; 1)$, i.e. $\del(y\otimes g)=y\otimes \del(g)$ for all $y\otimes g\in \fsl_2\otimes\cO(1; 1)$. 

Let $\fg:=(\fsl_2\otimes \cO(1; 1) )\rtimes (\Id_{\fsl_2}\otimes k \del)$. To ease notation we identify $\Id_{\fsl_2}\otimes k \del$ with $k \del$. It is easy to check that $\fg$ is a ($3p+1$)-dimensional Lie algebra with the Lie bracket defined in the natural way:
\[
[y\otimes g,\del]:=-y \otimes\del(g) 
\] 
for all $y\otimes g\in \fsl_2\otimes\cO(1; 1)$. By \thref{blockthmder}, we see that $\fg$ embeds in the restricted Lie algebra $\Der(\fsl_2\otimes \cO(1; 1))=\big(\fsl_2\otimes \cO(1; 1)\big) \rtimes \big(\Id_{\fsl_2} \otimes W(1; 1)\big)$. Then $\fg$ is restricted with the $[p]$-th power map given by $p$-th power of derivations. Note that $\del^p=0$. Let us look at the structure of $\fg$ in more detail.

\begin{lem}\thlabel{sl2noD}
\begin{enumerate}[\upshape(i)]
\item $[\fg, \fg]=\fsl_2\otimes \cO(1; 1)$ which is not semisimple.
\item $\fsl_2\otimes \cO(1; 1)$ is the unique minimal ideal of $\fg$.
\item The maximal dimension of tori in $\fsl_2\otimes \cO(1; 1)$ is $1$, i.e. $\MT(\fsl_2\otimes \cO(1; 1))=1$.
\item All irreducible components of $\cN(\fsl_2\otimes \cO(1; 1))$ have dimension $3p-1$. 
\end{enumerate}
\end{lem}

\begin{proof}
(i) Write $\fsl_2\otimes \cO(1; 1)=(\fsl_2\otimes 1)\oplus(\fsl_2\otimes \fm)$. Since $\fsl_2$ is simple, it is easy to check that $[\fsl_2\otimes \cO(1; 1), \fsl_2\otimes \cO(1; 1)]=\fsl_2\otimes \cO(1; 1)$. As $\fsl_2\otimes \cO(1; 1)$ is a Lie subalgebra of $\fg$, we have that $\fsl_2\otimes \cO(1; 1)=[\fsl_2\otimes \cO(1; 1), \fsl_2\otimes \cO(1; 1)] \subseteq [\fg, \fg]$. On the other hand, 
\begin{align*}
[y\otimes g_1+\lambda_1 \del, z\otimes g_2+\lambda_2 \del]
=[y,z]\otimes g_1g_2+\lambda_1z\otimes \del(g_2)-\lambda_2y\otimes \del(g_1)
\end{align*}
for all $y, z \in \fsl_2, g_1, g_2 \in \cO(1; 1)$ and $\lambda_1, \lambda_2 \in k$. Clearly, this is an element of \\$\fsl_2\otimes \cO(1; 1)$. Hence $[\fg, \fg]\subseteq\fsl_2\otimes \cO(1; 1)$. As a result, $[\fg, \fg]=\fsl_2\otimes \cO(1; 1)$. 

Now suppose $\fsl_2\otimes \cO(1; 1)$ is semisimple. This means the only abelian ideal in $\fsl_2\otimes \cO(1; 1)$ is the zero ideal. Since $\fsl_2$ is simple, any ideal $I$ of $\fsl_2\otimes \cO(1; 1)$ has the form $\fsl_2\otimes J$ for some ideal $J$ of $\cO(1; 1)$. Take $J=(x^{p-1})$, the principal ideal generated by $x^{p-1}$ in $\cO(1; 1)$. Then we find a nonzero ideal $I=\fsl_2\otimes (x^{p-1})$ in $\fsl_2\otimes \cO(1; 1)$ such that $[I, I]=0$, a contradiction. Hence $\fsl_2\otimes \cO(1; 1)$ is not semisimple. This proves (i).

(ii) Observe that if $w=y\otimes g_1+\lambda_1 \del$ is an element of $\fg$ such that $[w, z\otimes g_2]=0$ for all $z\otimes g_2$ in $\fsl_2 \otimes \cO(1; 1)$, then we must have that either $y=0$ or $g_1=0$, and $\lambda_1=0$. As a result, $w=0$. This shows that $\fg$ acts faithfully on $\fsl_2\otimes \cO(1; 1)$ via the adjoint representation, and embeds $\fsl_2\otimes \cO(1; 1)$ into its derivation algebra $\Der(\fsl_2\otimes \cO(1; 1))$. Hence for any nonzero ideal $J$ of $\fg$ we must have that $[J, \fsl_2\otimes \cO(1; 1)]\neq 0$. In particular, for any minimal ideal $I$ of $\fg$, $[I, \fsl_2\otimes \cO(1; 1)]\neq 0$. By (i) of this lemma, $\fsl_2\otimes \cO(1; 1)$ is an ideal of $\fg$, so is $[I, \fsl_2\otimes \cO(1; 1)]$. Since $[I, \fsl_2\otimes \cO(1; 1)]\subseteq I$, the minimality of $I$ implies that $[I, \fsl_2\otimes \cO(1; 1)]=I$. As $I=[I, \fsl_2\otimes \cO(1; 1)]\subseteq I \cap (\fsl_2\otimes \cO(1; 1))$, we have that $I \subseteq \fsl_2\otimes \cO(1; 1)$.

Let $v$ be any nonzero element of $I$. Since $I \subseteq \fsl_2\otimes \cO(1; 1)$, we can write $v=y\otimes (a_0+a_1x+\dots+a_jx^j+\dots +a_{p-1}x^{p-1})$ for some $0\neq y\in\fsl_2$ and $a_j \in k^{*}$. Then   
\[
(\ad \del)^{j}(v)=y\otimes (j!a_j+xg_3)
\]for some $g_3\in \cO(1; 1)$. This shows that $I \cap (\fsl_2\otimes 1)\neq \{0\}$. By the simplicity of $\fsl_2$, we have that $\fsl_2\otimes 1 \subset I$. Now consider the adjoint endomorphisms from $\fsl_2\otimes x$ to $\fsl_2\otimes 1$, one can deduce that $\fsl_2\otimes x \subset I$. Continuing in this way and by induction, we get $\fsl_2\otimes x^l\subset I$ for all $0\leq l\leq p-1$. Hence $I=\fsl_2\otimes \cO(1; 1)$. This proves (ii).

(iii) Since any element $y\otimes g$ in $\fsl_2 \otimes \fm$ satisfies $(y\otimes g)^{p}=0$, it follows that $\fsl_2 \otimes \fm$ is a $p$-ideal of $\fsl_2 \otimes \cO(1; 1)$ consisting of nilpotent elements. As a results, $\MT(\fsl_2 \otimes \fm)=0$. By \thref{MTlem}, we have that $\MT(\fsl_2\otimes \cO(1; 1))=\MT(\fsl_2\otimes 1)$. It is known that the maximal dimension of tori in $\fsl_2$ is $1$. Hence $\MT(\fsl_2\otimes \cO(1; 1))=1$. This proves (iii).

(iv) By \thref{nvarietythm} and (iii) of this lemma, we know that all irreducible components of $\cN(\fsl_2\otimes \cO(1; 1))$ have dimension $\dim (\fsl_2\otimes \cO(1; 1))-\MT(\fsl_2\otimes \cO(1; 1))=3p-1$. This proves (iv).
\end{proof}

Before we are going to prove a similar result for $\fg$, we need to construct some automorphisms of $\fg$.

\begin{lem}\thlabel{sl2automorphism}
Let $y\otimes q$ be an element of $\fsl_2 \otimes \cO(1;1)$. For any $\sigma \in \Aut(\fsl_2)$, define $\sigma(y\otimes q):=\sigma(y)\otimes q$. Then $\sigma$ extends to an automorphism of $\fg$.
\end{lem}

\begin{proof}
We first show that $\sigma$ extends to an automorphism of $\fsl_2 \otimes \cO(1;1)$. For any $y_1\otimes g_1, y_2\otimes g_2 \in \fsl_2 \otimes \cO(1;1)$, we have that 
\begin{align*}
\sigma([y_1\otimes g_1, y_2\otimes g_2 ])&= \sigma ([y_1, y_2]\otimes g_1g_2)=[\sigma(y_1), \sigma(y_2)]\otimes g_1g_2.
\end{align*}
On the other hand, 
\begin{align*}
[\sigma(y_1\otimes g_1), \sigma(y_2\otimes g_2)]=[\sigma(y_1)\otimes g_1, \sigma(y_2)\otimes g_2]=[\sigma(y_1), \sigma(y_2)]\otimes g_1g_2.
\end{align*}
Hence $\sigma$ is an automorphism of $\fsl_2 \otimes \cO(1;1)$. Then it induces an automorphism of the derivation algebra 
\begin{align*}
\Der(\fsl_2\otimes \cO(1; 1))=\big(\fsl_2\otimes \cO(1; 1)\big) \rtimes \big(\Id_{\fsl_2} \otimes W(1; 1)\big)
\end{align*}
via conjugation. Note that $\fg \subset \Der(\fsl_2\otimes \cO(1; 1))$. It remains to check $\sigma \circ \del \circ \sigma^{-1}$.  For any $y_1\otimes g_1 \in \fsl_2\otimes \cO(1; 1)$, we have that 
\begin{align*}
\sigma \circ \del \circ \sigma^{-1}(y_1\otimes g_1)&=\sigma\circ \del\big(\sigma^{-1}(y_1)\otimes g_1\big)\\
&=\sigma\big(\sigma^{-1}(y_1)\otimes \del(g_1)\big)\\
&=\sigma \sigma^{-1}(y_1) \otimes \del(g_1)\\
&=y_1 \otimes \del(g_1).
\end{align*}
Thus, $\sigma \circ \del \circ \sigma^{-1}=\del$. As a result, $\sigma$ preserves $\fg$ and it extends to an automorphism of $\fg$. This completes the proof.
\end{proof}

\begin{lem}\thlabel{autg}
Let $u=y\otimes g$ be an element of $\fsl_2\otimes \fm$. Then $\exp(\ad u)$ is an automorphism of $\fg$. Similarly, if $v=z\otimes q$ is an element of $\cN(\fsl_2)\otimes \cO(1;1)$, then $\exp(\ad v)$ is an automorphism of $\fg$.
\end{lem}

\begin{proof}
For any $u=y\otimes g \in\fsl_2\otimes \fm$, we first show that $\exp (\ad u)$ is an automorphism of $\fsl_2\otimes \cO(1; 1)$. Suppose that $L$ is a Lie algebra over $k$ and $D\in \Der L$ satisfies $D^{p}=0$. In order to show that $\exp (D)=\sum_{i=0}^{p-1}\frac{1}{i!}D^{i}$ is an automorphism of $L$ it suffices to show that 
\[
\text{$\sum_{0\leq i, j\leq p,\, i+j\geq p}\frac{1}{i!j!}[D^{i}(w_1), D^{j}(w_2)]= 0$}
\]
for all $w_1, w_2\in L$. Note that $\ad u$ with $u$ above is a derivation of $\fsl_2\otimes \cO(1; 1)$ such that $(\ad u)^p=\ad u^{p}=\ad (y^p\otimes g^p)=0$; see \thref{defres}(1). Set $D=\ad u=\ad (y\otimes g)$. Then for any $w_1=y_1\otimes g_1$ and $w_2=y_2\otimes g_2$ in $\fsl_2 \otimes \cO(1;1)$, we have that 
\begin{align*}
D^{i}(w_1)=(\ad y)^{i}(y_1)\otimes g^{i}g_1,\,\text{and}
\end{align*}
\begin{align*}
D^{j}(w_2)=(\ad y)^{j}(y_2)\otimes g^{j}g_2.
\end{align*}
Hence 
\begin{align*}
&\sum_{0\leq i, j\leq p,\, i+j\geq p}\frac{1}{i!j!}[D^{i}(w_1), D^{j}(w_2)]\\
=&\sum_{0\leq i, j\leq p,\, i+j\geq p}\frac{1}{i!j!}[(\ad y)^{i}(y_1), (\ad y)^{j}(y_2)]\otimes g^{i+j}g_1g_2.
\end{align*}
Since $g\in \fm$ and $i+j\geq p$, we have that $g^{i+j}=0$. As a result, 
\[
\sum_{0\leq i, j\leq p, \,i+j\geq p}\frac{1}{i!j!}[D^{i}(w_1), D^{j}(w_2)]= 0.
\] 
Therefore, $\exp(\ad u)$ is an automorphism of $\fsl_2\otimes \cO(1; 1)$. Then it induces an automorphism of the derivation algebra 
\[
\Der(\fsl_2\otimes \cO(1; 1))=\big(\fsl_2\otimes \cO(1; 1)\big) \rtimes \big(\Id_{\fsl_2} \otimes W(1; 1)\big)
\]
via conjugation. Note that 
\[
\fg=(\fsl_2\otimes \cO(1; 1) )\rtimes k \del \subset \Der(\fsl_2\otimes \cO(1; 1)).\]
To conclude that $\exp(\ad u)$ is an automorphism of $\fg$, we need to check that $\exp(\ad u)$ preserves $\fg$, i.e. $\exp (\ad u)\circ \del \circ \exp (\ad (-u))\in \fg$. Recall that $u=y\otimes g \in\fsl_2\otimes \fm$. We show that for any $y_1\otimes g_1\in \fsl_2\otimes \cO(1; \underline{1})$, 
\begin{equation}\label{expdelexpprecal}
\begin{aligned}
&\bigg(\exp (\ad (y\otimes g))\circ \del \circ \exp (\ad (-y\otimes g))\bigg)(y_1\otimes g_1)\\
=&y_1\otimes \del(g_1)-\big(\ad \big(y\otimes \del(g)\big)\big)(y_1\otimes g_1)-\big(\ad \big(y^{p}\otimes g^{p-1}\del(g)\big)\big)(y_1\otimes g_1).
\end{aligned}
\end{equation}
We first compute $\del \circ \exp (\ad (-y\otimes g))(y_1\otimes g_1)$. Then we apply $\exp (\ad (y\otimes g))$ to it. We will use the following in our computations:
\begin{enumerate}[\upshape(i)]
\item Since $\del$ is a derivation of $\cO(1; \underline{1})$, we have that $\del(g^{i}g_1)=g^{i}\del(g_1)+ig^{i-1}\del(g)g_1$ for all $0\leq i\leq p-1$.
\item $\exp (\ad (y\otimes g))\exp (\ad (-y\otimes g))=\Id$.
\item Since $g\in \fm$, we have that $g^{p}=0$. 
\item $(p-1)!\equiv -1 (\modd p)$.
\end{enumerate} 
Let us compute $\del \circ \exp (\ad (-y\otimes g))(y_1\otimes g_1)$.
\begin{align*}
&\del \circ \exp (\ad (-y\otimes g))(y_1\otimes g_1)\\
=&\del\bigg(\sum_{i=0}^{p-1}\frac{1}{i!} (\ad (-y))^{i}(y_1)\otimes g^i g_1\bigg)\\
=&\sum_{i=0}^{p-1}\frac{1}{i!} (\ad (-y))^{i}(y_1)\otimes \bigg(g^{i}\del(g_1)+ig^{i-1}\del(g)g_1\bigg)\quad \text{(by (i))}\\
=&\exp (\ad (-y\otimes g))(y_1\otimes \del(g_1))+\bigg(\sum_{i=1}^{p-1}\frac{1}{i!} (\ad (-y))^{i}(y_1)\otimes ig^{i-1}\del(g)g_1\bigg)\\
=&\exp (\ad (-y\otimes g))(y_1\otimes \del(g_1))+\bigg(\sum_{i=1}^{p-1}\frac{1}{(i-1)!} (\ad (-y))^{i}(y_1)\otimes g^{i-1}\del(g)g_1\bigg)\\
=&\exp (\ad (-y\otimes g))(y_1\otimes \del(g_1))\\
&+\bigg(\ad (-y\otimes\del(g))\sum_{i=1}^{p-1}\frac{1}{(i-1)!} (\ad (-y\otimes g))^{i-1}\bigg)(y_1\otimes g_1)\\
=&\exp (\ad (-y\otimes g))(y_1\otimes \del(g_1))\\
&+\bigg(\ad (-y\otimes\del(g))\bigg(\exp(\ad (-y\otimes g))-\frac{1}{(p-1)!} (\ad (-y\otimes g))^{p-1}\bigg)\bigg)(y_1\otimes g_1).
\end{align*}
Applying $\exp (\ad (y\otimes g))$ to the above, we get
\begin{align*}
&\bigg(\exp (\ad (y\otimes g))\circ \del \circ \exp (\ad (-y\otimes g))\bigg)(y_1\otimes g_1)\\
=&y_1\otimes \del(g_1) \\
&+\bigg(\ad (-y\otimes\del(g))\bigg(\Id-\big(\frac{1}{(p-1)!}(\ad (-y\otimes g))^{p-1}\big)\exp (\ad (y\otimes g))\bigg)\bigg)(y_1\otimes g_1)\quad \\
&\quad \quad \quad \quad \quad \quad \quad \quad \quad \quad \quad \quad\quad \quad \quad \quad \quad \quad \quad \quad \quad \quad \quad \quad\quad \quad \quad \quad \quad \quad \quad \quad \quad \text{(by (ii))}\\
=&y_1\otimes \del(g_1)+\bigg(\ad (-y\otimes\del(g))\bigg(\Id+(\ad (-y\otimes g))^{p-1}\bigg)\bigg)(y_1\otimes g_1)\quad \text{(by (iii) and (iv))}\\
=&y_1\otimes \del(g_1)+\bigg(\ad (-y\otimes\del(g))+(\ad (-y\otimes\del(g))(\ad (-y\otimes g))^{p-1} \bigg)(y_1\otimes g_1)\\
=&y_1\otimes \del(g_1)+\big(\ad (-y\otimes\del(g))\big)(y_1\otimes g_1)+(\ad (-y\otimes\del(g))(\ad (-y\otimes g))^{p-1} (y_1\otimes g_1)\\
=&y_1\otimes \del(g_1)-\big(\ad (y\otimes \del(g))\big)(y_1\otimes g_1)+(-1)^p(\ad y)^p(y_1)\otimes g^{p-1}\del(g) (g_1)\\
=&y_1\otimes \del(g_1)-\big(\ad \big(y\otimes \del(g)\big)\big)(y_1\otimes g_1)-\big(\ad \big(y^{p}\otimes g^{p-1}\del(g)\big)\big)(y_1\otimes g_1).
\end{align*}
The last line follows from \thref{defres}(1) that $(\ad y)^p=\ad y^p$. Hence we get \eqref{expdelexpprecal} as desired. It follows that
\begin{equation}\label{edele-1a}
\begin{aligned}
&\exp (\ad (y\otimes g))\circ \del \circ \exp (\ad (-y\otimes g))\\
=&\del-\ad \big(y\otimes \del(g)\big)-\ad \big(y^{p}\otimes g^{p-1}\del(g)\big)\\
=&\del-y\otimes \del(g)-y^{p}\otimes g^{p-1}\del(g).
\end{aligned}
\end{equation}
The last line follows from that $\fsl_2\cong \ad (\fsl_2)$ via the adjoint representation and hence we may identify $y\otimes \del(g)$ (respectively $y^{p}\otimes g^{p-1}\del(g)$) with its image $\ad \big(y\otimes \del(g)\big)$ (respectively $\ad \big(y^{p}\otimes g^{p-1}\del(g)\big)$) in $\fgl(\fsl_2\otimes \cO(1; \underline{1}))$ under $\ad$. As a result, we see that $\exp (\ad (y\otimes g))\circ \del \circ \exp (\ad (-y\otimes g))\in \fg$. Therefore, $\exp (\ad (y\otimes g))$ preserves $\fg$ and it is an automorphism of $\fg$. 

Now let $v=z\otimes q$ be an element of $\cN(\fsl_2)\otimes \cO(1;1)$. Note that $\fsl_2$ contains a self-centralizing maximal torus, namely $kh$. Hence $e(\fsl_2)=0$. Also, $\MT(\fsl_2)=1$. It follows from \thref{nvarietythm} that any nilpotent element $z$ of $\fsl_2$ satisfies $z^{p}=0$. As a result, $\ad v$ is a derivation of $\fsl_2\otimes \cO(1; 1)$ such that $(\ad v)^{p}=\ad v^{p}=\ad (z^p\otimes q^p)=0$. Since $(\ad z)^{p}=0$, one can show similarly that $\exp (\ad v)$ is an automorphism of \\$\fsl_2\otimes \cO(1; 1)$. Then it induces an automorphism of the derivation algebra \\$\Der(\fsl_2\otimes \cO(1; 1))$ via conjugation. By \eqref{edele-1a}, we have that 
\begin{align*}
\exp (\ad (z\otimes q))\circ \del \circ \exp (\ad (-z\otimes q))=\del-z\otimes \del(q)\in \fg.
\end{align*}
Hence $\exp (\ad (z\otimes q))$ preserves $\fg$ and it is an automorphism of $\fg$. This completes the proof.
\end{proof}

\begin{lem}\thlabel{sl2Dlem}
\begin{enumerate}[\upshape(i)]
\item $\fg$ is semisimple.
\item The maximal dimension of tori in $\fg$ equals $\MT(\fsl_2\otimes \cO(1; 1))$ which is $1$.
\item Let $\fg_{ss}$ denote the set of all semisimple elements of $\fg$. Then
\[
\text{$1=e(\fg):=\min\big\{r\in \Z_{\geq 0} \,|\, W^{p^{r}} \subseteq \fg_{ss}$ for some nonempty Zariski open $W \subset \fg$\big\}}.
\]
\item All irreducible components of $\cN(\fg)$ have dimension $3p$. 
\item Any $a \in \cN(\fg)$ satisfies $a^{p^{2}}=0$.
\end{enumerate}
\end{lem}

\begin{proof}
(i) By \thref{sl2noD}(ii), if $\fg$ is not semisimple, then any nonzero abelian \\$p$-ideal would have to contain the unique minimal ideal $\fsl_2\otimes \cO(1; 1)$. But $\fsl_2\otimes \cO(1; 1)$ is nonabelian, we get a contradiction. This proves (i).

(ii) Since $\fsl_2\otimes \cO(1; 1)$ is a $p$-ideal of $\fg$ and $\del^{p}=0$, it follows from \thref{MTlem} and \thref{sl2noD}(iii) that $\MT(\fg)=\MT(\fsl_2\otimes \cO(1; 1))+\MT(k\del)=1+0=1$. This proves (ii).

(iii) By \thref{autg}, we know that $\exp(t\ad (z\otimes q))$ is an automorphism of $\fg$ for any $t\in k$ and $z\otimes q\in \cN(\fsl_2)\otimes \cO(1;1)$. Let $G$ denote the group generated by these $\exp(t\ad (z\otimes q))$. Then $G$ is a connected algebraic group. Consider the set $G.(h\otimes \cO(1; 1)^{*}\rtimes k\del)$. Here $\cO(1; 1)^{*}$ denotes the set of all invertible elements of $\cO(1; 1)$. We want to show that this set contains a nonempty Zariski open subset of $\fg$. This can be done by showing that the morphism 
\begin{align*}
\theta: G \times (h\otimes \cO(1;1)\rtimes k\del)&\to \fg\\
(\tilde{g},v)&\mapsto \tilde{g}(v) 
\end{align*}
is dominant. Consider the differential of $\theta$ at $(1_G, h\otimes 1)$
\begin{align*}
(d\theta)_{(1_G, h\otimes 1)}: \Lie(G) \oplus (h\otimes \cO(1;1)\rtimes k\del)&\to \fg\\
(X, Y)&\mapsto[X, h\otimes 1]+Y.
\end{align*}
It is easy to see that $\Ima ((d \theta)_{(1_G, h\otimes 1)})$ contains $h\otimes \cO(1;1)\rtimes k\del$ and $[\Lie(G), h\otimes 1]$. Since $\exp(t\ad(e\otimes q))$ is in $G$ for any $t\in k$ and $q \in \cO(1; 1)$, we have that $e\otimes \cO(1; 1)\subset \Lie(G)$. Here we identify $\fsl_2\otimes \cO(1; 1)$ with its image in $\fgl(\fsl_2\otimes \cO(1; 1))$ under $\ad$. By our assumption, $p>2$. So 
\[
[e\otimes \cO(1; 1), h\otimes 1]=-2e\otimes \cO(1;1)\subset [\Lie(G), h\otimes 1].
\]
Similarly, one can show that $f\otimes \cO(1;1)\subset [\Lie(G), h\otimes 1]$. Hence $\Ima ((d\theta)_{(1_G, h\otimes 1)})$ contains $h\otimes \cO(1; 1)\rtimes k\del+e\otimes \cO(1;1)+f\otimes \cO(1; 1)$ which is $\fg$. Thus, $\Ima ((d\theta)_{(1_G, h\otimes 1)})=\fg$ and so $(d\theta)_{(1_G, h\otimes 1)}$ is surjective. It follows from \thref{dcdominant} that the morphism $\theta$ is dominant.

Note that for any $g=\sum_{i=0}^{p-1}a_i x^i \in \cO(1,1)$ and $\lambda \in k$, 
\[
(h\otimes g +\lambda \del)^p = h\otimes (a_0^p-\lambda^{p-1}a_{p-1})
\]
by Jacobson's formula. It follows that there is a nonempty Zariski open subset $U_0=h\otimes O(1,1)^{*}\rtimes k\del$ in  $h\otimes O(1,1)\rtimes k\del$ consisting of elements $u_0$ such that $u_{0}^p$ is semisimple. By \thref{zopen}, $G.U_0$ contains a nonempty Zariski open subset $U$ of $\fg$. We have that $u^p\in \fg_{ss}$ for all $u\in U$. Therefore, $e(\fg)\le 1$. If $e(\fg)=0$, then it follows from the definition of $e(\fg)$ that $\fg_{ss}$ contains a nonempty Zariski open subset $V$ of $\fg$. Then there is an element $v \in V\setminus \fsl_2\otimes \cO(1,1)$ such that $v\in \spn\big\{v^{p^{i}}\,|\,i\ge 1\big\}$. This is impossible as $w^{p^{i}}\in \fsl_2\otimes O(1,1)$ for all $w \in \fg$. Hence $e(\fg)=1$. This proves (iii).

(iv) By (ii) of this lemma and \thref{nvarietythm}, all irreducible components of $\cN(\fg)$ have dimension $\dim\fg-\MT(\fg)=3p$. This proves (iv).

(v) By \thref{nvarietythm} and the fact that $e(\fg)=1=\MT(\fg)$, we have that $a^{p^{2}}=0$ for all $a\in \cN(\fg)$. This proves (v).
\end{proof}

We need another result which shows that applying suitable automorphisms of $\fg$, some nilpotent elements of $\fg$ can be reduced to a nice form.

\begin{lem}\thlabel{autg1}
Let $a=\lambda \del+v$ be an element of $\fg$, where $\lambda \in k^{*}$ and $v \in \fsl_2\otimes \cO(1; 1)$. Then there exist automorphisms of $\fg$ such that $a$ is conjugate to $a'=\lambda \del+v'$, where $v'\in \fsl_2 \otimes \fm^{p-1}$. Moreover, if $a$ is nilpotent, then $v'=b\otimes x^{p-1}$ for some $b\in \cN(\fsl_2)$ (possibly $0$).
\end{lem}

\begin{proof}
Let $a=\lambda \del+v$ be an element of $\fg$, where $\lambda \in k^{*}$ and $v \in \fsl_2\otimes \cO(1; 1)$. Suppose $v=z_1\otimes \sum_{i=0}^{p-1}a_ix^{i}$ for some $z_1 \in \fsl_2$ and $a_i \in k$. By \thref{autg}, we know that $\exp(\ad (y\otimes g))$ with $y\otimes g\in \fsl_2\otimes \fm$ is an automorphism of $\fg$. Let $H$ denote the subgroup of $\Aut(\fg)$ generated by these $\exp(\ad (y\otimes g))$. Then $H$ is a connected algebraic group. We show that choosing suitable element $y\otimes g\in \fsl_2\otimes \fm$, $\big(\exp(\ad (y\otimes g))\big)(a)=\lambda \del+v_1$ for some $v_1 \in \fsl_2\otimes \fm$. If $a_0=0$, then $a$ is of the desired form and there is nothing to do. Suppose $a_0\neq 0$. By \eqref{edele-1a},
\begin{align*}
\big(\exp(\ad (y\otimes g))\big)(a)&=\exp(\ad (y\otimes g))\circ\lambda \del\circ\exp(\ad (-y\otimes g))+\big(\exp(\ad (y\otimes g))\big)(v)\\
=&\lambda \del -\lambda y\otimes \del(g)-\lambda y^{p}\otimes g^{p-1}\del(g)\\
&+z_1\otimes \sum_{i=0}^{p-1}a_ix^{i}+\sum_{j=1}^{p-1}\frac{1}{j!}(\ad y)^{j}(z_1)\otimes g^{j}\sum_{i=0}^{p-1}a_ix^{i}.
\end{align*} 
Choose $y=z_1/\lambda\in\fsl_2$ and $g=a_0x \in \fm$. Then 
\[
-\lambda y\otimes \del(g)+z_1\otimes a_0=-z_1\otimes a_0+z_1\otimes a_0=0,
\]
\[-\lambda y^{p}\otimes g^{p-1}\del(g)= -\lambda y^{p}\otimes a_0^{p}x^{p-1},
\] 
and 
\[
\sum_{j=1}^{p-1}\frac{1}{j!}(\ad y)^{j}(z_1)\otimes g^{j}\sum_{i=0}^{p-1}a_ix^{i}=0.
\]
Hence 
\begin{align*}
\big(\exp(\ad (y\otimes g))\big)(a)&=\lambda \del-\lambda y^{p}\otimes a_0^{p}x^{p-1}+z_1\otimes \sum_{i=1}^{p-1}a_ix^{i}.
\end{align*}
Since $-\lambda y^{p}\otimes a_0^{p}x^{p-1}+z_1\otimes \sum_{i=1}^{p-1}a_ix^{i}\in \fsl_2\otimes \fm$, we see that $\big(\exp(\ad (y\otimes g))\big)(a)=\lambda \del+v_1$ for some $v_1 \in \fsl_2\otimes \fm$. Continuing in this way and by induction we get $a$ is conjugate under $H$ to $\lambda \del+v_{p-2}$ for some $v_{p-2} \in \fsl_2 \otimes \fm^{p-2}$. Suppose 
\[
v_{p-2}=z_{p-2} \otimes (a_{p-2}x^{p-2}+a_{p-1}x^{p-1})
\]
for some $z_{p-2} \in \fsl_2$ and $a_{p-2}, a_{p-1} \in k$. If $a_{p-2}=0$, then $v_{p-2} \in \fsl_2 \otimes \fm^{p-1}$ and there is nothing to do. Suppose $a_{p-2}\neq 0$. Then for any $\exp (\ad (y\otimes g)) \in H$, we have that
\begin{align*}
&\big(\exp (\ad (y\otimes g))\big)(\lambda \del+v_{p-2})\\
=&\exp (\ad (y\otimes g))\circ \lambda \del\circ\exp (\ad (-y\otimes g))+\big(\exp (\ad (y\otimes g))\big)(v_{p-2})\\
=&\lambda \del -\lambda y\otimes \del(g)-\lambda y^{p}\otimes g^{p-1}\del(g)\\
&+z_{p-2}\otimes (a_{p-2}x^{p-2}+a_{p-1}x^{p-1})+\sum_{j=1}^{p-1}\frac{1}{j!}(\ad y)^{j}(z_{p-2})\otimes g^{j}(a_{p-2}x^{p-2}+a_{p-1}x^{p-1}).
\end{align*}
Choose $y=z_{p-2}/\lambda \in \fsl_2$ and $g=a_{p-2}x^{p-1}/(p-1) \in \fm$. Then 
\[
-\lambda y\otimes \del(g)+z_{p-2}\otimes a_{p-2}x^{p-2}=-z_{p-2}\otimes a_{p-2}x^{p-2}+z_{p-2}\otimes a_{p-2}x^{p-2}=0.
\]
Moreover, $-\lambda y^{p}\otimes g^{p-1}\del(g)=0$, and 
\[
\sum_{j=1}^{p-1}\frac{1}{j!}(\ad y)^{j}(z_{p-2})\otimes g^{j}(a_{p-2}x^{p-2}+a_{p-1}x^{p-1})=0.
\]
Hence
\[
\big(\exp (\ad (y\otimes g))\big)(\lambda \del+v_{p-2})=\lambda \del+z_{p-2}\otimes a_{p-1}x^{p-1}.
\] 
Therefore, we proved that if $a=\lambda \del+v \in \fg$ with $\lambda \in k^{*}$ and $v \in \fsl_2\otimes \cO(1; 1)$, then $a$ is conjugate under $H$ to $a'=\lambda \del+v'$, where $v'\in \fsl_2 \otimes \fm^{p-1}$.

By above, we may assume that $a=\lambda \del+b\otimes x^{p-1}$ for some $\lambda \in k^{*}$ and $b \in \fsl_2$. If $a$ is nilpotent, then \thref{sl2Dlem}(v) implies that $a^{p^{2}}=0$. By Jacobson's formula,  $a^{p} =(p-1)!\lambda^{p-1}b$. So $a^{p^{2}}=0$ implies that $b^p=0$. As a result, $b \in \cN(\fsl_2)$. This completes the proof.  
\end{proof}

We are now ready to prove that
\begin{thm}\thlabel{sl2thm}
The variety $\cN(\fg)$ is irreducible.
\end{thm}

\begin{proof}
Let $X$ be any irreducible component of $\cN(\fg)$. Put $X_{0}:=X\cap (\fsl_2\otimes \cO(1; 1))$. Then $X_0\subset \cN(\fsl_2\otimes \cO(1; 1))$. By \thref{sl2noD}(iv) and \thref{sl2Dlem}(iv), we have that $\dim \cN(\fsl_2\otimes \cO(1; 1))=3p-1$ and $\dim X=3p$. It follows that $X_0$ is a proper Zariski closed subset of $X$. Then the complement $X\setminus X_0$ is Zariski open and nonempty in $X$. As a result, $X$ equals the Zariski closure of $X\setminus X_0$. 

Let us describe $X\setminus X_0$ explicitly. If $a \in X\setminus X_0$, then $a=\lambda \del+v$ for some $\lambda \in k^{*}$ and $v\in \fsl_2 \otimes \cO(1; 1)$. By \thref{autg1}, we know that there exists a subgroup $H$ of $\Aut(\fg)$ generated by all $\exp(\ad u)$ with $u \in \fsl_2\otimes \fm$ such that $a$ is $H$-conjugate to $\lambda \del+b \otimes x^{p-1}$ for some $b \in \cN(\fsl_2)$ (possibly $0$). Since $\dim H=\dim \Lie (H)=\dim (\fsl_2\otimes \fm)=3(p-1)$, $X\setminus X_0$ is not contained in the union of $H.(\lambda \del)$. As a result, there are some $\lambda \del+b\otimes x^{p-1}$ with $b\neq 0$ in $X\setminus X_0$.

Let $\tilde{G}=\Aut(\fsl_2)H$. Note that $H$ is a normal subgroup of $\tilde{G}$ and $\tilde{G}\subset \Aut(\fg)$ is connected. Moreover, it is known that all nonzero nilpotent elements of $\fsl_2$ are conjugate under $\Aut (\fsl_2)$. By \thref{sl2automorphism}, we know that $\Aut (\fsl_2)$ fixes $\del$. Hence any $a \in X\setminus X_0$ is conjugate under $\tilde{G}$ to $\lambda \del+\mu e\otimes x^{p-1}$ for some $\lambda \in k^{*}$ and $\mu \in k$. This shows that $X\setminus X_0 \subseteq \tilde{G}.(\lambda \del+\mu e\otimes x^{p-1}\,|\, \lambda \in k^{*}, \mu \in k)$. It is clear that $\tilde{G}.(\lambda \del+\mu e\otimes x^{p-1}\,|\, \lambda \in k^{*}, \mu \in k) \subseteq X\setminus X_0$. Thus $X\setminus X_0=\tilde{G}.(\lambda \del+\mu e\otimes x^{p-1}\,|\, \lambda \in k^{*}, \mu \in k)$. This quasi-affine variety is the image of a morphism from the irreducible variety $\tilde{G}\times k^{*}\times k$, hence $X\setminus X_{0}$ is irreducible.  

The above result holds for any irreducible component $X$ of $\cN(\fg)$ and shows that $X\setminus X_{0}$ is independent of the choice of $X$. Therefore, there is only one $X$ and so $\cN(\fg)$ is irreducible. This completes the proof.
\end{proof}

\begin{rmk}
\begin{enumerate}
\item We see that the above proof relies heavily on the structure and automorphisms of $\fg$, and Premet's theorem on $\cN(\fg)$ (\thref{nvarietythm}). Since we are familiar with the special linear Lie algebra $\fsl_2$ and the truncated polynomial ring $\cO(1; 1)$, it is easy to do calculations using Jacobson's formula and construct such a proof. But for other semisimple restricted Lie algebras with little known structural information, it may be more difficult to verify Premet's conjecture.
\item Another useful thing to point out is that the above proof works if we replace $\fsl_2$ by any Lie algebra $\fg_2=\Lie (G_2)$, where $G_2$ is a reductive algebraic group. This is because 
\begin{enumerate}[\upshape(i)]
\item $\fg_2$ is a restricted Lie algebra by \thref{exres}, so is $(\fg_2 \otimes \cO(1; 1)) \rtimes k\del$.
\item One can show similarly that $\MT(\fg_2 \otimes \cO(1; 1))=\MT \big((\fg_2 \otimes \cO(1; 1)) \rtimes k\del\big)$.\\ Hence $\dim \cN(\fg_2\otimes \cO(1; 1))<\dim \cN\big((\fg_2\otimes \cO(1; 1)) \rtimes k\del\big)$.
\item We know that $\cN(\fg_2)$ is irreducible and it is the Zariski closure of a single nilpotent $G_2$-orbit. Hence all nonzero nilpotent elements of $\fg_2$ are conjugate under $G_2$; see \cite[Theorem 1, Sec.~2.8 and Sec.~6.3-6.4]{J04}.
\end{enumerate}
\end{enumerate}
\end{rmk}

\section{Socle involves $S$}\label{JacobWittinfo}
We would like to extend the previous example by replacing $\cO(1; 1)$ with the truncated polynomial ring $\cO(m; \underline{1})$ in $m$ variables. Then we need to replace $\del$. A recent result proved by A.~Premet and D.~I.~Stewart shows that $\del$ is one of the representatives of transitive subalgebras of the derivation algebra $\Der \cO(1; 1)$ under the action of $\Aut(\cO(1; 1))$; see \cite[Lemma 2.1]{PD17}. This suggests that a transitive subalgebra of $\Der \cO(m; \underline{1})$ will do the job. Let us properly explain the set up.

Assume that $k$ is an algebraically closed field of characteristic $p>2$. Suppose $m\geq 2$. Let $\cO(m; \underline{1})=k[X_1, \dots, X_m]/(X_1^{p}, \dots, X_{m}^{p})$ denote the truncated polynomial ring in $m$ variables. We write $x_{i}$ for the image of $X_{i}$ in $\cO(m; \underline{1})$. Note that $\cO(m; \underline{1})$ is a local ring with $\dim \cO(m; \underline{1}) =p^m$. The degree function on the polynomial ring $k[X_1, \dots, X_m]$ induces a grading on $\cO(m; \underline{1})$. For each $i\in \N_0$, set 
\begin{equation}\label{Omgrading2.2.2}
\cO(m; \underline{1})_{i}:=\{f \in \cO(m; \underline{1})\,|\, \deg f =i\}.
\end{equation}
Then $\cO(m; \underline{1})=\bigoplus_{i=0}^{m(p-1)}\cO(m; \underline{1})_{i}$. For each $j\in \N_0$, set 
\[
\cO(m; \underline{1})_{(j)}:=\bigoplus_{i\geq j}\cO(m; \underline{1})_{i}=\{f \in \cO(m; \underline{1})\,|\, \deg f \geq j\}.
\]
Then the $\Z$-grading on $\cO(m; \underline{1})$ induces a descending filtration on $\cO(m; \underline{1})$. Note that each $\cO(m; \underline{1})_{(j)}$ is an ideal of $\cO(m;\underline{1})$. If $j>m(p-1)$, then $\cO(m; \underline{1})_{(j)}=0$. The unique maximal ideal of $\cO(m; \underline{1})$, denoted $\fm$, is $\cO(m; \underline{1})_{(1)}$. Since $f^{p}=f(0)^p$ for any $f\in \cO(m; \underline{1})$, we have that $f^{p}=0$ for all $f\in\fm$.

Let $W(m;\underline{1})$ denote the derivation algebra of $\cO(m; \underline{1})$. It is called the \textit{$m$th Jacobson-Witt algebra}. By \thref{wittthm}, this restricted Lie algebra is simple. Note that any derivation $\ccD$ of $\cO(m;\underline{1})$ is uniquely determined by its effects on $x_1, \dots, x_m$. 
It is easy to see that $W(m;\underline{1})$ is a free $\cO(m; \underline{1})$-module of rank $m$ generated by the partial derivatives $\del_1,\dots, \del_m$ such that $\del_i(x_j)=\delta_{ij}$ for all $1\leq i, j\leq m$. Hence $\dim W(m; \underline{1})=mp^m$. Put
\begin{align}\label{wmgrading2.2.2}
W(m;\underline{1})_{l}:=\sum_{i=1}^{m}\cO(m;\underline{1})_{l+1}\del_i
\end{align}
for $-1\leq l\leq m(p-1)-1$. Then the $\Z$-grading on $\cO(m; \underline{1})$ induces a $\Z$-grading on $W(m; \underline{1})$, i.e.
\[
W(m; \underline{1})=W(m; \underline{1})_{-1}\oplus W(m; \underline{1})_{0}\oplus\dots\oplus W(m; \underline{1})_{m(p-1)-1}.
\]
Note that $W(m; \underline{1})_{-1}=\sum_{i=1}^{m}k\del_i$. Similarly, put 
\begin{align}\label{wmfiltration2.2.2}
W(m;\underline{1})_{(l)}:=\sum_{i=1}^{m}\cO(m;\underline{1})_{(l+1)}\del_i
\end{align}
for $-1\leq l\leq m(p-1)-1$. Then the natural filtration on $\cO(m;\underline{1})$ induces a descending filtration on $W(m;\underline{1})$, i.e.
\begin{align*}
W(m;\underline{1})=W(m;\underline{1})_{(-1)}\supset W(m;\underline{1})_{(0)}\supset\dots\supset W(m;\underline{1})_{(m(p-1)-1)}\supset 0.
\end{align*}
It is called the \textit{standard filtration}; see \thref{sfiltrations}. Note that each $W(m;\underline{1})_{(l)}$ is a Lie subalgebra of $W(m;\underline{1})$. The subalgebra $W(m;\underline{1})_{(0)}=\sum_{i=1}^{m}\fm\del_i$ is often referred to as the \textit{standard maximal subalgebra} of $W(m;\underline{1})$. This is because it can be characterized as the unique proper subalgebra of smallest codimension in $W(m;\underline{1})$. By \thref{W(m;n)_(0)restricted}, we know that $W(m;\underline{1})_{(0)}$ is a restricted subalgebra of $W(m; \underline{1})$. Moreover, for any $\ccD_1\in W(m; \underline{1})_{(0)}$, it is easy to check that $\ccD_1(\cO(m; \underline{1}))\subseteq \fm$. It is also easy to see that $W(m;\underline{1})_{(1)}=\sum_{i=1}^{m}\fm^{2}\del_i$ is the nilradical of $W(m;\underline{1})_{(0)}$.

A restricted subalgebra $Q'$ of $W(m; \underline{1})$ is called \textit{transitive} if it does not preserve any proper nonzero ideals of $\cO(m; \underline{1})$. Equivalently, $Q'+ W(m; \underline{1})_{(0)}=W(m; \underline{1})$; see \cite[Sec.~2.1]{PD17} and \cite[Definition 2.3.1]{S04}. 

Let $G$ denote the automorphism group of $\cO(m; \underline{1})$. Each $\sigma\in G$ is uniquely determined by its effects on $x_1, \dots, x_m$. An assignment $\sigma(x_i)=f_i $ extends to an automorphism of $\cO(m; \underline{1})$ if and only if $f_i\in \fm$, and the Jacobian $\Jac(f_1, \dots, f_m):=\big|\big(\frac{\del f_i}{\del x_j}\big)_{1\leq i, j\leq m}\big|\notin \fm$. It follows that $G$ is a connected algebraic group of dimension $mp^m-m$. By \cite[Theorem 10]{J43}, every automorphism of $W(m; \underline{1})$ is induced by a unique automorphism of $\cO(m; \underline{1})$ via the rule $\ccD^{\sigma}=\sigma \circ \ccD\circ \sigma^{-1}$ for all $\sigma \in G$ and $\ccD \in W(m; \underline{1})$. So we can identify $G$ with the automorphism group of $W(m; \underline{1})$. Note that for any $f\in \cO(m; \underline{1})$, $\ccD\in W(m; \underline{1})$ and $\sigma \in G$, 
\[
(f\ccD)^{\sigma}=f^{\sigma}\ccD^{\sigma},
\]
where $f^{\sigma}=f(\sigma(x_1), \dots, \sigma(x_m))$; see \eqref{autconj1} in Sec.~\ref{nilpotentvarietyresultsWn}. It follows that if $\ccD_2=\sum_{i=1}^{m}g_i\del_i\in W(m; \underline{1})$, then 
\[
\ccD_2^{\sigma}=\sum_{i,\, j=1}^{m}g_i^{\sigma}\bigg(\del_i\big(\sigma^{-1}(x_j)\big)\bigg)^{\sigma}\del_j;
\]
see \eqref{autoconjugationrule2} in Sec.~\ref{nilpotentvarietyresultsWn}. Note also that $G$ respects the standard filtration of $W(m; \underline{1})$, and $\Lie(G)=W(m; \underline{1})_{(0)}$. 

Let us introduce the Lie algebra that we are going to work with. Let $\dd{D}$ be any restricted transitive subalgebra of $W(m; \underline{1})$ such that $\cN(\dd D)$ is irreducible. Let $S$ be any simple restricted Lie algebra over $k$ such that all its derivations are inner. We assume further that $\cN(S)$ is irreducible. It is natural to form the semidirect product
\[
\cL:=(S\otimes \cO(m; \underline{1}))\rtimes (\Id_S\otimes \dd D).
\]
Then $\cL$ is known to be a semisimple restricted Lie algebra over $k$; see also \cite[Sec.~2.1]{PD17}. To ease notation we identify $\Id_S\otimes \dd D$ with $\dd D$. 

Let us explain the reason that we assumed that $\dd D$ is a transitive subalgebra of $W(m; \underline{1})$. Note that if $\dd{D}$ is not transitive, then $\cO(m; \underline{1})$ contains a proper nonzero $\dd{D}$-invariant ideal, say $\fn$. Since $\cO(m; \underline{1})$ is a local ring, this ideal is contained in the maximal ideal $\fm$ of $\cO(m; \underline{1})$. Hence $f^{p}=0$ for all $f\in \fn$. Now consider $S\otimes \fn$, a nonzero subspace of $\cL$. Since $\cO(m; \underline{1})\fn \subseteq \fn$, we have that $[S\otimes \cO(m; \underline{1}), S\otimes \fn]\subseteq S\otimes \fn$. Since $\fn$ is $\dd{D}$-invariant, it follows that $S\otimes \fn$ is a nonzero ideal of $\cL$ consisting of nilpotent elements of $\cL$. Hence $\cL$ is not semisimple and we are not interested in such $\cL$ in this chapter. Therefore, for $\cL$ to be semisimple, it is necessary that $\dd{D}$ is transitive. This is the reason that we assumed that $\dd D$ is transitive when we introduced the Lie algebra $\cL$.

We want to prove \thref{NLirr} which states that the variety $\cN(\cL)$ is irreducible. We will consider the surjective morphism $\tilde{\psi}: \cN(\cL)\to \cN(\dd D)$ and need some preliminary results. Let us outline these results in each section.

In Sec.~\ref{nileleminND} we consider nilpotent elements of $\dd D$. Since $\dd D$ is a restricted transitive subalgebra of $W(m; \underline{1})$, there exists $d\in \dd D$ such that $d\notin W(m; \underline{1})_{(0)}$. If $d$ is also nilpotent, then we show that $d$ is conjugate under $G=\Aut(W(m; \underline{1}))$ to an element in a nice form. For that, we need a few results. Thanks to \thref{D70Lem6part} in Sec.~\ref{nilpotentvarietyresultsWn} and \thref{dformlemma1} which state that any $z\in W(m; \underline{1})$ such that $z\notin W(m; \underline{1})_{(0)}$ is conjugate under $G$ to an element in a nice form. Then we show that any such nilpotent element $z$ is conjugate under $G$ to $d_0+u$, where $d_0=\del_1+x_1^{p-1}\del_2+\dots+x_1^{p-1}\cdots x_{s-1}^{p-1}\del_s$ with $1\leq s\leq m$, $u\in (I\del_1+\dots+I\del_m)\cap W(m; \underline{1})_{(p-1)}$ and $I$ is the ideal of $\cO(m; \underline{1})$ generated by $x_{s+1}, \dots, x_m$; see \thref{dformlemma2} for notations. For the proof, we need \thref{regularderithm} which was proved by A.~Premet in \cite{P2015}. Then we show that $(d_0+u)^{p^{s}}\in W(m; \underline{1})_{(0)}$; see \thref{zpscalculationslemmma}. We show in \thref{GpreservesND} that any $d\in\cN(\dd D)$ such that $d\notin W(m; \underline{1})_{(0)}$ is conjugate under $G$ to $d_0+u$. Moreover, $d^{p^{s}}\in W(m; \underline{1})_{(0)}$ by \thref{zpscalculationslemmma}. We shall denote $d_0+u\in \cN(\dd D)$ by $z$.

Next we define a monomial ordering DegLex on $\cO(m; \underline{1})$; see \thref{DegLexdefn}. Then we 
consider the subspace 
\[
M=I+z(\fm),
\] 
where $\fm$ is the maximal ideal of $\cO(m; \underline{1})$, $I$ is the ideal of $\cO(m; \underline{1})$ generated by $x_{s+1}, \dots, x_m$ with $s\geq 1$, and $z=d_0+u$ as above. We show $M$ has the property that $M\oplus kx_1^{p-1}\cdots x_s^{p-1}=\cO(m; \underline{1})$; see \thref{zcalculations} and \thref{d-d0d_0image}.

In Sec.~\ref{nileleminNL} we consider nilpotent elements of $\cL$. Let $D'=\sum_{i=0}^{m(p-1)}s_i\otimes f_i+d$ be an element of $\cL$, where $s_i\in S$, $f_i\in \cO(m;\underline{1})$ with $\deg f_i=i$, and $d\in \cN(\dd D)$ with $d\in W(m; \underline{1})_{(0)}$. In \thref{D1pNcalculations}, we compute $p$-th powers of $D'$ and show that $D'\in \cN(\cL)$ if and only if $s_0\in \cN(S)$. 

We want to prove a similar result for the element $D=\sum_{i=0}^{m(p-1)}s_i\otimes f_i+z$ of $\cL$, where $s_i\in S$, $f_i\in \cO(m;\underline{1})$ with $\deg f_i=i$, and $z=d_0+u \in \cN(\dd D)$ as above. We first construct some automorphisms of $\cL$, i.e. we show that $\exp(\ad(\tilde{s}\otimes f))$ with $\tilde{s}\otimes f \in S \otimes \fm$ is an automorphism of $\cL$; see \thref{Sautoexps}. Let $H$ denote the subgroup of $\Aut(\cL)$ generated by these $\exp(\ad(\tilde{s}\otimes f))$. Then we show that $D=\sum_{i=0}^{m(p-1)}s_i\otimes f_i+z$ is conjugate under $H$ to 
\begin{align*}
D_1=s_0'\otimes x_1^{p-1}\cdots x_{s}^{p-1}+v'+z, 
\end{align*}
where $s_0'\in S$ (possibly $0$), $v'\in S\otimes I$, $I$ is the ideal of $\cO(m; \underline{1})$ generated by $x_{s+1}, \dots, x_m$ with $s\geq 1$, and $z=d_0+u\in \cN(\dd D)$ as above; see \thref{DconjugatetoD1}. Then $D$ is nilpotent if and only if $D_1$ is nilpotent. In \thref{D1psnilpotentcalculations}, we show that $D_1\in \cN(\cL)$ if and only if $s_0'\in \cN(S)$.

In Sec.~\ref{socleSmainthmproof}, we prove that the variety $\cN(\cL)$ is irreducible.

\subsection{Nilpotent elements of $\dd D$}\label{nileleminND}
Since $\dd D$ is a restricted transitive subalgebra of $W(m; \underline{1})$, there are elements in $\dd D$ which are not in $W(m; \underline{1})_{(0)}$. In this section we consider these elements. If they are nilpotent, we show that they are conjugate under $G=\Aut(W(m; \underline{1}))$ to an element in a nice form; see \thref{GpreservesND}. For that, we start with elements of $W(m; \underline{1})$ which are not in $W(m; \underline{1})_{(0)}$.

\begin{lem}\cite[p.~154, line 17]{P91}\thlabel{dformlemma1}
Let $z$ be an element of $W(m; \underline{1})$ such that \\$z\notin W(m; \underline{1})_{(0)}$. Then there exists an $1\leq s\leq m$ such that $z$ is conjugate under $G$ to 
\begin{align*}
\sum_{i=1}^{s}x_1^{p-1}\cdots x_{i-1}^{p-1}(1+x_i^{p-1}\psi_i)\del_i+x_1^{p-1}\cdots x_s^{p-1}\sum_{i=s+1}^{m}\psi_i\del_i,
\end{align*}
where $\psi_i\in k[X_{i+1}, \dots, X_{m}]/(X_{i+1}^{p}, \dots, X_m^{p})$ for $1\leq i\leq s$ and $\psi_i\in \fm_{m-s}$, the maximal ideal of $k[X_{s+1}, \dots, X_{m}]/(X_{s+1}^{p}, \dots, X_m^{p})$, for $s+1\leq i\leq m$. 
\end{lem}
\begin{myproof}[Sketch of proof]
Recall that if $\sigma\in G$ and $\ccD=\sum_{i=1}^{m}g_i\del_i\in W(m; \underline{1})$, then 
\begin{equation}\label{lem3.2.1eq}
\ccD^{\sigma}=\sigma \circ \ccD\circ \sigma^{-1}=\sum_{i,\, j=1}^{m}g_i^{\sigma}\big(\del_i(\sigma^{-1}(x_j))\big)^{\sigma}\del_j,
\end{equation}
where $g_i^{\sigma}=g_i(\sigma(x_1), \dots, \sigma(x_m))$; see \eqref{autconj1} and \eqref{autoconjugationrule2} in Sec.~\ref{nilpotentvarietyresultsWn}.

Let $z$ be an element of $W(m; \underline{1})$ such that $z\notin W(m; \underline{1})_{(0)}$. By \thref{D70Lem6part}, there exists $\Phi_1 \in G$ such that
\[
z^{\Phi_1}=\del_1+x_1^{p-1}\sum_{i=1}^{m}\varphi_i\del_i=(1+x_1^{p-1}\varphi_1)\del_1+x_1^{p-1}\sum_{i=2}^{m}\varphi_i\del_i,
\]
where $\varphi_i\in k[X_2,\dots, X_{m}]/(X_2^{p}, \dots, X_{m}^{p})$ for all $1\leq i\leq m$. Set $z_2=\sum_{i=2}^{m}\varphi_i\del_i$. If $z_2\in W(m; 1)_{(0)}$, then we find an $s=1$ such that $z^{\Phi_1}$ is of the desired form. If $z_2\notin W(m; 1)_{(0)}$, then \thref{D70Lem6part} implies that there exists $\Phi_2 \in G$ such that
\[
z_2^{\Phi_2}=(1+x_2^{p-1}\varphi'_{2})\del_2+x_2^{p-1}\sum_{i=3}^{m}\varphi'_{i}\del_i,
\]
where $\varphi'_{i} \in k[X_3, \dots, X_m]/(X_3^{p}, \dots, X_m^{p})$ for all $2\leq i\leq m$. Here we may assume that $\Phi_2(x_1)=x_1$ and $\Phi_2(x_j)\in  k[X_2, \dots, X_m]/(X_2^{p}, \dots, X_m^{p})\cap \fm$, i.e. the intersection of $k[X_2, \dots, X_m]/(X_2^{p}, \dots, X_m^{p})$ with the maximal ideal $\fm$ of $\cO(m; \underline{1})$, for $2\leq j\leq m$. By \eqref{lem3.2.1eq}, it is easy to check that
\[
z^{\Phi_1\Phi_2}=(z^{\Phi_1})^{\Phi_2}=(1+x_1^{p-1}\varphi_1^{\Phi_2})\del_1+x_1^{p-1}(1+x_2^{p-1}\varphi'_{2})\del_2+x_1^{p-1}x_2^{p-1}\sum_{i=3}^{m}\varphi'_{i}\del_i.
\]
Set $z_3=\sum_{i=3}^{m}\varphi'_{i}\del_i$. If $z_3\in W(m; 1)_{(0)}$, then we find an $s=2$ such that  $z^{\Phi_1\Phi_2}$ is of the desired form. If $z_3\notin W(m; 1)_{(0)}$, then \thref{D70Lem6part} implies that there exists $\Phi_3 \in G$ such that
\[
z_3^{\Phi_3}=(1+x_3^{p-1}\varphi''_{3})\del_3+x_3^{p-1}\sum_{i=4}^{m}\varphi''_{i}\del_i,
\]
where $\varphi''_{i} \in k[X_4, \dots, X_m]/(X_4^{p}, \dots, X_m^{p})$ for all $3\leq i\leq m$. Here we may assume that $\Phi_3(x_1)=x_1$, $\Phi_3(x_2)=x_2$ and $\Phi_3(x_j)\in  k[X_3, \dots, X_m]/(X_3^{p}, \dots, X_m^{p})\cap \fm$ for $3\leq j\leq m$. By \eqref{lem3.2.1eq}, it is easy to check that
\begin{align*}
z^{\Phi_1\Phi_2\Phi_3}=(z^{\Phi_1\Phi_2})^{\Phi_3}=&(1+x_1^{p-1}(\varphi_1^{\Phi_2})^{\Phi_3})\del_1+x_1^{p-1}(1+x_2^{p-1}(\varphi_{2}')^{\Phi_3})\del_2\\
&+x_1^{p-1}x_2^{p-1}(1+x_3^{p-1}\varphi''_{3})\del_3+x_1^{p-1}x_2^{p-1}x_3^{p-1}\sum_{i=4}^{m}\varphi''_{i}\del_i.
\end{align*}
Continue doing the above process until we find an $1\leq s\leq m$ and $\Phi_1, \dots, \Phi_s\in G$ such that 
\begin{align*}
z^{\Phi_1\Phi_2\Phi_3\dots\Phi_s}=&(1+x_1^{p-1}\psi_1)\del_1+x_1^{p-1}(1+x_2^{p-1}\psi_2)\del_2
+x_1^{p-1}x_2^{p-1}(1+x_3^{p-1}\psi_3)\del_3+\dots\\
&+x_1^{p-1}\cdots x_{s-1}^{p-1}(1+x_s^{p-1}\psi_s)\del_s+x_1^{p-1}\cdots x_s^{p-1}\sum_{i=s+1}^{m}\psi_i\del_i\\
=&\sum_{i=1}^{s}x_1^{p-1}\cdots x_{i-1}^{p-1}(1+x_i^{p-1}\psi_i)\del_i+x_1^{p-1}\cdots x_s^{p-1}\sum_{i=s+1}^{m}\psi_i\del_i,
\end{align*}
where $\psi_i\in k[X_{i+1}, \dots, X_{m}]/(X_{i+1}^{p}, \dots, X_m^{p})$ for $1\leq i\leq s$ and $\psi_i\in \fm_{m-s}$, the maximal ideal of $k[X_{s+1}, \dots, X_{m}]/(X_{s+1}^{p}, \dots, X_m^{p})$, for $s+1\leq i\leq m$. Note that if $s=m$, then the above process shows that $z$ is $G$-conjugate to $\sum_{i=1}^{m}x_1^{p-1}\cdots x_{i-1}^{p-1}(1+x_i^{p-1}\psi_i)\del_i$, where $\psi_i\in k[X_{i+1}, \dots, X_{m}]/(X_{i+1}^{p}, \dots, X_m^{p})$ for $1\leq i\leq m$. Note that $\psi_m=0$. This completes the sketch of proof.
\end{myproof}

Next we assume that $z$ is a nilpotent element of $W(m; \underline{1})$ such that $z\notin W(m; \underline{1})_{(0)}$. We want to prove that $z$ is conjugate under $G$ to an element in a nicer form; see \thref{dformlemma2}. For that, we need \thref{regularderithm} which was proved by A.~Premet in \cite{P2015}. This theorem characterizes all regular elements of $W(m; \underline{1})$. To state that theorem, we need to introduce some notations used in \cite[Sec.~2 and 3]{P2015}. Then we state a few preliminary results and give a sketch proof of that theorem.

Let $\fg$ be any finite dimensional restricted Lie algebra over $k$. Given $x\in \fg$, let $\fg_{x}^{0}$ denote the set of all $y\in \fg$ for which $(\ad x)^{N}(y)=0$, where $N\gg 0$. It is known that $\fg_{x}^{0}$ is a restricted subalgebra of $\fg$ containing the centralizer $\fc_{\fg}(x)$. Define $\rk(\fg):=\min_{x\in\fg}\dim\fg_{x}^{0}$. We say that $x\in \fg$ is \textit{regular} if its centralizer $\fc_{\fg}(x)$ has the smallest possible dimension. Note that $\dim\fg_{x}^{0}=\rk(\fg)$ if and only if $x$ is regular and in this case $\fg_{x}^{0}$ is a Cartan subalgebra of minimal dimension in $\fg$. It is known that the set of all regular elements of $\fg$ is Zariski open in $\fg$. We consider the restricted Lie algebra $W(m; \underline{1})$. In \cite{P91}, A.~Premet proved that 
\begin{lem}\cite[Lemma 9]{P91}\thlabel{SingP91}
The set of singular elements of $\cN(W(m; \underline{1}))$ coincides with 
\[
\cN_{0}=\{\ccD_1 \in \cN(W(m; \underline{1}))\,|\, \ccD_{1}^{p^{m-1}}\in W(m; \underline{1})_{(0)}\};
\]
see also \thref{WnN0resultslem5}. 
\end{lem}
In \cite{P2015}, A.~Premet characterized all regular elements of $W(m; \underline{1})$. The following result on tori will be used in part (a) of the proof.
\begin{thm}\cite[Theorem 7.5.1]{S04}\thlabel{S03thm7.5.1}
Let $\ft\subset W(m; \underline{1})$ be a torus and set $\ft_0:=\ft\cap W(m; \underline{1})_{(0)}$. Let $t_1, \dots, t_s$ be toral elements of $\ft$ linearly independent $(\modd \ft_0)$. Then there exists $\Phi \in G=\Aut(W(m; \underline{1}))$ such that $\Phi(t_i)=(1+x_i)\del_i$ for $1\leq i\leq s$ and $\Phi(\ft_0) \subset \sum_{i=s+1}^{m}kx_i\del_i$. 
\end{thm}

Let $\ft$ be any torus in $W(m; \underline{1})$. We denote by $(\ft ^{\text{tor}})^{*}$ the set of all linear functions $\alpha: \ft\to k$ such that $\alpha(t)\in \F_{p}$ for all toral elements $t\in \ft$. It is known that $(\ft ^{\text{tor}})^{*}$ is an $\F_p$-form of the dual space $\ft^{*}$ with $\Card ((\ft ^{\text{tor}})^{*})=p^l$, where $l=\dim \ft$. Note that $\cO(m; \underline{1})$ is tautologically a $W(m; \underline{1})$-module. Then $\cO(m; \underline{1})$ decomposes as \\$\cO(m; \underline{1})=\bigoplus_{\lambda \in \ft^{*}}\cO(m; \underline{1})^{\lambda}$, where 
\[
\cO(m; \underline{1})^{\lambda}=\{f\in \cO(m; \underline{1})\,|\, t.f=\lambda(t)f\,\,\text{for all $t\in \ft$}\}.
\]
If $\cO(m; \underline{1})^{\lambda}\neq \{0\}$, then we say that $\lambda \in \ft^{*}$ is a \textit{weight} of $\cO(m; \underline{1})$ with respect to $\ft$ or a $\ft$-weight. We denote the set of all $\ft$-weights of $\cO(m; \underline{1})$ by $\Lambda(\cO(m; \underline{1}))$. Note that $\Lambda(\cO(m; \underline{1}))\subseteq (\ft ^{\text{tor}})^{*}$. 
\begin{lem}\cite[Lemma 1]{P2015}\thlabel{P2015lem1}
Let $\ft$ be an $r$-dimensional torus in $W(m; \underline{1})$. Then $\Lambda(\cO(m; \underline{1}))= (\ft ^{\text{tor}})^{*}$ and $\dim \cO(m; \underline{1})^{\lambda}=p^{m-r}$ for all $\lambda \in\Lambda(\cO(m; \underline{1}))$.
\end{lem}

Next we consider a particular torus $\ft_{\ccD}$ which is defined as follows: for any $\ccD \in  W(m; \underline{1})$, there exist a unique semisimple element $\ccD_s$ and a unique nilpotent element $\ccD_n$ such that $\ccD=\ccD_s+\ccD_n$ and $[\ccD_s, \ccD_n]=0$; see \thref{JCthm}. Let $\ft_{\ccD}$ denote the torus of $W(m; \underline{1})$ generated by the semisimple part $\ccD_s$ of $\ccD$. For any $\ccD \in  W(m; \underline{1})$, we also know that there exist homogeneous polynomials $\varphi_0, \dots, \varphi_{m-1}$ in $k[W(m; \underline{1})]$ with $\deg \varphi_i=p^{m}-p^{i}$ such that $\ccD^{p^{m}}+\sum_{i=0}^{m-1}\varphi_i(\ccD)\ccD^{p^{i}}=0$; see \thref{nvarietythm}. Define 
\begin{equation}\label{rddefnition}
\begin{aligned}
r=r(\ccD)&:=\min\{0\leq i\leq m-1\,|\, \varphi_i(\ccD)\neq 0\}\,\text{ for $\ccD\notin \cN(W(m; \underline{1}))$, and}\\
r=r(\ccD)&:=m \,\, \text{for $\ccD \in \cN(W(m; \underline{1}))$}.
\end{aligned}
\end{equation}
Then the following hold:
\begin{lem}\cite[Lemma 2(i)]{P2015}\thlabel{P2015lem2(i)}
$\dim \ft_{\ccD}=m-r(\ccD)=m-r$ and $\ccD_n^{p^{r}}=0$.
\end{lem}

In \cite[Sec~3.3]{P2015}, A.~Premet first observed that
\begin{lem}\cite[Remark 1]{P2015}\thlabel{P2015rmk1}
If $\ccD$ is a regular element of $W(m; \underline{1})$ and $r=r(\ccD)$, then $(\ad\ccD)^{p^{r}-1}$ maps $W(m; \underline{1})_{\ccD}^{0}=\{\ccD_1\in W(m; \underline{1})\,|\, (\ad \ccD)^{N}(\ccD_1)=0\, \text{for $N\gg 0$}\}$ onto the centralizer $\fc_{W(m; \underline{1})}(\ccD)=\sum_{i=0}^{m-1}k\ccD^{p^{i}}$, and the derivations $\ccD, \ccD^{p}, \dots, \ccD^{p^{m-1}}$ are linearly independent.
\end{lem}

Then A.~Premet proved the following theorem which characterizes all regular elements of $W(m; \underline{1})$. We will use (ii) and (iii) in the proof of \thref{dformlemma2}.
\begin{thm}\cite[Theorem 2]{P2015}\thlabel{regularderithm}
Suppose $\ccD \in W(m; \underline{1})$ and let $r=r(\ccD)$ be defined as in \eqref{rddefnition}. Then the following are equivalent:
\begin{enumerate}[\upshape(i)]
\item $\ccD$ is a regular element of $W(m; \underline{1})$.
\item The kernel of $\ccD$ on $\cO(m; \underline{1})$ is $k1$ which is $1$-dimensional.
\item There exist $z_{r+1}, \dots, z_m \in \{\varepsilon_i+x_i\,|\, r+1\leq i\leq m\}$ for some $\varepsilon_i\in \{0, 1\}$ and $\sigma \in G$ such that $\ccD_s^{\sigma}=\sum_{i=r+1}^{m}\lambda_i(z_i\del_i)$ for some $\lambda_i\in k$, the torus $(\ft_\ccD)^{\sigma}$ is spanned by $z_{r+1}\del_{r+1}, \dots, z_m\del_m$, and $\ccD_n^{\sigma}=\del_1+x_1^{p-1}\del_2+\dots+x_1^{p-1}\cdots x_{r-1}^{p-1}\del_r$.
\item All Jordan blocks of $\ccD_n$ have size $p^r$.
\end{enumerate}
\end{thm}
\begin{myproof}[Sketch of proof]
We prove the above statements are equivalent in the following order: (a) (i) implies (ii), (b) (ii) implies (iii), (c) (iii) implies (iv), (d) (iv) implies (ii) and finally (e) (ii) implies (i).

\textbf{(a)} Suppose $\ccD$ is a regular element of $W(m; \underline{1})$. We prove by contradiction that $\Ker \ccD=k1$. Suppose $\dim \Ker \ccD\geq 2$. Then there exists a nonzero $f\in \Ker \ccD \cap \fm$. Since $p>2$, it is easy to see that if $f\in \fm\setminus \fm^{2}$, then $f^{2}\neq 0$. Hence we may assume that $f\in \fm^2$. Note that $f\ccD \in \fc_{W(m; \underline{1})}(\ccD)$ and $(f\ccD)^{p}=f^{p}\ccD^{p}=0$. It follows from \thref{P2015rmk1} and \thref{P2015lem2(i)} that $f\ccD=\lambda \ccD_{n}^{p^{r-1}}$ for some $\lambda\in k$.

We split the proof into two cases: $\ccD \notin W(m; \underline{1})_{(0)}$ and $\ccD \in W(m; \underline{1})_{(0)}$. Suppose $\ccD \notin W(m; \underline{1})_{(0)}$. Then $f\ccD \neq 0$ and we may assume that $f\ccD=\ccD_n^{p^{r-1}}$. It follows that $\ccD_n^{p^{r-1}} \in W(m; \underline{1})_{(1)}$. As $\ad \ccD|_{W(m; \underline{1})_{\ccD}^{0}}=\ad \ccD_n|_{W(m; \underline{1})_{\ccD}^{0}}$, it follows from \thref{P2015rmk1} that there exists some $y \in W(m; \underline{1})_{\ccD}^{0}$ such that $(\ad \ccD_n)^{p^{r}-1}(y)=\ccD$. Since $p>2$, we have that
\begin{equation}\label{regeq1}
\begin{aligned}
\ccD=&(\ad \ccD_n^{p^{r-1}})^{p-1}((\ad \ccD_n)^{p^{r-1}-1}(y)) \\
\in& [W(m; \underline{1})_{(1)},[W(m; \underline{1})_{(1)}, W(m; \underline{1})]]\subseteq W(m; \underline{1})_{(1)}.
\end{aligned}
\end{equation}
This contradicts our assumption that $\ccD \notin W(m; \underline{1})_{(0)}$. So this case cannot occur.

Suppose $\ccD \in W(m; \underline{1})_{(0)}$. By \thref{S03thm7.5.1}, we may assume that $\ft_{\ccD}\subseteq T_m$, where $T_m$ is a maximal torus in $W(m; \underline{1})$ with basis $\{ x_1\del_1,\dots,  x_m\del_m\}$. Let $\{\theta_1, \dots, \theta_m\}$ be the corresponding dual basis in $T_m^{*}$. Let $\nu_i$ denote the restriction of $\theta_i$ to $\ft_\ccD$. Let $\ad_{-1}$ denote the representation of $W(m; \underline{1})_0$ in $\fgl(W(m; \underline{1})_{-1})$ induced by the adjoint action of $W(m; \underline{1})_0$ to $W(m; \underline{1})_{-1}$; see \eqref{wmgrading2.2.2} for the $\Z$-grading on $W(m; \underline{1})$. Then the set of $\ft_{\ccD}$-weights on $W(m; \underline{1})_{-1}$ coincides with $\Lambda:=\{-\nu_1, \dots, -\nu_m\}$. By \thref{P2015lem2(i)}, we know that $\dim \ft_{\ccD}=m-r$. Since $\Lambda$ spans $\ft_{\ccD}^{*}$, we have that $\Card (\Lambda)\geq m-r$. 

For $\nu \in \Lambda$, set $\eta(\nu):=\dim W(m; \underline{1})_{-1}^{\nu}$. Then $\eta(\nu)\geq1$ and $\sum_{\nu_\in \Lambda}\eta(\nu)=m$. Set $\eta:=\max_{\nu\in \Lambda}\eta(\nu)$. Then one can show that
\begin{equation}\label{regeq2}
\eta\leq r+1.
\end{equation}
Write $\ccD_n=\sum_{i\geq 0}\ccD_{n, i}$, where $\ccD_{n, i}\in W(m; \underline{1})_{i}$; see \eqref{wmgrading2.2.2} for notations. Since $\ft_{\ccD}\subseteq T_m$, we have that $\ccD_{n, i} \in \fc_{W(m; \underline{1})}(\ft_{\ccD})$ for all $i$. In particular, each weight space $W(m; \underline{1})_{-1}^{\nu}$ is invariant under $\ad _{-1}(\ccD_{n,0})$. As $\ccD_n$ is nilpotent, we get 
\begin{equation}\label{regeq3}
(\ad _{-1}(\ccD_{n,0}))^\eta=0.
\end{equation} 
Since $W(m; \underline{1})_{(1)}$ is a $p$-ideal of $W(m; \underline{1})_{(0)}$, then one can show using Jacobson's formula that 
\begin{equation}\label{inL1}
\ccD_n^{p^{K}}-\ccD_{n, 0}^{p^{K}} \in W(m; \underline{1})_{(1)}
\end{equation}
for all $K\geq 0$.

Suppose $r\geq 2$. Then one can show that $p^{r-1}\geq r+1$. By \eqref{regeq2} and \eqref{regeq3}, we have that $\ad_{-1}(\ccD_{n,0}^{p^{r-1}})=0$. Since $\ad_{-1}$ is a faithful representation of $W(m; \underline{1})_{0}$, we get $\ccD_{n,0}^{p^{r-1}}=0$. By \eqref{inL1}, we get $\ccD_{n}^{p^{r-1}}\in W(m; \underline{1})_{(1)}$. Applying \eqref{regeq1}, we get $\ccD \in W(m; \underline{1})_{(1)}$. In particular, $\ccD$ is nilpotent. By our assumption, $\ccD$ is regular. This contradicts \thref{SingP91}.

Suppose $r=1$. Then $\Card (\Lambda)=m-1$ or $m$. If $\Card (\Lambda)=m$, then $\sum_{\nu_\in \Lambda}\eta(\nu)=m$ implies that $\eta=1$. By \eqref{regeq3}, we have that $\ad_{-1}(\ccD_{n,0})=0$. This again gives that $\ccD_n^{p^{r-1}}\in  W(m; \underline{1})_{(1)}$. Arguing as above, we get a contradiction. If $\Card (\Lambda)=m-1$, then one can show that $\eta=2$. Since $\dim \ft_{\ccD}=m-1$, we may assume that $\ft_{\ccD}$ is spanned by elements $t_i=x_i\del_i+c_i(x_m\del_m)$, $1\leq i\leq m-1$, for some $c_i\in \F_{p}$. Note that $\nu_1, \dots, \nu_{m-1}$ are linearly independent and $\nu_m=c_1\nu_1+\dots+c_{m-1}\nu_{m-1}$. Since $\Card (\Lambda)=m-1$, we may assume that $\nu_m=\nu_{m-1}$, $\ccD_{n, 0}=\lambda (x_{m-1}\del_m)$ for some $\lambda \in k$, and 
\[
\ccD_s=\sum_{i=1}^{m-2}\alpha_i(x_i\del_i)+\alpha_{m-1}(x_{m-1}\del_{m-1}+x_m\del_m)
\]
for some $\alpha_i\in k^{*}$. Then $\fc_{W(m; \underline{1})}(\ft_{\ccD})\subset W(m;\underline{1})_{(0)}$. If $\lambda=0$, then $\ccD_n\in W(m;\underline{1})_{(1)}$. Applying \eqref{regeq1} with $r=1$, we get $\ccD\in W(m;\underline{1})_{(1)}$. By our assumption, $\ccD$ is regular. This contradicts \thref{SingP91}. If $\lambda \neq 0$, then $\ccD_n \notin W(m; \underline{1})_{(1)}$. Then one can show that $\fc_{W(m; \underline{1})}(\ccD) \cap W(m; \underline{1})_{(1)}\neq \{0\}$. By \thref{P2015rmk1}, any nilpotent element of $\fc_{W(m; \underline{1})}(\ccD)$ is a scalar multiple of $\ccD_n$. Since $\ccD_n \notin W(m; \underline{1})_{(1)}$, this is a contradiction. So this case cannot occur. 

If $r=0$, then $\ccD_n=0$ and $\ft_{\ccD}=T_m$. Hence $\Ker \ccD=k1$, the zero weight space of $T_m$ in $\cO(m; \underline{1})$. This shows that (i) implies (ii).

\textbf{(b)} Suppose (ii) holds for $\ccD$, i.e. $\Ker \ccD=k1$. We show that (ii) implies (iii). Let $B$ denote the zero weight space of $\ft_{\ccD}$ in $\cO(m; \underline{1})$. By \thref{P2015lem1}, $\dim B=p^r$. Note that the restriction of $\ccD$ to $B$, denoted $\ccD_{|B}$, is a nilpotent derivation of $B$. Let $\fm_{B}=B \cap \fm$ be the maximal ideal of the local ring $B$. By our assumption, $\fm_B$ is not $\ccD$-stable. Hence $B$ is differentiably simple; see \thref{defndfiffsimple}. By \thref{Blockthmdsimple}, $B\cong \cO(r; \underline{1})$ as $k$-algebras. Since $\ccD_{|B}$ is a nilpotent derivation of $B$ and $\Ker \ccD_{|B}=k1$, it follows from \thref{WnDresultslem3} and \thref{Wnproof} that there exist $y_1, \dots, y_r \in \fm_B$ whose cosets in $\fm_B/\fm_B^{2}$ are linearly independent such that
\[
\ccD_{|B}=\frac{\del}{\del y_1}+y_1^{p-1} \frac{\del}{\del y_2}+\dots+y_1^{p-1}\cdots y_{r-1}^{p-1}\frac{\del}{\del y_r}.
\]
By \thref{WnDresultslem3}(ii), $\Der B$ is a free $B$-module with basis $\ccD_{|B}, \dots \ccD^{p^{r-1}}_{|B}$. Hence there exists $\{b_{i, j}\,|\, 0\leq i, j\leq r-1\}\subset B$ such that 
\[\frac{\del}{\del y_i}=\bigg(\sum_{j=0}^{r-1}b_{i,j}\ccD_n^{p^{j}}\bigg)_{|B}\]
for $1\leq i\leq r$. This show that the set of partial derivatives $\{\frac{\del}{\del y_i}\,|\,1\leq i\leq r\}\subset \Der B$ can be lifted to a system of commuting derivations of $\cO(m; \underline{1})$.

Since each derivation $\frac{\del}{\del y_i}$ of $\cO(m; \underline{1})$ maps $B\cap \fm^{2}$ to $\fm_B=B\cap \fm$, it follows that $\fm_{B}^{2}=B\cap \fm^2$. Then $\fm_B/\fm_B^{2}$ embeds into $\fm/\fm^{2}$. As a result, there exist $y'_{r+1}, \dots, y_{m}' \in \fm$ such that the cosets of $y_1, \dots, y_r, y'_{r+1}, \dots, y_{m}'$ in $\fm/\fm^{2}$ are linearly independent. Since $\dim \ft_{\ccD}=m-r$ and $\ft_{\ccD}$ acts semisimply on $\cO(m; \underline{1})/k1$, we may assume that there exist $\gamma_{r+1}, \dots, \gamma_{m} \in (\ft_{\ccD}^{\text{tor}})^{*}$ such that $y_i'+k1 \in (\cO(m; \underline{1})/k1)^{\gamma_i}$. Since $B=\cO(m; \underline{1})^{0}$, then $\{\gamma_{r+1}, \dots, \gamma_{m}\}$ forms a basis of the dual space $\ft_{\ccD}^{*}$. It follows that $t(y_i')=\gamma_i(t)y_i'+\gamma'_i(t)1$ for all $t\in \ft_{\ccD}$, where $\gamma_i': \ft_{\ccD}\to k$ is a linear function. Then one can show that $\gamma_i'$ is proportional to $\gamma_i$. So for each $i\geq r+1$, there exists $\varepsilon_i\in k$ such that $t(y_i'+\varepsilon_i)=\gamma_{i}(t)(y_i'+\varepsilon_i)$. Rescaling the $y_i'$'s if need be, we may assume that $\varepsilon_i\in \{0, 1\}$. Set $y_i:=y_i'+\varepsilon_i$ for $r+1\leq i\leq m$.

By our choice of $y_1, \dots, y_m$, there is a unique $\sigma \in \Aut(\cO(m; \underline{1}))$ such that $\sigma^{-1}(x_i)=y_i$ for $1\leq i\leq r$ and $\sigma^{-1}(x_i)=y_i-\varepsilon_i$ for $r+1\leq i\leq m$. Then 
\begin{align*}
\ccD_n^\sigma &=\sigma\circ\ccD_n\circ \sigma^{-1}=\del_1+x_1^{p-1}\del_2+\dots+ x_1^{p-1}\cdots x_{r-1}^{p-1}\del_r,\, \text{and}\\
\ccD_s^\sigma&=\sigma\circ\ccD_s\circ \sigma^{-1}=\sum_{i=r+1}^{m}\lambda_i(\varepsilon_i+x_i)\del_i
\end{align*}
for some $\lambda_i \in k$. Since $\dim \ft_{\ccD}=m-r$ and each $(\varepsilon_i+x_i)\del_i$ is toral, it is easy to see that $(\ft_{\ccD})^\sigma$ is spanned by $(\varepsilon_{r+1}+x_{r+1})\del_{r+1}, \dots, (\varepsilon_m+x_m)\del_m$. This shows that (ii) implies (iii).

\textbf{(c)} Suppose (iii) holds for $\ccD$ and adapt the notations introduced in (b). Let $\gamma=\sum_{i=r+1}^{m} \alpha_i \gamma_i \in (\ft_{\ccD}^{\text{tor}})^{*}$. Then $\alpha_i \in \F_p$. By \thref{P2015lem1}, the weight space $\cO(m; \underline{1})^{\gamma}$ is a free $B$-module of rank $1$ generated by $y^{\gamma}=y_{r+1}^{\alpha_{r+1}}\cdots y_m^{\alpha_m}$. Moreover, $(\ccD_{|B})^{p^{r}-1}\neq 0$. Since $\ccD_n(y^{\gamma})=0$, then $\ccD_n$ acts on each $\cO(m; \underline{1})^{\gamma}$ as a Jordan block of size $p^r$. This shows that (iii) implies (iv).

\textbf{(d)} Suppose (iv) holds for $\ccD$. Consider the zero weight space $\cO(m; \underline{1})^{0}$ of $\ft_{\ccD}$. By \thref{P2015lem1}, we see that $\ccD_n$ acts on $\cO(m; \underline{1})^{0}$ for $\ft_{\ccD}$ as a single Jordan block of size $p^r$. Since $\Ker \ccD \subseteq \cO(m; \underline{1})^{0}$, we must have that $\Ker \ccD=k1$. This shows that (iv) implies (ii). 

\textbf{(e)} Suppose (ii) holds for $\ccD$. We show that (ii) implies (i). Since we have shown in (b) that (ii) implies (iii), we get an explicit description of $\ccD$. So we may assume that $\ccD=\ccD_s+\ccD_n$, where $\ccD_s=\sum_{i=r+1}^{m}\lambda_i z_i\del_i$ for some $\lambda_{i} \in k$ and $z_i=\varepsilon_i+x_i$ as in statement (iii), and $\ccD_n=\del_1+x_1^{p-1}\del_2+\dots+x_1^{p-1}\cdots x_{r-1}^{p-1}\del_r$. Due to the form of $\ccD$, it is easy to see that $\ccD$ preserves each direct summand $\bigoplus_{K=1}^{m}\cO(r; \underline{1})z_{r+1}^{a_{r+1}}\cdots z_{m}^{a_{m}}\del_{K}$ of $W(m; \underline{1})$, where $0\leq a_i\leq p-1$ are fixed. Note that there are $p^{m-r}$ of such summands. Moreover, $\ccD$ acts on each summand as a direct sum of Jordan blocks of size $p^{r}$ with eigenvalue $(a_{r+1}\lambda_{r+1}+\dots+a_{m}\lambda_{m})-\lambda_{K}$. Since the $\lambda_i$'s are linearly independent over the prime field $\F_{p}$, this eigenvalue is zero if and only if either $K\leq r$ and $a_i=0$ for all $r+1\leq i\leq m$ or $K\geq r+1, a_K=1$ and $a_i=0$ for all $r+1\leq i\leq m$ except for $i= K$. Due to the form of $\ccD_n$, this implies that the kernel of $\ad \ccD$ has dimension $r+(m-r)=m$. Hence $\ccD$ is regular and (ii) implies (i). This completes the sketch of proof.
\end{myproof}

Now we want to state \thref{dformlemma2}. But we need to define $W$ and understand the following:

\begin{lem}\thlabel{IidealofOW}
Let $I$ be the ideal of $\cO(m; \underline{1})$ generated by $x_{s+1}, \dots, x_{m}$, where $s\geq 1$. Define 
\begin{align*}
W:=I\del_1+\dots+I \del_m.
\end{align*}
Let $d_0=\del_1+x_1^{p-1}\del_2+\dots+x_1^{p-1}\cdots x_{s-1}^{p-1}\del_s$ be the derivation of $\cO(s; \underline{1})$. Then $W$ is an $\ad d_0$-invariant restricted Lie subalgebra of $W(m; \underline{1})$ contained in $W(m; \underline{1})_{(0)}$.
\end{lem}

\begin{proof}
Since $I$ is an ideal of $\cO(m; \underline{1})$ and $W(m; \underline{1})=\sum_{j=1}^{m}\cO(m; \underline{1})\del_j$, we can describe $W$ as $IW(m; \underline{1})$, i.e. the set of all $f\ccD$ with $f\in I$ and $\ccD\in W(m; \underline{1})$. Let $f\ccD, gE$ be any elements of $W$, where $f, g \in I$ and  $\ccD, E \in W(m; \underline{1})$. By \eqref{LbracketWm}, we have that 
\begin{align*}
[f\ccD, gE]=f\ccD(g)E-gE(f)\ccD+fg[\ccD, E]. 
\end{align*}
Since $f, g\in I$ and $I$ is an ideal of $\cO(m; \underline{1})$, it is easy to see that $[f\ccD, gE] \in W$. Hence $W$ is a Lie subalgebra of $W(m; \underline{1})$. Since $\cO(m; \underline{1})$ is a local ring, the ideal $I$ is contained in the maximal ideal $\fm$ of $\cO(m; \underline{1})$. It follows that $W \subset W(m; \underline{1})_{(0)}$.  

Next we show that $W$ is $\ad d_0$-invariant. Let $f\ccD$ be any element of $W$, where $f\in I$ and $\ccD\in W(m; \underline{1})$. Note that $I$ is $d_0$-invariant. Then $d_0(f)\in I$. By \eqref{LbracketWm} again, we have that 
\[
[d_0, f\ccD]=d_0(f)\ccD+f[d_0, \ccD]\in I\ccD+fW(m; \underline{1})\subseteq IW(m; \underline{1})=W.
\]
Hence $W$ is $\ad d_0$-invariant.

It remains to show that $W$ is restricted. Let $\ccD'$ be any element of $W(m; \underline{1})$. Note that $\ccD' \in W$ if and only if $\ccD'(x_i)\in I$ for all $1\leq i\leq m$. Let $\ccD_1\in W$. Then $\ccD_1(x_i)\in I$ for all $1\leq i\leq m$ and $\ccD_1$ preserves $I$. It follows that $\ccD_1^{n}$ preserves $I$ for all $n\in \Z_{>0}$. Hence 
\[
\ccD_1^{p}(x_i)=\ccD_1^{p-1}(\ccD_1(x_i))\in \ccD_1^{p-1}(I)\subseteq I
\]
for all $1\leq i\leq m$. Therefore, $\ccD_1^{p}\in W$ and $W$ is restricted.

It follows from the above that $W$ is an $\ad d_0$-invariant restricted Lie subalgebra of $W(m; \underline{1})$ contained in $W(m; \underline{1})_{(0)}$. This completes the proof.
\end{proof}

\begin{lem}\thlabel{dformlemma2}
Let $z$ be a nilpotent element of $W(m; \underline{1})$ such that $z\notin W(m; \underline{1})_{(0)}$. Then $z$ is conjugate under $G$ to $d_0+u$, where 
\[
d_0=\del_1+x_1^{p-1}\del_2+\dots+x_1^{p-1}\cdots x_{s-1}^{p-1}\del_s
\]
with $1\leq s\leq m$ and $u\in W\cap W(m; \underline{1})_{(p-1)}$; see \thref{IidealofOW} for $W=I\del_1+\dots+I \del_m$ and \eqref{wmfiltration2.2.2} for $W(m; \underline{1})_{(p-1)}=\{\sum_{j=1}^{m}f_j\del_j\,|\, \deg f_j\geq p \,\text{for all $j$}\}$ the component of the standard filtration of $W(m; \underline{1})$.
\end{lem}

\begin{proof}
Let $z$ be a nilpotent element of $W(m; \underline{1})$ such that $z\notin W(m; \underline{1})_{(0)}$. By \thref{dformlemma1}, we may assume that 
\begin{align*}
z=&\sum_{i=1}^{s}x_1^{p-1}\cdots x_{i-1}^{p-1}(1+x_i^{p-1}\psi_i)\del_i+x_1^{p-1}\cdots x_s^{p-1}\sum_{i=s+1}^{m}\psi_i\del_i\\
=&d_0+\sum_{i=1}^{s}x_1^{p-1}\cdots x_i^{p-1}\psi_i\del_i+x_1^{p-1}\cdots x_s^{p-1}\sum_{i=s+1}^{m}\psi_i\del_i,
\end{align*}
where $1\leq s\leq m$, $d_0=\del_1+x_1^{p-1}\del_2+\dots+x_1^{p-1}\cdots x_{s-1}^{p-1}\del_s$,\\ $\psi_i\in k[X_{i+1}, \dots, X_{m}]/(X_{i+1}^{p}, \dots, X_m^{p})$ for $1\leq i\leq s$ and $\psi_i\in \fm_{m-s}$, the maximal ideal of $k[X_{s+1}, \dots, X_{m}]/(X_{s+1}^{p}, \dots, X_m^{p})$, for $s+1\leq i\leq m$. 

Let $I$ be the ideal of $\cO(m; \underline{1})$ generated by $x_{s+1}, \dots, x_{m}$. Since $\psi_i\in \fm_{m-s}$ for all $s+1\leq i\leq m$, it is easy to see that $z$ preserves the ideal $I$. Note that the factor ring $\cO(m;\underline{1})/I$ is isomorphic to $\cO(s; \underline{1})=k[X_1, \dots, X_s]/(X_1^{p}, \dots, X_s^{p})$, and $z$ acts on $\cO(s; \underline{1})$ as a derivation. This derivation, call it $y$, has the same form as $z$ except that we forget the last summand $x_1^{p-1}\cdots x_s^{p-1}\sum_{i=s+1}^{m}\psi_i\del_i$ and replace $\psi_i$, $1\leq i\leq s$, by their images $\bar{\psi_i}$ in $\cO(s; \underline{1})$, i.e. $\psi_i=\bar{\psi_i}+\psi_i'$ with $\psi_i'\in I$, and
\[
y=d_0+\sum_{i=1}^{s}x_1^{p-1}\cdots x_i^{p-1}\bar{\psi_i}\del_i.
\]
Note also that 
\begin{equation}\label{degreegeqp-1}
\sum_{i=1}^{s}x_1^{p-1}\cdots x_i^{p-1}\psi_i'\del_i+x_1^{p-1}\cdots x_s^{p-1}\sum_{i=s+1}^{m}\psi_i\del_i \in  W(m; \underline{1})_{(p-1)}.
\end{equation}

We show that the kernel of $y$ on $\cO(s; \underline{1})$ is $1$-dimensional, i.e. $\Ker y=k1$. Suppose the contrary. Then there exists a nonzero $f\in \Ker y$ with $f(0)=0$. We want to show that $y(f)\neq 0$. Hence we get a contradiction. Note that $f$ is a linear combination of monomials $\boldsymbol{x}^{A}=x_1^{a_1}\cdots x_{s}^{a_s}$, where $0\leq a_i\leq p-1$ and $\sum_{i=1}^{s}a_i>0$. We want to look at the ``smallest'' monomial involved in $f$. Note that the standard degree of monomials is not a good choice. This is because $y=d_0+\sum_{i=1}^{s}x_1^{p-1}\cdots x_i^{p-1}\bar{\psi_i}\del_i$ and applying $d_0=\del_1+x_1^{p-1}\del_2+\dots+x_1^{p-1}\cdots x_{s-1}^{p-1}\del_s$ to $\boldsymbol{x}^{A}$ in $f$, the standard degree either decreases or increases; see \eqref{d0eq1} and \eqref{d0eq2} below. So it is difficult to compare $d_0(\boldsymbol{x}^{A})$ with $\big(\sum_{i=1}^{s}x_1^{p-1}\cdots x_i^{p-1}\bar{\psi_i}\del_i\big)(\boldsymbol{x}^{A})$ and the effects of $y$ on the other monomials in $f$. This means that we cannot easily deduce that $y(f)\neq 0$. Hence we are going to use a monomial ordering introduced later; see \thref{DegLexdefn}\footnote{Note that this monomial ordering was defined after we have proved this lemma. Due to the forms of $d_0$ and $u$, we defined the $|\,\,|_p$-degree of monomials as in \thref{DegLexdefn}; see \thref{DegLexexample}(2) and (3) later for the reasons. Here we are working with the subring $\cO(s; \underline{1})$ of $\cO(m; \underline{1})$ and the $|\,\,|_p$-degree is the first part $\sum_{i=1}^{s}a_i''p^{i-1}$.}. Let $\boldsymbol{x}^{A''}=x_1^{a_1''}\cdots x_s^{a_s''}$ be any monomial in $\cO(s; \underline{1})$. Define the $|\,\,|_p$-degree of $\boldsymbol{x}^{A''}$ by 
\[
|A''|_p:=\sum_{i=1}^{s}a_i''p^{i-1}=a_1''+a_2''p+\dots +a_s''p^{s-1},
\]
i.e. the $p$-adic expansion of nonnegative integers with digits $a_i''$. It is known that for a fixed prime number $p$, every nonnegative integer has a unique $p$-adic expansion. Hence for $l=0, 1, \dots, p^s-1$, there is a unique $\boldsymbol{x}^{A''}\in\cO(s; \underline{1})$ with $|A''|_p=l$. Note that if the product of two monomials is nonzero, then the $|\,\,|_p$-degree of this product is given by the sum of $|\,\,|_p$-degrees of the monomials; see \thref{DegLexrmk}(iii) and (iv)(a) later. 

Consider $d_0=\del_1+x_1^{p-1}\del_2+\dots+x_1^{p-1}\cdots x_{s-1}^{p-1}\del_s$ in $y$. By \thref{zcalculations} later, we know that applying $d_0$ to any $\boldsymbol{x}^{A}$ in $f$, we get
\begin{equation}\label{d0lem3.2.7}
d_0(\boldsymbol{x}^{A})=\alpha\boldsymbol{x}^{A'},
\end{equation}
where $1\leq\alpha\leq p-1$ and $\boldsymbol{x}^{A'}\in \cO(s; \underline{1})$ with $|A'|_p=|A|_p-1$. 

Consider $\sum_{i=1}^{s}x_1^{p-1}\cdots x_i^{p-1}\bar{\psi_i}\del_i$ in $y$. We first show that all nonzero summands in $\sum_{i=1}^{s}x_1^{p-1}\cdots x_i^{p-1}\bar{\psi_i}\del_i$ have positive $|\,\,|_p$-degrees. Since $\bar{\psi_i}\in \cO(s; \underline{1})$, it is a linear combination of monomials in $\cO(s; \underline{1})$. Then we can write 
\[
\bar{\psi_i}=\sum_{A''}\lambda_{A''}\boldsymbol{x}^{A''},
\]
where $\lambda_{A''}\in k$ and $\boldsymbol{x}^{A''}\in \cO(s; \underline{1})$ with $|A''|_p=l\geq 0$.
Hence 
\[
\sum_{i=1}^{s}x_1^{p-1}\cdots x_i^{p-1}\bar{\psi_i}\del_i=\sum_{i=1}^{s}\sum_{A''}\lambda_{A''}x_1^{p-1}\cdots x_i^{p-1}\boldsymbol{x}^{A''}\del_i.
\]
Consider each summand $x_1^{p-1}\cdots x_i^{p-1}\boldsymbol{x}^{A''}\del_i$. By \thref{DegLexexample}(1) later, we know that for $1\leq i\leq s$, $\del_i$ has $|\,\,|_p$-degree $-p^{i-1}$. If $x_1^{p-1}\cdots x_i^{p-1}\boldsymbol{x}^{A''}\del_i\neq 0$, then \thref{DegLexrmk}(iii) implies that $x_1^{p-1}\cdots x_i^{p-1}\boldsymbol{x}^{A''}\del_i$ has $|\,\,|_p$-degree
\begin{align*}
&\big((p-1)+(p-1)p+\dots+(p-1)p^{i-1}\big)+l-p^{i-1}\\
=&(p^i-1)+l-p^{i-1}\\
=&(p^i-p^{i-1})+(l-1).
\end{align*}
By our assumption, $p\geq 3$. Since $1\leq i\leq s$, an easy induction on $i$ shows that $p^i-p^{i-1}\geq p-p^0=p-1\geq 2$.  
Since $l\geq 0$, we have that $l-1\geq -1$. It follows that 
\begin{align*}
&\big((p-1)+(p-1)p+\dots+(p-1)p^{i-1}\big)+l-p^{i-1}\\
=&(p^i-p^{i-1})+(l-1)\geq 2+(-1)=1>0.
\end{align*}
Hence all nonzero summands $x_1^{p-1}\cdots x_i^{p-1}\boldsymbol{x}^{A''}\del_i$ in $\sum_{i=1}^{s}x_1^{p-1}\cdots x_i^{p-1}\bar{\psi_i}\del_i$ have positive $|\,\,|_p$-degrees. Then applying $\sum_{i=1}^{s}x_1^{p-1}\cdots x_i^{p-1}\bar{\psi_i}\del_i$ to any $\boldsymbol{x}^{A}$ in $f$, we get 
\begin{equation}\label{d0lem3.2.72}
\bigg(\sum_{i=1}^{s}x_1^{p-1}\cdots x_i^{p-1}\bar{\psi_i}\del_i\bigg)(\boldsymbol{x}^{A})=\sum_{A'''}\mu_{A'''}\boldsymbol{x}^{A'''},
\end{equation}
where $\mu_{A'''}\in k$ and $\boldsymbol{x}^{A'''}\in \cO(s; \underline{1})$ with $|A'''|_p>|A|_p$. By \eqref{d0lem3.2.7}, $|A'|_p=|A|_p-1$. Hence $|A'''|_p>|A'|_p$. 

Now we look at $f$ and take the monomial of smallest $|\,\,|_p$-degree with nonzero coefficient, say it is $\boldsymbol{x}^{A_1}=x_1^{a_1}\cdots x_s^{a_s}$, where $0\leq a_i\leq p-1$ and $\sum_{i=1}^{s}a_i>0$. Rescaling $f$ if need be, we may assume that the coefficient of $\boldsymbol{x}^{A_1}$ is $1$. If $a_1\neq 0$, then 
\begin{align}\label{d0eq1}
d_0(\boldsymbol{x}^{A_1})=a_1\boldsymbol{x}^{A_1'}=a_1x_1^{a_1-1}\cdots x_s^{a_s}.
\end{align}
By our choice of $\boldsymbol{x}^{A_1}$ and \eqref{d0lem3.2.7}-\eqref{d0lem3.2.72}, we have that
\begin{align*}
y(f)=d_0(\boldsymbol{x}^{A_1})+\sum_{A_2}\gamma_{A_2}\boldsymbol{x}^{A_2}=a_1\boldsymbol{x}^{A_1'}+\sum_{A_2}\gamma_{A_2}\boldsymbol{x}^{A_2},
\end{align*}
where $\gamma_{A_2}\in k$ and $\boldsymbol{x}^{A_2}\in \cO(s; \underline{1})$ with $|A_2|_p>|A_1'|_p$.
Since all $\boldsymbol{x}^{A_2}$ in $y(f)$ have $|A_2|_p>|A_1'|_p$, they do not cancel $\boldsymbol{x}^{A_1'}$. As $a_1\boldsymbol{x}^{A_1'}\neq 0$, we have that $y(f)\neq 0$. If $a_1=\dots=a_{K-1}=0$ and $a_K\neq 0$, then $\boldsymbol{x}^{A_1}=x_K^{a_K}x_{K+1}^{a_{K+1}}\cdots x_s^{a_s}$ and 
\begin{align}\label{d0eq2}
d_0(\boldsymbol{x}^{A_1})=a_K\boldsymbol{x}^{A_1'}=a_K x_1^{p-1}\cdots x_{K-1}^{p-1}x_K^{a_K-1}x_{K+1}^{a_{K+1}}\cdots x_s^{a_s}.
\end{align}
Hence
\begin{align*}
y(f)=d_0(\boldsymbol{x}^{A_1})+\sum_{A_3}\gamma_{A_3}\boldsymbol{x}^{A_3}=a_K\boldsymbol{x}^{A_1'}+\sum_{A_3}\gamma_{A_3}\boldsymbol{x}^{A_3},
\end{align*}
where $\gamma_{A_3}\in k$ and $\boldsymbol{x}^{A_3}\in \cO(s; \underline{1})$ with $|A_3|_p>|A_1'|_p$. Arguing similarly as above, we get $y(f)\neq 0$. So $f\notin \Ker y$. This is a contradiction. Hence $\Ker y=k1$.

Note that since $z$ is nilpotent so is $y$. Since $\Ker y=k1$, it follows from \thref{regularderithm}(i) and (ii) that $y$ is a regular element of $W(s; \underline{1})$. Then \thref{regularderithm}(iii) implies that $y$ is conjugate under $\Aut(W(s; \underline{1}))$ to 
\[
d_0=\del_1+x_1^{p-1}\del_2+\dots+x_1^{p-1}\cdots x_{s-1}^{p-1}\del_s.
\]
Since $W(s; \underline{1})$ is a Lie subalgebra of $W(m; \underline{1})$, we may identify $\Aut(W(s; \underline{1}))$ with a subgroup of $G=\Aut(W(m; \underline{1}))$ by letting $\sigma(x_i)=x_i$ for all $\sigma \in \Aut(W(s; \underline{1}))$ and $s+1\leq i\leq m$. Since $y^{\sigma}=\sigma y\sigma^{-1}=d_0$ for some $\sigma \in \Aut(W(s; \underline{1}))$, then $z^{\sigma}=\sigma z\sigma^{-1}=d_0+u$, where $u\in W\cap W(m; \underline{1})_{(p-1)}=(I\del_1+\dots+I \del_m)\cap W(m; \underline{1})_{(p-1)}$.
Note that $u\in W(m; \underline{1})_{(p-1)}$ follows from \eqref{degreegeqp-1} and that $G$ preserves the standard filtration of $W(m; \underline{1})$, in particular, $G$ preserves the component $W(m; \underline{1})_{(p-1)}$. This completes the proof.
\end{proof}

\begin{lem}\thlabel{zpscalculationslemmma}
Let $z=d_0+u$ be as in \thref{dformlemma2}. Then $z^{p^{s}}\in W(m; \underline{1})_{(0)}$.
\end{lem}

\begin{proof}
Recall that $z=d_0+u$, where $d_0=\del_1+x_1^{p-1}\del_2+\dots+x_1^{p-1}\cdots x_{s-1}^{p-1}\del_s$ and $u\in W\cap W(m; \underline{1})_{(p-1)}$; see \thref{IidealofOW} for $W$.
By \thref{generalJacobF}, 
\[
z^{p^{s}}=d_0^{p^{s}}+u^{p^{s}}+\sum_{l=0}^{s-1}v_l^{p^{l}},
\]
where $v_l$ is a linear combination of commutators in $d_0$ and $u$. By Jacobi identity, we can rearrange each $v_l$ so that $v_l$ is in the span of $[w_t,[w_{t-1},[\dots,[w_2,[w_1, u]\dots]$, where $t=p^{s-l}-1$ and each $w_\alpha, 1\leq \alpha\leq t$, is equal to $d_0$ or to $u$.

We show that $[w_t,[w_{t-1},[\dots,[w_2,[w_1, u]\dots]\in W$ and so $v_l\in W$. By \thref{IidealofOW}, $W$ is an $\ad d_0$-invariant restricted Lie subalgebra of $W(m; \underline{1})$ contained in $W(m; \underline{1})_{(0)}$. Since $u\in W$, it follows that $[d_0, u]\in W$. It is clear that $[u, u]=0\in W$. This shows that $[w_1, u] \in W$. Consider $w_2$. If $w_2=d_0$, then by the same reason, we get $[d_0, [d_0, u]]\in W$. If $w_2=u$, then $W$ is a Lie subalgebra of $W(m; \underline{1})$  implies that $[u,[d_0, u]]\in W$. This shows that $[w_2,[w_1, u]] \in W$. Continuing in this way and using the same reasons, we get $[w_t,[w_{t-1},[\dots,[w_2,[w_1, u]\dots]\in W$ and so $v_l\in W$ for all $0\leq l\leq s-1$. Since $W$ is a restricted Lie subalgebra of $W(m; \underline{1})$, we have that $v_l^{p^{l}}\in W$ for all $1\leq l\leq s-1$. Hence $\sum_{l=0}^{s-1}v_l^{p^{l}}\in W \subset W(m; \underline{1})_{(0)}$. Similarly, since $u\in W$ and $W$ is restricted, we have that $u^{p^{s}}\in W\subset W(m; \underline{1})_{(0)}$. By \thref{WnDresultslem3}(i), we know that $d_0^{p^{s}}=0$. Hence 
\[
z^{p^{s}}=u^{p^{s}}+\sum_{l=0}^{s-1}v_l^{p^{l}}\in  W(m; \underline{1})_{(0)}.
\]
This completes the proof.
\end{proof}

Now we consider the restricted transitive subalgebra $\dd D$ of $W(m; \underline{1})$ and show the last two results hold for any $d\in \cN(\dd D)$ such that $d\notin W(m; \underline{1})_{(0)}$.

\begin{lem}\thlabel{GpreservesND}
Let $d$ be any element of $\cN(\dd D)$ such that $d\notin W(m; \underline{1})_{(0)}$. Then $d$ is conjugate under $G$ to $d_0+u$, where $d_0+u$ is the element in \thref{dformlemma2}. Moreover, $d^{p^{s}}\in W(m; \underline{1})_{(0)}$.
\end{lem}
\begin{proof}
Let $d$ be any element of $\cN(\dd D)$ such that $d\notin W(m; \underline{1})_{(0)}$. By \thref{dformlemma2}, we know that $d$ is conjugate under $G$ to $d_0+u$. Here we want to point out that the use of the group action does not suppose that $\dd D$ or $\cN(\dd D)$ is $G$-stable, it is only used to bring elements of $\cN(\dd D)$ to a nice form. Since $d$ is conjugate under $G$ to $d_0+u$, there exists $\sigma \in G$ such that $d=\sigma(d_0+u) \sigma^{-1}$. Then
\[
d^{p^{s}}=\big(\sigma(d_0+u) \sigma^{-1}\big)^{p^{s}}=\sigma (d_0+u)^{p^{s}}\sigma^{-1}.
\]
By \thref{zpscalculationslemmma}, $(d_0+u)^{p^{s}} \in W(m; \underline{1})_{(0)}$. Since $\sigma \in G$, it preserves the standard filtration of $W(m; \underline{1})$. Hence $d^{p^{s}}\in W(m; \underline{1})_{(0)}$. This completes the proof.
\end{proof}

\begin{rmk}\label{notationzdsame}
Let us denote $d_0+u\in\cN(\dd D)$ by $z$. Note that $z^{p^{s}}\in W(m; \underline{1})_{(0)}$ implies that $z^{p^{s}}$ preserves the maximal ideal $\fm$ of $\cO(m; \underline{1})$. This will be used in the proof of \thref{D1psnilpotentcalculations}; see step 2. 
\end{rmk}

Define 
\[
M:=I+z(\fm)
\]
to be the subspace of $\cO(m; \underline{1})$, where $\fm$ is the maximal ideal of $\cO(m; \underline{1})$, $I$ is the ideal of $\cO(m; \underline{1})$ generated by $x_{s+1}, \dots, x_m$ with $s\geq 1$, and $z=d_0+u$ as in \thref{dformlemma2}. We are aiming for \thref{d-d0d_0image} which shows $M$ has the property that $M\oplus kx_1^{p-1}\cdots x_s^{p-1}=\cO(m; \underline{1})$. To prove that result, we need to consider the image of $\fm$ under $z=d_0+u$. In the proof of \thref{dformlemma2}, we have seen that applying $d_0$ to any monomial in $\fm$, the standard degree of monomials either decreases or increases. It is difficult to get much useful information. Therefore, we need to define a monomial ordering on $\cO(m; \underline{1})$. Let us first recall the definitions of total ordering and monomial ordering.

\begin{defn}\cite[p.~55]{ALOCOX2015}\thlabel{totalorderingdefn}
A binary relation $\leq$ is a \textnormal{total ordering} on a set $X$ if the following hold for all $a, b, c\in X$:
\begin{enumerate}[\upshape(i)]
\item Antisymmetry: if $a\leq b$ and $b\leq a$, then $a=b$.
\item Transitivity: if $a\leq b$ and $b\leq c$, then $a\leq c$.
\item Comparability: either $a\leq b$ or $b\leq a$.
\end{enumerate}
\end{defn}

Consider the polynomial ring $k[X_1, \dots, X_m]$. Note that there is a natural bijection between $\Z^{m}_{\geq 0}$ and the set of monomials in $k[X_1, \dots, X_m]$ given by
\[
(\alpha_1, \dots, \alpha_m)\longleftrightarrow x_1^{\alpha_1}\cdots x_m^{\alpha_m},
\]
and addition in $\Z^{m}_{\geq 0}$ corresponds to multiplication of monomials in $k[X_1, \dots, X_m]$. So if $\prec$ is an ordering on $\Z^{m}_{\geq 0}$, then it gives an ordering on monomials: if 
$(\alpha_1, \dots, \alpha_m)\prec(\beta_1, \dots, \beta_m)$, then $x_1^{\alpha_1}\cdots x_m^{\alpha_m}\prec x_1^{\beta_1}\cdots x_m^{\beta_m}$. Recall that

\begin{defn}\cite[Definition 1, Sec.~2 and Corollary 6, Sec.~4, Chap.~2]{ALOCOX2015}\thlabel{monomialordering}
A \textnormal{monomial ordering} on $k[X_1, \dots, X_m]$ is a relation $\prec$ on $\Z^{m}_{\geq 0}$, or equivalently, a relation $\prec$ on the set of monomials $x_1^{\alpha_1}\cdots x_m^{\alpha_m}$, $(\alpha_1, \dots, \alpha_m)\in \Z^{m}_{\geq 0}$, satisfying: 
\begin{enumerate}[\upshape(i)]
\item $\prec$ is a total ordering on $\Z^{m}_{\geq 0}$.
\item If $(\alpha_1, \dots, \alpha_m)\prec(\beta_1, \dots, \beta_m)$ and $(\gamma_1, \dots, \gamma_m)\in \Z^{m}_{\geq 0}$, then 
\[
(\alpha_1+\gamma_1, \dots, \alpha_m+\gamma_m)\prec(\beta_1+\gamma_1, \dots, \beta_m+\gamma_m).
\]
\item $\prec$ is a well-ordering on $\Z^{m}_{\geq 0}$, i.e. every nonempty subset of $\Z^{m}_{\geq 0}$ has a smallest element under $\prec$. This is equivalent to the condition that $(0, \dots, 0)\prec(\alpha_1, \dots, \alpha_m)$ for every $(\alpha_1, \dots, \alpha_m)\neq (0, \dots, 0)$.
\end{enumerate} 
\end{defn}
It follows from the above definition that $(0, \dots, 0)$ (or $1$) is the smallest element in $\Z^{m}_{\geq 0}$ (or $k[X_1, \dots, X_m]$) under any monomial ordering. 

We are interested in $\cO(m; \underline{1})=k[X_1, \dots, X_m]/(X_1^p, \dots, X_m^{p})$, where every monomial $x_1^{a_1}\cdots x_m^{a_m}$ in $\cO(m; \underline{1})$ has $0\leq a_i\leq p-1$ for all $1\leq i\leq m$. There are many monomial orderings on $\cO(m; \underline{1})$. In particular, the following is an example. 

\begin{defn}\cite[Definition 3 and Proposition 4, Sec.~2, Chap.~2]{ALOCOX2015}\thlabel{lexordering2.2.2}
The \textnormal{lexicographic ordering} $\prec_{\text{lex}}$ on $\cO(m; \underline{1})$  with 
\[
x_1 \prec_{\text{lex}}x_2\prec_{\text{lex}} \dots \prec_{\text{lex}}x_m
\]
is defined as follows: for any two non-equal monomials $x_1^{a_1}\cdots x_m^{a_m}$ and $x_1^{a'_1} \cdots x_m^{a'_m}$,
\[
\text{$x_1^{a_1}\cdots x_m^{a_m}\prec_{\text{lex}} x_1^{a'_1} \cdots x_m^{a'_m}$ if and only if $a_i<a'_i$},
\]
where $i$ is the largest number in $\{1, \dots, m\}$ for which $a_i\neq a'_i$. 
\end{defn}

Next we define the $|\,\,|_p$-degree of monomials in $\cO(m; \underline{1})$. Then it induces a monomial ordering on $\cO(m; \underline{1})$.
\begin{defn}\thlabel{DegLexdefn}
Let $A=(a_1,\dots, a_s, a_{s+1}, \dots, a_m)$, where $0\leq a_i\leq p-1$. Set
\[
\boldsymbol{x}^{A}:=x_1^{a_1}\cdots x_s^{a_s}x_{s+1}^{a_{s+1}}\cdots x_m^{a_m}.
\]
Define the $|\,\,|_p$-degree of $\boldsymbol{x}^{A}$ by
\[
|A|_p:=\sum_{i=1}^{s}a_ip^{i-1}+p^s\sum_{i=s+1}^{m}a_i.
\]
Then we have a version of the total ordering called \textnormal{DegLex}: \\if $A=(a_1,\dots, a_s, a_{s+1}, \dots, a_m)$ and $A'=(a_1', \dots, a_{s}', a_{s+1}', \dots, a_m')$ with $0\leq a_i, a_i'\leq p-1$, then we say that $A\prec_{\text{DegLex}} A'$ if either $|A|_p<|A'|_p$ or $|A|_p=|A'|_p$ and $A$ precedes $A'$ in the lexicographic ordering $\prec_{lex}$ defined in \thref{lexordering2.2.2}.
\end{defn}

\begin{rmk}\thlabel{DegLexrmk}
\begin{enumerate}[\upshape(i)]
\item Note that $0\leq a_i\leq p-1$ are well defined nonnegative integers, and the first summand $\sum_{i=1}^{s}a_ip^{i-1}=a_1+a_2p+\dots+a_sp^{s-1}$ in $|A|_p$ is the $p$-adic expansion of nonnegative integers with digits $a_i$. It is well known that for a fixed prime number $p$, every nonnegative integer has a unique $p$-adic expansion.
\item It is easy to check that DegLex is a monomial ordering on $\cO(m; \underline{1})$. Moreover, we see that DegLex first orders monomials by the $|\,\,|_p$-degree, then it uses the lexicographic ordering to break ties.
\item Let $\boldsymbol{x}^{A}$ and $\boldsymbol{x}^{A'}$ be any monomials in $\cO(m; \underline{1})$. If $\boldsymbol{x}^{A}\boldsymbol{x}^{A'}=\boldsymbol{x}^{A''}\neq 0$, i.e. all exponents of $x_1, \dots, x_m$ are strictly less than $p$, then it follows from \thref{DegLexdefn} that 
\[
|A''|_p=|A|_p+|A'|_p.
\]
\item Let $\boldsymbol{x}^{A}=x_1^{a_1}\cdots x_s^{a_s}x_{s+1}^{a_{s+1}}\cdots x_m^{a_m}$ be any monomial in $\cO(m; \underline{1})$. Then  
\[
|A|_p=a_1+a_2p+\dots+a_sp^{s-1}+p^s\sum_{i=s+1}^{m}a_i.
\]
Let 
\[
Q_0=(p-1)+(p-1)p+\dots+(p-1)p^{s-1}=p^s-1,
\]
and 
\[
Q=Q_0+p^s((m-s)(p-1)).
\]
\begin{enumerate}
\item If $0\leq |A|_p\leq Q_0$, then we must have that $\sum_{i=s+1}^{m}a_i=0$. Indeed, if $\sum_{i=s+1}^{m}a_i\geq 1$, then $|A|_p\geq p^s>Q_0$, a contradiction. Hence $\boldsymbol{x}^{A}=x_1^{a_1}\cdots x_s^{a_s}$ with $|A|_p=a_1+a_2p+\dots+a_sp^{s-1}$. By (i), we know that for a fixed prime $p$, the $p$-adic expansion of any nonnegative integer is unique. Hence for $l=0, 1, \dots, Q_0$, there is a unique monomial $\boldsymbol{x}^{A}=x_1^{a_1}\cdots x_s^{a_s}$ with $|A|_p=l$.
\item If $p^s=Q_0+1\leq |A|_p\leq Q$, then $1\leq \sum_{i=s+1}^{m}a_i\leq (m-s)(p-1)$ and $0\leq \sum_{i=1}^{s}a_i\leq s(p-1)$ by (a). Note that in this case, we could have distinct monomials with the same $|\,\,|_p$-degree and such monomials are easy to construct. We could simply fix $a_1, \dots, a_s$ and choose $a_{s+1}, \dots, a_m$ so that they have the same $\sum_{i=s+1}^{m}a_i$. For example, the monomials $x_1^{p-1}\cdots x_s^{p-1}x_{s+1}$ and $x_1^{p-1}\cdots x_s^{p-1}x_{s+2}$ have the same $|\,\,|_p$-degree which is $Q_0+p^s$. By \thref{lexordering2.2.2}, we have that $x_1^{p-1}\cdots x_s^{p-1}x_{s+1}\prec_{lex} x_1^{p-1}\cdots x_s^{p-1}x_{s+2}$.
\end{enumerate}
\item Let us describe $\cO(m; \underline{1})$ using the $|\,\,|_p$-degree. Keeping in mind the numbers $Q_0$ and $Q$ introduced in (iv). Since $\cO(m; \underline{1})$ is spanned by monomials 
\[
\text{$\boldsymbol{x}^{A}=x_1^{a_1}\cdots x_m^{a_m}$, \quad $0\leq a_i\leq p-1$,} 
\]
we can describe $\cO(m; \underline{1})$ as 
\[
\cO(m; \underline{1})=\spn\{\boldsymbol{x}^{A}\,|\, 0\leq |A|_p\leq Q\}.
\]
Let $I$ be the ideal of $\cO(m; \underline{1})$ generated by $x_{s+1}, \dots, x_m$, where $s\geq 1$. Let $f$ be any polynomial in $\cO(m; \underline{1})$. Then $f$ is a linear combination of monomials $\boldsymbol{x}^{A}$ in $\cO(m; \underline{1})$. Since $\cO(m; \underline{1})=I\oplus \cO(s; \underline{1})$, these monomials $\boldsymbol{x}^{A}$ are either in $I$ or in $\cO(s; \underline{1})$. If $\boldsymbol{x}^{A}\in I$, then $A=(a_1, \dots, a_s, a_{s+1}, \dots, a_m)$, where $0\leq \sum_{i=1}^{s}a_i\leq s(p-1)$ and $1\leq \sum_{i=s+1}^{m}a_i\leq (m-s)(p-1)$. It follows that $p^s\leq |A|_p\leq Q$. If $\boldsymbol{x}^{A}\notin I$, i.e. $\boldsymbol{x}^{A}\in \cO(s; \underline{1})$, then $A=(a_1, \dots, a_s, 0, \dots, 0)$, where $0\leq \sum_{i=1}^{s}a_i\leq s(p-1)$. It follows that $0\leq |A|_p\leq Q_0$. By (iv)(a), we can write $f$ as 
\[
f=\sum_{l=0}^{Q_0}\lambda_{A_l}\boldsymbol{x}^{A_l}+g, 
\]
where $\lambda_{A_l}\in k$, $\boldsymbol{x}^{A_l}=x_1^{a_1}\cdots x_s^{a_s}\in \cO(s; \underline{1})$ with $|A_l|_p=l$, and $g\in I$. We will come back to this in \thref{DconjugatetoD1}, Sec.~\ref{nileleminNL}.
\end{enumerate}
\end{rmk}

To get familiar with the $|\,\,|_p$-degree defined in \thref{DegLexdefn}, let us look at some examples.
Moreover, we want to explain the reason that we define $|\,\,|_p$-degree as $|A|_p=\sum_{i=1}^{s}a_ip^{i-1}+p^s\sum_{i=s+1}^{m}a_i$. Initially, we want to compute $z(\boldsymbol{x}^{A})$, where $z=d_0+u$ and $\boldsymbol{x}^{A}$ is any monomial in the maximal ideal $\fm$ of $\cO(m; \underline{1})$. Due to the form of $d_0$, the standard degree of monomials is not useful. Hence we want to define a monomial ordering so that $d_0$ decreases the ``degree'' of the monomial, i.e. $|\,\,|_p$-degree, whereas $u$ increases the $|\,\,|_p$-degree of the monomial; see \thref{DegLexexample}(2) and (3) below. As a result, $z(\boldsymbol{x}^{A})$ is nonzero. Then we want to use these results to show the subspace $M=I+z(\fm)$ has the property that $M\oplus kx_1^{p-1}\cdots x_s^{p-1}=\cO(m; \underline{1})$. But later we realize there is an alternative way to prove this and we only need to do computations for $d_0$; see \thref{zcalculations} and \thref{d-d0d_0image}. However, the $|\,\,|_p$-degree of monomials is still useful for the proof of \thref{DconjugatetoD1} as we may lost of control if we are using the standard degree of monomials. 
 
\begin{ex}\thlabel{DegLexexample}
\begin{enumerate}[\upshape (1)]
\item Let $\boldsymbol{x}^A=x_1^{a_1}\cdots x_s^{a_s}x_{s+1}^{a_{s+1}}\cdots x_m^{a_m}\in\cO(m; \underline{1})$. Let 
\[
|A|:=a_1+\dots+a_s+a_{s+1}+\dots+a_m
\]
be the standard degree of $\boldsymbol{x}^{A}$. Suppose $|A|>0$. Consider the partial derivatives $\del_1, \dots, \del_m$. Note that for $1\leq j\leq m$, if $a_j\neq 0$, then 
\[
\del_j(\boldsymbol{x}^A)=a_jx_1^{a_1}\cdots x_j^{a_j-1}\cdots x_m^{a_m}.
\]
Hence the degree of $\del_j$ with respect to the standard degree of monomials is $-1$. Consider the $|\,\,|_p$-degree.
Since 
\[
|A|_p=a_1+a_2p+\dots +a_j{p^{j-1}}+\dots+a_sp^{s-1}+p^s\sum_{j=s+1}^{m}a_j,
\]
we see that for $1\leq j\leq s$, the degree of $\del_j$ with respect to $|\,\,|_p$ is $-p^{j-1}$, and for $s+1\leq j\leq m$, the degree of $\del_j$ with respect to $|\,\,|_p$ is $-p^s$.
\item Consider $\boldsymbol{x}^A(x_i\del_j)$, where $\boldsymbol{x}^A\in \cO(m; \underline{1})$ with $|A|>0, s+1\leq i\leq m$ and \\$1\leq j\leq m$. Suppose $\boldsymbol{x}^A x_i\neq 0$. By (1), we know that the degree of $\boldsymbol{x}^A(x_i\del_j)$ with respect to the standard degree of monomials is 
\[
|A|+(1-1)=|A|>0.
\] 
Consider the $|\,\,|_p$-degree. By \thref{DegLexrmk}(iii) and (1), we know that for $1\leq j\leq s$, the degree of $\boldsymbol{x}^A(x_i\del_j)$ with respect to $|\,\,|_p$ is
\[
|A|_p+(p^s-p^{j-1}).
\]
Since $|A|>0$, we have that $|A|_p>0$. Since $1\leq j\leq s$, we have that $p^s-p^{j-1}>0$. Hence $|A|_p+(p^s-p^{j-1})>0$. By (1) again, we know that for $s+1\leq j\leq m$, the degree of $\boldsymbol{x}^A(x_i\del_j)$ with respect to $|\,\,|_p$ is 
\[
|A|_p+(p^s-p^s)=|A|_p,
\]
which is positive, too. Recall that $u\in (\sum_{j=1}^{m}I\del_j)\cap W(m; \underline{1})_{(p-1)}$, where $I$ is the ideal of $\cO(m; \underline{1})$ generated by $x_{s+1}, \dots, x_m$ with $s\geq 1$, and 
\[
W(m; \underline{1})_{(p-1)}=\bigg\{\sum_{j=1}^{m}f_j\del_j\,|\, \deg f_j\geq p \,\text{for all $j$}\bigg\}.
\] 
Then $u$ is a linear combination of $\boldsymbol{x}^{A}(x_i\del_j)$, where $\boldsymbol{x}^{A}x_i$, $s+1\leq i\leq m$, is a monomial in $I$ with $p\leq |A|+1\leq m(p-1)$, and $1\leq j\leq m$. By above, we know that each $\boldsymbol{x}^{A}(x_i\del_j)$ has degree $|A|$ with respect to the standard degree of monomials. Hence applying $\boldsymbol{x}^{A}(x_i\del_j)$ to any $\boldsymbol{x}^{A'}=x_1^{a_1'}\cdots x_m^{a_m'}\in\fm$ with $a_j'\neq 0$, the standard degree increases by $|A|\geq p-1$, and the $|\,\,|_p$-degree increases by at least $|A|_p>0$. Since $\boldsymbol{x}^{A}x_i\in I$ and $I$ is an ideal of $\cO(m; \underline{1})$, it follows that $\big(\boldsymbol{x}^{A}(x_i\del_j)\big)(\boldsymbol{x}^{A'})\in I$ and so
\begin{equation*}\label{sumxiDiaction}
u(\boldsymbol{x}^{A'})=\sum_{A''}\lambda_{A''}\boldsymbol{x}^{A''},
\end{equation*}
where $\lambda_{A''} \in k$ and $\boldsymbol{x}^{A''}\in I$ with $p-1\leq |A|= |A''|-|A'|\leq m(p-1)-1$ and $|A''|_p-|A'|_p\geq|A|_p>0$.
\item Consider $x_1^{p-1}\cdots x_{r-1}^{p-1}\del_{r}$, where $1\leq r\leq s$. By (1), the degree of $x_1^{p-1}\cdots x_{r-1}^{p-1}\del_{r}$ with respect to $|\,\,|_p$ is 
\[
\big((p-1)+(p-1)p+\dots+(p-1)p^{r-2}\big)-p^{r-1}=(p^{r-1}-1)-p^{r-1}=-1.
\]
Since $d_0=\del_1+x_1^{p-1}\del_2+\dots+x_1^{p-1}\cdots x_{s-1}^{p-1}\del_s$ and each summand has $|\,\,|_p$-degree $-1$, it follows that $d_0$ has $|\,\,|_p$-degree $-1$.
\end{enumerate}
\end{ex} 

Now we consider the action of $z=d_0+u$ on the maximal ideal $\fm$ of $\cO(m; \underline{1})$. Note that $\fm=\fm_s\oplus I$, where $\fm_s$ is the maximal ideal of $\cO(s; \underline{1})$ and $I$ is the ideal of $\cO(m; \underline{1})$ generated by $x_{s+1}, \dots, x_m$ with $s\geq 1$. Since $d_0=\del_1+x_1^{p-1}\del_2+\dots+x_1^{p-1}\cdots x_{s-1}^{p-1}\del_s$ is a derivation of $\cO(s; \underline{1})$, we first consider $d_0(\fm_s)$.

\begin{lem}\thlabel{zcalculations}
Let $d_0=\del_1+x_1^{p-1}\del_2+\dots+x_1^{p-1}\cdots x_{s-1}^{p-1}\del_s$ be the derivation of $\cO(s; \underline{1})$. Let $\fm_s$ be the maximal ideal of $\cO(s; \underline{1})$. Let $\boldsymbol{x}^{A_1}=x_r^{a_r}x_{r+1}^{a_{r+1}}\cdots x_s^{a_s}$ be any monomial in $\fm_s$, where $1\leq r\leq s$ is the smallest index such that $1\leq a_r\leq p-1$, and $0\leq a_i\leq p-1$ for $r+1\leq i\leq s$.
Then 
\[
d_0(\boldsymbol{x}^{A_1})=a_r \boldsymbol{x}^{A_2}=a_r x_1^{p-1}\cdots x_{r-1}^{p-1}x_r^{a_r-1}x_{r+1}^{a_{r+1}}\cdots x_s^{a_s}
\]
with $|A_2|_p=|A_1|_p-1$. Hence $d_0(\fm_s)$ is spanned by all monomials $\boldsymbol{x}^{A_2}\in \cO(s; \underline{1})$ except $x_1^{p-1}\cdots x_s^{p-1}$. In particular, $\dim d_0(\fm_s)=p^s-1$.
\end{lem}
\begin{proof}
Let $\boldsymbol{x}^{A_1}=x_r^{a_r}x_{r+1}^{a_{r+1}}\cdots x_s^{a_s}$ be any monomial in $\fm_s$, where $1\leq r\leq s$ is the smallest index such that $1\leq a_r\leq p-1$, and $0\leq a_i\leq p-1$ for $r+1\leq i\leq s$. Since 
\[
d_0=\del_1+x_1^{p-1}\del_2+\dots+x_1^{p-1}\cdots x_{r-1}^{p-1}\del_{r}+x_1^{p-1}\cdots x_{r}^{p-1}\del_{r+1}+\dots+x_1^{p-1}\cdots x_{s-1}^{p-1}\del_s,
\]
we have that 
\[
d_0(\boldsymbol{x}^{A_1})=a_r x_1^{p-1}\cdots x_{r-1}^{p-1}x_r^{a_r-1}x_{r+1}^{a_{r+1}}\cdots x_s^{a_s}=a_r \boldsymbol{x}^{A_2}.
\]
Since $0\leq a_r-1\leq p-2<p-1$, we see that $\boldsymbol{x}^{A_2}\neq x_1^{p-1}\cdots x_s^{p-1}$. By \thref{DegLexexample}(3), we know that the degree of $d_0$ with respect to $|\,\,|_p$ is $-1$. Hence $|A_2|_p=|A_1|_p-1$. By \thref{WnDresultslem3}(iii), we know that $d_0$ acts on $\cO(s; \underline{1})$ as a single Jordan block of size $p^{s}$ with zeros on the main diagonal. So $\dim d_0(\cO(s; \underline{1}))=p^s-1$. Since $\cO(s; \underline{1})=\fm_s\oplus k$ and $d_0(c)=0$ for all $c\in k$, we have that $\dim d_0(\fm_s)=p^s-1$. It follows from the above calculation that $d_0(\fm_s)$ is spanned by all monomials $\boldsymbol{x}^{A_2}\in \cO(s; \underline{1})$ except $x_1^{p-1}\cdots x_s^{p-1}$. Indeed, 
since $A_2=(p-1, \dots, p-1, a_r-1, a_{r+1}, \dots, a_s)$ with $0\leq a_r-1\leq p-2$ and $0\leq a_i\leq p-1$ for $r+1\leq i\leq s$, we see that as $r$ varies from $1$ to $s$, there are $\sum_{r=1}^{s} (p-1)p^{s-r}=p^s-1$ different possibilities for $A_2$ except $(p-1, \dots, p-1)$ (i.e. $p-1$ appears $s$ times). So there are $p^s-1$ distinct monomials $\boldsymbol{x}^{A_2}\in\cO(s; \underline{1})$ except $x_1^{p-1}\cdots x_s^{p-1}$. This completes the proof.
\end{proof}

\begin{lem}\thlabel{d-d0d_0image}
Let $z=d_0+u$ be as in \thref{dformlemma2}. Let 
\[
M=I+z(\fm)
\]
be the subspace of $\cO(m; \underline{1})$, where $\fm$ is the maximal ideal of $\cO(m; \underline{1})$ and $I$ is the ideal of $\cO(m; \underline{1})$ generated by $x_{s+1}, \dots, x_m$ with $s\geq 1$. Then $M$ is spanned by all monomials $\boldsymbol{x}^{A}=x_1^{a_1}\cdots x_m^{a_m}\in\cO(m; \underline{1})$ except $x_1^{p-1}\cdots x_s^{p-1}$. In particular, $\dim M=p^{m}-1$.
\end{lem}
\begin{proof}
Let $M=I+z(\fm)$ be the subspace of $\cO(m; \underline{1})$, where $\fm$ is the maximal ideal of $\cO(m; \underline{1})$ and $I$ is the ideal of $\cO(m; \underline{1})$ generated by $x_{s+1}, \dots, x_m$ with $s\geq 1$. Recall that $z=d_0+u$, where $d_0=\del_1+x_1^{p-1}\del_2+\dots+x_1^{p-1}\cdots x_{s-1}^{p-1}\del_s$ is a derivation of $\cO(s; \underline{1})$ and $u\in (\sum_{i=1}^{m}I\del_i)\cap W(m; \underline{1})_{(p-1)}$. Note that $d_0$ preserves the ideal $I$. Since $u\in \sum_{i=1}^{m}I\del_i$ and $I$ is an ideal of $\cO(m; \underline{1})$, we have that $u(\fm)\subseteq I$. Hence $M=I+z(\fm)=I+d_0(\fm)$. Consider $d_0(\fm)$. Note that $\fm=\fm_s\oplus I$, where $\fm_s$ is the maximal ideal of $\cO(s; \underline{1})$. Since $d_0(\fm_s)\subset \cO(s; \underline{1})$, $d_0(I)\subseteq I$ and $\cO(s; \underline{1})\cap I=\{0\}$, we have that $d_0(\fm)=d_0(\fm_s)\oplus d_0(I)$. Since $d_0(I) \subseteq I$, we have that 
\begin{equation}\label{Moplus}
M=I+z(\fm)=I+d_0(\fm)=I\oplus d_0(\fm_s).
\end{equation}
Since $I$ is the ideal of $\cO(m; \underline{1})$ generated by $x_{s+1}, \dots, x_m$ with $s\geq 1$, then $I$ is spanned by all monomials $x_1^{a_1}\cdots x_s^{a_s}x_{s+1}^{a_{s+1}}\cdots x_m^{a_m}$, where $0\leq \sum_{i=1}^{s}a_i\leq s(p-1)$ and \\$1\leq \sum_{i=s+1}^{m}a_i\leq (m-s)(p-1)$. So $\dim I=p^s(p^{m-s}-1)=p^m-p^s$. By \thref{zcalculations}, $d_0(\fm_s)$ is spanned by all monomials $x_1^{a_1}\cdots x_s^{a_s}\in \cO(s; \underline{1})$ except $x_1^{p-1}\cdots x_s^{p-1}$ and so $\dim d_0(\fm_s)=p^s-1$. Therefore, $M=I\oplus d_0(\fm_s)$ is spanned by all monomials $\boldsymbol{x}^{A}=x_1^{a_1}\cdots x_m^{a_m}\in\cO(m; \underline{1})$ except $x_1^{p-1}\cdots x_s^{p-1}$. In particular, $\dim M=(p^m-p^s)+(p^s-1)=p^{m}-1$. This completes the proof.
\end{proof}

\subsection{Nilpotent elements of $\cL$}\label{nileleminNL}
In the last section we have obtained all the results required for nilpotent elements of $\dd D$. Now let us consider nilpotent elements of $\cL=(S\otimes \cO(m; \underline{1}))\rtimes \dd D$.

\begin{lem}\thlabel{D1pNcalculations}
Let $D'=\sum_{i=0}^{m(p-1)}s_i\otimes f_i+d$ be an element of $\cL$, where $s_i\in S$, $f_i\in \cO(m;\underline{1})$ with $\deg f_i=i$, and $d\in \cN(\dd D)$ with $d\in W(m; \underline{1})_{(0)}$. Then $D'\in \cN(\cL)$ if and only if $s_0\in \cN(S)$.
\end{lem}
\begin{proof}
Let $D'=\sum_{i=0}^{m(p-1)}s_i\otimes f_i+d$ be as in the lemma. We first show that for every $N\geq 1$,
\begin{equation}\label{DpNresults}
\text{$(D')^{p^{N}}=d^{p^{N}}+(s_0^{p^{N}}\otimes f_{0}^{p^{N}})$ $+$ (other terms in $S\otimes\fm$)}.
\end{equation}
By \thref{generalJacobF}, 
\[
(D')^{p^{N}}=d^{p^{N}}+\bigg(\sum_{i=0}^{m(p-1)}s_i\otimes f_i\bigg)^{p^{N}}+\sum_{l=0}^{N-1}v_{l}^{p^{l}}, 
\]
where $v_l$ is a linear combination of commutators in $\sum_{i=0}^{m(p-1)}s_i\otimes f_i$ and $d$. By Jacobi identity, we can rearrange each $v_l$ so that $v_l$ is in the span of 
\[
\bigg[w_t,\bigg[w_{t-1},\bigg[\dots,\bigg[w_2,\bigg[w_1, \sum_{i=0}^{m(p-1)}s_i\otimes f_i\bigg]\dots\bigg],
\] 
where $t=p^{N-l}-1$ and each $w_\alpha, 1\leq \alpha\leq t$, is equal to $\sum_{i=0}^{m(p-1)}s_i\otimes f_i$ or to $d$. 

We show that $[w_t,[w_{t-1},[\dots,[w_2,[w_1, \sum_{i=0}^{m(p-1)}s_i\otimes f_i]\dots]\in S\otimes \fm$. If $w_1=\sum_{i=0}^{m(p-1)}s_i\otimes f_i$, then $[w_1, \sum_{i=0}^{m(p-1)}s_i\otimes f_i]=0$. Suppose $w_1=d$. By our assumption, $d\in W(m; \underline{1})_{(0)}=\sum_{j=1}^{m}\fm \del_j$. It is easy to see that $d(\cO(m; \underline{1}))\subseteq \fm$. Hence 
\[
\bigg[d, \sum_{i=0}^{m(p-1)}s_i\otimes f_i\bigg]=\sum_{i=0}^{m(p-1)}s_i\otimes d(f_i)\in S\otimes \fm. 
\]
This shows that $[w_1,\sum_{i=0}^{m(p-1)}s_i\otimes f_i]\in S\otimes \fm$. Consider $[w_2, [w_1, \sum_{i=0}^{m(p-1)}s_i\otimes f_i]]$. If $w_2=d$, then $[d,[w_1, \sum_{i=0}^{m(p-1)}s_i\otimes f_i]]\in [d, S\otimes \fm]=S\otimes d(\fm)\subseteq S\otimes \fm$. If $w_2=\sum_{i=0}^{m(p-1)}s_i\otimes f_i$, then 
\[
\bigg[\sum_{i=0}^{m(p-1)}s_i\otimes f_i,\bigg[w_1, \sum_{i=0}^{m(p-1)}s_i\otimes f_i\bigg]\bigg]\in[S\otimes \cO(m; \underline{1}), S\otimes \fm]\subseteq S\otimes \fm.
\]
This shows that $[w_2, [w_1, \sum_{i=0}^{m(p-1)}s_i\otimes f_i]] \in S\otimes \fm$. Continuing in this way, i.e. using that $d$ preserves $S\otimes \fm$ and $S\otimes \fm$ is an ideal of $S\otimes \cO(m; \underline{1})$, we eventually get \\$[w_t,[w_{t-1},[\dots,[w_2,[w_1, \sum_{i=0}^{m(p-1)}s_i\otimes f_i]\dots]\in S\otimes \fm$. Since each $v_l$ is a linear combination of these commutators, we have that $v_l\in S\otimes \fm$ for all $0\leq l\leq N-1$. Consider $v_l^{p^{l}}$, where $0<l\leq N-1$. Since $v_l$ is a linear combination of elements in $S\otimes \fm$, it follows from \thref{generalJacobF} and $[S\otimes \fm, S\otimes \fm]\subseteq S\otimes \fm$ that $v_l^{p}\in S\otimes \fm$. Continuing in this way, we get $v_l^{p^{l}}\in S\otimes \fm$ for all $0<l\leq N-1$. Therefore, 
\begin{equation}\label{DpNfirstsummandeq}
\sum_{l=0}^{N-1}v_{l}^{p^{l}}\in S\otimes \fm.
\end{equation}

We consider $\big(\sum_{i=0}^{m(p-1)}s_i\otimes f_i\big)^{p^{N}}$ and show that  
\begin{equation}\label{DpNlastsummandeq} 
\bigg(\text{$\sum_{i=0}^{m(p-1)}s_i\otimes f_i\bigg)^{p^{N}}=(s_0^{p^{N}}\otimes f_0^{p^{N}})$ $+$ (other terms in $S\otimes \fm$)}.
\end{equation}
Since $f_i\in \fm$ for $0< i\leq m(p-1)$, we have that $f_i^{p^{N}}=0$. It follows from \thref{generalJacobF} that 
\begin{equation*}
\bigg(\sum_{i=0}^{m(p-1)}s_i\otimes f_i\bigg)^{p^{N}}=\sum_{i=0}^{m(p-1)}s_i^{p^{N}}\otimes f_i^{p^{N}}+\sum_{l=0}^{N-1}u_l^{p^{l}}=(s_0^{p^{N}}\otimes f_0^{p^{N}})+\sum_{l=0}^{N-1}u_l^{p^{l}},
\end{equation*}
where $u_l$ is a linear combination of commutators in $s_i\otimes f_i$, $0\leq i\leq m(p-1)$. By Jacobi identity, we can rearrange each $u_l$ so that $u_l$ is in the span of 
\[
[\eta_t,[\eta_{t-1},[\dots,[\eta_2,[\eta_1, s_0\otimes f_0]\dots],
\] 
where $t=p^{N-l}-1$ and each $\eta_\alpha, 1\leq \alpha\leq t$, is equal to some $s_i\otimes f_i$, $0\leq i\leq m(p-1)$. Since $S\otimes\fm$ is an ideal of $S\otimes \cO(m; \underline{1})$, one 
can show similarly that $[\eta_t,[\eta_{t-1},[\dots,[\eta_2,[\eta_1, s_0\otimes f_0]\dots]\in S\otimes \fm$. Hence $u_l\in S\otimes \fm$ for all $0\leq l\leq N-1$ and $u_l^{p^{l}}\in S\otimes \fm$ for all $0<l\leq N-1$. Therefore, $\sum_{l=0}^{N-1}u_l^{p^{l}} \in S\otimes \fm$. This gives \eqref{DpNlastsummandeq} as desired.

It follows from \eqref{DpNlastsummandeq} and \eqref{DpNfirstsummandeq} that for every $N\geq 1$, the $p$-th powers of $D'$ is given by 
\begin{align*}
(D')^{p^{N}}&=d^{p^{N}}+\bigg(\sum_{i=0}^{m(p-1)}s_i\otimes f_i\bigg)^{p^{N}}+\sum_{l=0}^{N-1}v_{l}^{p^{l}}\\
&=\text{$d^{p^{N}}+(s_0^{p^{N}}\otimes f_0^{p^{N}})$ $+$ (other terms in $S\otimes \fm$)}. 
\end{align*}
By our assumption, $d$ is a nilpotent element of $\dd D$ such that $d\in W(m; \underline{1})_{(0)}$. Hence $d$ preserves $S\otimes \fm$. Therefore, for $N\gg0$, $D'$ is a nilpotent element of $\cL$ if and only if  $s_0$ is a nilpotent element of $S$. This completes the proof.
\end{proof}

Next we consider elements of the form $D=\sum_{i=0}^{m(p-1)}s_i\otimes f_i+z$, where $s_i\in S$, $f_i\in \cO(m;\underline{1})$ with $\deg f_i=i$, and $z=d_0+u \in \cN(\dd D)$ as in \thref{dformlemma2}. We first construct some automorphisms of $\cL$; see \thref{Sautoexps}. Then we show that $D$ can be reduced to an element in a nice form; see \thref{DconjugatetoD1}.

\begin{lem}\thlabel{Sautoexps}
Let $\tilde{s}\otimes f$ be an element of $S \otimes \fm$. Then $\exp(\ad(\tilde{s}\otimes f))$ is an automorphism of $\cL$.
\end{lem}

\begin{proof}
The proof is similar to \thref{autg} in Sec.~\ref{sl2section2.2.1}. Let $\tilde{s}\otimes f$ be an element of $S \otimes \fm$. We first show that $\exp(\ad(\tilde{s}\otimes f))$ is an automorphism of $S\otimes \cO(m; \underline{1})$. Since $f\in \fm$, we have that $f^{p}=0$. Hence 
$(\ad(\tilde{s}\otimes f))^{p}=\ad (\tilde{s}^{p}\otimes f^p)=0$. Moreover, it is easy to check that for any $w_1, w_2 \in S\otimes \cO(m; \underline{1})$,
\[
\sum_{0\leq i, j\leq p,\, i+j\geq p}\frac{1}{i!j!}\big[(\ad(\tilde{s}\otimes f))^{i}(w_1), (\ad(\tilde{s}\otimes f))^{j}(w_2)\big]=0.
\]
Hence $\exp(\ad(\tilde{s}\otimes f))$ is an automorphism of $S\otimes \cO(m; \underline{1})$. Then it induces an automorphism of the derivation algebra 
\[
\Der(S\otimes \cO(m; \underline{1}))=\big(S\otimes \cO(m; \underline{1})\big) \rtimes \big(\Id_S \otimes W(m; \underline{1})\big)
\]
via conjugation. Note that 
\[
\cL=(S\otimes \cO(m; \underline{1}))\rtimes \dd D\subset \Der(S\otimes \cO(m; \underline{1})).
\]
To conclude that $\exp(\ad(\tilde{s}\otimes f))$ is an automorphism of $\cL$, we need to check that $\exp(\ad(\tilde{s}\otimes f))$ preserves $\cL$. Explicitly, we need to show that for any $d_1\in \dd D$, \\$\exp(\ad(\tilde{s}\otimes f))\circ d_1\circ \exp(\ad(-\tilde{s}\otimes f)) \in \cL$. In the proof of \thref{autg} (see \eqref{edele-1a}), we did a similar computation. By the same computational method and the following reasons:
\begin{enumerate}[\upshape(i)]
\item $d_1$ is a derivation of $\cO(m; \underline{1})$,
\item $\exp(\ad(\tilde{s}\otimes f))\exp(\ad(-\tilde{s}\otimes f))=\Id$,
\item $f^{p}=0$,
\item $(p-1)!\equiv -1 (\modd p)$,\,\text{and}
\item $S\cong \ad S$ via the adjoint representation and hence we may identify $S\otimes \cO(m; \underline{1})$ with its image in $\fgl(S\otimes \cO(m; \underline{1}))$ under $\ad$,
\end{enumerate}
one can show that 
\begin{equation}\label{expformulasec2}
\exp(\ad(\tilde{s}\otimes f))\circ d_1\circ \exp(\ad(-\tilde{s}\otimes f)) =d_1-\tilde{s}\otimes d_1(f)-\tilde{s}^{p}\otimes f^{p-1}d_1(f).
\end{equation}
It is clear that $\exp(\ad(\tilde{s}\otimes f))\circ d_1\circ \exp(\ad(-\tilde{s}\otimes f)) \in \cL$. Hence $\exp(\ad(\tilde{s}\otimes f))$ preserves $\cL$ and it is an automorphism of $\cL$. This completes the proof.
\end{proof}

Let $H$ denote the subgroup of $\Aut(\cL)$ generated by $\exp(\ad(\tilde{s}\otimes f))$, where $\tilde{s}\otimes f \in S \otimes \fm$. Then $H$ is a connected algebraic group with $S\otimes \fm\subseteq\Lie(H)$. We show that applying suitable elements of $H$, the following element of $\cL$ can be reduced to an element in a nice form.

\begin{lem}\thlabel{DconjugatetoD1}
Let $D=\sum_{i=0}^{m(p-1)}s_i\otimes f_i+z$ be an element of $\cL$, where $s_i\in S$, $f_i\in \cO(m;\underline{1})$ with $\deg f_i=i$, and $z=d_0+u \in \cN(\dd D)$ as in \thref{dformlemma2}. Then $D$ is conjugate under $H$ to 
\begin{align*}
D_1=s_0'\otimes x_1^{p-1}\cdots x_{s}^{p-1}+v'+z,
\end{align*}
where $s_0'\in S$ (possibly $0$), $v'\in S\otimes I$, $I$ is the ideal of $\cO(m; \underline{1})$ generated by $x_{s+1}, \dots, x_m$ with $s\geq 1$, and $z=d_0+u\in \cN(\dd D)$ as in \thref{dformlemma2}.
\end{lem}

\begin{myproof}[Strategy of the proof]\renewcommand{\qedsymbol}{}
\textbf{\textit{Step 1}}. Let $D$ be as in the lemma. By \thref{DegLexrmk}(v), we can rewrite $D$ as 
\[
D=\sum_{l=0}^{Q_0}s_{A_l}\otimes \boldsymbol{x}^{A_l}+v+z,
\]
where $s_{A_l}\in S$, $\boldsymbol{x}^{A_l}=x_1^{a_1}\cdots x_s^{a_s}\in \cO(s; \underline{1})$ with $0\leq |A_l|_p=l\leq Q_0=p^s-1$, $v\in S\otimes I$, and $z=d_0+u \in \cN(\dd D)$ as in \thref{dformlemma2}. We want to show that $D$ is conjugate under $H=\langle\exp(\ad(\tilde{s}\otimes f))\,|\,\tilde{s}\otimes f\in S\otimes \fm\rangle$ to $D_1=s_0'\otimes x_1^{p-1}\cdots x_{s}^{p-1}+v'+z$ for some $s_0'\in S$ (possibly $0$) and $v'\in S\otimes I$. We first observe that for any $\exp(\ad(\tilde{s}\otimes f))$ in $H$, $\exp(\ad(\tilde{s}\otimes f))(v+u)=u+v_3$ for some $v_3=v_3(\tilde{s}, f)\in S\otimes I$. Note that $(S\otimes I)\cap (S\otimes \cO(s; \underline{1}))=\{0\}$. Moreover, all $\boldsymbol{x}^{A_l}$ are in $\cO(s; \underline{1})$ and $d_0$ is a derivation of $\cO(s; \underline{1})$. Hence to show that $D$ is conjugate under $H$ to $D_1$, we just need to show that $\sum_{l=0}^{Q_0}s_{A_l}\otimes \boldsymbol{x}^{A_l}+d_0$ is conjugate under $H$ to $s_0'\otimes x_1^{p-1}\cdots x_{s}^{p-1}+d_0$ for some $s_0'\in S$ (possibly $0$). For that, we only need to apply automorphisms $\exp(\ad(\tilde{s}\otimes f))$ with $\tilde{s}\otimes f\in S\otimes \fm_s$. Here $\fm_s$ denotes the maximal ideal of $\cO(s; \underline{1})$.

\textbf{\textit{Step 2}}. Let $D_0=\sum_{l=0}^{Q_0}s_{A_l}\otimes \boldsymbol{x}^{A_l}+d_0$. We show that $D_0$ is conjugate under $H$ to $s_0'\otimes x_1^{p-1}\cdots x_{s}^{p-1}+d_0$ for some $s_0'\in S$ (possibly $0$). If $s_{A_l}=0$ for all $l$, then $D_0$ is of the desired form. If not all $s_{A_l}$ are zero, then we look at $\boldsymbol{x}^{A_l}$'s with $s_{A_l}\neq 0$ and take the one with the smallest $|\,\,|_p$-degree, say it is $\boldsymbol{x}^{A_K}$ with $|A_K|_p=K$. Then 
\[
D_0=\sum_{l=K}^{Q_0}s_{A_l}\otimes \boldsymbol{x}^{A_l}+d_0.
\]
If $\boldsymbol{x}^{A_K}=x_1^{p-1}\cdots x_s^{p-1}$, then $|A_K|_p=K=Q_0$ and $D_0$ is of the desired form. If $\boldsymbol{x}^{A_K}=x_1^{a_1}\cdots x_s^{a_s}\neq x_1^{p-1}\cdots x_s^{p-1}$, then $0\leq |A_K|_p=K<Q_0$ and there exist $0\leq a_i<p-1$ for some $1\leq i\leq s$. Let $1\leq r\leq s$ be the smallest index such that $0\leq a_r< p-1$. Then $\boldsymbol{x}^{A_K}=x_1^{p-1}\dots x_{r-1}^{p-1}x_r^{a_r}x_{r+1}^{a_{r+1}}\cdots x_s^{a_s}$.
Let $f=\boldsymbol{x}^{\tilde{A_K}}=x_r^{a_r+1}x_{r+1}^{a_{r+1}}\cdots x_s^{a_s}\in \fm_s$ and $\tilde{s}=(a_{r}+1)^{-1}s_{A_K}\in S$. It follows from \thref{zcalculations} that applying $\exp(\ad(\tilde{s}\otimes f))$ to $D_0$, we may assume that 
\[
D_0=\sum_{l=K+1}^{Q_0}s'_{A_l}\otimes \boldsymbol{x}^{A_l}+d_0,
\]
where $s'_{A_l}\in S$, $\boldsymbol{x}^{A_l}=x_1^{a_1}\cdots x_s^{a_s}\in \cO(s; \underline{1})$ with $|A_l|_p=l$. Continue doing the above until we get $D_0$ is conjugate under $H$ to $s_0'\otimes x_1^{p-1}\cdots x_{s}^{p-1}+d_0$ for some $s_0'\in S$ (possibly $0$). It follows that $D$ is conjugate under $H$ to $D_1$.
\end{myproof}

\begin{proof}
\textbf{\textit{Step 1}}. Let $D=\sum_{i=0}^{m(p-1)}s_i\otimes f_i+z$ be as in the lemma. Let $I$ be the ideal of $\cO(m; \underline{1})$ generated by $x_{s+1}, \dots, x_m$, where $s\geq 1$. Since $f_i\in \cO(m; \underline{1})$, it follows from \thref{DegLexrmk}(v) that we can write 
\[
f_i=\sum_{l=0}^{Q_0}\lambda_{A_l, i}\boldsymbol{x}^{A_l}+g_i, 
\]
where $\lambda_{A_l, i}\in k$, $\boldsymbol{x}^{A_l}=x_1^{a_1}\cdots x_s^{a_s}\in \cO(s; \underline{1})$ with $0\leq |A_l|_p=l\leq Q_0=p^s-1$, and $g_i\in I$. Then we can rewrite $D$ as 
\begin{align}\label{newformD}
D=\sum_{l=0}^{Q_0}s_{A_l}\otimes \boldsymbol{x}^{A_l}+v+z,
\end{align}
where $s_{A_l}\in S$, $\boldsymbol{x}^{A_l}=x_1^{a_1}\cdots x_s^{a_s}\in \cO(s; \underline{1})$ with $0\leq |A_l|_p=l\leq Q_0=p^s-1$, $v\in S\otimes I$, and $z=d_0+u \in \cN(\dd D)$ as in \thref{dformlemma2}. We want to show that $D$ is conjugate under $H=\langle\exp(\ad(\tilde{s}\otimes f))\,|\,\tilde{s}\otimes f\in S\otimes \fm\rangle$ to $D_1=s_0'\otimes x_1^{p-1}\cdots x_{s}^{p-1}+v'+z$ for some $s_0'\in S$ (possibly $0$) and $v'\in S\otimes I$. We claim that this problem can be reduced to show that $\sum_{l=0}^{Q_0}s_{A_l}\otimes \boldsymbol{x}^{A_l}+d_0$ is conjugate under $H$ to $s_0'\otimes x_1^{p-1}\cdots x_{s}^{p-1}+d_0$ for some $s_0'\in S$ (possibly $0$). Take $D$ as in \eqref{newformD} and let $\exp(\ad (\tilde{s}\otimes f))$ be any element in $H$. By \thref{Sautoexps} (see \eqref{expformulasec2}), 
\begin{align*}
\exp(\ad(\tilde{s}\otimes f))(D)=&\exp(\ad(\tilde{s}\otimes f))\bigg(\sum_{l=0}^{Q_0}s_{A_l}\otimes \boldsymbol{x}^{A_l}+d_0\bigg)+\exp(\ad(\tilde{s}\otimes f))(v+u)\\
=&\sum_{l=0}^{Q_0}s_{A_l}\otimes \boldsymbol{x}^{A_l}+\sum_{j=1}^{p-1}\sum_{l=0}^{Q_0}\frac{1}{j!}(\ad \tilde{s})^{j}(s_{A_l})\otimes f^{j}\boldsymbol{x}^{A_l}\\
&+d_0-\tilde{s}\otimes d_0(f)-\tilde{s}^{p}\otimes f^{p-1}d_0(f)+\exp(\ad(\tilde{s}\otimes f))(v+u).
\end{align*}
We show that $\exp(\ad(\tilde{s}\otimes f))(v+u)=u+v_3$ for some $v_3=v_3(\tilde{s},f)\in S\otimes I$. Since $I$ is an ideal of $\cO(m; \underline{1})$, this implies that $S\otimes I$ is an ideal of $S\otimes \cO(m; \underline{1})$. Hence $S\otimes I$ is stabilized by all $\exp(\ad(\tilde{s}\otimes f))$ in $H$. In particular, $\exp(\ad(\tilde{s}\otimes f))(v)=v_1$ for some $v_1=v_1(\tilde{s}, f)\in S\otimes I$. By \thref{Sautoexps} (see \eqref{expformulasec2}) again,
\[
\exp(\ad (\tilde{s}\otimes f))(u)=u-\tilde{s}\otimes u(f)-\tilde{s}^{p}\otimes f^{p-1}u(f).
\]
Since $u\in (I\del_1+\dots+I\del_m)$, we have that $u(\fm)\subseteq I$. In particular, $u(f)\in I$. Hence 
$\exp(\ad (\tilde{s}\otimes f))(u)=u+v_2$ for some $v_2=v_2(\tilde{s},f)\in S\otimes I$. It follows that $\exp(\ad(\tilde{s}\otimes f))(v+u)=u+v_3$ for some $v_3=v_3(\tilde{s},f)=v_1(\tilde{s},f)+v_2(\tilde{s},f)\in S\otimes I$. Hence for any $\exp(\ad (\tilde{s}\otimes f))\in H$, 
\begin{align*}
\exp(\ad(\tilde{s}\otimes f))(D)=&\exp(\ad(\tilde{s}\otimes f))\bigg(\sum_{l=0}^{Q_0}s_{A_l}\otimes \boldsymbol{x}^{A_l}+d_0\bigg)+(u+v_3)\\
=&\sum_{l=0}^{Q_0}s_{A_l}\otimes \boldsymbol{x}^{A_l}+\sum_{j=1}^{p-1}\sum_{l=0}^{Q_0}\frac{1}{j!}(\ad \tilde{s})^{j}(s_{A_l})\otimes f^{j}\boldsymbol{x}^{A_l}\\
&+d_0-\tilde{s}\otimes d_0(f)-\tilde{s}^{p}\otimes f^{p-1}d_0(f)+(u+v_3).
\end{align*}
Applying another element $\exp(\ad (\tilde{s_1}\otimes f'))\in H$ to the above, we still get \\$\exp(\ad (\tilde{s_1}\otimes f'))(u+v_3)=u+v_4$ for some $v_4=v_4(\tilde{s_1}, f')\in S\otimes I$. Note that $(S\otimes I)\cap (S\otimes\cO(s; \underline{1}))=\{0\}$. Moreover, all $\boldsymbol{x}^{A_l}$ are in $\cO(s; \underline{1})$ and $d_0$ is a derivation of $\cO(s; \underline{1})$. Hence to show that $D=\sum_{l=0}^{Q_0}s_{A_l}\otimes \boldsymbol{x}^{A_l}+v+d_0+u$ is $H$-conjugate to $D_1=s_0'\otimes x_1^{p-1}\cdots x_{s}^{p-1}+v'+d_0+u$, we just need to show that $\sum_{l=0}^{Q_0}s_{A_l}\otimes \boldsymbol{x}^{A_l}+d_0$ is $H$-conjugate to $s_0'\otimes x_1^{p-1}\cdots x_{s}^{p-1}+d_0$. For that, we only need to apply automorphisms $\exp(\ad(\tilde{s}\otimes f))$ with $\tilde{s}\otimes f\in S\otimes \fm_s$ and do computations in $\cO(s; \underline{1})$ and $W(s; \underline{1})$. Here $\fm_s$ denotes the maximal ideal of $\cO(s; \underline{1})$.

\textbf{\textit{Step 2}}. Let $D_0=\sum_{l=0}^{Q_0}s_{A_l}\otimes \boldsymbol{x}^{A_l}+d_0$. We show that $D_0$ is $H$-conjugate to $s_0'\otimes x_1^{p-1}\cdots x_{s}^{p-1}+d_0$ for some $s_0'\in S$ (possibly $0$). If $s_{A_l}=0$ for all $l$, then $D_0$ is of the desired form. If not all $s_{A_l}$ are zero, then we look at $\boldsymbol{x}^{A_l}$'s with $s_{A_l}\neq 0$ and take the one with the smallest $|\,\,|_p$-degree, say it is $\boldsymbol{x}^{A_K}$ with $|A_K|_p=K$. Then 
\[
D_0=\sum_{l=K}^{Q_0}s_{A_l}\otimes \boldsymbol{x}^{A_l}+d_0.
\]
If $\boldsymbol{x}^{A_K}=x_1^{p-1}\cdots x_s^{p-1}$, then $|A_K|_p=K=Q_0$ and $D_0$ is of the desired form. If $\boldsymbol{x}^{A_K}\neq x_1^{p-1}\cdots x_s^{p-1}$, then $0\leq |A_K|_p=K<Q_0$ and we need to apply $\exp(\ad(\tilde{s}\otimes f))$ to $D_0$ and clear $s_{A_K}\otimes \boldsymbol{x}^{A_K}$. By \thref{Sautoexps} (see \eqref{expformulasec2}), we know that for any $\exp(\ad (\tilde{s}\otimes f))\in H$,
\begin{equation}\label{expclearform}
\begin{aligned}
\exp(\ad(\tilde{s}\otimes f))(D_0)=&\sum_{l=K}^{Q_0}s_{A_l}\otimes \boldsymbol{x}^{A_l}+\sum_{j=1}^{p-1}\sum_{l=K}^{Q_0}\frac{1}{j!}(\ad \tilde{s})^{j}(s_{A_l})\otimes f^{j}\boldsymbol{x}^{A_l}\\
&+d_0-\tilde{s}\otimes d_0(f)-\tilde{s}^{p}\otimes f^{p-1}d_0(f).
\end{aligned}
\end{equation}
Let $A_K=(a_1, \dots, a_s)$. Since $\boldsymbol{x}^{A_K}\neq x_1^{p-1}\cdots x_s^{p-1}$, there exist $0\leq a_i<p-1$ for some $1\leq i\leq s$. Note that if $|A_K|_p=0$, then $a_i=0$ for all $1\leq i\leq s$. Let $1\leq r\leq s$ be the smallest index such that $0\leq a_r< p-1$. Then 
\[
\boldsymbol{x}^{A_K}=x_1^{p-1}\dots x_{r-1}^{p-1}x_r^{a_r}x_{r+1}^{a_{r+1}}\cdots x_s^{a_s}.
\]
Let $f=\boldsymbol{x}^{\tilde{A_K}}=x_r^{a_r+1}x_{r+1}^{a_{r+1}}\cdots x_s^{a_s}\in \fm_s$. By \thref{zcalculations}, 
\begin{equation*}
d_0(f)=d_0(\boldsymbol{x}^{\tilde{A_K}})=(a_{r}+1)\boldsymbol{x}^{A_K}
\end{equation*}
with $|A_K|_p=|\tilde{A_K}|_p-1$. Let $\tilde{s}=(a_{r}+1)^{-1}s_{A_K}\in S$. Substituting $\tilde{s}$ and $f$ into \eqref{expclearform}, we get
\begin{equation}\label{expclearA_K}
\begin{aligned}
\exp(\ad(\tilde{s}\otimes f))(D_0)=&\sum_{l=K+1}^{Q_0}s_{A_l}\otimes \boldsymbol{x}^{A_l}\\
&+\sum_{j=1}^{p-1}\sum_{l=K+1}^{Q_0}\frac{1}{j!}(a_{r}+1)^{-j}(\ad s_{A_K})^{j}(s_{A_l})\otimes(\boldsymbol{x}^{\tilde{A_K}})^{j}\boldsymbol{x}^{A_l}\\
&+d_0-(a_{r}+1)^{-p+1}s_{A_K}^{p}\otimes (\boldsymbol{x}^{\tilde{A_K}})^{p-1}\boldsymbol{x}^{A_K}.
\end{aligned}
\end{equation}

We want to show that $\exp(\ad(\tilde{s}\otimes f))(D_0)=\sum_{l=K+1}^{Q_0}s'_{A_l}\otimes \boldsymbol{x}^{A_l}+d_0$, where $s'_{A_l}\in S$ and $\boldsymbol{x}^{A_l}=x_1^{a_1}\cdots x_s^{a_s}\in \cO(s; \underline{1})$ with $|A_l|_p=l$. Look at the second and the last summands in \eqref{expclearA_K}. Since 
$\boldsymbol{x}^{\tilde{A_{K}}}\in \fm_s$ and $\fm_s$ is an ideal of $\cO(s; \underline{1})$, we see that these two summands are in $S\otimes \fm_s$. This implies that every nonzero monomial $\boldsymbol{x}^{A}$ in these two summands has $|A|_p\leq Q_0$. Moreover, we show that every nonzero monomial $\boldsymbol{x}^{A}$ in these two summands has $|A|_p>K+1$. We start with the second summand. By \thref{DegLexrmk}(iii), we know that for any monomials $\boldsymbol{x}^{A'}$ and $\boldsymbol{x}^{A''}$ in $\cO(s; \underline{1})$, if $\boldsymbol{x}^{A'}\boldsymbol{x}^{A''}=\boldsymbol{x}^{A'''}\neq 0$, then $|A'''|_p=|A'|_p+|A''|_p$.
Consider $(\boldsymbol{x}^{\tilde{A_K}})^{j}$, where $1\leq j\leq p-1$. By our choice, $\boldsymbol{x}^{\tilde{A_K}}\in\fm_s$ with $|\tilde{A_K}|_p=|A_K|_p+1=K+1$. Hence for $2\leq j\leq p-1$, 
if $(\boldsymbol{x}^{\tilde{A_K}})^{j}\neq 0$, then $(\boldsymbol{x}^{\tilde{A_K}})^{j}\in\fm_s$ with $|\,\,|_p$-degree $j|\tilde{A_K}|_p\geq 2(K+1)>K+1$. Now look at $(\boldsymbol{x}^{\tilde{A_K}})^{j}\boldsymbol{x}^{A_l}$ in the second summand. Since $l\geq K+1\geq 1$, every monomial $\boldsymbol{x}^{A_l}$ is in $\fm_s$ and has $|A_l|_p=l\geq K+1$. If $(\boldsymbol{x}^{\tilde{A_K}})^{j}\boldsymbol{x}^{A_l}\neq 0$, then $(\boldsymbol{x}^{\tilde{A_K}})^{j}\boldsymbol{x}^{A_l}$ is in $\fm_s$ and has $|\,\,|_p$-degree $j|\tilde{A_K}|_p+|A_l|_p\geq |\tilde{A_K}|_p+|A_l|_p\geq(K+1)+(K+1)>K+1$.

Look at $(\boldsymbol{x}^{\tilde{A_K}})^{p-1}\boldsymbol{x}^{A_K}$ in the last summand. By above, if $(\boldsymbol{x}^{\tilde{A_K}})^{p-1}\neq 0$, then $(\boldsymbol{x}^{\tilde{A_K}})^{p-1}$ is in $\fm_s$ and has $|\,\,|_p$-degree $>K+1$. Note that $\boldsymbol{x}^{A_K}\in \cO(s; 1)$ and has $|A_K|_p=K\geq 0$. So if $(\boldsymbol{x}^{\tilde{A_K}})^{p-1}\boldsymbol{x}^{A_K}\neq 0$, then $(\boldsymbol{x}^{\tilde{A_K}})^{p-1}\boldsymbol{x}^{A_K}$ is in $\fm_s$ and has $|\,\,|_p$-degree $(p-1)|\tilde{A_K}|_p+|A_K|_p>(K+1)+K\geq K+1$.

Hence every nonzero monomial $\boldsymbol{x}^{A}$ in the second and the last summands has $K+1<|A|_p\leq Q_0$. Therefore, applying $\exp(\ad(\tilde{s}\otimes f))$ to $D_0$, we may assume that 
\[
D_0=\sum_{l=K+1}^{Q_0}s'_{A_l}\otimes \boldsymbol{x}^{A_l}+d_0
\]
for some $s'_{A_l}\in S$ and $\boldsymbol{x}^{A_l}=x_1^{a_1}\cdots x_s^{a_s}\in \cO(s; \underline{1})$ with $|A_l|_p=l$. 

Now look at $D_0$ and repeat the above process, i.e. take $\boldsymbol{x}^{A_l}$ with the smallest $|\,\,|_p$-degree for which $s'_{A_l}\neq 0$. If $A_l=(a_1, \dots, a_s)=(p-1, \dots, p-1)$, then $D_0$ is of the desired form. If not, then applying a similar automorphism $\exp(\ad(\tilde{s}\otimes f))$ to $D_0$ and clear $s'_{A_l}\otimes \boldsymbol{x}^{A_l}$. Continuing in this way, we eventually get $D_0$ is $H$-conjugate to $s_0'\otimes x_1^{p-1}\cdots x_{s}^{p-1}+d_0$ for some $s_0'\in S$ (possibly $0$). Therefore, $D=D_0+v+u$ is $H$-conjugate to $D_1=s_0'\otimes x_1^{p-1}\cdots x_{s}^{p-1}+d_0+v'+u$ for some $s_0'\in S$ (possibly $0$) and $v'\in S\otimes I$. This completes the proof.
\end{proof}

\begin{lem}\thlabel{D1psnilpotentcalculations}
Let $D_1=s_0'\otimes x_1^{p-1}\cdots x_{s}^{p-1}+v'+z$ be an element of $\cL$, where $s_0'\in S$, $v'\in S\otimes I$, $I$ is the ideal of $\cO(m; \underline{1})$ generated by $x_{s+1}, \dots, x_m$ with $s\geq 1$, and $z=d_0+u \in \cN(\dd D)$ as in \thref{dformlemma2}. Then $D_1\in \cN(\cL)$ if and only if $s_0'\in \cN(S)$.
\end{lem}

\begin{myproof}[Strategy of the proof]\renewcommand{\qedsymbol}{}
This is a computational proof with the following key steps:\\
\textbf{\textit{Step 1}}. Let $D_1=s_0'\otimes x_1^{p-1}\cdots x_{s}^{p-1}+v'+z$ be as in the lemma. We show that 
\begin{align*}
D_1^{p^{s}}=\text{$z^{p^{s}}+(s_0'\otimes (-1)^{s})$ $+$ (other terms in $S\otimes \fm$)}.
\end{align*}
Set $w=s_0'\otimes x_1^{p-1}\cdots x_{s}^{p-1}+v'$. By \thref{generalJacobF}, 
\begin{equation}\label{D1pscalculationfirst}
D_1^{p^{s}}=z^{p^{s}}+w^{p^{s}}+\sum_{r=0}^{s-1}u_r^{p^{r}}, 
\end{equation}
where $u_r$ is a linear combination of commutators in $z$ and $w$. We first show that $w^{p^{s}}\in S\otimes \fm$. Then we consider $\sum_{r=0}^{s-1}u_r^{p^{r}}$. By Jacobi identity, we can rearrange each $u_r$ so that $u_r$ is in the span of $[w_t,[w_{t-1},[\dots,[w_2,[w_1, w]\dots]$, where $t=p^{s-r}-1$ and each $w_\alpha, 1\leq \alpha\leq t$, is equal to $z$ or to $w$. We consider all such iterated commutators $[w_t,[w_{t-1},[\dots,[w_2,[w_1, w]\dots]$ and show that 
\begin{align*}
\text{$u_0=(\ad z)^{p^{s}-1}(w)=(s_0'\otimes (-1)^{s})$ $+$ (other terms in $S\otimes \fm$)},
\end{align*}
and for $p-1\leq t=p^{s-r}-1\leq p^{s-1}-1$,
\begin{align*}
\text{$[w_t,[w_{t-1},[\dots,[w_2,[w_1, w]\dots]\in S\otimes \fm$.}
\end{align*}
Hence for $1\leq r\leq s-1$, $u_r\in S\otimes \fm$ and $u_r^{p^{r}}\in S\otimes \fm$.

Note that $[z, w]=s_0'\otimes z(x_1^{p-1}\cdots x_{s}^{p-1})+[z, v']$. To show the above claims, we need to consider the action of $z=d_0+u$ on $x_1^{p-1}\cdots x_{s}^{p-1}$; see parts (i)-(iii). Then we consider commutators in $z$ and $w$; see parts (iv)-(vi).
\begin{enumerate}[\upshape(i)]
\item We show that for any $0<l<p^s-1$, $0\neq d_0^{l}(x_1^{p-1}\cdots x_{s}^{p-1})\in \fm_s$, where $\fm_s$ is the maximal ideal of $\cO(s; \underline{1})$. 
\item We show by induction that for any $0<l<p^s-1$,
\begin{equation*}
\text{$z^{l}(x_1^{p-1}\cdots x_{s}^{p-1})=d_0^{l}(x_1^{p-1}\cdots x_{s}^{p-1})$ $+$ (other terms in $I$)}\in \fm_s\oplus I.
\end{equation*}
\item We show that $\text{$z^{p^s-1}(x_1^{p-1}\cdots x_{s}^{p-1})=(-1)^s$ $+$ (other terms in $I$)}$.
\item We show that for any $0<l<p^s-1$, $\text{$(\ad z)^{l}(w)\in (S\otimes \fm_s)\oplus (S\otimes I)$}$, and\\
$\text{$(\ad z)^{p^{s}-1}(w)=(s_0'\otimes (-1)^{s})$ $+$ (other terms in $S\otimes \fm$)}$.
\item We show that for any $0<l<p^s-1$, $[w, (\ad z)^{l}(w)]\in S\otimes I$.
\item We show that for $p-1\leq t=p^{s-r}-1\leq p^{s-1}-1$, 
\[
[w_t,[w_{t-1},[\dots,[w_2,[w_1, w]\dots] \in S\otimes \fm,
\] 
where each $w_\alpha, 1\leq \alpha\leq t$, is equal to $z$ or to $w$. Hence for $1\leq r\leq s-1$, $u_r\in S\otimes \fm$ and $u_r^{p^{r}}\in S\otimes \fm$.
\end{enumerate}
It follows from the above that $D_1^{p^{s}}=\text{$z^{p^{s}}+(s_0'\otimes (-1)^{s})$ $+$ (other terms in $S\otimes \fm$)}$.

\textbf{\textit{Step 2}}. We show that $D_1\in\cN(\cL)$ if and only if $s_0' \in \cN(S)$. By the last step, $D_1^{p^{s}}=\text{$z^{p^{s}}+(s_0'\otimes (-1)^{s})$ $+$ (other terms in $S\otimes \fm$)}$. By our assumption, $z$ is nilpotent. By \thref{zpscalculationslemmma}, $z^{p^{s}}\in W(m; \underline{1})_{(0)}$. Hence $z^{p^{s}}$ preserves $S\otimes \fm$. Applying \thref{D1pNcalculations} with $D'=D_1^{p^{s}}$, $d=z^{p^{s}}$, $s_0=s_0'$ and $f_0=(-1)^s$, we get for $N\geq 1$ (see \eqref{DpNresults}), 
\[
\text{$D_1^{p^{s+N}}=z^{p^{s+N}}+((s_0')^{p^{N}}\otimes (-1)^{sp^{N}})$ $+$ (other terms in $S\otimes \fm$)}.
\]
Hence for $s+N\gg 0$, $D_1\in\cN(\cL)$ if and only if $s_0' \in \cN(S)$. 
\end{myproof}

\begin{proof}
\textbf{\textit{Step {1}}}. Let $D_1=s_0'\otimes x_1^{p-1}\cdots x_{s}^{p-1}+v'+z$ be as in the lemma. We show that 
\begin{align*}
D_1^{p^{s}}=\text{$z^{p^{s}}+(s_0'\otimes (-1)^{s})$ $+$ (other terms in $S\otimes \fm$)}.
\end{align*}
Set $w=s_0'\otimes x_1^{p-1}\cdots x_{s}^{p-1}+v'$. By \thref{generalJacobF}, 
\begin{equation}\label{D1pscalculationfirst}
D_1^{p^{s}}=z^{p^{s}}+w^{p^{s}}+\sum_{r=0}^{s-1}u_r^{p^{r}}, 
\end{equation}
where $u_r$ is a linear combination of commutators in $z$ and $w$. We first show that $w^{p^{s}}\in S\otimes \fm$. Since $x_1^{p-1}\cdots x_{s}^{p-1}\in \fm$, then $(x_1^{p-1}\cdots x_{s}^{p-1})^{p^{s}}=0$. By \thref{generalJacobF},
\[
w^{p^{s}}=(v')^{p^s}+\sum_{r=0}^{s-1}{\eta_r}^{p^{r}},
\]
where ${\eta_r}$ is a linear combination of commutators in $s_0'\otimes x_1^{p-1}\cdots x_{s}^{p-1}$ and $v'$. We show that $(v')^{p^s}\in S\otimes I$. Since $v'\in S\otimes I$, we can write $v'=\sum_{i=1}^{n}\tilde{s_i}\otimes g_i$ for some $\tilde{s_i}\in S$ and $g_i\in I$. Since $\cO(m; \underline{1})$ is a local ring, the ideal $I$ is contained in the maximal ideal $\fm$ of $\cO(m; \underline{1})$. Hence $g_i^{p}=0$ for all $i$. Moreover, $[S\otimes I, S\otimes I]\subseteq S\otimes I$. By \thref{generalJacobF}, we have that that $(v')^{p}\in S\otimes I$. Continuing in this way, we get $(v')^{p^s}\in S\otimes I$. Next we show that $\sum_{r=0}^{s-1}{\eta_r}^{p^{r}}\in S\otimes I$. Since $I\subset\fm$, we have that $\fm I\subseteq \fm\cap I=I$. Since $s_0'\otimes x_1^{p-1}\cdots x_{s}^{p-1}\in S\otimes \fm$ and $v'\in S\otimes I$, we have that 
\[
[s_0'\otimes x_1^{p-1}\cdots x_{s}^{p-1}, v']\in [S\otimes \fm, S\otimes I]\subseteq S\otimes \fm I\subseteq S\otimes I.
\]
It is clear that $[S\otimes I, S\otimes I]\subseteq S\otimes I$. So any iterated commutators in  $s_0'\otimes x_1^{p-1}\cdots x_{s}^{p-1}$ and $v'$ are in $S\otimes I$. Since ${\eta_r}$ is a linear combination of these commutators, we have that ${\eta_r}\in S\otimes I$ for all $0\leq r\leq s-1$. By a similar argument as above, one can show that ${\eta_r}^{p^{r}}\in S\otimes I$ for all $0< r\leq s-1$. Hence $\sum_{r=0}^{s-1}{\eta_r}^{p^{r}}\in S\otimes I$. Therefore, $w^{p^{s}}=(v')^{p^s}+\sum_{r=0}^{s-1}{\eta_r}^{p^{r}}\in S\otimes I\subset S\otimes \fm$.

Now we consider $\sum_{r=0}^{s-1}u_r^{p^{r}}$ in \eqref{D1pscalculationfirst}, where $u_r$ is a linear combination of commutators in $z$ and $w$. By Jacobi identity, we can rearrange each $u_r$ so that $u_r$ is in the span of $[w_t,[w_{t-1},[\dots,[w_2,[w_1, w]\dots]$, where $t=p^{s-r}-1$ and each $w_\alpha, 1\leq \alpha\leq t$, is equal to $z$ or to $w$. We consider all such iterated commutators and show that 
\begin{align}\label{claim2}
\text{$u_0=(\ad z)^{p^{s}-1}(w)=(s_0'\otimes (-1)^{s})$ $+$ (other terms in $S\otimes \fm$)},
\end{align}
and for $p-1\leq t=p^{s-r}-1\leq p^{s-1}-1$,
\begin{align}\label{claim1}
\text{$[w_t,[w_{t-1},[\dots,[w_2,[w_1, w]\dots]\in S\otimes \fm$.}
\end{align}
Hence for $1\leq r\leq s-1$, $u_r\in S\otimes \fm$ and $u_r^{p^{r}}\in S\otimes \fm$. 

Recall that $z=d_0+u$, where $d_0=\del_1+x_1^{p-1}\del_2+\dots +x_1^{p-1}\cdots x_{s-1}^{p-1}\del_s$ is a derivation of $\cO(s; \underline{1})$ and $u\in (I\del_1+\dots +I\del_m)\cap W(m; \underline{1})_{(p-1)}$; see \thref{dformlemma2} for notations. Note that $[z, w]=s_0'\otimes z(x_1^{p-1}\cdots x_{s}^{p-1})+[z, v']$. Since $d_0$ preserves the ideal $I$ and so does $z$, we have that $[z,v']\in[z, S \otimes I]=S\otimes z(I)\subseteq S\otimes I$. Hence to show \eqref{claim2} and \eqref{claim1}, we need to consider the action of $z$ on $x_1^{p-1}\cdots x_{s}^{p-1}$. We split our work into the following parts (i)-(vi):

\textbf{(i)} We show that for any $0<l<p^s-1$, $0\neq d_0^{l}(x_1^{p-1}\cdots x_{s}^{p-1})\in \fm_s$, where $\fm_s$ is the maximal ideal of $\cO(s; \underline{1})$. 

There are two ways to prove this result: (a) using the $|\,\,|_p$-degree of monomials or (b) using the standard degree of monomials and some results on $d_0$. Let us do both ways and see the difference.

(a) Let $\boldsymbol{x}^{A}=x_1^{a_1}\cdots x_s^{a_s}$ be any monomial in $\cO(s; \underline{1})$. Note that $\boldsymbol{x}^{A}\in \fm_s$ if and only if $|A|_p>0$. Consider $x_1^{p-1}\cdots x_{s}^{p-1}$, where $A=(p-1, \dots, p-1)$. Then $|A|_p=p^s-1$. By \thref{DegLexexample}(3), we know that the degree of $d_0$ with respect to $|\,\,|_p$ is $-1$. Hence for any $0<l<p^s-1$, $d_0^{l}(x_1^{p-1}\cdots x_{s}^{p-1})$ is a monomial in $\cO(s; \underline{1})$ with $|\,\,|_p$-degree $p^s-1-l>0$. Therefore, $0\neq d_0^{l}(x_1^{p-1}\cdots x_{s}^{p-1})\in \fm_s$.

(b) By \thref{WnDresultslem3}(i) and (iii), we know that 
\[
d_0^{p^{K}}=(-1)^{K}(\del_{K+1}+x_{K+1}^{p-1}\del_{K+2}+\dots+x_{K+1}^{p-1}\cdots x_{s-1}^{p-1}\del_s)
\] 
for all $0\leq K\leq s-1$ and $d_0^{p^{s}}=0$. Moreover, $d_0^{p^{s}-1}(x_1^{p-1}\cdots x_s^{p-1})=(-1)^{s}$. Hence $d_0^{l}(x_1^{p-1}\cdots x_{s}^{p-1})\neq 0$ for any $0<l<p^{s}-1$. We show that $d_0^{l}(x_1^{p-1}\cdots x_{s}^{p-1}) \in \fm_s$. Here we use a similar argument given in \cite[p.~151, line~-8 and p.~153, line 4]{P91}. By \eqref{Omgrading2.2.2}, $\cO(s; \underline{1})$ is a graded $W(s; \underline{1})$-module, i.e. 
\[
\cO(s;\underline{1})=\cO(s;\underline{1})_0\oplus\cO(s;\underline{1})_1\oplus \dots\oplus \cO(s;\underline{1})_{s(p-1)}.
\]
It is easy to see that $\cO(s;\underline{1})_{\eta}\subseteq \fm_s^{\eta}$ for any $\eta\geq 1$. Hence $x_1^{p-1}\cdots x_s^{p-1}\in \fm_s^{s(p-1)}$. Recall that $W(s; \underline{1})$ is a free $\cO(s; \underline{1})$-module with basis $\del_1, \dots, \del_s$. Then for any $\ccD \in W(s; \underline{1})$ and $0<c<s(p-1)$, we have that
\begin{equation}\label{generalDlequations}
\ccD^{c}(x_1^{p-1}\cdots x_s^{p-1}) \in \fm_s^{s(p-1)-c}.
\end{equation}
Since $l<p^s-1=\sum_{K=0}^{s-1}(p-1)p^K$, we have that $l=\sum_{K=0}^{s-1}a_Kp^K$, where $0\leq a_K\leq p-1$ and $\sum_{K=0}^{s-1}{a_K}<s(p-1)$. Then
\begin{equation*}
d_0^{l}(x_1^{p-1}\cdots x_s^{p-1})=\bigg(\prod_{K=0}^{s-1}{\big(d_0^{p^{K}}\big)}^{a_K}\bigg)(x_1^{p-1}\cdots x_s^{p-1}).
\end{equation*}
Applying \eqref{generalDlequations} with $\ccD=d_0^{p^{K}}$ and $c=a_K$ for $0\leq K\leq s-1$, we get 
\begin{equation*}
d_0^{l}(x_1^{p-1}\cdots x_s^{p-1})=\bigg(\prod_{K=0}^{s-1}{(d_0^{p^{K}})}^{a_K}\bigg)(x_1^{p-1}\cdots x_s^{p-1})\in \fm_s^{s(p-1)-\sum_{K=0}^{s-1}{a_K}}\subseteq \fm_s.
\end{equation*}
This proves (i). 

\textbf{(ii)} We show by induction that for any $0<l<p^s-1$,
\begin{equation*}
\text{$z^{l}(x_1^{p-1}\cdots x_{s}^{p-1})=d_0^{l}(x_1^{p-1}\cdots x_{s}^{p-1})$ $+$ (other terms in $I$)}\in \fm_s\oplus I.
\end{equation*}

For $l=1$, we have that $z(x_1^{p-1}\cdots x_{s}^{p-1})=d_0(x_1^{p-1}\cdots x_{s}^{p-1})+u(x_1^{p-1}\cdots x_{s}^{p-1})$. By (i), we know that $d_0(x_1^{p-1}\cdots x_{s}^{p-1})\in \fm_s$. Since $u\in\sum_{j=1}^{m}I\del_j$ and $I$ is an ideal of $\cO(m; \underline{1})$, we have that $u(x_1^{p-1}\cdots x_{s}^{p-1}) \in I$. It is clear that $\fm_s\cap I=\{0\}$. Hence the result holds for $l=1$. Suppose the result holds for $l=r_1<p^s-2$, i.e. $z^{r_1}(x_1^{p-1}\cdots x_{s}^{p-1})=d_0^{r_1}(x_1^{p-1}\cdots x_{s}^{p-1})+g$ for some $g\in I$. Applying $z$ again, we get 
\begin{align*}
z^{r_1+1}(x_1^{p-1}\cdots x_{s}^{p-1})&=d_0^{r_1+1}(x_1^{p-1}\cdots x_{s}^{p-1})+d_0(g)+u\big(d_0^{r_1}(x_1^{p-1}\cdots x_{s}^{p-1})\big)+u(g).
\end{align*}
It is clear that $u\big(d_0^{r_1}(x_1^{p-1}\cdots x_{s}^{p-1})\big)+u(g)\in I$. Since $d_0$ preserves $I$, we have that $d_0(g)\in I$. Since $r_1+1<p^s-1$, it follows from (i) that $d_0^{r_1+1}(x_1^{p-1}\cdots x_{s}^{p-1})\in \fm_s$. Hence the result holds for $l=r_1+1$. Therefore, for any $0<l<p^s-1$, the result holds. This proves (ii).

\textbf{(iii)} We show that $\text{$z^{p^s-1}(x_1^{p-1}\cdots x_{s}^{p-1})=(-1)^s$ $+$ (other terms in $I$)}$.

By (ii), we know that $z^{p^s-2}(x_1^{p-1}\cdots x_{s}^{p-1})=d_0^{p^{s}-2}(x_1^{p-1}\cdots x_{s}^{p-1})+g$ for some $g\in I$. Applying $z$ again and using a similar argument as in (ii), we get
\[
z^{p^s-1}(x_1^{p-1}\cdots x_{s}^{p-1})=d_0^{p^{s}-1}(x_1^{p-1}\cdots x_{s}^{p-1})+g'
\]
for some $g'\in I$. By \thref{WnDresultslem3}(iii), we know that $d_0^{p^{s}-1}(x_1^{p-1}\cdots x_{s}^{p-1})=(-1)^s$. Hence $z^{p^s-1}(x_1^{p-1}\cdots x_{s}^{p-1})=(-1)^s+g'$ as desired. This proves (iii).

Now we consider commutators in $z$ and $w$.

\textbf{(iv)} We show that for any $0<l<p^s-1$, $\text{$(\ad z)^{l}(w)\in (S\otimes \fm_s)\oplus (S\otimes I)$}$, and\\
$\text{$(\ad z)^{p^{s}-1}(w)=(s_0'\otimes (-1)^{s})$ $+$ (other terms in $S\otimes \fm$)}$.

Note that for any $0<l<p^s-1$, $(\ad z)^{l}(w)=s_0'\otimes z^{l}(x_1^{p-1}\cdots x_{s}^{p-1})+(\ad z)^{l}(v')$. Since $v'\in S\otimes I$ and $z$ preserves $I$, we have that $[z, v']\in [z,S\otimes I]=S\otimes z(I)\subseteq S\otimes I$. Then an easy induction on $l$ shows that $(\ad z)^{l}(v')\in S\otimes I$. By (ii), we know that $z^{l}(x_1^{p-1}\cdots x_{s}^{p-1})\in \fm_s\oplus I$. Hence 
\[
(\ad z)^{l}(w)=s_0'\otimes z^{l}(x_1^{p-1}\cdots x_{s}^{p-1})+(\ad z)^{l}(v')\in s_0'\otimes (\fm_s\oplus I)+S\otimes I \subseteq (S\otimes \fm_s)\oplus (S\otimes I).
\]
Similarly, 
\[(\ad z)^{p^{s}-1}(w)=s_0'\otimes z^{p^{s}-1}(x_1^{p-1}\cdots x_{s}^{p-1})+(\ad z)^{p^{s}-1}(v').
\]
Arguing as above, one can show that $(\ad z)^{p^{s}-1}(v')\in S\otimes I$. By (iii), we know that $\text{$z^{p^s-1}(x_1^{p-1}\cdots x_{s}^{p-1})=(-1)^s$ $+$ (other terms in $I$)}$. Since $I\subset \fm$, we have that
\[
\text{$(\ad z)^{p^{s}-1}(w)=(s_0'\otimes (-1)^s)$ $+$ (other terms in $S\otimes \fm$)}.
\]
This proves (iv) and \eqref{claim2}.

\textbf{(v)} We show that for any $0<l<p^s-1$, $[w, (\ad z)^{l}(w)]\in S\otimes I$.

By (iv), we know that $(\ad z)^{l}(w)\in (S\otimes \fm_s)\oplus (S\otimes I)$. Then 
\begin{align*}
[w, (\ad z)^{l}(w)]=&[s_0'\otimes x_1^{p-1}\cdots x_{s}^{p-1}+v', (\ad z)^{l}(w)]\\
\in &[s_0'\otimes x_1^{p-1}\cdots x_{s}^{p-1}+v', (S\otimes \fm_s)\oplus (S\otimes I)]\subseteq S\otimes I.
\end{align*}
This proves (v).

\textbf{(vi)} We show that for $p-1\leq t=p^{s-r}-1\leq p^{s-1}-1$, 
\[
[w_t,[w_{t-1},[\dots,[w_2,[w_1, w]\dots] \in S\otimes \fm,
\] 
where each $w_\alpha, 1\leq \alpha\leq t$, is equal to $z$ or to $w$. Hence for $1\leq r\leq s-1$, $u_r\in S\otimes \fm$ and $u_r^{p^{r}}\in S\otimes \fm$.

Let us consider all such iterated commutators $[w_t,[w_{t-1},[\dots,[w_2,[w_1, w]\dots]$. If $w_1=w$, then $[w_1, w]=0$. If $w_1=\dots= w_t=z$, then (iv) implies that 
\begin{align*}\label{adztw}
(\ad z)^{t}(w)\in (S\otimes \fm_s)\oplus (S\otimes I)= S\otimes \fm.
\end{align*}
Hence we need to consider commutators $(\ad w)^{\beta}(\ad z)^{\gamma}(w)$, where $1\leq \gamma<t$ and $\beta>0$. By (v), we know that $[w, (\ad z)^{\gamma}(w)]\in S\otimes I$. Since $w\in S\otimes \fm$, we have that 
\[
[w,[w, (\ad z)^{\gamma}(w)]]\in[S\otimes \fm, S\otimes I]\subseteq [S, S]\otimes \fm I\subseteq S\otimes I.
\]
Continuing in this way, we get $(\ad w)^{\beta}(\ad z)^{\gamma}(w)\in S\otimes I$. If $\beta+\gamma=t$, then we are done. If $\beta+\gamma<t$, then we need to consider $[w_\alpha, (\ad w)^{\beta}(\ad z)^{\gamma}(w)]$, where $w_\alpha=z$ or $w$. If $w_\alpha=z$, then $z$ preserves the ideal $I$. Hence
\[
[z,(\ad w)^{\beta}(\ad z)^{\gamma}(w)]\in [z, S\otimes I]=S\otimes z(I)\subseteq S\otimes I.
\]
If $w_\alpha=w\in S\otimes \fm$, then by the same reason as above, we get 
\[
[w,(\ad w)^{\beta}(\ad z)^{\gamma}(w)]\in S\otimes I.
\]
Hence $[w_\alpha, (\ad w)^{\beta}(\ad z)^{\gamma}(w)]\in S\otimes I$. If $1+\beta+\gamma=t$, then we are done. If $1+\beta+\gamma<t$, then we need to consider $[w_\nu,[w_\alpha, (\ad w)^{\beta}(\ad z)^{\gamma}(w)]]$, where $w_\nu=z$ or $w$. Arguing similarly, we get $[w_\nu,[w_\alpha, (\ad w)^{\beta}(\ad z)^{\gamma}(w)]]\in S\otimes I$. Continuing in this way, we eventually get $[w_t,[w_{t-1},[\dots,[w_2,[w_1, w]\dots] \in S\otimes \fm$ for all $p-1\leq t=p^{s-r}-1\leq p^{s-1}-1$.

Since $u_r$, $1\leq r\leq s-1$, is in the span of $[w_t,[w_{t-1},[\dots,[w_2,[w_1, w]\dots]$, we have that $u_r\in S\otimes \fm$ for all $1\leq r\leq s-1$. By \thref{generalJacobF} and $[S\otimes \fm, S\otimes \fm]\subseteq S\otimes \fm$, one can show that $u_r^{p^{r}}\in S\otimes \fm$ for all $1\leq r\leq s-1$. This proves (vi) and \eqref{claim1}.

It follows from \eqref{claim2} and \eqref{claim1} that
\begin{align*}
D_1^{p^{s}}=&z^{p^{s}}+w^{p^{s}}+\sum_{r=0}^{s-1}u_r^{p^{r}}=z^{p^{s}}+w^{p^{s}}+(\ad z)^{p^{s}-1}(w)+\sum_{r=1}^{s-1}u_r^{p^{r}}\\
=&\text{$z^{p^{s}}+(s_0'\otimes (-1)^{s})$ $+$ (other terms in $S\otimes \fm$)}.
\end{align*}

\textbf{\textit{Step 2}}. We show that $D_1$ is a nilpotent element of $\cL$ if and only if $s_0'$ is a nilpotent element of $S$. By step 1, we know that 
\[
\text{$D_1^{p^{s}}=z^{p^{s}}+(s_0'\otimes (-1)^{s})$ $+$ (other terms in $S\otimes \fm$)}.
\]
By our assumption, $z$ is nilpotent. By \thref{zpscalculationslemmma}, we know that $z^{p^{s}}\in W(m; \underline{1})_{(0)}$. Hence $z^{p^{s}}$ preserves $S\otimes\fm$. Applying \thref{D1pNcalculations} with $D'=D_1^{p^{s}}$, $d=z^{p^{s}}$, $s_0=s_0'$ and $f_0=(-1)^s$, we get for $N\geq 1$ (see \eqref{DpNresults}), 
\[
\text{$D_1^{p^{s+N}}=z^{p^{s+N}}+((s_0')^{p^{N}}\otimes (-1)^{sp^{N}})$ $+$ (other terms in $S\otimes \fm$)}.
\]
Hence for $s+N\gg 0$, $D_1$ is a nilpotent element of $\cL$ if and only if $s_0'$ is a nilpotent element of $S$. This completes the proof. 
\end{proof}

\subsection{The irreducibility of $\cN(\cL)$}\label{socleSmainthmproof}
We are now ready to prove that the nilpotent variety of $\cL=(S\otimes \cO(m; \underline{1}))\rtimes \dd D$ is irreducible. Recall our assumptions that $\dd{D}$ is a restricted transitive subalgebra of $W(m; \underline{1})$ such that $\cN(\dd D)$ is irreducible, $S$ is a simple restricted Lie algebra such that all its derivations are inner and $\cN(S)$ is irreducible. 
\begin{thm}\thlabel{NLirr}
The variety $\cN(\cL)$ is irreducible.
\end{thm}

\begin{proof}
Let $D$ be an element of $\cL=(S\otimes \cO(m; \underline{1}))\rtimes \dd D$. Then we can write 
\begin{equation}\label{Dform}
D=\sum_{i=0}^{m(p-1)}s_i\otimes f_i+d,
\end{equation}
where $s_i\in S$, $f_i\in \cO(m;\underline{1})$ with $\deg f_i=i$, and $d \in \dd D$. Note that the surjective Lie algebra homomorphism $\psi: \cL \to \dd D, D\mapsto d$ induces a surjective morphism 
\[
\tilde{\psi}: \cN(\cL)\to \cN(\dd D).
\]
By our assumption, $\cN(\dd D)$ is irreducible. By \thref{nvarietythm}, $\cN(\cL)$ is equidimensional. If we can prove that all fibres of $\tilde{\psi}$ are irreducible and have the same dimension, then the irreducibility of $\cN(\cL)$ follows from \thref{l:irr}.

Since $\dd D$ is a restricted transitive subalgebra of $W(m; \underline{1})$, i.e. $\dd D+ W(m; \underline{1})_{(0)}=W(m; \underline{1})$, there exist elements in $\dd D$ which are not in $W(m; \underline{1})_{(0)}$. Hence for $d\in\cN(\dd D)$, we have two cases to consider: either $d\in W(m; \underline{1})_{(0)}$ or $d\notin W(m; \underline{1})_{(0)}$. Let us compute $\tilde{\psi}^{-1}(d)$ in each case. 

\textbf{\textit{Case 1}}:  $d\in \cN(\dd D)$ and $d\in W(m; \underline{1})_{(0)}$. Let $D=\sum_{i=0}^{m(p-1)}s_i\otimes f_i+d$ be an element of $\cL$ such that $\tilde{\psi}(D)=d$; see \eqref{Dform} for notations. By \thref{D1pNcalculations}, we know that $D \in \cN(\cL)$ if and only if $s_0\in \cN(S)$. As a result, 
\[
\tilde{\psi}^{-1}(d)=\cN(S)\otimes 1+S\otimes \fm +d \cong \cN(S)\otimes 1+S\otimes \fm.
\]
Since $S\otimes \fm$ is irreducible and $\cN(S)$ is irreducible by our assumption, it follows that all fibres $\tilde{\psi}^{-1}(d)$ are irreducible. Moreover, 
\begin{align*}
\dim \tilde{\psi}^{-1}(d)&=\dim (\cN(S)\otimes 1)+\dim (S\otimes \fm)\\
&=\big(\dim(S\otimes 1)-\MT(S)\big)+\dim (S\otimes \fm)\quad (\text{by \thref{nvarietythm}(iii)})\\
&=\dim(S\otimes \cO(m;\underline{1}))-\MT(S)\\
&=p^m\dim S-\MT(S).
\end{align*}

\textbf{\textit{Case 2}}: $d\in \cN(\dd D)$ and $d\notin W(m; \underline{1})_{(0)}$. 

\textbf{\textit{Step 1}}. We compute $\tilde{\psi}^{-1}(d)$ for all $d\in \cN(\dd D)$ with $d\notin W(m; \underline{1})_{(0)}$. Then we deduce that they are irreducible. By \thref{dformlemma2} and \thref{GpreservesND}, we may assume that 
\begin{equation}\label{dform}
d=d_0+u,
\end{equation}
where 
\[
d_0=\del_1+x_1^{p-1}\del_2+\dots+x_1^{p-1}\cdots x_{s-1}^{p-1}\del_s
\]
with $1\leq s\leq m$, $u\in (I\del_1+\dots+I \del_m)\cap W(m; \underline{1})_{(p-1)}$ and $I$ is the ideal of $\cO(m; \underline{1})$ generated by $x_{s+1}, \dots, x_m$.

Let $D=\sum_{i=0}^{m(p-1)}s_i\otimes f_i+d$ be an element of $\cL$ such that $\tilde{\psi}(D)=d$; see \eqref{Dform} and \eqref{dform} for notations. Recall the subgroup $H$ of $\Aut(\cL)$ which is generated by all $\exp(\ad(\tilde{s}\otimes f))$, where $\tilde{s}\otimes f \in S \otimes \fm$; see \thref{Sautoexps}. Note that $H$ is a connected algebraic group with $S\otimes \fm\subseteq\Lie(H)$. Since $\exp(\ad(\tilde{s}\otimes f))=\sum_{j=0}^{p-1}\frac{1}{j!}(\ad (\tilde{s}\otimes f))^{j}$ and $D$ is of the form \eqref{Dform}, it is easy to see that $H$ stabilizes the fibres of $\tilde{\psi}$. By \thref{DconjugatetoD1}, we know that $D$ is conjugate under $H$ to 
\begin{align*}
D_1=s_0'\otimes x_1^{p-1}\cdots x_{s}^{p-1}+v'+d, 
\end{align*}
where $s_0'\in S$ (possibly $0$), $v'\in S\otimes I$ and $d=d_0+u$ as in \eqref{dform}. Then $D$ is nilpotent if and only if $D_1$ is nilpotent. By \thref{D1psnilpotentcalculations}, we know that $D_1\in \cN(\cL)$ if and only if $s_0'\in \cN(S)$. Hence all fibres of $\tilde{\psi}$ have the form 
\[
H.\big(\cN(S)\otimes x_1^{p-1}\cdots x_{s}^{p-1}+S\otimes I+d\big)\cong H.\big(\cN(S)\otimes x_1^{p-1}\cdots x_{s}^{p-1}+S\otimes I\big).
\] 
Since $H$ is connected, $\cN(S)$ and $S\otimes I$ are irreducible, it follows that all fibres $\tilde{\psi}^{-1}(d)$ are irreducible.

\textbf{\textit{Step 2}}. We show that all fibres of $\tilde{\psi}$ have the same dimension. By case 1, we know that $\dim \tilde{\psi}^{-1}(d)=p^m\dim S-\MT(S)$ for all $d \in \cN(\ccD)$ with $d\in W(m;\underline{1})_{(0)}$. In particular,
\begin{equation}\label{0dim}
\dim \tilde{\psi}^{-1}(0)=p^m\dim S-\MT(S).
\end{equation}
To finish the proof we just need to show that 
\[
\dim \tilde{\psi}^{-1}(d)=\dim \tilde{\psi}^{-1}(0)
\]
for all $d\in \cN(\dd D)$ with $d \notin W(m;\underline{1})_{(0)}$. 

\textbf{\textit{Step 2(i)}}. We first show that 
\[
\dim \tilde{\psi}^{-1}(0)\geq\dim \tilde{\psi}^{-1}(d)
\] 
for all $d\in \cN(\dd D)$ with $d \notin W(m;\underline{1})_{(0)}$. Suppose the contrary, i.e. $\dim \tilde{\psi}^{-1}(0)<\dim \tilde{\psi}^{-1}(d)$ for some $d\in \cN(\dd D)$ with $d \notin W(m;\underline{1})_{(0)}$. By \thref{chethm}, the set 
\[
W_1=\big\{x\in \cN(\cL)\,|\,\dim \tilde{\psi}^{-1}(\tilde{\psi}(x))\geq r\big\}
\]
is Zariski closed in $\cN(\cL)$ for every $r\in \N_0$. We now take $r=\dim\tilde{\psi}^{-1}(0)+1$. If $W_1$ is empty, then we are done. If $W_1$ is nonempty, then it contains $w+d\in \cN(\cL)$ with $w\in S\otimes \fm$ such that $\dim \tilde{\psi}^{-1}(\tilde{\psi}(w+d))\geq r$. Note that for all $\lambda\in k^{*}$, 
\[
\tilde{\psi}(\lambda (w+d))=\lambda\tilde{\psi}(w+d).
\]  
Then 
\[
\tilde{{\psi}}^{-1}\big(\tilde{\psi}(\lambda (w+d))\big)=\lambda\tilde{{\psi}}^{-1}\big(\tilde{\psi}(w+d)\big).
\]
So $W_1$ is $k^{*}$-stable. Since $W_1$ is Zariski closed, it contains $0$. But this contradicts our choice of $r$. As a result, 
\begin{equation}\label{samedim1}
\dim\tilde{\psi}^{-1}(0)\geq\dim \tilde{\psi}^{-1}(d)
\end{equation}
for all $d\in \cN(\dd D)$ with $d \notin W(m;\underline{1})_{(0)}$.

\textbf{\textit{Step 2(ii)}}. Next we show that 
\[
\dim \tilde{\psi}^{-1}(0)\leq\dim \tilde{\psi}^{-1}(d)
\]
for all $d\in \cN(\dd D)$ with $d \notin W(m;\underline{1})_{(0)}$. Consider the morphism 
\begin{align*}
\theta: H \times \big(\cN(S)\otimes x_1^{p-1}\cdots x_{s}^{p-1}+S\otimes I+d\big)&\to \tilde{\psi}^{-1}(d)\\
(h, D_1)&\mapsto h(D_1).
\end{align*}
By the work in step 1, we know that any point in the fibre $\tilde{\psi}^{-1}(d)$ has the form 
\[
\theta\big(h, s_0'\otimes x_1^{p-1}\cdots x_{s}^{p-1}+v'+d\big)
\]
for some $h\in H, s_0'\in \cN(S)$ and $v'\in S\otimes I$. As $H$ acts on the fibre it preserves smooth points. Hence we may assume that $h=1$. Instead of computing $\dim \tilde{\psi}^{-1}(d)$ which is difficult, we can compute the differential of $\theta$ at a smooth point of the fibre with $h=1$. Since $\cN(S)$ is irreducible, the set of smooth points in $\cN(S)$ is nonempty. Let $s_0'$ be a smooth point of $\cN(S)$. Then $\dim \T_{s_0'}(\cN(S))=\dim \cN(S)$. Take 
\[
D_1=s_0'\otimes x_1^{p-1}\cdots x_{s}^{p-1}+v'+d\in \tilde{\psi}^{-1}(d).
\]
Then the differential of $\theta$ at $(1, D_1)$ is 
\begin{align*}
(d\theta)_{(1, D_1)}: \Lie(H) \oplus\big(\T_{s_{0}'}(\cN(S))\otimes x_1^{p-1}\cdots x_{s}^{p-1}+S\otimes I\big)&\to \T_{D_1}(\tilde{\psi}^{-1}(d))\\
(X,Y)&\mapsto [X,D_1]+Y.
\end{align*}
Since $D_1$ is a smooth point, this implies that $\dim \T_{D_1}(\tilde{\psi}^{-1}(d))=\dim \tilde{\psi}^{-1}(d)$. In order to show that $\dim \tilde{\psi}^{-1}(0)\leq\dim \tilde{\psi}^{-1}(d)$, we just need to show that 
\[
\dim \tilde{\psi}^{-1}(0)\leq \dim \T_{D_1}(\tilde{\psi}^{-1}(d)).
\]
It is clear that 
\begin{equation}\label{diff1}
\T_{s_{0}'}(\cN(S))\otimes x_1^{p-1}\cdots x_{s}^{p-1}+S\otimes I\subseteq \Ima ((d\theta)_{(1, D_1)}). 
\end{equation}
Moreover, $[\Lie(H), D_1]\subseteq \Ima ((d\theta)_{(1, D_1)})$. Since $S\otimes \fm \subseteq \Lie(H)$, we have that 
\begin{equation}\label{diff2}
[S\otimes \fm, D_1]=[S\otimes \fm, s_0'\otimes x_1^{p-1}\cdots x_{s}^{p-1}+v'+d]\subseteq \Ima ((d\theta)_{(1, D_1)})
\end{equation}

Observe \eqref{diff1} and \eqref{diff2} carefully. We see that $\Ima ((d\theta)_{(1, D_1)})$ contains the following subspaces:
\begin{enumerate}[\upshape(1)]
\item $\T_{s_{0}'}(\cN(S))\otimes x_1^{p-1}\cdots x_{s}^{p-1}$ which has dimension equal to $\dim \cN(S)$.
\item $S\otimes I$.
\item $[S\otimes \fm, s_0'\otimes x_1^{p-1}\cdots x_{s}^{p-1}+v'+d]$.  
\end{enumerate} 
Look at the last two subspaces. Note that $\fm=\fm_s\oplus I$, where $\fm_s$ is the maximal ideal of $\cO(s; \underline{1})$. Moreover, $\fm I\subseteq I$. Since $S\otimes I$ is in $\Ima ((d\theta)_{(1, D_1)})$, 
\[
[S\otimes \fm, s_0'\otimes x_1^{p-1}\cdots x_{s}^{p-1}]=[S\otimes (\fm_s\oplus I), s_0'\otimes x_1^{p-1}\cdots x_{s}^{p-1}]\subseteq S\otimes I,
\]
and 
\[
[S\otimes \fm, v']\subseteq [S\otimes \fm, S\otimes I]\subseteq S\otimes \fm I\subseteq S\otimes I,
\]
we see that $\Ima ((d\theta)_{(1, D_1)})$ contains
\begin{equation*}
\begin{aligned}
&S\otimes I+[S\otimes \fm, s_0'\otimes x_1^{p-1}\cdots x_{s}^{p-1}+v'+d]\\
=&S\otimes I+[S\otimes \fm, d]\\
=&S\otimes I+S\otimes d(\fm).
\end{aligned}
\end{equation*}
Recall \eqref{dform} that $d=d_0+u$, where $d_0=\del_1+x_1^{p-1}\del_2+\dots+x_1^{p-1}\cdots x_{s-1}^{p-1}\del_s$ is a derivation of $\cO(s; \underline{1})$ and $u\in (I\del_1+\dots+I \del_m)\cap W(m; \underline{1})_{(p-1)}$. Since $u(\fm)\subseteq I$, we have that $S\otimes I+S\otimes d(\fm)=S\otimes I+S\otimes d_0(\fm)$. Since $\fm=\fm_s\oplus I$, $d_0(\fm_s)\subset \cO(s; \underline{1})$, $d_0(I)\subseteq I$ and $\cO(s; \underline{1})\cap I=\{0\}$, we have that $d_0(\fm)=d_0(\fm_s)\oplus d_0(I)$. Since $d_0(I)\subseteq I$, we have that 
\begin{equation}\label{lastwosubs}
S\otimes I+S\otimes d(\fm)=S\otimes I+S\otimes d_0(\fm)=(S\otimes I)\oplus (S\otimes d_0(\fm_s))\cong S\otimes (I\oplus d_0(\fm_s)).
\end{equation}
By \thref{d-d0d_0image} (see \eqref{Moplus}), we know the subspace $M=I+d(\fm)=I\oplus d_0(\fm_s)$ has the property that $M\oplus kx_1^{p-1}\cdots x_s^{p-1}=\cO(m; \underline{1})$. By \eqref{lastwosubs}, we see that the sum of the last two subspaces equals $S\otimes M$. This subspace of dimension $(p^m-1)\dim S$ is contained in $\Ima ((d\theta)_{(1, D_1)})$ and the complement of the first subspace $\T_{s_{0}'}(\cN(S))\otimes x_1^{p-1}\cdots x_{s}^{p-1}$. Therefore, 
\begin{align*}
\dim\Ima ((d\theta)_{(1, D_1)})&\geq \dim(\T_{s_{0}'}(\cN(S))\otimes x_1^{p-1}\cdots x_{s}^{p-1})+\dim (S\otimes M)\\
&=\dim \cN(S)+(p^m-1)\dim S\\
&=(\dim S-\MT(S))+ (p^m-1)\dim S\\
&=p^m\dim S-\MT(S)\\
&=\dim \tilde{\psi}^{-1}(0)\quad \quad \text{(see \eqref{0dim})}.
\end{align*}
Since $\Ima ((d\theta)_{(1, D_1)})\subseteq \T_{D_1}(\tilde{\psi}^{-1}(d))$ we have that
\begin{equation}\label{samedim2}
\dim \tilde{\psi}^{-1}(0)\leq \dim\T_{D_1}(\tilde{\psi}^{-1}(d))=\dim \tilde{\psi}^{-1}(d).
\end{equation}
It follows from \eqref{samedim1} and \eqref{samedim2} that all fibres of $\tilde{\psi}$ have the same dimension. 

By case 1 and 2, we see that all fibres of $\tilde{\psi}$ are irreducible and have the same dimension. Hence $\cN(\cL)$ is irreducible by \thref{l:irr}. This completes the proof.
\end{proof}

\begin{rmk}
\begin{enumerate}
\item We verified Premet's conjecture for a class of semisimple restricted Lie algebras $\cL=(S\otimes \cO(m; \underline{1}))\rtimes \dd D$ under the assumptions that $S$ is a simple restricted Lie algebra over $k$ with $\ad S=\Der S$ and $\cN(S)$ is irreducible, $\dd D$ is a restricted transitive subalgebra of $W(m;\underline{1})$ with $\cN(\dd D)$ is irreducible. 
\item A similar argument works for $\bigoplus _{i=1}^{t}(S_{i}\otimes \cO(m_{i}; \underline{1}))\rtimes (\Id_{S_{i}}\otimes \dd D_{i})$, where
\begin{enumerate}[\upshape(i)]
\item each $S_{i}$ is a simple restricted Lie algebra such that $\ad S_{i}=\Der S_{i}$ and $\cN(S_{i})$ is irreducible, and 
\item each $\dd D_{i}$ is a restricted transitive subalgebra of $W(m_i; \underline{1})$ such that $\cN(\dd D_{i})$ is irreducible.
\end{enumerate}
\item There are further cases to consider such as $\ad S \subsetneq \Der S$ and other semisimple restricted Lie algebras which are not of the form given in \thref{blockthm}. It is unclear to the author how to tackle these problems.
\end{enumerate}
\end{rmk}

\chapter{The nilpotent variety of $W(1;n)_p$ is irreducible}
In this chapter we assume that $k$ is an algebraically closed field of characteristic $p>3$, and $n\in \N_{\geq 2}$. We are interested in the minimal $p$-envelope $W(1;n)_p$ of the Zassenhaus algebra $W(1;n)$. This restricted Lie algebra is semisimple. Recent studies have shown that the variety $\cN(W(1;n)):=\cN(W(1;n)_p)\cap W(1;n)$ is reducible \cite[Theorem 4.8(i)]{YS16}. So investigating the variety $\cN(W(1;n)_p)$ becomes critical for verifying Premet's conjecture. Note that this chapter contains the main result of the research paper \cite{C19} written by the author of this thesis.

This chapter is organized as follows. We first recall some basic results on $W(1;n)$ and $W(1;n)_p$. Then we study some nilpotent elements of $W(1;n)_p$ and identify an irreducible component of $\cN(W(1; n)_p)$. Finally, we prove that $\cN(W(1; n)_p)$ is irreducible. The proof is similar to Premet's proof for the Jacobson-Witt algebra $W(n;\underline{1})$; see \thref{Wnproof}. But the $(n+1)$-dimensional subspace $V$ used in $W(n;\underline{1})$ has no obvious analogue for $W(1;n)_p$. Therefore, a new $V$ is constructed using the original definition of $W(1;n)$ due to H.~Zassenhaus. In general, constructing analogues of $V$ for the minimal $p$-envelopes of $W(n;\underline{m})$, where $\underline{m}=(m_1, \dots, m_n)$ and $m_i>1$ for some $i$, would enable one to check Premet's conjecture for this class of restricted Lie algebras.

\section{Preliminaries}
\subsection{$W(1;n)$ and $W(1;n)_p$}\label{2.1}
Let $k$ be an algebraically closed field of characteristic $p>3$ and $n\in \N$. The divided power algebra $\cO(1;n)$ has a $k$-basis $\{x^{(a)}=\frac{1}{a!}x^a\,|\, 0\leq a \leq p^{n}-1\}$, and the product in $\cO(1;n)$ is given by $x^{(a)}x^{(b)} = \binom {a + b}{a} x^{(a + b)}$ if $0\leq a+b \leq p^n-1$ and $0$ otherwise. In the following, we write $x^{(1)}$ as $x$. It is straightforward to see that $\cO(1;n)$ is a local algebra with the unique maximal ideal $\fm$ spanned by all $x^{(a)}$ such that $a\geq 1$. A system of divided powers is defined on $\fm$, $f\mapsto f^{(r)} \in \cO(1;n)$, where $r\geq0$; see \thref{dpower}.

A derivation $\ccD$ of $\cO(1; n)$ is called \textit{special} if $\ccD(x^{(a)})=x^{(a-1)}\ccD(x)$ for $1\leq a \leq p^{n}-1$ and $0$ otherwise; see \thref{specialderivation}. The set of all special derivations of $\cO(1;n)$ forms a Lie subalgebra of $\Der \cO(1;n)$ denoted $\fL=W(1;n)$ and called the \textit{Zassenhaus algebra}. When $n=1$, $\fL$ coincides with the Witt algebra $W(1;1):=\Der \cO(1;1)$, a simple and restricted Lie algebra. When $n\geq 2$, $\fL$ provides the first example of a simple, non-restricted Lie algebra; see \thref{wittthm}. From now on, we always assume that $n\geq 2$.

The Zassenhaus algebra $\fL$ admits an $\cO(1;n)$-module structure via $(f\ccD)(x)=f\ccD(x)$ for all $f\in \cO(1;n)$ and $\ccD\in \fL$. Since each $\ccD\in \fL$ is uniquely determined by its effect on $x$, it is easy to see that $\fL$ is a free $\cO(1;n)$-module of rank $1$ generated by the special derivation $\del$ such that $\del(x^{(a)})=x^{(a-1)}$ if $1\leq a \leq p^{n}-1$ and $0$ otherwise; see \thref{wittthm}. Hence the Lie bracket in $\fL$ is given by $[x^{(i)}\del,x^{(j)}\del]=\big(\binom {i+j-1}{i}-\binom {i+j-1}{j}\big)x^{(i+j-1)}\del$ if $1\leq i+j \leq p^{n}$ and $0$ otherwise; see formula \eqref{LbracketWm}.

There is a $\Z$-grading on $\fL$, i.e. $\fL=\bigoplus_{i=-1}^{p^{n}-2} kd_{i}$ with $d_{i}:=x^{(i+1)}\partial$. Put $\fL_{(i)}:=\bigoplus_{j \geq i}^{p^{n}-2} kd_{j}$ for $-1\leq i\leq p^{n}-2$. Then this  $\Z$-grading induces a natural filtration
\begin{equation}\label{filtrationoffL}
\fL=\fL_{(-1)}\supset\fL_{(0)}\supset\fL_{(1)}\supset\dots \supset\fL_{(p^n-2)}\supset 0
\end{equation}
on $\fL$. It is known that for $i\geq 0$, 
\begin{align}\label{eip}
d_{i}^p=
\begin{cases}
d_{i}, &\text{if $i=0$,}\\
d_{pi}, &\text{if $i=p^{t}-1$ for some $1\leq t\leq n-1$,}\\
0, &\text{otherwise}.
\end{cases}
\end{align}
In particular, $\fL_{(0)}$ is a restricted Lie subalgebra of $\Der \cO(1;n)$ and $\fL_{(1)}=\,\text{nil\,}(\fL_{(0)})$; see \cite[p.~3]{YS16}. We show that this implies that all nonzero tori in $\fL_{(0)}$ have dimension $1$. Consider the surjective map $\pi: \fL_{(0)}\twoheadrightarrow \fL_{(0)}/\fL_{(1)}$. Let $\ft$ be any nonzero torus in $\fL_{(0)}$. Let $\pi|_\ft$ be the restriction of $\pi$ to $\ft$. Then $\Ker \pi|_\ft \subset\fL_{(1)}$. Since $\fL_{(1)}$ is a nilpotent $p$-ideal of $\fL_{(0)}$, we have that $\Ker \pi|_\ft =0$. Hence any nonzero torus $\ft$ in $\fL_{(0)}$ has dimension $1$; see \cite[Sec.~2]{T98}. As an example, $kd_0=kx\del$ is a $1$-dimensional torus in $\fL_{(0)}$.

Note that the Zassenhaus algebra $\fL$ has another presentation. Let $q=p^n$ and let $\F_q \subset k$ be the set of all roots of $x^q-x=0$. This is a finite field with $q$ elements. Then $\fL$ has a $k$-basis $\{e_{\alpha}\,|\, \alpha\in \F_q\}$ with the Lie bracket given by $[e_\alpha, e_\beta]=(\beta-\alpha)e_{\alpha+\beta}$ \cite[Theorem~7.6.3(1)]{S04}. We will use this presentation in Sec.~\ref{3.2.3irredsection}.

It is also useful to mention that there is an embedding from $\fL$ into the Jacobson-Witt algebra $W(n; \underline{1})$; see \cite[Lemma 3.1 and Proposition 4.3]{YS16}. Recall that $W(n; \underline{1})$ is the derivation algebra of $\cO(n; \underline{1})$, where $\cO(n; \underline{1})=k[X_1, \dots, X_n]/(X_1^{p}, \dots, X_{n}^{p})$ is the truncated polynomial ring in $n$ variables. For each $1\leq i\leq n$, we write $x_i$ for the image of $X_i$ in $\cO(n; \underline{1})$. Note that $W(n; \underline{1})$ is a free $\cO(n; \underline{1})$-module of rank $n$ generated by the partial derivatives $\del_1, \dots, \del_n$ such that $\del_i(x_j)=\delta_{ij}$ for all $1\leq i, j\leq n$; see Sec.~\ref{nilpotentvarietyresultsWn} for details. The Zassenhaus algebra $\fL$ is the set of all special derivations of the divided power algebra $\cO(1;n)$. Note that $\cO(1;n)$ has a $k$-basis $\{x^{(a)}\,|\, 0\leq a\leq p^n-1\}$. For each $0\leq a \leq p^n-1$, let $a=\sum_{K=0}^{n-1}a_Kp^{K}$, $0\leq a_K\leq p-1$, be the $p$-adic expansion of $a$. Define 
\begin{equation}\label{embed1}
\begin{aligned}
\phi: \cO(1;n) &\rightarrow \cO(n; \underline{1})\\
x^{(a)}&\mapsto \prod_{i=1}^{n} \frac{x_i^{a_{i-1}}}{a_{i-1}!};
\end{aligned}
\end{equation}
see \cite[(3.1.1)]{YS16}. Then $\phi$ is an algebra isomorphism and it induces the following Lie algebra isomorphism: 
\begin{equation}\label{embed2}
\begin{aligned}
\varphi: \Der\cO(1;n) &\xrightarrow{\sim} W(n; \underline{1})=\Der\cO(n; \underline{1})\\
D &\mapsto \varphi(D),
\end{aligned}
\end{equation}
where $(\varphi(D))(u)=\phi\big(D(\phi^{-1}(u))\big)$ for all $u\in \cO(n; \underline{1})$. Moreover, $\varphi(D^{p})=\varphi(D)^{p}$ for all $D\in \Der\cO(1;n)$; see \cite[(3.1.2)]{YS16}. Since $\fL$ is a Lie subalgebra of $\Der\cO(1;n)$, the above $\varphi$ induces an embedding 
\begin{equation}\label{embed3}
\iota=\varphi|_{\fL}: \fL \hookrightarrow W(n; \underline{1});
\end{equation}
see \cite[(4.3.2)]{YS16}. More precisely, $\iota$ is induced by $\phi$ defined in \eqref{embed1}. By a direct computation, we have that 
\begin{equation*}
\iota(\del)=\sD_1=\del_1+\sum_{l=1}^{n-1}(-1)^{l}x_1^{p-1}\cdots x_l^{p-1}\del_{l+1}. 
\end{equation*}
Note that $\sD_1$ has a similar expression to $\sD$ in \thref{WnDresultslem3}. We will use the above embedding $\iota$ in the sketch proof of \thref{YS161}.

Let $\fL_p=W(1;n)_p$ denote the $p$-envelope of $\fL\cong \ad \fL$ in $\Der \fL$. This semisimple restricted Lie algebra is referred to as the \textit{minimal $p$-envelope} of $\fL$; see \thref{ssminpenvelope}. By \thref{wittpenvelope}, we have that $\fL_p=\fL +\sum_{i=1}^{n-1}k\partial^{p^{i}}$. Here we identify $\fL$ with $\ad \fL\subset \Der \fL$ and regard $\del^{p^{n}}$ as $0$. Then $\dim \fL_p=p^n+(n-1)$. By formula \eqref{LbracketWm}, we have that for any $1\leq i\leq n-1$ and $0\leq a\leq p^n-1$, the brackets $[\del^{p^{i}}, x^{(a)}\del]=x^{(a-p^{i})}\del$ if $a\geq p^i$ and $0$ otherwise.

Let $\cN$ denote the variety of nilpotent elements in $\fL_p$. It is well known that $\cN$ is Zariski closed in $\fL_p$. One should note that the maximal dimension of toral subalgebras in $\fL_p$ equals $n$ \cite[Theorem~7.6.3(2)]{S04}. Moreover, $\fL_p$ possesses a toral Cartan subalgebra; see \cite[p. 555]{P90}. Hence the set of all semisimple elements of $\fL_p$ is Zariski dense in $\fL_p$; see \cite[Theorem~2]{P87}. It follows from these facts and \thref{nvarietythm} that there exist nonzero homogeneous polynomial functions $\varphi_0, \dots, \varphi_{n-1}$ on $\fL_p$ such that $\cN$ coincides with the set of all common zeros of $\varphi_0, \dots, \varphi_{n-1}$. The variety $\cN$ is equidimensional of dimension $p^n-1$. Furthermore, any $\ccD \in \cN$ satisfies 
\begin{equation}\label{nilpotentconditioninL}
\ccD^{p^{n}}=0.
\end{equation}

\subsection{The automorphism group $G$}\label{2.2}
An automorphism $\Phi$ of $\cO(1;n)$ is called \textit{admissible} if $\Phi(f^{(r)})=\Phi(f)^{(r)}$ for all $f\in \fm$ and $r\geq 0$. By \cite[Lemma 8]{W71}, this is equivalent to the condition that $\Phi(x^{(p^{j})})=\Phi(x)^{(p^{j})}$ for any $1\leq j\leq n-1$. Let $G$ denote the group of all admissible automorphisms of $\cO(1;n)$. It is well known that $G$ is a connected algebraic group, and each $\Phi \in G$ is uniquely determined by its effect on $x$. By \cite[Theorem~2]{W71}, an assignment $\Phi(x):=\sum_{i=1}^{p^{n}-1}\alpha_{i}x^{(i)}$ with $\alpha_{i}\in k$ such that $\alpha_{1}\neq 0$ and $\alpha_{p^{j}}=0$ for $1\leq j\leq n-1$ extends to an admissible automorphism of $\cO(1;n)$. Conversely, if $\Phi$ is an admissible automorphism of $\cO(1;n)$, then $\Phi(x)$ has to be of this form. Hence $\dim G=p^{n}-n$.

Any automorphism of the Zassenhaus algebra $\fL$ is induced by a unique admissible automorphism $\Phi$ of $\cO(1;n)$ via the rule $\ccD^{\Phi}=\Phi \circ \ccD \circ \Phi^{-1}$, where $\ccD\in \fL$ \cite[Theorem~12.8]{R56}. So from now on, we shall identify $G$ with the automorphism group of $\fL$. It is known that $G$ respects the natural filtration of $\fL$. In \cite[Sec.~1]{T98}, Tyurin stated explicitly that if $\Phi\in G$ is such that $\Phi(x)=y$, then $\Phi(g(x)\partial)=(y')^{-1}g(y)\partial$ for any $g(x)\in \cO(1;n)$. Extend this by defining $\Phi(\del^{p^{i}})=\Phi(\del)^{p^{i}}$ for $1\leq i\leq n-1$, one gets an automorphism of $\fL_p$.

It follows from the above description of $G$ that $\Lie(G)\subseteq \fL_{(0)}$. More precisely,
\begin{lem}\thlabel{LieG}
The set $\{d_i=x^{(i+1)}\partial\,|\,\text{$0 \leq i \leq p^n-2$ and $i\neq p^l-1$ for $1\leq l\leq n-1$}\}$ forms a $k$-basis of $\Lie(G)$.
\end{lem}
\begin{proof}
Let $\psi :\A^{1}\to G$ be the map defined by $t\mapsto (x\mapsto x+tx^{(i+1)})$, where $0 \leq i \leq p^n-2$ and $i\neq p^l-1$ for $1\leq l\leq n-1$. It is easy to check that $\psi$ is a morphism of algebraic varieties. Note that $\psi(0)=\Id$. So the differential of $\psi$ at $0$ is the map $d_0\psi: k \to \Lie(G)$. Hence $d_0\psi(k)\subseteq \Lie(G)$.

Let us compute $d_0\psi(k)$. The morphism $\psi$ sends $\A^{1}$ to the set of admissible automorphisms $X=\{\Phi_t\,|\, t \in \A^{1}\}$, where 
\[
\Phi_t(x)=x+tx^{(i+1)}.
\]
Since $\Phi_t$ is uniquely determined by its effect on $x$ and ``admissible'' is equivalent to the condition that $\Phi_t(x^{(p^{j})})=\Phi_t(x)^{(p^{j})}$ for all $1\leq j\leq n-1$ (see Sec.~\ref{2.2}), we have that
\begin{align*}
\Phi_t(x^{(p^{j})})&=(x+tx^{(i+1)})^{(p^{j})}=x^{(p^{j})}+tx^{(p^{j}-1)}x^{(i+1)}+\text{terms of higher degree in $t$}
\end{align*}
for all $1\leq j\leq n-1$; see \thref{dpower}(iv). Now consider $d_0\psi: k \to \T_{\Id}(X)$. Since $X\subset G$ is closed in $G$, we have that $\T_{\Id}(X)\subseteq \Lie(G)$. Passing to $\T_{\Id}(X)$, i.e. calculating $\frac{\del}{\del t}(\Phi_t(x))|_{t=0}$ and $\frac{\del}{\del t}(\Phi_t(x^{(p^{j})}))|_{t=0}$, we get
\begin{align*}
x&\mapsto x^{(i+1)},\\
x^{(p^{j})}&\mapsto x^{(p^{j}-1)}x^{(i+1)}.
\end{align*}
The above results are the same for $d_i=x^{(i+1)}\del$ acting on $x$ and $x^{(p^{j})}$, respectively. Hence $d_i\in \T_{\Id}(X)\subseteq \Lie(G)$. Note that the set
\[
\{d_i=x^{(i+1)}\partial\,|\,\text{$0 \leq i \leq p^n-2$ and $i\neq p^l-1$ for $1\leq l\leq n-1$}\}
\]
consists of $p^n-n$ linearly independent vectors. Since $\dim\Lie(G)=\dim G=p^n-n$, they form a basis of $\Lie(G)$. This completes the proof.
\end{proof}

\section{The variety $\cN$}
\subsection{Some elements in $\cN$}
In Sec.~\ref{2.1}, we observed that any elements of $\fL_{(1)}$ are nilpotent, but they do not tell us much information about $\cN$. The interesting nilpotent elements are contained in the complement of $\fL_{(1)}$ in $\cN$, denoted $\cN\setminus\fL_{(1)}$. They are of the form $\sum_{i=0}^{n-1}\alpha_{i}\partial^{{p^{i}}}+f(x)\partial$ for some $f(x)\in \fm$ and $\alpha_i\in k$ with at least one $\alpha_{i}\neq 0$. In this section, we study elements of this form in the following way: take $\ccD=\sum_{i=0}^{n-1}\alpha_{i}\partial^{{p^{i}}}+f(x)\partial$ (not necessarily nilpotent) and we show that $\ccD$ is conjugate under $G$ to an element in a nice form; see \thref{YS161pre} and \thref{l:l13}. Then we assume  that $\ccD$ is nilpotent, i.e. $\ccD^{p^{n}}=0$ by \eqref{nilpotentconditioninL}. We do some calculations to check if $\ccD^{p^{n-1}}\in \fL_{(0)}$. If $\ccD^{p^{n-1}}\notin \fL_{(0)}$, then we show further that  $\ccD$ is conjugate under $G$ to an element in a nicer form; see \thref{YS161}, \thref{p:p1} and \thref{c:c2}. In doing this, we get the  results required to prove \thref{l:l14} in the next section, i.e. 
\[
\big\{\ccD\in\cN \,|\,\ccD^{p^{n-1}}\notin \fL_{(0)}\big\}=G.(\del+k\del^p+\dots+k\del^{p^{n-1}}).
\]

Let us begin with elements of the form $\alpha_0\del+f(x)\del$, where $\alpha_0\neq 0$ and $f(x)\in \fm$. In \cite{YS16}, Y.-F.~Yao and B.~Shu proved the following:

\begin{lem}\cite[Proposition 4.1]{YS16}\thlabel{YS161pre}
Let $\ccD=\alpha_0\del+f(x)\del$ be an element of $\fL\subset\fL_p$, where $\alpha_0\neq 0$ and $f(x)\in \fm$. Then $\ccD$ is conjugate under $G$ to $\del+\sum_{i=1}^{n}l_ix^{(p^{i}-1)}\del$ for some $l_i\in k$.
\end{lem}
\begin{myproof}[Sketch of proof]
Take $\ccD$ as in the lemma. Then we can write $\ccD=\sum_{i=0}^{p^{n}-1}\alpha_ix^{(i)}\del$ for some $\alpha_i\in k$ with $\alpha_0\neq 0$. Let $\Phi\in G$ be such that $\Phi(x)=\alpha_0 x$. Then 
\[
\Phi(\ccD)=\del+\sum_{i=1}^{p^n-1}\alpha_0^{i-1}\alpha_i x^{(i)}\del.
\]
So we may assume that  
\[
\ccD=\del+\sum_{i=1}^{p^n-1}\beta_i x^{(i)}\del=\del+\beta_1 x \del+\beta_2 x^{(2)}\del+\dots+\beta_{p^n-1} x^{(p^n-1)}\del
\]
for some $\beta_i\in k$. Let $\Phi_0\in G$ be such that $\Phi_0(x)=x+\beta_1 x^{(2)}$. Then one can check that 
\[
\Phi_0(\ccD)=\Phi_0\big(\del+\sum_{i=1}^{p^n-1}\beta_i x^{(i)}\del\big)=
\del+\beta_2'x^{(2)}\del+\beta_3'x^{(3)}\del+\dots+\beta_{p^n-1}'x^{(p^n-1)}\del
\]
for some $\beta_2', \dots,\beta_{p^n-1}' \in k$. In general, if 
\[
\Phi_{i-1}\cdots\Phi_0(\ccD)=\del+\mu_{i+1}x^{(i+1)}\del+\mu_{i+2}x^{(i+2)}\del+\dots +\mu_{p^{n}-1}x^{(p^n-1)}\del
\]
for $1\leq i\leq p-3$, then take $\Phi_i\in G$ with $\Phi_i(x)=x+\mu_{i+1}x^{(i+2)}$. Applying $\Phi_i$ to the above, we get
\[
\Phi_i\Phi_{i-1}\cdots\Phi_0(\ccD)=\del+\mu_{i+2}'x^{(i+2)}\del+\mu_{i+3}'x^{(i+3)}\del+\dots+\mu_{p^n-1}'x^{(p^n-1)}\del
\]
for some $\mu_{i+2}', \dots, \mu_{p^n-1}'\in k$. Consequently, we get
\[
\Phi_{p-3}\Phi_{p-4}\cdots\Phi_0(\ccD)=\del+\nu_{p-1}x^{(p-1)}\del+\nu_px^{(p)}\del+\dots+\nu_{p^n-1}x^{(p^n-1)}\del
\]
for some $\nu_{p-1}, \dots\nu_{p^n-1}\in k$. Set $\Phi_{p-2}=\Id$ and take $\Phi_{p-1}\in G$ with $\Phi_{p-1}(x)=x+\nu_p x^{(p+1)}$. Then one can check that 
\[
\Phi_{p-1}\Phi_{p-2}\Phi_{p-3}\cdots\Phi_0(\ccD)=\del+\nu_{p-1}x^{(p-1)}\del+\nu_{p+1}'x^{(p+1)}\del+\dots+\nu_{p^n-1}'x^{(p^n-1)}\del
\]
for some $\nu_{p-1}, \nu_{p+1}',\dots, \nu_{p^n-1}'\in k$. Continuing in this way, we finally get 
\[
\Phi_{p^{n}-3}\Phi_{p^{n}-4}\cdots\Phi_0(\ccD)=\del+\sum_{i=1}^{n}l_i x^{(p^{i}-1)}\del
\]
for some $l_i\in k$. This completes the sketch of proof.
\end{myproof}

Then Y.-F~Yao and B.~Shu assumed that $\ccD$ is nilpotent and they proved that 
\begin{lem}\cite[Proposition~4.3]{YS16}\thlabel{YS161}
Let $\ccD=\alpha_0\del+f(x)\del$ be a nilpotent element of $\fL\subset\fL_p$, where $\alpha_0\neq 0$ and $f(x)\in \fm$. Then 
\begin{enumerate}[\upshape(i)]
\item $\ccD^{p^{n-1}}\notin \fL_{(0)}$.
\item $\ccD$ is conjugate under $G$ to $\del$.
\end{enumerate}
\end{lem}
 
\begin{myproof}[Sketch of proof]
(i) Take $\ccD$ as in the lemma. Then we can write $\ccD=\sum_{i=0}^{p^{n}-1}\alpha_ix^{(i)}\del$ for some $\alpha_i\in k$ with $\alpha_0\neq 0$. By Jacobson's formula, one can show that
\[
\ccD^{p^{n-1}}=\alpha_0^{p^{n-1}}\del^{p^{n-1}}+\sum_{j=0}^{n-2}\alpha_j'\del^{p^{j}}+w
\]
for some $\alpha_j'\in k$ and $w\in \fL_{(0)}$. Since $\alpha_0\neq 0$, this implies that $\ccD^{p^{n-1}} \not \in \fL_{(0)}$. This proves (i).

(ii) The proof splits into three steps:

\textbf{\textit{Step 1}}. Take $\ccD$ as in the lemma. By \thref{YS161pre}, $\ccD$ is conjugate under $G$ to $\del+\sum_{i=1}^{n}l_i x^{(p^{i}-1)}\del$ for some $l_i\in k$. Since $\ccD$ is nilpotent, it follows from \eqref{nilpotentconditioninL} that 
\[
\big(\del+\sum_{i=1}^{n}l_i x^{(p^{i}-1)}\del\big)^{p^n}=0.
\]
We show that this implies that $\del+\sum_{i=1}^{n}l_i x^{(p^{i}-1)}\del$ is conjugate under $G$ to $\del$. Here we will embed the Zassenhaus algebra $\fL$ into the Jacobson-Witt algebra $W(n; \underline{1})$ and use A.~Premet's results on $\cN(W(n; \underline{1}))$. In Sec.~\ref{2.1}, we have described the construction process of this embedding. Recall the algebra isomorphism 
\[
\phi: \cO(1;n) \rightarrow \cO(n; \underline{1})=k[X_1, \dots, X_n]/(X_1^{p}, \dots, X_{n}^{p})
\]
and the induced Lie algebra isomorphism 
\[
\varphi: \Der\cO(1;n) \xrightarrow{\sim} W(n; \underline{1})=\Der\cO(n; \underline{1});
\]
see \eqref{embed1} and \eqref{embed2} for their definitions. Then $\varphi$ gives rise to the following embedding:
\begin{align*}
\iota=\varphi|_{\fL}: \fL &\hookrightarrow W(n; \underline{1});
\end{align*}
see \eqref{embed3}. It follows from the definition of $\phi$ that 
\begin{equation}\label{iotadeland2}
\begin{aligned}
\iota(\del)&=\sD_1=\del_1+\sum_{l=1}^{n-1}(-1)^{l}x_1^{p-1}\cdots x_l^{p-1}\del_{l+1},\,\text{and}\\
\iota\big(\del+\sum_{i=1}^{n}l_i x^{(p^{i}-1)}\del\big)&=\sD_1+\sum_{i=1}^{n}(-1)^{i}l_i x_1^{p-1}\cdots x_i^{p-1}\del_1.
\end{aligned}
\end{equation}
Note that $\sD_1$ has a similar expression to $\sD$ in \thref{WnDresultslem3}. Since $\varphi(D^{p})=\varphi(D)^{p}$ for all $D\in \Der\cO(1;n)$, we have that 
\begin{align*}
\varphi\big((\del+\sum_{i=1}^{n}l_i x^{(p^{i}-1)}\del)^{p^n}\big)=\big(\sD_1+\sum_{i=1}^{n}(-1)^{i}l_i x_1^{p-1}\cdots x_i^{p-1}\del_1\big)^{p^{n}}.
\end{align*}
Since $\big(\del+\sum_{i=1}^{n}l_i x^{(p^{i}-1)}\del\big)^{p^n}=0$, the above implies that 
\[
\big(\sD_1+\sum_{i=1}^{n}(-1)^{i}l_i x_1^{p-1}\cdots x_i^{p-1}\del_1\big)^{p^{n}}=0,
\]
i.e. $\sD_1+\sum_{i=1}^{n}(-1)^{i}l_i x_1^{p-1}\cdots x_i^{p-1}\del_1$ is a nilpotent element of $W(n; \underline{1})$.

\textbf{\textit{Step 2}}. Set 
\[
D(n)=\sD_1+\sum_{i=1}^{n}(-1)^{i}l_i x_1^{p-1}\cdots x_i^{p-1}\del_1.
\] 
We show that $D(n)^{p^{n}}=0$ implies that $D(n)$ is conjugate under $\Aut(W(n; \underline{1}))$ to $\sD_1$. For $n=1$, $\sD_1=\del_1$ and $D(1)=\del_1-l_1 x_1^{p-1}\del_1$. By considering $D(1)^{p}(x_1)$,  it is easy to show that $D(1)^{p}=l_1(1-l_1x_1^{p-1})\del_1$. By our assumption, $D(1)^{p}=0$. Since $1-l_1x_1^{p-1}$ is invertible in $\cO(n; \underline{1})$, this implies that $l_1=0$. Hence $D(1)=\del_1=\Id(\sD_1)$. Suppose now $n\geq 2$. Then one can check that 
\begin{align*}
D(n)^{p^{n-1}}(x_n)=&D(n)^{p^{n-1}-1}((-1)^{n-1} x_1^{p-1}\cdots x_{n-1}^{p-1})\\
\equiv&(\del_1-x_1^{p-1}\del_2+\dots+(-1)^{n-2}x_1^{p-1}\cdots x_{n-2}^{p-1}\del_{n-1})^{p^{n-1}-1}\\&\cdot((-1)^{n-1} x_1^{p-1}\cdots x_{n-1}^{p-1})\quad (\modd \fM)\\
\equiv& \pm 1 \quad (\modd \fM)\quad \text{(by \thref{WnDresultslem3}(iii))}. 
\end{align*}
Here $\fM$ denotes the unique maximal ideal of $\cO(n; \underline{1})$ generated by $x_1, \dots, x_n$. It follows that $D(n)^{p^{n-1}}\notin W(n; \underline{1})_{(0)}=\{\sum_{i=1}^{n}f_i\del_i\,|\, \text{$f_i\in \fM$ for all $i$}\}$. Since $D(n)^{p^{n}}=0$ and $D(n)^{p^{n-1}}\notin W(n; \underline{1})_{(0)}$, it follows from \thref{WnOresultslem4} that $D(n)\in \Aut(W(n; \underline{1})).\sD_1$.

\textbf{\textit{Step 3}}. Let $u_1, u_2\in \fL$. Then \cite[Lemma A.3]{YS16} states that $u_1, u_2$ are in the same $G$-orbit if and only if $\iota(u_1), \iota(u_2)$ are in the same $\Aut(W(n; \underline{1}))$-orbit. Set $u_1=\del+\sum_{i=1}^{n}l_i x^{(p^{i}-1)}\del$. By \eqref{iotadeland2}, $\iota(u_1)=D(n)$. 
Set $u_2=\del$. By \eqref{iotadeland2}, $\iota(u_2)=\sD_1$.
Since $D(n)\in \Aut(W(n; \underline{1})).\sD_1$, the above result implies that $\del+\sum_{i=1}^{n}l_i x^{(p^{i}-1)}\del\in G.\del$. This proves (ii).
\end{myproof}

Now we consider the other elements of $\cN\setminus\fL_{(1)}$, i.e. elements of the form \\$\del^{p^{t}}+\sum_{i=0}^{t-1}\beta_i\del^{p^{i}}+g(x)\del$, where $1\leq t\leq n-1, \beta_i\in k$ and $g(x)\in\fm$. We want to show that they are conjugate under $G$ to elements in nice forms. For that, we need a result proved by S.~Tyurin which tells us how admissible automorphisms of $\cO(1; n)$ with identical linear part act on these elements.

\begin{lem}\cite[Theorem 1]{T98}\thlabel{Tyurinthm1}
Let $\Phi\in G$ be such that $\Phi(x)=y=x+\sum_{j=2}^{p^{n}-1}\nu_jx^{(j)}$, where $\nu_j\in k$ with $\nu_{p^{i}}=0$ for $1\leq i\leq n-1$. Then 
\begin{align*}
\Phi(\del)&=(y')^{-1}\del\equiv \del \,(\modd \fL_{(0)}),\,\text{and}\\
\Phi(\del^{p^{i}})&=\del^{p^{i}}-(y')^{-1}(\del^{p^{i}} y)\del\equiv \del^{p^{i}}\,(\modd \fL_{(0)}) \,\text{for $1\leq i\leq n-1$}.
\end{align*}
Hence for any $\ccD=\partial^{p^{t}}+\sum_{i=0}^{t-1}\beta_{i}\partial^{p^{i}}+g(x)\partial\in \fL_p$ with $1\leq t\leq n-1$, $\beta_{i} \in k$ and $g(x)\in \fm$, 
\[
\Phi(\ccD)=\partial^{p^{t}}+\sum_{i=1}^{t-1}\beta_{i}\partial^{p^{i}}+(y')^{-1}\bigg(\beta_{0}+\Phi(g(x))-\sum_{i=1}^{t-1}\beta_{i}\partial^{p^{i}}y-\partial^{p^{t}}y\bigg)\partial.
\]
\end{lem}
\begin{myproof}[Sketch of proof]
Recall from Sec.~\ref{2.2} that if $\Phi$ is any admissible automorphism of $\cO(1; n)$ with $\Phi(x)=y$ (not necessarily with identical linear part), then 
\begin{equation}\label{thm1st}
\text{$\Phi(q(x)\del)=(y')^{-1}\Phi(q(x))\partial$ for any $q(x)\in \cO(1;n)$},
\end{equation} and 
\begin{equation}\label{thm2nd}
\text{$\Phi(\del^{p^{i}})=\Phi(\del)^{p^{i}}$ for $1\leq i\leq n-1$}.
\end{equation}
In this lemma, we only consider admissible automorphisms of $\cO(1; n)$ with identical linear part, i.e. $\Phi\in G$ with $\Phi(x)=y=x+\sum_{j=2}^{p^{n}-1}\nu_jx^{(j)}$, where $\nu_j\in k$ with $\nu_{p^{i}}=0$ for $1\leq i\leq n-1$. By \eqref{thm1st}, we have that $\Phi(\del)=(y')^{-1}\del$. We show that $\Phi(\del)\equiv \del\,(\modd \fL_{(0)})$. This is equivalent to show that 
\begin{equation}\label{thm3rd}
(y')^{-1}\del-\del \in \fL_{(0)}.
\end{equation}
Note that $y'=1+\sum_{j=2}^{p^{n}-1}\nu_jx^{(j-1)}$ which is invertible in $\cO(1;n)$. Since $\fL_{(0)}$ is invariant under multiplication of invertible elements of $\cO(1;n)$, we can multiply both sides of \eqref{thm3rd} by $y'$ and show that $(1-y')\del \in \fL_{(0)}$. This is clearly true. Hence
\[\Phi(\del)=(y')^{-1}\del\equiv \del\,(\modd \fL_{(0)}).\] 

It remains to compute $\Phi(\del^{p^{i}})$ for $1\leq i\leq n-1$. Since $\Phi(\del)\equiv \del\,(\modd \fL_{(0)})$, we may write $\Phi(\del)=\del+\phi_0(x)\del$ for some $\phi_0(x) \in\fm$. By \eqref{thm2nd}, $\Phi(\del^{p})=\Phi(\del)^{p}$. By Jacobson's formula (\thref{defres}(3)) and the fact that $\fL_{(0)}$ is a restricted Lie subalgebra of $\Der \cO(1;n)$, we get $\Phi(\del^{p})=\del^{p}+\phi_1(x)\del$ for some $\phi_1(x)\in \cO(1;n)$. Since $\del^{p}x=0$, it follows that 
\begin{align*}
0=\Phi(\del^p x)=\Phi(\del^p)y=(\del^{p}+\phi_1(x)\del)y=\del^p y+\phi_1(x)y'.
\end{align*}
Hence $\phi_1(x)=-(y')^{-1}(\del^p y)$ and so $\Phi(\del^{p})=\del^{p}-(y')^{-1}(\del^p y)\del$. We show that $\Phi(\del^{p})\equiv\del^{p}\,(\modd \fL_{(0)})$. This is equivalent to show that $-(y')^{-1}(\del^p y)\del \in \fL_{(0)}$. By the same reason as above, we can multiply both sides by $-y'$ and show that $(\del^p y)\del \in \fL_{(0)}$. Due to the form of $y$, i.e. $\nu_p=0$, this is clearly true. Hence  
\[\Phi(\del^{p})=\del^{p}-(y')^{-1}(\del^p y)\del\equiv \del^{p}\,(\modd \fL_{(0)}).\]
In a similar way, one can show that 
\[
\text{$\Phi(\del^{p^{i}})=\del^{p^{i}}-(y')^{-1}(\del^{p^{i}} y)\del\equiv\del^{p^{i}}\,(\modd \fL_{(0)})$ for $2\leq i\leq n-1$}.
\]

Let $\ccD=\partial^{p^{t}}+\sum_{i=0}^{t-1}\beta_{i}\partial^{p^{i}}+g(x)\partial$ be an element of $\fL_p$, where $1\leq t\leq n-1$, $\beta_{i} \in k$ and $g(x)\in \fm$.  It follows from \eqref{thm1st} and the above that 
\[
\Phi(\ccD)=\partial^{p^{t}}+\sum_{i=1}^{t-1}\beta_{i}\partial^{p^{i}}+(y')^{-1}\bigg(\beta_{0}+\Phi(g(x))-\sum_{i=1}^{t-1}\beta_{i}\partial^{p^{i}}y-\partial^{p^{t}}y\bigg)\partial.
\]
This completes the sketch of proof.
\end{myproof}

\begin{lem}\cite[Theorem 1]{T98}\thlabel{l:l13}
Let $\ccD=\partial^{p^{t}}+\sum_{i=0}^{t-1}\beta_{i}\partial^{p^{i}}+g(x)\partial$ be an element of $\fL_p$, where $1\leq t\leq n-1$, $\beta_{i} \in k$ and $g(x)\in \fm$. Then $\ccD$ is conjugate under $G$ to
\begin{align*}
\partial^{p^{t}}+\sum_{i=0}^{t-1}\beta_{i}\partial^{p^{i}}+x^{(p^n-p^{t})}h(x)\partial
\end{align*}
for some $h(x)=\sum_{\eta=0}^{p^{t}-1}\mu_{\eta}x^{(\eta)}$ with $\mu_{\eta}\in k$.
\end{lem}
\begin{myproof}[Strategy of the proof]\renewcommand{\qedsymbol}{}
The proof splits into two parts. The first part (a) was proved by S.~Tyurin; see \cite[Theorem 1, p.~68, line -8]{T98}. The second part (b) follows from (a) and \thref{bcmodpapp}. Explicitly, we prove the following:\\
\textbf{(a)} Take $\ccD$ as in the lemma. We show that for $1\leq j\leq p^n-p^t-1$, if $g(x)\del\equiv \gamma_{j}x^{(j)}\del \,(\modd \fL_{(j)})$ for some $\gamma_j \in k$ and $\Phi_j\in G$ is such that $\Phi_j(x)=y_j=x+\gamma_{j}x^{(p^{t}+j)}$, then
\begin{align*}
\Phi_j(\ccD)\equiv \del^{p^{t}}+\sum_{i=0}^{t-1}\beta_{i}\del^{p^{i}}\quad (\modd \fL_{(j)}).
\end{align*}
We start with $j=1$ and continue checking the above equivalence until we get $\ccD$ is conjugate under $G$ to 
\[
\del^{p^{t}}+\sum_{i=0}^{t-1}\beta_{i}\del^{p^{i}}\quad (\modd \fL_{(p^{n}-p^{t}-1)}).
\]
\textbf{(b)} By part (a), we may assume that
\begin{align*}
\ccD=\partial^{p^{t}}+\sum_{i=0}^{t-1}\beta_{i}\partial^{p^{i}}+\sum_{\eta=0}^{p^{t}-1}\mu_{\eta}x^{(p^{n}-p^{t}+\eta)}\del
\end{align*}
for some $\mu_\eta\in k$. Note that for any $0\leq \eta\leq p^{t}-1$,
\begin{equation*}
x^{(p^{n}-p^{t})}x^{(\eta)}={{p^{n}-p^{t}+\eta}\choose{p^{n}-p^{t}}}x^{(p^{n}-p^{t}+\eta)}.
\end{equation*}
The result then follows from \thref{bcmodpapp} which states that for any $0\leq \eta\leq p^{t}-1$, the following congruence holds:
\begin{align*}
{{p^{n}-p^{t}+\eta}\choose{p^{n}-p^{t}}}\equiv 1\quad (\modd p).
\end{align*}
\end{myproof}
\begin{proof}
\textbf{(a)} Take $\ccD=\partial^{p^{t}}+\sum_{i=0}^{t-1}\beta_{i}\partial^{p^{i}}+g(x)\partial$ as in the lemma. By \thref{Tyurinthm1}, we know that if $\Phi(x)=y$ is any admissible automorphism of $\cO(1; n)$ with identical linear part, then 
\[
\Phi(\ccD)=\partial^{p^{t}}+\sum_{i=1}^{t-1}\beta_{i}\partial^{p^{i}}+(y')^{-1}\bigg(\beta_{0}+\Phi(g(x))-\sum_{i=1}^{t-1}\beta_{i}\partial^{p^{i}}y-\partial^{p^{t}}y\bigg)\partial.
\]
If $g(x)\partial\equiv \gamma_{1}x\partial\,(\modd \fL_{(1)})$ for some $\gamma_1\in k$ and $\Phi_1\in G$ is such that $\Phi_1(x)=y_1=x+\gamma_{1}x^{(p^{t}+1)}$, then we show that 
\[
\Phi_1(\ccD) \equiv \partial^{p^{t}}+\sum_{i=0}^{t-1}\beta_{i}\partial^{p^{i}}\quad(\modd \fL_{(1)}).
\]
If $\gamma_1=0$, then $g(x)\partial\in\fL_{(1)}$ and the result is clear. If $\gamma_1\neq 0$, then we show this congruence by proving that $\Phi_1(\ccD)- \partial^{p^{t}}-\sum_{i=0}^{t-1}\beta_{i}\partial^{p^{i}}\in \fL_{(1)}$, i.e.
\begin{align}\label{e:q1}
(y_1')^{-1}\bigg(\beta_{0}+\Phi_1(g(x))-\sum_{i=1}^{t-1}\beta_{i}\partial^{p^{i}}y_1-\partial^{p^{t}}y_1\bigg)\del- \beta_{0}\del \in \fL_{(1)}.
\end{align}
Note that $y_1'=1+\gamma_1x^{(p^{t})}$ which is invertible in $\cO(1; n)$. Since $\fL_{(1)}$ is invariant under multiplication of invertible elements of $\cO(1;n)$, we can multiply both sides of \eqref{e:q1} by $y_1'$ and show that
\begin{align*}
\bigg(\beta_{0}+\Phi_1(g(x))-\sum_{i=1}^{t-1}\beta_{i}\partial^{p^{i}}y_1-\partial^{p^{t}}y_1\bigg)\del- \beta_{0} y_1'\del\in\fL_{(1)}.
\end{align*}
Since $g(x)\partial\equiv \gamma_{1}x\partial\,(\modd \fL_{(1)})$ and $\Phi_1$ preserves the natural filtration of $\fL$, in particular, it preserves $\fL_{(1)}$, hence
\begin{equation*}
\begin{aligned}
&\bigg(\beta_{0}+\Phi_1(g(x))-\sum_{i=1}^{t-1}\beta_{i}\partial^{p^{i}}y_1-\partial^{p^{t}}y_1\bigg)\del- \beta_{0} y_1'\del\\
\equiv&\bigg(\beta_{0}+\gamma_{1}(x+\gamma_1 x^{(p^{t}+1)})-\sum_{i=1}^{t-1}\beta_{i}\gamma_{1}x^{(p^{t}-p^{i}+1)}-\gamma_1 x\bigg)\del\\
\quad &-\beta_0(1+\gamma_1x^{(p^{t} )})\del\\
\equiv & 0\quad(\modd \fL_{(1)}).
\end{aligned}
\end{equation*}
Here we used our assumption that $p>3$, i.e. $p^{t}-p^{i}\geq p^t-p^{t-1}\geq p-1>2$. Therefore, $\ccD$ is conjugate to $\del^{p^{t}}+\sum_{i=0}^{t-1}\beta_{i}\partial^{p^{i}}\, (\modd \fL_{(1)})$. 

In general, if $l<p^n-p^t-1$ and $\ccD \equiv \partial^{p^{t}}+\sum_{i=0}^{t-1}\beta_{i}\partial^{p^{i}}+\gamma_lx^{(l)}\del\,(\modd \fL_{(l)})$ for some $\gamma_l \in k$, i.e. $\ccD =\partial^{p^{t}}+\sum_{i=0}^{t-1}\beta_{i}\partial^{p^{i}}+g_l(x)\del$ with $g_l(x)\del\equiv\gamma_lx^{(l)}\del\,(\modd\fL_{(l)})$, then applying $\Phi_l\in G$ with $\Phi_l(x)=y_l=x+\gamma_l x^{(p^{t}+l)}$ to $\ccD$, we get
\[
\Phi_l(\ccD)=\partial^{p^{t}}+\sum_{i=1}^{t-1}\beta_{i}\partial^{p^{i}}+(y_{l}')^{-1}\bigg(\beta_{0}+\Phi_l(g_l(x))-\sum_{i=1}^{t-1}\beta_{i}\partial^{p^{i}}y_l-\partial^{p^{t}}y_l\bigg)\partial.
\]
We show that
\[
\Phi_l(\ccD)\equiv\partial^{p^{t}}+\sum_{i=0}^{t-1}\beta_{i}\partial^{p^{i}}\quad(\modd\fL_{(l)}).
\]
If $\gamma_l=0$, then $g_l(x)\del\in\fL_{(l)}$ and the result is clear. If $\gamma_l \neq 0$, then we show this congruence by proving that
\begin{align}
&(y_l')^{-1}\bigg(\beta_{0}+\Phi_l(g_l(x))-\sum_{i=1}^{t-1}\beta_{i}\partial^{p^{i}}y_l-\partial^{p^{t}}y_l\bigg)\partial-\beta_{0}\del\in \fL_{(l)}.\label{e:q3}
\end{align}
By the same reason as before, we can multiply both sides of \eqref{e:q3} by $y_l'$ and show that
\begin{align*}
\bigg(\beta_{0}+\Phi_l(g_l(x))-\sum_{i=1}^{t-1}\beta_{i}\partial^{p^{i}}y_l-\partial^{p^{t}}y_l\bigg)\del - \beta_{0}y_l'\del \in\fL_{(l)}.
\end{align*}
Indeed, since $\Phi_l$ preserves $\fL_{(l)}$, we have that
\begin{equation*}
\begin{aligned}
&\bigg(\beta_{0}+\Phi_l(g_l(x))-\sum_{i=1}^{t-1}\beta_{i}\partial^{p^{i}}y_l-\partial^{p^{t}}y_l\bigg)\del - \beta_{0}y_l'\del\\
\equiv & \bigg(\beta_{0}+\gamma_{l}(x+\gamma_{l}x^{(p^{t}+l)})^{(l)}-\sum_{i=1}^{t-1}\beta_{i}\gamma_{l}x^{(p^{t}-p^{i}+l)}-\gamma_{l}x^{(l)}\bigg)\del\\
\quad & -\beta_{0} (1+\gamma_{l}x^{(p^{t}+l-1)})\del\\
\equiv & 0 \quad(\modd\fL_{(l)}).
\end{aligned}
\end{equation*}
Hence $\ccD$ is conjugate to $\partial^{p^{t}}+\sum_{i=0}^{t-1}\beta_{i}\partial^{p^{i}}\, (\modd \fL_{(l)})$. Then
\begin{align*}
\ccD\equiv\partial^{p^{t}}+\sum_{i=0}^{t-1}\beta_{i}\partial^{p^{i}}+\gamma_{l+1}x^{(l+1)}\partial\quad(\modd \fL_{(l+1)})
\end{align*}
for some $\gamma_{l+1}\in k$. If $\gamma_{l+1}\neq 0$, then applying $\Phi_{l+1}\in G$ with $\Phi_{l+1}(x)=y_{l+1}=x+\gamma_{l+1}x^{(p^{t}+l+1)}$ to $\ccD$ we can show that $\ccD$ is conjugate to $\partial^{p^{t}}+\sum_{i=0}^{t-1}\beta_{i}\partial^{p^{i}}\,(\modd \fL_{(l+1)})$. Continue doing this until we get $\ccD$ is conjugate to $\partial^{p^{t}}+\sum_{i=0}^{t-1}\beta_{i}\partial^{p^{i}}\, (\modd \fL_{(p^{n}-p^{t}-1)})$. Then
\begin{align*}
\ccD\equiv\partial^{p^{t}}+\sum_{i=0}^{t-1}\beta_{i}\partial^{p^{i}}+\gamma_{p^{n}-p^{t}}x^{(p^{n}-p^{t})}\del\quad (\modd \fL_{(p^{n}-p^{t})})
\end{align*}
for some $\gamma_{p^{n}-p^{t}} \in k$. If $\gamma_{p^{n}-p^{t}}\neq 0$, then we were supposed to apply $\Phi_{p^{n}-p^{t}}\in G$ with $\Phi_{p^{n}-p^{t}}(x)=x+\gamma_{p^{n}-p^{t}}x^{(p^n)}$ to $\ccD$. But since $x^{(j)}=0$  for $j\geq p^n$ in $\cO(1; n)$, the automorphism $\Phi_{p^{n}-p^{t}}$ is the identity automorphism and we stop here. Therefore, $\ccD$ is conjugate under $G$ to
\[
\partial^{p^{t}}+\sum_{i=0}^{t-1}\beta_{i}\partial^{p^{i}}\quad (\modd \fL_{(p^{n}-p^{t}-1)}).
\]

\textbf{(b)} By part (a), we may assume that
\begin{align*}
\ccD=\partial^{p^{t}}+\sum_{i=0}^{t-1}\beta_{i}\partial^{p^{i}}+\sum_{\eta=0}^{p^{t}-1}\mu_{\eta}x^{(p^{n}-p^{t}+\eta)}\del
\end{align*}
for some $\mu_\eta\in k$. Note that for $0\leq \eta\leq p^{t}-1$,
\begin{equation*}
x^{(p^{n}-p^{t})}x^{(\eta)}={{p^{n}-p^{t}+\eta}\choose{p^{n}-p^{t}}}x^{(p^{n}-p^{t}+\eta)}\equiv x^{(p^{n}-p^{t}+\eta)}\quad (\modd p)
\end{equation*}
by \thref{bcmodpapp}. As a result,
\begin{align*}
\ccD&\equiv\partial^{p^{t}}+\sum_{i=0}^{t-1}\beta_{i}\partial^{p^{i}}+\sum_{\eta=0}^{p^{t}-1}\mu_{\eta}x^{(p^{n}-p^{t})}x^{(\eta)}\del\quad (\modd p)\\
&=\partial^{p^{t}}+\sum_{i=0}^{t-1}\beta_{i}\partial^{p^{i}}+x^{(p^{n}-p^{t})}\sum_{\eta=0}^{p^{t}-1}\mu_{\eta}x^{(\eta)}\del.
\end{align*}
Set $h(x)=\sum_{\eta=0}^{p^{t}-1}\mu_{\eta}x^{(\eta)}$, we get the desired result. This completes the proof.
\end{proof}

Next we assume that $\ccD=\partial^{p^{t}}+\sum_{i=0}^{t-1}\beta_{i}\partial^{p^{i}}+x^{(p^{n}-p^{t})}\sum_{\eta=0}^{p^{t}-1}\mu_{\eta}x^{(\eta)}\del$ is nilpotent. We check that if $\ccD^{p^{n-1}}\in \fL_{(0)}$. Let us first consider the case $t=n-1$.

\begin{lem}\thlabel{p:p1}
Let $\ccD=\partial^{p^{n-1}}+\sum_{i=0}^{n-2}\beta_{i}\partial^{p^{i}}+x^{(p^{n}-p^{n-1})}\sum_{\eta=0}^{p^{n-1}-1}\mu_{\eta}x^{(\eta)}\partial$ be a nilpotent element of $\fL_p$.
\begin{enumerate}[\upshape(i)]
\item If $\beta_{i}=0$ for all $i$, then $\mu_{0}=\mu_{1}=0$ and $\ccD^{p^{n-1}}\in \fL_{(1)}$.
\item
\begin{enumerate}[\upshape(a)]
\item Let $j\geq 0$ be the smallest index such that $\beta_{j}\neq 0$. Then $\mu_{0}=0$ and $\ccD^{p^{n-1-j}}$ is conjugate under $G$ to
\[
\del^{p^{n-1}}+x^{(p^{n}-p^{n-1})}\sum_{\eta=2}^{p^{n-1}-1}\nu_{\eta}x^{(\eta)}\del
\]
for some $\nu_{\eta}\in k$. Hence $\ccD^{p^{n-1}}\in \fL_{(1)}$ if $j\geq 1$.
\item In particular, if $\beta_{0}\neq 0$, then $\ccD^{p^{n-1}}$ is conjugate under $G$ to $\del^{p^{n-1}}$. Hence $\ccD=\del^{p^{n-1}}+\sum_{i=0}^{n-2}\gamma_{i}\del^{p^{i}}$ for some $\gamma_{i}\in k$ with $\gamma_{0}\neq 0$.
\end{enumerate}
\end{enumerate}
\end{lem}

The proof is long as it involves many calculations. Let us first explain the strategy of the proof.
\begin{myproof}[Strategy of the proof]\renewcommand{\qedsymbol}{}
This is a computational proof and the following steps are crucial:

\textbf{\textit{Step 1}}. Take $\ccD=\partial^{p^{n-1}}+\sum_{i=0}^{n-2}\beta_{i}\partial^{p^{i}}+x^{(p^{n}-p^{n-1})}\sum_{\eta=0}^{p^{n-1}-1}\mu_{\eta}x^{(\eta)}\partial$ as in the lemma. Since $\ccD$ is nilpotent, then $\ccD^{p^{n}}=0$ by \eqref{nilpotentconditioninL}. We first calculate $\ccD^{p}$ which will be used in step 2. Set $\ccD_1=x^{(p^{n}-p^{n-1})}\sum_{\eta=0}^{p^{n-1}-1}\mu_{\eta}x^{(\eta)}\partial$ and $\ccD_2=\partial^{p^{n-1}}+\sum_{i=0}^{n-2}\beta_{i}\partial^{p^{i}}$. By Jacobson's formula,
\begin{align*}
\ccD^{p}&=\sum_{i=0}^{n-2}\beta_{i}^{p}\partial^{p^{i+1}}+(\ad \ccD_2)^{p-1}(\ccD_1)+\mu_{(1)} x^{(p^{n}-1)}\partial
\end{align*}
for some $\mu_{(1)}\in k$; see \eqref{cal} in the proof.

\textbf{\textit{Step 2}}. We consider the scalars $\beta_i$ in the following two cases and prove statements (i) and (ii) in the lemma.

\textbf{\textit{Step 2(i)}}. Suppose $\beta_{i}=0$ for all $i$. Then  $\ccD=\partial^{p^{n-1}}+x^{(p^{n}-p^{n-1})}\sum_{\eta=0}^{p^{n-1}-1}\mu_{\eta}x^{(\eta)}\partial$. We show that $\mu_{0}=\mu_{1}=0$ and $\ccD^{p^{n-1}}\in \fL_{(1)}$. By the calculation in step 1, we have that $\ccD^{p}=\sum_{\eta=0}^{p^{n-1}-1}\mu_{\eta}x^{(\eta)}\del+\mu_{(1)} x^{(p^{n}-1)}\partial$; see \eqref{betaiall0firsteq}. We show that if $\mu_0\neq0$, then $\ccD^{p^{n}}\neq 0$, a contradiction. Hence $\mu_0=0$. Similarly, we show that $\mu_1=0$. As a result, $\ccD^{p}\in \fL_{(1)}$. Since $\fL_{(1)}$ is restricted, we get $\ccD^{p^{n-1}}\in \fL_{(1)}$ as desired.

\textbf{\textit{Step 2(ii)(a)}}. Let $j\geq 0$ be the smallest index such that $\beta_{j}\neq 0$, and let $l$ be the largest index such that $\beta_{l}\neq 0$, i.e. $0\leq j\leq l\leq n-2$ and
\[
\ccD=\partial^{p^{n-1}}+\sum_{i=j}^{l}\beta_{i}\partial^{p^{i}}+x^{(p^{n}-p^{n-1})}\sum_{\eta=0}^{p^{n-1}-1}\mu_{\eta}x^{(\eta)}\partial.
\]
We show that $\mu_0=0$ and $\ccD^{p^{n-1-j}}$ is conjugate under $G$ to 
\[
\del^{p^{n-1}}+x^{(p^{n}-p^{n-1})}\sum_{\eta=2}^{p^{n-1}-1}\nu_{\eta}x^{(\eta)}\del
\] 
for some $\nu_{\eta}\in k$. Let us start with the special case $j=l$. We prove by induction that for any $1\leq r\leq n-1-j$, $\ccD^{p^{r}}$ is conjugate under $G$ to
\begin{align*}
\partial^{p^{j+r}}+\beta_{0,(1)}^{p^{r-1}}\del^{p^{r-1}}+x^{(p^n-p^{j+r})}\sum_{\eta=0}^{p^{j+r}-1}\mu_{\eta,(r)}x^{(\eta)}\del
\end{align*}
for some $\beta_{0,(1)}\in k^{*}\mu_0$ and $\mu_{\eta,(r)} \in k$. Here we will use Jacobson's formula and \thref{l:l13}. In particular, $\ccD^{p^{n-1-j}}$ is conjugate under $G$ to
\begin{align*}
\partial^{p^{n-1}}+\beta_{0,(1)}^{p^{n-2-j}}\del^{p^{n-2-j}}+x^{(p^n-p^{n-1})}\sum_{\eta=0}^{p^{n-1}-1}\mu_{\eta,(n-1-j)}x^{(\eta)}\del; 
\end{align*}
see \eqref{ee1}. Then we calculate $\ccD^{p^{n}}$ and use $\ccD^{p^{n}}=0$ to show that $\beta_{0,(1)}=\mu_0=0$ and $\mu_{0,(n-1-j)}=\mu_{1,(n-1-j)}=0$. This gives the desired result in this case.

For the general case, $j<l$, one can show similarly that $\ccD^{p^{n-1-j}}$ is conjugate under $G$ to
\begin{equation*}
\partial^{p^{n-1}}+\lambda\del^{p^{n-2-j}}+\sum_{i=0}^{n-3-j}\lambda_i\del^{p^{i}}+x^{(p^n-p^{n-1})}\sum_{\eta=0}^{p^{n-1}-1}\nu_\eta x^{(\eta)}\del
\end{equation*}
for some $\lambda\in k^{*}\mu_0$ and $\lambda_i, \nu_\eta \in k$; see \eqref{generaljlcase}. Then we show similarly that $\ccD^{p^{n}}=0$ implies that $\lambda=\mu_{0}=0$, $\lambda_i=0$ for $0\leq i\leq n-3-j$ and $\nu_0=\nu_1=0$. This gives the desired result.

Suppose now $j\geq 1$. We show that $(\del^{p^{n-1}}+x^{(p^{n}-p^{n-1})}\sum_{\eta=2}^{p^{n-1}-1}\nu_{\eta}x^{(\eta)}\del)^p\in \fL_{(1)}$; see \eqref{Dpn-jcaseiia}. Since $\ccD^{p^{n-1-j}}$ is conjugate under $G$ to $\del^{p^{n-1}}+x^{(p^{n}-p^{n-1})}\sum_{\eta=2}^{p^{n-1}-1}\nu_{\eta}x^{(\eta)}\del$ and $G$ preserves $\fL_{(1)}$, we have that $\ccD^{p^{n-j}}\in \fL_{(1)}$.
As $\fL_{(1)}$ is restricted, we get $\ccD^{p^{n-1}}\in \fL_{(1)}$ as desired.

\textbf{\textit{Step 2(ii)(b)}}. Suppose $\beta_0\neq 0$. By step 2(ii)(a), $\ccD^{p^{n-1}}$ is conjugate under $G$ to $\del^{p^{n-1}}+x^{(p^{n}-p^{n-1})}\sum_{\eta=2}^{p^{n-1}-1}\nu_{\eta}x^{(\eta)}\del$ for some $\nu_{\eta}\in k$. We show that $\ccD^{p^{n}}= 0$ implies that $\nu_{\eta}=0$ for all $\eta$. Hence $\ccD^{p^{n-1}}$ is conjugate to $\del^{p^{n-1}}$. Then we find an expression for $\ccD$. Let
\begin{align*}
\mathcal{S}:=\bigg\{\ccD\in \big(\del^{p^{n-1}}+\sum_{i=1}^{n-2}k\del^{p^{i}}+\fL\big)\cap \cN\,|\, \text{$\ccD^{p^{n-1}}$ is conjugate under $G$ to $\partial^{p^{n-1}}$}\bigg\}.
\end{align*}
Note that $\mathcal{S}$ is a subset of the centralizer $\fc_{\fL_{p}}(\del^{p^{n-1}})$. Then we consider elements of $\fc_{\fL_{p}}(\del^{p^{n-1}})$ and show that any $\ccD$ in $\mathcal{S}$ has the form $\ccD=\del^{p^{n-1}}+\sum_{i=0}^{n-2}\gamma_{i}\del^{p^{i}}$ for some $\gamma_{i}\in k$ with $\gamma_{0}\neq 0$.
\end{myproof}

\begin{proof}
\textbf{\textit{Step 1}}. Let $\ccD=\partial^{p^{n-1}}+\sum_{i=0}^{n-2}\beta_{i}\partial^{p^{i}}+x^{(p^{n}-p^{n-1})}\sum_{\eta=0}^{p^{n-1}-1}\mu_{\eta}x^{(\eta)}\partial$ be a nilpotent element of $\fL_p$. Then $\ccD^{p^{n}}=0$ by \eqref{nilpotentconditioninL}. Let us first calculate $\ccD^{p}$. Recall Jacobson's formula,
\begin{align}\label{J1}
(\ccD_1+\ccD_2)^{p}=\ccD_1^{p}+\ccD_2^{p}+\sum_{i=1}^{p-1}s_{i}(\ccD_1, \ccD_2)
\end{align}
for all $\ccD_1, \ccD_2\in \fL_p$, and $s_{i}(\ccD_1, \ccD_2)$ can be computed by the formula
\begin{align*}
\ad(t\ccD_1+\ccD_2)^{p-1}(\ccD_1)=\sum_{i=1}^{p-1}is_{i}(\ccD_1, \ccD_2)t^{i-1},
\end{align*}
where $t$ is a variable. Set 
\[
\ccD_1=x^{(p^{n}-p^{n-1})}\sum_{\eta=0}^{p^{n-1}-1}\mu_{\eta}x^{(\eta)}\partial
\]
and 
\[
\ccD_2=\partial^{p^{n-1}}+\sum_{i=0}^{n-2}\beta_{i}\partial^{p^{i}}.
\]
We first show that $\ccD_1^{p}=0$. By \thref{bcmodpapp}, $x^{(p^{n}-p^{n-1})}x^{(\eta)}\equiv x^{(p^{n}-p^{n-1}+\eta)}\,(\modd p)$. Then $\ccD_1=\sum_{\eta=0}^{p^{n-1}-1}\mu_{\eta}x^{(p^{n}-p^{n-1}+\eta)}\partial$. By \thref{generalJacobF},
\begin{align*}
\ccD_1^{p}=\sum_{\eta=0}^{p^{n-1}-1}\big(\mu_{\eta}x^{(p^{n}-p^{n-1}+\eta)}\partial\big)^{p}+w,
\end{align*}
where $w$ is a linear combination of commutators in $\mu_{\eta}x^{(p^{n}-p^{n-1}+\eta)}\partial$, $0\leq \eta\leq p^{n-1}-1$. By Jacobi identity, we can rearrange $w$ so that $w$ is in the span of commutators\\
$[w_{p-1},[w_{p-2},[\dots,[w_2,[w_1, w_0]\dots]$, where each $w_{\nu}, 0\leq \nu\leq p-1$, is equal to some \\
$\mu_{\eta}x^{(p^{n}-p^{n-1}+\eta)}\partial$, $0\leq \eta\leq p^{n-1}-1$. We show that such iterated commutator equals $0$ and so $w=0$. Recall \eqref{filtrationoffL} the natural filtration $\{\fL_{(\alpha)}\}_{\alpha\geq -1}$ of $\fL$, where $\fL_{(\alpha)}=0$ for all $\alpha>p^n-2$. Since $0\leq \eta\leq p^{n-1}-1$, we have that $\mu_{\eta}x^{(p^{n}-p^{n-1}+\eta)}\partial\in \fL_{(p^n-p^{n-1}-1)}$ for all $\eta$. Then $[w_{p-1},[w_{p-2},[\dots,[w_2,[w_1, w_0]\dots]\in \fL_{(p(p^n-p^{n-1}-1))}$ by \thref{fildefn}(ii). We show that $\fL_{(p(p^n-p^{n-1}-1))}=0$, i.e. $p(p^n-p^{n-1}-1)>p^n-2$. Note that 
\[
p(p^n-p^{n-1}-1)>p^n-2\iff p^{n+1}-p^{n}-p>p^n-2\iff p^{n+1}-p+2> 2p^n.
\]
We consider the $p$-adic expansions of these two numbers. By our assumption, $p>3$. Then $p-1>2$. So $2p^n+0p^{n-1}+\dots+0p+0$ is the $p$-adic expansion of $2p^n$. By \eqref{paeofrs1} in \thref{bcmodpapp}, the $p$-adic expansion of $p^{n+1}-p+2$ is 
\[
p^{n+1}-p+2=\sum_{j'=1}^{n}(p-1)p^{j'}+2=(p-1)p^n+(p-1)p^{n-1}+\dots+(p-1)p+2.
\]
Since $p-1>2$, it is clear that $p^{n+1}-p+2>2p^n$. Hence $\fL_{(p(p^n-p^{n-1}-1))}=0$ and so $[w_{p-1},[w_{p-2},[\dots,[w_2,[w_1, w_0]\dots]=0$. As a result, $w=0$. By \eqref{eip}, we have that $\big(\mu_{\eta}x^{(p^{n}-p^{n-1}+\eta)}\partial\big)^{p}=0$ for all $0\leq \eta\leq p^{n-1}-1$. Therefore, $\ccD_1^{p}=0$. 

By \thref{generalJacobF} again, we get $\ccD_2^{p}=\sum_{i=0}^{n-2}\beta_{i}^{p}\partial^{p^{i+1}}$. By \eqref{filtrationoffL} the natural filtration of $\fL$, we have that for any $1\leq s\leq p-2$,
\begin{align*}
[\ccD_1, (\ad \ccD_2)^{s}(\ccD_1)]&\in [\fL_{(p^{n}-p^{n-1}-1)}, \fL_{(p^{n}-(s+1)p^{n-1}-1)}]\\
&\subseteq[\fL_{(p^{n}-p^{n-1}-1)}, \fL_{(p^{n-1}-1)}]\\
&\subseteq \fL_{(p^{n}-2)}=\text{span$\{x^{(p^{n}-1)}\partial\}$}.
\end{align*}
This last term will appear if and only if $s=p-2$. So
\begin{align}\label{cal}
\ccD^{p}&=\sum_{i=0}^{n-2}\beta_{i}^{p}\partial^{p^{i+1}}+(\ad \ccD_2)^{p-1}(\ccD_1)+\mu_{(1)} x^{(p^{n}-1)}\partial
\end{align}
for some $\mu_{(1)}\in k$.

\textbf{\textit{Step 2}}. We consider the scalars $\beta_i$ in the following two cases.

\textbf{\textit{Step 2(i)}}. If $\beta_{i}=0$ for all $i$, then $\ccD_2=\del^{p^{n-1}}$. By \eqref{cal},  
\[
\ccD^{p}=(\ad \del^{p^{n-1}})^{p-1}(\ccD_1)+\mu_{(1)} x^{(p^{n}-1)}\partial.
\]
Since $\partial^{p^{n-1}}$ is a derivation of $\fL$ and $\partial^{p^{n-1}}(\sum_{\eta=0}^{p^{n-1}-1}\mu_{\eta}x^{(\eta)})=0$, we have that
\begin{align}\label{betaiall0firsteq}
\ccD^{p}
=\sum_{\eta=0}^{p^{n-1}-1}\mu_{\eta}x^{(\eta)}\del+\mu_{(1)} x^{(p^{n}-1)}\partial.
\end{align}
If $\mu_{0}\neq 0$, then $\ccD^{p}\equiv \mu_{0}\partial \,(\modd \fL_{(0)})$. By \thref{generalJacobF},
\[
\ccD^{p^{n}}\equiv \mu_{0}^{p^{n-1}}\partial^{p^{n-1}}+\sum_{i=0}^{n-2} \mu_{i}'\del^{p^{i}}\quad(\modd \fL_{(0)})
\]
for some $\mu_{i}'\in k$. As $\mu_0\neq 0$, this implies that $\ccD^{p^{n}}\not \equiv 0\,(\modd \fL_{(0)})$ and so it is not equal to $0$. This contradicts that $\ccD$ is nilpotent. Hence $\mu_{0}=0$. Similarly, if $\mu_{1}\neq 0$ then $\ccD^{p}\equiv \mu_{1}x\partial \,(\modd \fL_{(1)})$. But $\ccD^{p^{n}}\equiv \mu_{1}^{p^{n-1}}x\partial\not \equiv 0\,(\modd \fL_{(1)})$, a contradiction. Thus $\mu_{1}=0$. Therefore, $\ccD^{p}$ is an element of $\fL_{(1)}$. Since $\fL_{(1)}$ is restricted we have that $\ccD^{p^{n-1}}\in \fL_{(1)}$. This proves (i).

\textbf{\textit{Step 2(ii)(a)}}. Let $j\geq 0$ be the smallest index such that $\beta_{j}\neq 0$, and let $l$ be the largest index such that $\beta_{l}\neq 0$, i.e. $0\leq j\leq l\leq n-2$ and
\[
\ccD=\partial^{p^{n-1}}+\sum_{i=j}^{l}\beta_{i}\partial^{p^{i}}+x^{(p^{n}-p^{n-1})}\sum_{\eta=0}^{p^{n-1}-1}\mu_{\eta}x^{(\eta)}\partial.
\]
We first consider the special case $j=l$, i.e.
\begin{align*}
\ccD=\partial^{p^{n-1}}+\beta_{j}\partial^{p^{j}}+x^{(p^{n}-p^{n-1})}\sum_{\eta=0}^{p^{n-1}-1}\mu_{\eta}x^{(\eta)}\partial.
\end{align*}
We prove by induction that for any $1\leq r\leq n-1-j$, $\ccD^{p^{r}}$ is conjugate under $G$ to
\begin{align*}
\partial^{p^{j+r}}+\beta_{0,(1)}^{p^{r-1}}\del^{p^{r-1}}+x^{(p^n-p^{j+r})}\sum_{\eta=0}^{p^{j+r}-1}\mu_{\eta,(r)}x^{(\eta)}\del
\end{align*}
for some $\beta_{0,(1)}\in k^{*}\mu_0$ and $\mu_{\eta,(r)} \in k$. For $r=1$, the previous calculation \eqref{cal} gives
\begin{align*}
\ccD^{p}=\beta_{j}^{p}\partial^{p^{j+1}}+\ad \big(\partial^{p^{n-1}}+\beta_{j}\partial^{p^{j}}\big)^{p-1}\bigg(x^{(p^{n}-p^{n-1})}\sum_{\eta=0}^{p^{n-1}-1}\mu_{\eta}x^{(\eta)}\partial\bigg)+\mu_{(1)} x^{(p^{n}-1)}\partial.
\end{align*}
Note that
\begin{equation*}
\begin{aligned}
&\ad \bigg(\partial^{p^{n-1}}+\beta_{j}\partial^{p^{j}}\bigg)^{p-1}\bigg(x^{(p^{n}-p^{n-1})}\sum_{\eta=0}^{p^{n-1}-1}\mu_{\eta}x^{(\eta)}\partial\bigg)\\
=&\ad\bigg(\sum_{m=0}^{p-1}(-1)^{m}\beta_{j}^{p-1-m}\partial^{mp^{n-1}+(p-1-m)p^{j}}\bigg)\bigg(\sum_{\eta=0}^{p^{n-1}-1}\mu_{\eta}x^{(p^{n}-p^{n-1}+\eta)}\partial\bigg)\\
=&\sum_{\eta=0}^{p^{n-1}-1}\mu_{\eta}x^{(\eta)}\partial-\beta_{j}\sum_{\eta=0}^{p^{n-1}-1}\mu_{\eta}x^{(p^{n-1}-p^{j}+\eta)}\partial+\dots\\
\quad& +\beta_{j}^{p-1}\sum_{\eta=0}^{p^{n-1}-1}\mu_{\eta}x^{(p^{n}-p^{n-1}-(p-1)p^{j}+\eta)}\partial.
\end{aligned}
\end{equation*}
The above result can be rewritten as $\mu_{0}\del+g(x)\del$ for some $g(x)\in \fm$. Hence
\[
\ccD^{p}=\beta_{j}^{p}\partial^{p^{j+1}}+\mu_{0}\del+g(x)\del+\mu_{(1)} x^{(p^{n}-1)}\partial.
\]
Then the automorphism $\Phi(x)=\alpha x$ with $\alpha^{p^{j+1}}=\beta_j^{p}$ reduces $\ccD^{p}$ to the form
\[
\ccD^{p}=\partial^{p^{j+1}}+\beta_{0,(1)}\del+f_{1}(x)\del,
\]
where $\beta_{0,(1)}\in k^{*}\mu_{0}$ and $f_{1}(x)\in \fm$. Then \thref{l:l13} implies that $\ccD^{p}$ is conjugate under $G$ to
\begin{align*}
\partial^{p^{j+1}}+\beta_{0,(1)}\del+x^{(p^n-p^{j+1})}\sum_{\eta=0}^{p^{j+1}-1}\mu_{\eta,(1)}x^{(\eta)}\del
\end{align*}
for some $\mu_{\eta,(1)}\in k$. Thus, the result is true for $r=1$. Suppose the result is true for $r=K-1<n-1-j$, i.e. $\ccD^{p^{K-1}}$ is conjugate under $G$ to
\[
\partial^{p^{j+K-1}}+\beta_{0,(1)}^{p^{K-2}}\del^{p^{K-2}}+x^{(p^n-p^{j+K-1})}\sum_{\eta=0}^{p^{j+K-1}-1}\mu_{\eta,(K-1)}x^{(\eta)}\del
\]
for some $\beta_{0,(1)}\in k^{*}\mu_0$ and $\mu_{\eta,(K-1)}\in k$. Let us calculate $\ccD^{p^{K}}$. Set 
\[
\ccD_1=x^{(p^n-p^{j+K-1})}\sum_{\eta=0}^{p^{j+K-1}-1}\mu_{\eta,(K-1)}x^{(\eta)}\del
\] 
and 
\[\ccD_2=\partial^{p^{j+K-1}}+\beta_{0,(1)}^{p^{K-2}}\del^{p^{K-2}}\]
in the Jacobson's formula~\eqref{J1}. Then $\ccD_1^p\in\fL_{(1)}$ and $\ccD_2^p=\partial^{p^{j+K}}+\beta_{0,(1)}^{p^{K-1}}\del^{p^{K-1}}$. By \eqref{filtrationoffL} the natural filtration of $\fL$, we have that
\begin{align*}
(\ad \ccD_2)^{p-1}(\ccD_1)\in \fL_{(p^n-p^{j+K}-1)}\subseteq\fL_{(1)}.
\end{align*}
Similarly, for any $1\leq s\leq p-2$,
\begin{align*}
[\ccD_1, (\ad \ccD_2)^s(\ccD_1)]&\in [\fL_{(p^n-p^{j+K-1}-1)}, \fL_{(p^n-(s+1)p^{j+K-1}-1)}]\\
&\subseteq[\fL_{(1)}, \fL_{(p^n-(s+1)p^{j+K-1}-1)}]\\
&\subseteq\fL_{(p^n-(s+1)p^{j+K-1})},\\
&\subseteq \fL_{(p^n-(p-1)p^{j+K-1})},\\
&\subseteq \fL_{(p^n-(p-1)p^{n-2})}\quad\text{(since $j+K-1\leq n-2$)}\\
&\subseteq \fL_{(1)}.
\end{align*}
Hence $\ccD^{p^{K}}=\partial^{p^{j+K}}+\beta_{0,(1)}^{p^{K-1}}\del^{p^{K-1}}+f_{K}(x)\del$ for some $f_{K}(x)\in \fm$ with $f_{K}(x)\del \in \fL_{(1)}$. By \thref{l:l13}, $\ccD^{p^{K}}$ is conjugate under $G$ to
\[
\partial^{p^{j+K}}+\beta_{0,(1)}^{p^{K-1}}\del^{p^{K-1}}+x^{(p^n-p^{j+K})}\sum_{\eta=0}^{p^{j+K}-1}\mu_{{\eta}, (K)}x^{(\eta)}\del
\]for some $\mu_{{\eta}, (K)}\in k$, i.e. the result is true for $r=K$. Therefore, we proved by induction that for any $1\leq r\leq n-1-j$, $\ccD^{p^{r}}$ is conjugate under $G$ to
\begin{align*}
\partial^{p^{j+r}}+\beta_{0,(1)}^{p^{r-1}}\del^{p^{r-1}}+x^{(p^n-p^{j+r})}\sum_{\eta=0}^{p^{j+r}-1}\mu_{\eta,(r)}x^{(\eta)}\del
\end{align*}
for some $\beta_{0,(1)}\in k^{*}\mu_{0}$ and $\mu_{\eta,(r)} \in k$. In particular, $\ccD^{p^{n-1-j}}$ is conjugate under $G$ to
\begin{align}
\partial^{p^{n-1}}+\beta_{0,(1)}^{p^{n-2-j}}\del^{p^{n-2-j}}+x^{(p^n-p^{n-1})}\sum_{\eta=0}^{p^{n-1}-1}\mu_{\eta,(n-1-j)}x^{(\eta)}\del\label{ee1}.
\end{align}
By Jacobson's formula,
\begin{align}
\ccD^{p^{n-j}}=\beta_{0,(1)}^{p^{n-1-j}}\del^{p^{n-1-j}}+\sum_{\eta=0}^{p^{n-1}-1}\mu_{\eta,(n-1-j)}x^{(\eta)}\del+f_{n-j}(x) \del+\mu_{(n-j)}x^{(p^{n}-1)}\del\,\label{ee2}
\end{align} for some $f_{n-j}(x)\del \in \fL_{(1)}$ and $\mu_{(n-j)}\in k$. Then
\begin{align*}
\ccD^{p^{n}}\equiv \beta_{0,(1)}^{p^{n-1}}\del^{p^{n-1}}+\mu_{0,(n-1-j)}^{p^{j}}\del^{p^{j}}+\sum_{i=0}^{j-1}\mu_i''\del^{p^{i}}\quad(\modd \fL_{(0)})
\end{align*}
for some $\mu_i'' \in k$. But $\ccD^{p^{n}}=0$, this implies that $\beta_{0,(1)}=0$. Since $\beta_{0,(1)}\in k^{*}\mu_{0}$, we have that $\mu_{0}=0$. We must also have that $\mu_{0,(n-1-j)}=0$ and $\mu_i''=0$ for all $i$. Substituting these into \eqref{ee2}, we get
\begin{align*}
\ccD^{p^{n-j}}&=\sum_{\eta=1}^{p^{n-1}-1}\mu_{\eta,(n-1-j)}x^{(\eta)}\del +f_{n-j}(x) \del+\mu_{(n-j)}x^{(p^{n}-1)}\del\\
&\equiv \mu_{1,(n-1-j)}x\del \quad(\modd \fL_{(1)}).
\end{align*}
Then one can show similarly that $\mu_{1,(n-1-j)}=0$. Hence $\ccD^{p^{n-1-j}}$ \eqref{ee1} is conjugate under $G$ to
\begin{align*}
\partial^{p^{n-1}}+x^{(p^n-p^{n-1})}\sum_{\eta=2}^{p^{n-1}-1}\mu_{\eta,(n-1-j)}x^{(\eta)}\del.
\end{align*}

If $j<l$, i.e. $\ccD=\partial^{p^{n-1}}+\sum_{i=j}^{l}\beta_{i}\partial^{p^{i}}+x^{(p^{n}-p^{n-1})}\sum_{\eta=0}^{p^{n-1}-1}\mu_{\eta}x^{(\eta)}\partial$, then one can show similarly that $\ccD^{p^{n-1-j}}$ is conjugate under $G$ to
\begin{equation}\label{generaljlcase}
\partial^{p^{n-1}}+\lambda\del^{p^{n-2-j}}+\sum_{i=0}^{n-3-j}\lambda_i\del^{p^{i}}+x^{(p^n-p^{n-1})}\sum_{\eta=0}^{p^{n-1}-1}\nu_\eta x^{(\eta)}\del
\end{equation}
for some $\lambda\in k^{*}\mu_0$ and $\lambda_i, \nu_\eta \in k$. Then by the same arguments as above, one can show that $\lambda=\mu_{0}=0$, $\lambda_i=0$ for $0\leq i\leq n-3-j$ and $\nu_0=\nu_1=0$. As a result, $\ccD^{p^{n-1-j}}$ is conjugate under $G$ to
\begin{align*}
&\partial^{p^{n-1}}+x^{(p^{n}-p^{n-1})}\sum_{\eta=2}^{p^{n-1}-1}\nu_\eta x^{(\eta)}\partial.
\end{align*}

Suppose now $j\geq 1$. By a similar calculation as in \eqref{cal} and \eqref{betaiall0firsteq}, we get
\begin{equation}\label{Dpn-jcaseiia}
\begin{aligned}
\big(\partial^{p^{n-1}}+x^{(p^{n}-p^{n-1})}\sum_{\eta=2}^{p^{n-1}-1}\nu_\eta x^{(\eta)}\partial\big)^p=&\sum_{\eta=2}^{p^{n-1}-1}\nu_{\eta}x^{(\eta)}\partial+\mu_{(n-j)}x^{(p^n-1)}\del
\end{aligned}
\end{equation}
for some $\mu_{(n-j)}\in k$. This is an element of $\fL_{(1)}$. Since $\ccD^{p^{n-1-j}}$ is conjugate under $G$ to $\partial^{p^{n-1}}+x^{(p^{n}-p^{n-1})}\sum_{\eta=2}^{p^{n-1}-1}\nu_\eta x^{(\eta)}\partial$ and $G$ preserves $\fL_{(1)}$, we have that $\ccD^{p^{n-j}}\in \fL_{(1)}$. As $\fL_{(1)}$ is restricted, we have that $\ccD^{p^{n-1}}\in \fL_{(1)}$. This proves (ii)(a).

\textbf{\textit{Step 2(ii)(b)}}. If $\beta_{0}\neq 0$, then (ii)(a) implies that $\ccD^{p^{n-1}}$ is conjugate under $G$ to
\[
\partial^{p^{n-1}}+x^{(p^{n}-p^{n-1})}\sum_{\eta=2}^{p^{n-1}-1}\nu_{\eta}x^{(\eta)}\del
\]
for some $\nu_{\eta}\in k$. If $q$ is the smallest index such that $\nu_{q}\neq 0$, then
\begin{align*}
\ccD^{p^{n}}=&\sum_{\eta=q}^{p^{n-1}-1}\nu_{\eta}x^{(\eta)}\del+\mu_{(n)}x^{(p^{n}-1)}\del
\end{align*}
for some $\mu_{(n)}\in k$. As $\nu_{q}\neq 0$, this implies that $\ccD^{p^{n}}\neq 0$, a contradiction. Hence $\nu_{\eta}=0$ for all $\eta$. Therefore, we are interested in the set
\begin{align*}
\mathcal{S}:=\bigg\{\ccD\in \big(\del^{p^{n-1}}+\sum_{i=1}^{n-2}k\del^{p^{i}}+\fL\big)\cap \cN\,|\, \text{$\ccD^{p^{n-1}}$ is conjugate under $G$ to $\partial^{p^{n-1}}$}\bigg\}.
\end{align*}
Since $[\ccD, \ccD^{p^{n-1}}]=0$, the above set $\mathcal{S}$ is a subset of the centralizer $\fc_{\fL_{p}}(\del^{p^{n-1}})$. It is easy to verify that $\fc_{\fL_{p}}(\del^{p^{n-1}})$ is spanned by $\del^{p^{n-1}}$ and $W(1,n-1)_{p}$. Since $W(1,n-1)_{p}$ is a restricted Lie subalgebra of $\fL_{p}$, we may regard the automorphism group of $W(1,n-1)_{p}$ as a subgroup of $G$. Let $\ccD=\del^{p^{n-1}}+\sum_{i=1}^{n-2}\gamma_{i}\del^{p^{i}}+v$ be an element of $\fc_{\fL_{p}}(\del^{p^{n-1}})$, where $\gamma_{i}\in k$ and $v\in W(1,n-1)$. If $v=0$, then $\ccD^{p^{n-1}}=0$ which is not conjugate to $\del^{p^{n-1}}$. So $v\neq 0$. If $v\not\in W(1, n-1)_{(0)}$, then $v=\gamma_{0}\del$ for some $\gamma_{0}\neq 0$. It is easy to see that $\ccD^{p^{n-1}}$ is conjugate under $G$ to $\del^{p^{n-1}}$. If $v \in W(1, n-1)_{(0)}$, then $v=\sum_{i=1}^{p^{n-1}-1}\lambda_{i}x^{(i)}\del$ with $\lambda_{i}\neq 0$ for some $i$. It follows from \thref{l:l13} that $\ccD$ is conjugate under $G$ to
\begin{align*}
\del^{p^{n-1}}+\sum_{i=1}^{n-2}\gamma_{i}\del^{p^{i}}+x^{(p^n-p^{n-1})}\sum_{\eta=0}^{p^{n-1}-1}\lambda_{\eta}'x^{(\eta)}\del
\end{align*}
for some $\lambda_{\eta}'\in k$. If $\gamma_{i}=0$ for all $i$, then (i) of this lemma implies that $\ccD^{p^{n-1}} \in \fL_{(1)}$ which is not conjugate to $\del^{p^{n-1}}$. Similarly, if $j\geq 1$ is the smallest index such that $\gamma_{j}\neq 0$, then (ii)(a) of this lemma implies that $\ccD^{p^{n-1}}\in \fL_{(1)}$ which is again not conjugate to $\del^{p^{n-1}}$. Therefore, the set $\mathcal{S}$ consists of elements of the form $\ccD=\del^{p^{n-1}}+\sum_{i=0}^{n-2}\gamma_{i}\del^{p^{i}}$, where $\gamma_{i}\in k$ with $\gamma_{0}\neq 0$. This proves (ii)(b).
\end{proof}

We see that the last proof involves calculations using Jacobson's formula and applications of \thref{l:l13}. The only property of $\ccD$ that we used is $\ccD^{p^{n}}=0$. Hence using the same arguments, we can prove a very similar result for nilpotent elements $\partial^{p^{m}}+\sum_{i=0}^{m-1}\alpha_{i}\partial^{p^{i}}+x^{(p^n-p^{m})}\sum_{\eta=0}^{p^{m}-1}\mu_{\eta}x^{(\eta)}\del$, where $1\leq m \leq n-2$.

\begin{cor}\thlabel{c:c2}
Let $E=\partial^{p^{m}}+\sum_{i=0}^{m-1}\alpha_{i}\partial^{p^{i}}+x^{(p^n-p^{m})}\sum_{\eta=0}^{p^{m}-1}\mu_{\eta}x^{(\eta)}\del$ with $1\leq m \leq n-2$ be a nilpotent element of $\fL_p$.
\begin{enumerate}[\upshape(i)]
\item If $\alpha_{i}=0$ for all $i$, then $E^{p^{n-1}}\in \fL_{(1)}$.
\item
\begin{enumerate}[\upshape(a)]
\item Let $q\geq 0$ be the smallest index such that $\alpha_{q}\neq 0$. Then $E^{p^{n-1-q}}$ is conjugate under $G$ to
\[
\del^{p^{n-1}}+x^{(p^{n}-p^{n-1})}\sum_{\eta=2}^{p^{n-1}-1}\nu_{\eta}x^{(\eta)}\del
\]
for some $\nu_{\eta}\in k$. Hence $E^{p^{n-1}}\in \fL_{(1)}$ if $q \geq 1$.
\item In particular, if $\alpha_{0}\neq 0$, then $E^{p^{n-1}}$ is conjugate under $G$ to $\del^{p^{n-1}}$. Hence $E=\del^{p^{m}}+\sum_{i=0}^{m-1}\gamma_{i}\del^{p^{i}}$ for some $\gamma_{i}\in k$ with $\gamma_{0}\neq 0$.
\end{enumerate}
\end{enumerate}
\end{cor}

\begin{myproof}[Strategy of the proof]\renewcommand{\qedsymbol}{}
By a similar argument as in step 2(ii)(a) of \thref{p:p1}, one can show that $E^{p^{n-1-m}}$ is conjugate under $G$ to \[
\ccD_1=\partial^{p^{n-1}}+\sum_{i=0}^{m-1}\alpha_{i}^{p^{n-1-m}}\del^{p^{i+n-1-m}}+x^{(p^n-p^{n-1})}\sum_{\eta=0}^{p^{n-1}-1}\mu_{\eta,(n-1-m)}x^{(\eta)}\del
\]
for some $\mu_{\eta,(n-1-m)}\in k$. Note that $\ccD_1$ has a similar expression to $\ccD$ in \thref{p:p1}. Applying that lemma to $\ccD_1$, we get most of the desired results. The other results such as showing that $E^{p^{n-1}}\in \fL_{(1)}$ in (i) and finding an expression for $E$ in (ii)(b) follow from the same arguments as in the proof of that lemma. 
\end{myproof}
\begin{proof}
Take $E=\partial^{p^{m}}+\sum_{i=0}^{m-1}\alpha_{i}\partial^{p^{i}}+x^{(p^n-p^{m})}\sum_{\eta=0}^{p^{m}-1}\mu_{\eta}x^{(\eta)}\del$ as in the corollary. By a similar argument as in step 2(ii)(a) of \thref{p:p1}, i.e. using induction on $r$, one can show that for any $1\leq r\leq n-1-m$, $E^{p^{r}}$ is conjugate under $G$ to
\begin{align*}
\partial^{p^{m+r}}+\sum_{i=0}^{m-1}\alpha_{i}^{p^{r}}\del^{p^{i+r}}+x^{(p^n-p^{m+r})}\sum_{\eta=0}^{p^{m+r}-1}\mu_{\eta,(r)}x^{(\eta)}\del
\end{align*}
for some $\mu_{\eta,(r)} \in k$. In particular, $E^{p^{n-1-m}}$ is conjugate under $G$ to
\begin{align*}
\ccD_1=\partial^{p^{n-1}}+\sum_{i=0}^{m-1}\alpha_{i}^{p^{n-1-m}}\del^{p^{i+n-1-m}}+x^{(p^n-p^{n-1})}\sum_{\eta=0}^{p^{n-1}-1}\mu_{\eta,(n-1-m)}x^{(\eta)}\del.
\end{align*}
Note that $\ccD_1$ has a similar expression to $\ccD$ in \thref{p:p1}. Since $E$ is nilpotent, we have that $E^{p^{n}}=0$ by \eqref{nilpotentconditioninL}. As $E^{p^{n-1-m}}$ is conjugate under $G$ to $\ccD_1$, this implies that $\ccD_1^{p^{m+1}}=0$.

(i) Suppose $\alpha_{i}=0$ for all $i$. Applying \thref{p:p1}(i) to $\ccD_1$, we get $\mu_{0,(n-1-m)}=\mu_{1,(n-1-m)}=0$. Then \eqref{betaiall0firsteq} in step 2(i) of \thref{p:p1} gives 
\[
\ccD_1^{p}=\sum_{\eta=2}^{p^{n-1}-1}\mu_{\eta,(n-1-m)}x^{(\eta)}\del+\mu_{(1)} x^{(p^{n}-1)}\partial
\]
for some $\mu_{(1)}\in k$. Hence $\ccD_1^{p}\in \fL_{(1)}$. Since $G$ preserves $\fL_{(1)}$, we have that $E^{p^{n-m}}\in\fL_{(1)}$. As $\fL_{(1)}$ is restricted, we have that $E^{p^{n-1}}\in \fL_{(1)}$. This proves (i).

(ii)(a) Let $q\geq 0$ be the smallest index such that $\alpha_{q}\neq 0$. Then $E^{p^{n-1-m}}$ is conjugate under $G$ to
\[
\ccD_1=\partial^{p^{n-1}}+\sum_{i=q}^{m-1}\alpha_{i}^{p^{n-1-m}}\del^{p^{i+n-1-m}}+x^{(p^n-p^{n-1})}\sum_{\eta=0}^{p^{n-1}-1}\mu_{\eta,(n-1-m)}x^{(\eta)}\del.
\]
Applying \thref{p:p1}(ii)(a) to $\ccD_1$, we get $\mu_{0,(n-1-m)}=0$ and $\ccD_1^{p^{m-q}}$ is 
conjugate under $G$ to
\begin{equation*}
\del^{p^{n-1}}+x^{(p^n-p^{n-1})}\sum_{\eta=2}^{p^{n-1}-1}\nu_{\eta}x^{(\eta)}\del
\end{equation*}
for some $\nu_{\eta}\in k$. Hence $(E^{p^{n-1-m}})^{p^{m-q}}=E^{p^{n-1-q}}$ is conjugate under $G$ to the above element. It follows from \thref{p:p1}(ii)(a) that $E^{p^{n-1}}\in \fL_{(1)}$ if $q\geq 1$. This proves (ii)(a).

(b) If $\alpha_{0}\neq 0$, then it follows from the above and \thref{p:p1}(ii)(b) that $E^{p^{n-1}}$ is conjugate under $G$ to $\del^{p^{n-1}}$. Here we used that $E^{p^{n}}=0$. By the same arguments as in step 2(ii)(b) of \thref{p:p1}, one can show that $E=\del^{p^{m}}+\sum_{i=0}^{m-1}\gamma_{i}\del^{p^{i}}$ for some $\gamma_{i}\in k$ with $\gamma_{0}\neq 0$. This proves (ii)(b).
\end{proof}

\subsection{An irreducible component of $\cN$}
The results in the last section enable us to prove the following: 

\begin{pro}\thlabel{l:l14}
Define $\cN_\text{reg}:=\big\{\ccD\in\cN \,|\,\ccD^{p^{n-1}}\notin \fL_{(0)}\big\}$. Then
\[
\cN_\text{reg}=G.(\del+k\del^p+\dots+k\del^{p^{n-1}}).
\]
\end{pro}
\begin{proof}
Since $(\del+\sum_{i=1}^{n-1}\alpha_{i}\del^{p^{i}})^{p^{n}}=0$ and $(\del+\sum_{i=1}^{n-1}\alpha_{i}\del^{p^{i}})^{p^{n-1}}=\del^{p^{n-1}}$, this shows that any elements which are conjugate under $G$ to $\del+\sum_{i=1}^{n-1}\alpha_{i}\del^{p^{i}}$ are contained in $\cN_\text{reg}$. So $G.(\del+k\del^p+\dots+k\del^{p^{n-1}})\subseteq \cN_\text{reg}$. To show that $\cN_\text{reg}$ is a subset of $G.(\del+k\del^p+\dots+k\del^{p^{n-1}})$, we observe that if $\ccD_1\in \fL_{(1)}$, then $\ccD_1$ is nilpotent and $\ccD_1^{p^{n-1}}\in \fL_{(1)}\subset \fL_{(0)}$. Hence $\ccD_1\not\in \cN_\text{reg}$. As a result, $\cN_\text{reg}\subseteq \cN\setminus\fL_{(1)}$.

Let $\ccD \in \cN_\text{reg}$. Note that elements of $\cN\setminus\fL_{(1)}$ have the form 
\[
\sum_{i=0}^{n-1}\alpha_{i}\partial^{{p^{i}}}+f(x)\partial,
\]
where $f(x)\in \fm$ and $\alpha_i \in k $ with at least one $\alpha_{i}\neq 0$. If $\alpha_{0}\neq 0$ and $\alpha_{i}= 0$ for all $i\geq 1$, then $(\alpha_0\del+f(x)\del)^{p^{n-1}}\not \in \fL_{(0)}$ by \thref{YS161}(i). Hence $\alpha_0\del+f(x)\del\in\cN_\text{reg}$. Take $\ccD$ to be such an element. It follows from \thref{YS161}(ii) that $\ccD$ is conjugate under $G$ to $\partial$. Thus $\ccD \in G.(\del+k\del^p+\dots+k\del^{p^{n-1}})$ in this case.

For the other elements of $\cN\setminus\fL_{(1)}$, let $1\leq t\leq n-1$ be the largest index such that $\alpha_{t}\neq 0$, i.e. we consider elements of the form $\sum_{i=0}^{t}\alpha_{i}\partial^{{p^{i}}}+f(x)\partial$. Let $\Phi\in G$ be such that $\Phi(x)=\alpha x$, where $\alpha^{p^{t}}=\alpha_t$. Then $\Phi$ reduces $\sum_{i=0}^{t}\alpha_{i}\partial^{{p^{i}}}+f(x)\partial$ to 
\[
\del^{p^{t}}+\sum_{i=0}^{t-1}\beta_{i}\partial^{{p^{i}}}+g(x)\del
\]
for some $\beta_{i}\in k^{*}\alpha_i$ and $g(x)\in \fm$. By \thref{l:l13}, $\del^{p^{t}}+\sum_{i=0}^{t-1}\beta_{i}\partial^{{p^{i}}}+g(x)\del$ is conjugate under $G$ to
\[
\del^{p^{t}}+\sum_{i=0}^{t-1}\beta_{i}\partial^{{p^{i}}}+x^{(p^n-p^t)}\sum_{\eta=0}^{p^t-1}\mu_{\eta}x^{(\eta)}\del\]
for some $\mu_{\eta}\in k$. If $\beta_{i}=0$ for all $i$, then \thref{p:p1} and \thref{c:c2}(i) imply that
\[
\big(\del^{p^{t}}+x^{(p^n-p^t)}\sum_{\eta=0}^{p^t-1}\mu_{\eta}x^{(\eta)}\del\big)^{p^{n-1}}\in \fL_{(1)}\subset \fL_{(0)},
\]
and so elements of this form are not in $\cN_\text{reg}$. Now let $j\geq 1$ be the smallest index such that $\beta_{j}\neq 0$. Then \thref{p:p1} and \thref{c:c2}(ii)(a) imply that 
\[
\big(\del^{p^{t}}+\sum_{i=j}^{t-1}\beta_{i}\partial^{{p^{i}}}+x^{(p^n-p^t)}\sum_{\eta=0}^{p^t-1}\mu_{\eta}x^{(\eta)}\del\big)^{p^{n-1}}\in \fL_{(1)}\subset \fL_{(0)},
\]
and so elements of this form are not in $\cN_\text{reg}$. But if $\beta_{0}\neq 0$, then it is easy to see that 
\[
\big(\del^{p^{t}}+\sum_{i=0}^{t-1}\beta_{i}\partial^{{p^{i}}}+x^{(p^n-p^t)}\sum_{\eta=0}^{p^t-1}\mu_{\eta}x^{(\eta)}\del\big)^{p^{n-1}} \not\in \fL_{(0)},
\]
and so elements of this form are in $\cN_\text{reg}$. Take $\ccD$ to be such an element. It follows from \thref{p:p1} and \thref{c:c2}(ii)(b) that $\ccD=\del^{p^{t}}+\sum_{i=0}^{t-1}\gamma_{i}\del^{p^{i}}$ for some $\gamma_{i}\in k$ with $\gamma_{0}\neq 0$. Let $\Phi_1\in G$ be such that $\Phi_1(x)=\gamma_0 x$. Applying $\Phi_1$ to $\ccD$, we get $\Phi_1(\ccD)=\del+\sum_{i=1}^{t-1}\gamma_{i}'\del^{p^{i}}+(\gamma_0)^{-p^{t}}\del^{p^{t}}$ for some $\gamma_{i}'\in k^{*}\gamma_i$. Hence $\ccD\in G.(\del+k\del^p+\dots+k\del^{p^{n-1}})$ in this case. Since we have exhausted all elements of $\cN_\text{reg}$, this completes the proof.
\end{proof}

Before we proceed to show that the Zariski closure of $\cN_\text{reg}$ is an irreducible component of $\cN$, we need the following results:

\begin{lem}\thlabel{kerd}
Let $\ccD=\del+\sum_{i=1}^{n-1}\lambda_i\del^{p^{i}}$ with $\lambda_i\in k$ and denote by $\fc_\fL(\ccD)$ (respectively $\fc_{\fL_{p}}(\ccD))$ the centralizer of $\ccD$ in $\fL$ $($respectively $\fL_p)$. Then
\begin{enumerate}[\upshape(i)]
\item $\fc_\fL(\ccD)=\text{span$\{\del\}$}$.
\item $\fc_{\fL_{p}}(\ccD)=\text{span$\{\del, \del^p, \dots, \del^{p^{n-1}}$\}}$.
\item $\fc_\fL(\ccD)\cap \Lie(G)=\{0\}$.
\end{enumerate}
\end{lem}
\begin{proof}
(i) Clearly, $\text{span$\{\del\}$} \subseteq \fc_\fL(\ccD)$. Since $(\ad\ccD)^{p^{n}-1}\neq 0$ and $(\ad\ccD)^{p^{n}}=0$, the theory of canonical Jordan normal form says that there exists a basis $\mathcal{B}$ of $\fL$ such that the matrix of $\ad\ccD$ with respect to $\mathcal{B}$ is a single Jordan block of size $p^n$ with zeros on the main diagonal. Hence the matrix of $\ad\ccD$ has rank $p^n-1$. This implies that $\Ker (\ad\ccD)$ has dimension $1$. By definition, $\Ker (\ad\ccD)=\fc_\fL(\ccD)$. Hence $\fc_\fL(\ccD)=\text{span$\{\del\}$}$.

(ii) It is clear that $\text{span$\{\del, \del^p, \dots, \del^{p^{n-1}}\}$}\subseteq \fc_{\fL_{p}}(\ccD)$. Suppose $v\in \fc_{\fL_p}(\ccD)$. Then we can write $v=\sum_{i=0}^{n-1}\alpha_i\del^{p^{i}}+v_1$ for some $v_1 \in \fL_{(0)}$. Since $\sum_{i=0}^{n-1}\alpha_i\del^{p^{i}}\in \fc_{\fL_p}(\ccD)$, we must have that $v_1 \in \fc_{\fL_p}(\ccD)$. By (i), the centralizer of $\ccD$ in $\fL$ is $k\del$ which is not in $\fL_{(0)}$. Hence $v_1=0$ and $\fc_{\fL_{p}}(\ccD)=\text{span$\{\del, \del^p, \dots, \del^{p^{n-1}}$\}}$.

(iii) It follows from (i) and \thref{LieG}. This completes the proof.
\end{proof}

\begin{lem}\thlabel{l:l15}
The Zariski closure of $\cN_\text{reg}$ is an irreducible component of $\cN$.
\end{lem}
\begin{proof}
By \thref{l:l14}, it suffices to show that the Zariski closure of \newline
$G.(\del+k\del^p+\dots+k\del^{p^{n-1}})$ is an irreducible component of $\cN$. Set 
\[
X=\del+X_0,
\]
where $X_0=k\del^p+\dots+k\del^{p^{n-1}}$. Note that $X \cong X_0\cong \A^{n-1}$. So $X$ is irreducible. Moreover, $G$ is a connected algebraic group and so $\overline{G.X}$ is an irreducible variety contained in $\cN$. Hence $\dim \overline{G.X} \leq \dim \cN$. If $\dim \overline{G.X} \geq \dim \cN$, then we get the desired result.

Define $\Psi$ to be the morphism
\begin{align*}
\Psi: G\times X &\to \overline{G.X} \\
(g, \ccD)&\mapsto g.\ccD
\end{align*}
Since $G.X$ is dense in $\overline{G.X}$, it contains smooth points of $\overline{G.X}$. As the set of smooth points is $G$-invariant, there exists $\ccD\in X$ such that $\Psi(1, \ccD)=\ccD$ is a smooth point in $\overline{G.X}$. We may assume that $\ccD=\del+\sum_{i=1}^{n-1}\lambda_i\del^{p^{i}}$ for some $\lambda_i\in k$. Then $\T_{\ccD} (X)=X_0$ and the differential of $\Psi$ at the smooth point $(1, \ccD)$ is the map
\begin{align*}
(d\Psi)_{(1,\ccD)}: \Lie(G)\oplus X_0 &\to \T_{\ccD}(\overline{G.X})\\
(Y, Z) &\mapsto [Y, \ccD]+Z.
\end{align*}
Since $\dim \T_{\ccD}(\overline{G.X})=\dim \overline{G.X}$, it is enough to show that $\dim \T_{\ccD}(\overline{G.X})\geq \dim \cN=p^n-1$. This can be done by showing that $(d\Psi)_{(1,\ccD)}$ is injective. By \thref{LieG}, we know that any $Y\in \Lie(G)$ has the form $Y=\sum_{j=1}^{p^n-1}\mu_jx^{(j)}\del$, where $\mu_j\in k$ with $\mu_{p^{l}}=0$ for all $1\leq l\leq n-1$. By \eqref{LbracketWm},
\begin{align*}
[Y, \ccD]&=\bigg[\sum_{j=1}^{p^n-1}\mu_jx^{(j)}\del, \del+\sum_{i=1}^{n-1}\lambda_i\del^{p^{i}}\bigg]\\
&=-\sum_{j=1}^{p^n-1}\mu_jx^{(j-1)}\del-\sum_{i=1}^{n-1}\sum_{j=p^i+1}^{p^n-1}\lambda_i\mu_j x^{(j-p^i)}\del.
\end{align*}
Hence $[Y, \ccD]\in \fL$ for all $Y \in \Lie(G)$. As $\fL\cap X_0=\{0\}$, we have that 
\[
\Ker((d\Psi)_{(1,\ccD)})\cong \fc_{\Lie(G)}(\ccD).
\]
By \thref{kerd}(iii), $\fc_{\Lie(G)}(\ccD)=\{0\}$. Hence $\Ker((d\Psi)_{(1,\ccD)})=0$ and $(d\Psi)_{(1,\ccD)}$ is injective. As a result, 
\begin{align*}
\dim \T_{\ccD}(\overline{G.X})&\geq \dim \Ima ((d\Psi)_{(1,\ccD)})\\
&= \dim \Lie(G)+\dim X_0\\
&=(p^n-n)+(n-1)=p^n-1=\dim \cN.
\end{align*}
This completes the proof.
\end{proof}

\subsection{The irreducibility of $\cN$}\label{3.2.3irredsection}
Our goal is to prove the irreducibility of the variety $\cN$. To achieve this, we need the following result:

\begin{pro}\thlabel{l:l16}
Define $\cN_{\text{sing}}:=\cN\setminus \cN_\text{reg}=\big \{\ccD \in \cN \,|\, \ccD^{p^{n-1}} \in \fL_{(0)} \big\}$. Then
\[
\dim \cN_{\text{sing}}<\dim \cN.
\]
\end{pro}
Clearly, $\cN_{\text{sing}}$ is Zariski closed in $\cN$. To prove this proposition, we need to construct an $(n+1)$-dimensional subspace $V$ in $\fL_{p}$ such that $V \cap \cN_{\text{sing}}=\{0\}$. Then the result follows from \thref{ADthm}; see \thref{WnN0resultslem5} for a similar proof. The way $V$ is constructed relies on the original definition of $\fL$ due to H.~Zassenhaus and the following lemmas. Recall that $\fL$ has a $k$-basis $\{e_{\alpha}\,|\, \alpha\in \F_q\}$ with the Lie bracket given by $[e_\alpha, e_\beta]=(\beta-\alpha)e_{\alpha+\beta}$. Here $\F_q\subset k$ is a finite field with $q=p^n$ elements. The multiplicative group $\F_q^{*}$ of $\F_q$ is cyclic of order $p^n-1$ with generator $\xi$; see Sec.~\ref{2.1} for detail. Since $\fL$ is simple, it follows from \thref{sspenvelope}(iii) that all its minimal $p$-envelopes are isomorphic as restricted Lie algebras. Moreover, $\fL\cong\ad\fL$ via the adjoint representation. It follows from \thref{ssminpenvelope} that the minimal $p$-envelope $\fL_p=(\ad \fL)_p$ is the $p$-subalgebra of $\Der \fL$ generated by $\ad \fL$. We identify $\fL$ with $\ad \fL$. Since $\fL$ is a subalgebra of the restricted Lie algebra $\Der \cO(1; n)$ and $\{e_\alpha\,|\, \alpha \in \F_q\}$ is a basis of $\fL$, it follows from \thref{computeSp} that 
\begin{equation}\label{Lplemma1.1.3}
\fL_p=\sum_{\alpha\in \F_q, i\geq 0}ke_{\alpha}^{p^{i}},
\end{equation}
i.e. $\fL_p$ consists of all iterated $p$-th powers of $e_\alpha$ and $\fL$.

\begin{lem}\thlabel{l:l17}
\begin{enumerate}[\upshape(i)]
\item Let $\sigma\in \GL(\fL)$ be such that $\sigma(e_\alpha):=\xi^{-1} e_{\xi\alpha}$ for any $\alpha \in\F_q$. Then $\sigma$ is a diagonalizable automorphism of $\fL$. 
\item Let $\sigma$ be as in (i). Define $\sigma(e_\alpha^{p^{i}}):=\xi^{-p^{i}} e_{\xi\alpha}^{p^{i}}$ for any $i\geq 1$ and $\alpha\in \F_q$. Then $\sigma$ extends to an automorphism of $\fL_p$.
\end{enumerate}
\end{lem}
\begin{proof}
(i) By definition,
\begin{align*}
[\sigma(e_\alpha), \sigma(e_\beta)]&=[\xi^{-1} e_{\xi\alpha}, \xi^{-1} e_{\xi\beta}]=\xi^{-2}(\xi\beta-\xi\alpha)e_{\xi\alpha+\xi\beta}=\xi^{-1} (\beta-\alpha)e_{\xi(\alpha+\beta)}\\
&=\sigma([e_\alpha, e_\beta])
\end{align*}
for any $\alpha, \beta \in \F_q$. So the endomorphism $\sigma$ is an automorphism of $\fL$. Since $\xi^{p^n-1}=1$, we have that $\sigma^{p^n-1}=\Id$. As $k$ is an algebraically closed field, the automorphism $\sigma$ is diagonalizable. This proves (i).

(ii) Define $\sigma(e_\alpha^{p^{i}}):=\xi^{-p^{i}} e_{\xi\alpha}^{p^{i}}$ for any $i\geq 1$ and $\alpha\in \F_q$. We show that $\sigma$ extends to an automorphism of $\fL_p$. By definition,
\begin{align*}
[\sigma(e_\alpha^{p^{i}}), \sigma(e_\beta^{p^{j}})]=[\xi^{-p^{i}} e_{\xi\alpha}^{p^{i}}, \xi^{-p^{j}} e_{\xi\beta}^{p^{j}}]
\end{align*}
for any $i, j\geq 1$ and $\alpha, \beta \in \F_q$. On the other hand, it follows from \thref{defres}(1) that 
\begin{equation*}
[e_\alpha^{p^{i}}, e_\beta^{p^{j}}]=-(\ad e_\alpha)^{p^i-1}\circ (\ad e_\beta)^{p^{j}}(e_\alpha).
\end{equation*}
Since $\sigma$ is an automorphism of $\fL$ such that $\sigma(e_\alpha)=\xi^{-1} e_{\xi\alpha}$ for any $\alpha \in\F_q$, we have that 
\begin{align*}
\sigma[e_\alpha^{p^{i}}, e_\beta^{p^{j}}]&=\sigma\big(-(\ad e_\alpha)^{p^i-1}\circ (\ad e_\beta)^{p^{j}}(e_\alpha)\big)\\
&=-\xi^{-p^{i}}\xi^{-p^{j}}(\ad e_{\xi\alpha})^{p^i-1}\circ(\ad e_{\xi\beta})^{p^{j}}(e_{\xi\alpha})\\
&=[\xi^{-p^{i}}e_{\xi\alpha}^{p^i},\xi^{-p^{j}}e_{\xi\beta}^{p^{j}}]\\
&=[\sigma(e_\alpha^{p^{i}}), \sigma(e_\beta^{p^{j}})].
\end{align*}
Hence $\sigma$ extends to an automorphism of $\fL_p$. This proves (ii).
\end{proof}

Note that $[e_0, e_\beta]=\beta e_\beta$ for any $\beta \in \F_q$. So $\ad e_0$ is a semisimple endomorphism of $\fL$. Since $\ad \fL\cong \fL$, then $e_0$ generates a torus under the $[p]$-th power map.
\begin{lem}\thlabel{Ttorus}
Let $T$ denote the $p$-envelope $(ke_{0})_{p}$ in $\fL_p$. Then $T$ is an $n$-dimensional torus of $\fL_p$ such that $T\cap \fL=ke_0$.
\end{lem}
\begin{proof}
Let $T$ denote the $p$-envelope $(ke_{0})_{p}$ in $\fL_p$. Then $T=\sum_{i\geq 0} ke_0^{p^{i}}$ is the $p$-subalgebra of $\fL_p$ generated by $ke_{0}$; see \thref{Spdefn} and \thref{computeSp}. We first show that $T$ is a torus of $\fL_p$. By \thref{defres}(1), 
\[
[e_0^{p^{i}}, e_0^{p^{j}}]=-(\ad e_0)^{p^{i}-1}\circ (\ad e_0)^{p^{j}}(e_0)=0
\]
for any $i, j\geq 0$. Hence $T$ is an abelian $p$-subalgebra of $\fL_p$. Then we show that $T$ consists of semisimple elements. By \thref{sselemt}(iii), it is enough to show that $e_0$ is semisimple. By \thref{l:l17}, we know that $\sigma$ acts on $T$ as 
\begin{equation}\label{sigmaaction}
\sigma(e_0^{p^{i}})=\xi^{-p^{i}}e_0^{p^{i}}
\end{equation}
for all $i\geq 0$. Since $\xi^{p^{n}-1}=1$, we see that $\xi^{-1}, \xi^{-p}, \dots, \xi^{-p^{n-1}}$ are the eigenvalues of $\sigma$ on $T$. Moreover, $\sigma(e_0)=\xi^{-1}e_0$ and $\sigma(e_0^{p^{n}})=\xi^{-p^{n}}e_0^{p^{n}}=\xi^{-1}e_0^{p^{n}}$. So $e_0$ and $e_0^{p^{n}}$ correspond to the same eigenvalue $\xi^{-1}$ and they are in the same eigenspace. We show that $e_0$ and $e_0^{p^{n}}$ are linearly dependent with $e_0=e_0^{p^{n}}$. Recall that for any $\alpha, \beta \in \F_q$, the Lie bracket of $\fL$ is given by $[e_\alpha, e_\beta]=(\beta-\alpha)e_{\alpha+\beta}$. Then for any $\alpha\in \F_q$ and $i>0$, we have that 
\begin{align*}
[e_0-e_0^{p^{n}}, e_\alpha^{p^{i}}]&=[e_0, e_\alpha^{p^{i}}]-[e_0^{p^{n}}, e_\alpha^{p^{i}}]\\
&=-(\ad e_\alpha)^{p^{i}}(e_0)+(\ad e_0)^{p^{n}-1}\circ (\ad e_\alpha)^{p^{i}}(e_0)\\
&=0+0=0.
\end{align*}
Similarly, for any $\alpha\in \F_q$ and $i=0$, we have that 
\[
[e_0-e_0^{p^{n}}, e_\alpha]=[e_0, e_\alpha]-(\ad e_0)^{p^{n}}(e_\alpha)=\alpha e_\alpha-\alpha^{p^{n}}e_\alpha=\alpha e_\alpha-\alpha e_\alpha=0.
\]
We show the above calculations imply that $e_0-e_0^{p^{n}}\in \fz(\fL_p)$. By \eqref{Lplemma1.1.3}, the minimal $p$-envelope $\fL_p$ of $\fL$ is given by $\fL_p=\sum_{\alpha\in \F_q, i\geq 0}ke_{\alpha}^{p^{i}}$, i.e. $\fL_p$ consists of all iterated $p$-th powers of $e_\alpha$ and $\fL$. It follows from the above calculations that $e_0-e_0^{p^{n}}\in \fz(\fL_p)$. Since $\fL$ is simple, it follows from \thref{sspenvelope}(iii) that $\fL_p$ is semisimple. Since $\fz(\fL_p)$ is an abelian $p$-ideal of $\fL_p$, we must have that $\fz(\fL_p)=0$. As a result, $e_0=e_0^{p^{n}}$. This implies that $e_0$ is semisimple. By \thref{sselemt}(iii), $e_0^{p^{i}}$ is semisimple for every $i\geq 1$ and $t$ is semisimple for every $t\in (ke_{0})_{p}=T$. Hence $T$ is an abelian $p$-subalgebra of $\fL_p$ consisting of semisimple elements, i.e. $T$ is a torus of $\fL_p$. 

Next we show that $\dim_k T=n$. Since $e_0=e_0^{p^{n}}$, this implies that the $[p]$-th power map on $T$ is periodic, i.e. $e_0^{p^{i}}=e_0^{p^{n+i}}$ for all $i\geq 0$. Hence any $t\in T$ has the form $t=\sum_{i=0}^{n-1}\mu_i e_0^{p^{i}}$ for some $\mu_i\in k$. Thus $\{e_0, e_0^{p}, \dots, e_0^{p^{n-1}}\}$ spans $T$. Moreover, it follows from \eqref{sigmaaction} that the eigenvectors $e_0, e_0^{p}, \dots, e_0^{p^{n-1}}$ correspond to distinct eigenvalues $\xi^{-1}, \xi^{-p}, \dots, \xi^{-p^{n-1}}$. Hence they are linearly independent. Therefore, $\{e_0, e_0^{p}, \dots, e_0^{p^{n-1}}\}$ is a basis for $T$ and $\dim_k T=n$.

It remains to show that $T\cap \fL=ke_0$. Note that $e_{0}\notin \fL_{(0)}$. Indeed, if $e_0\in \fL_{(0)}$, then $T$ is contained in $\fL_{(0)}$ as $\fL_{(0)}$ is restricted. But this contradicts that any nonzero torus of $\fL_{(0)}$ has dimension $1$; see \eqref{eip} in Sec.~\ref{2.1}. Therefore, $e_{0}\notin \fL_{(0)}$. Since $e_0$ is semisimple, we may assume that $e_0=\del+z_0$ for some $0\neq z_0\in \fL_{(0)}$. Let us compute $e_0^{p^{i}}$ for $1\leq i\leq n-1$. By Jacobson's formula, $e_0^p=\del^p+\alpha_{1, 0}\del+z_1$ for some $\alpha_{1, 0}\in k$ and $0\neq z_1\in \fL_{(0)}$ (otherwise $e_0^p$ is nilpotent). Continue doing this, one can show that for $1\leq i\leq n-1$, 
\[
e_0^{p^{i}}=\del^{p^{i}}+\sum_{j=0}^{i-1}\alpha_{i, j}\del^{p^{j}}+z_i
\]
for some $\alpha_{i, j}\in k$ and $0\neq z_i\in \fL_{(0)}$ (otherwise $e_0^{p^{i}}$ is nilpotent). It is clear that $e_0^{p^{i}} \notin \fL$ for all $1\leq i\leq n-1$. Hence $T\cap \fL=ke_0$. This completes the proof.
\end{proof}

\begin{rmk}\thlabel{e_0info}
It follows from \thref{l:l17} and the last proof that $e_0^{p^{i}}$ with $i\geq 0$ satisfies the following:
\begin{enumerate}[\upshape(i)]
\item $\sigma(e_0^{p^{i}})=\xi^{-p^{i}}e_0^{p^{i}}$ for all $i\geq 0$.
\item $e_0^{p^{i}}$ is semisimple with $e_0^{p^{i}}=e_0^{p^{n+i}}$ for all $i\geq 0$.
\item $e_0$ is an element of $\fL=\fL_{(-1)}$ such that $e_0\notin \fL_{(0)}$. Hence for any $0\leq i\leq n-1$, we may assume that
\[
e_0^{p^{i}}=\del^{p^{i}}+\sum_{j=0}^{i-1}\alpha_{i, j}\del^{p^{j}}+z_i
\]
for some $\alpha_{i, j}\in k$ and $0\neq z_i\in \fL_{(0)}$. Note that $e_0^{p^{i}} \notin \fL$ for all $1\leq i\leq n-1$. Moreover, $\{e_0, e_0^p, \dots, e_0^{p^{n-1}}\}$ forms a basis for the $n$-dimensional torus $T=(ke_0)_p$.
\end{enumerate}
\end{rmk}

Since we have proved in \thref{l:l17} that $\sigma$ is an automorphism of $\fL$, then it preserves the natural filtration $\{\fL_{(i)}\}_{i\geq -1}$ of $\fL$. Moreover, we can prove the following:
\begin{lem}\thlabel{l:l18}
For $-1\leq i\leq p^n-2$, the automorphism $\sigma$ acts on each $1$-dimensional vector space $\fL_{(i)}/\fL_{(i+1)}$ as $\xi^{i}\Id$.
\end{lem}
\begin{proof}
We prove this result by induction on $i$. For $i=-1$, consider the surjective map $\pi_{-1}: \fL_{(-1)} \twoheadrightarrow \fL_{(-1)}/\fL_{(0)}$. By \thref{e_0info}(iii), we know that $e_0$ is an element of $\fL_{(-1)}$ such that $e_{0}\notin \fL_{(0)}$. Then the vector space $\fL_{(-1)}/\fL_{(0)}$ is spanned by $\pi(e_0)$. This implies that $\sigma$ acts on $\fL_{(-1)}/\fL_{(0)}$ as $\xi^{-1}\Id$. Hence the result holds for $i=-1$.

Suppose the result holds for $-1\leq i=j\leq p^n-3$, i.e. $\sigma$ acts on $\fL_{(j)}/\fL_{(j+1)}$ as $\xi^{j}\Id$. We want to show the result holds for $i=j+1\leq p^n-2$, i.e. $\sigma$ acts on $\fL_{(j+1)}/\fL_{(j+2)}$ as $\xi^{j+1}\Id$. Since $\fL=\fL_{(0)}+ke_0$ and $[\fL, \fL_{(i+1)}]=\fL_{(i)}$, we have that $[e_0, \fL_{(i+1)}]\nsubseteq \fL_{(i+1)}$. So we can find for each $i\in \{0, \dots, p^n-2\}$ a $u_i\in \fL_{(i)}$ such that $[e_0, u_i]\notin \fL_{(i)}$. By writing $u_i$ as a sum of $\sigma$-eigenvectors corresponding to the same eigenvalue, we see that we may assume that each $u_i$ is a $\sigma$-eigenvector corresponding to an eigenvalue, say $\lambda_i$. Then 
\begin{align*}
\sigma[e_0, u_i]=[\sigma(e_0), \sigma(u_i)]=[\xi^{-1}e_0, \lambda_i u_i]=\xi^{-1}\lambda_i[e_0, u_i].
\end{align*}
So $\xi^{-1}\lambda_i$ is the eigenvalue of $\sigma$ on $\fL_{(i-1)}/\fL_{(i)}$. Since $i\in \{0, \dots, p^n-2\}$, then $i-1\in \{-1, \dots, p^n-3\}$. By the induction hypothesis, $\sigma$ acts on $\fL_{(i-1)}/\fL_{(i)}$ as $\xi^{i-1}\Id$. So we must have that $\xi^{-1}\lambda_i=\xi^{i-1}$ and hence $\lambda_i=\xi^{i}$. Now consider the surjective map $\pi_i: \fL_{(i)} \twoheadrightarrow \fL_{(i)}/\fL_{(i+1)}$. Since $u_i\in \fL_{(i)}$ is such that $[e_0, u_i]\notin \fL_{(i)}$, i.e. $u_i\notin \fL_{(i+1)}$, then the vector space $\fL_{(i)}/\fL_{(i+1)}$ is spanned by $\pi_i(u_i)$. It follows that $\sigma$ acts on $\fL_{(i)}/\fL_{(i+1)}$ as $\xi^i\Id$. Therefore, we proved by induction that for $-1\leq i\leq p^n-2$, $\sigma$ acts on each $1$-dimensional vector space $\fL_{(i)}/\fL_{(i+1)}$ as $\xi^{i}\Id$. This completes the proof.
\end{proof}

\begin{lem}\thlabel{utoral}
Let $u$ denote the element $u_0$ in the last proof. Then $ku$ is a $1$-dimensional torus in $\fL_{(0)}$.
\end{lem}

\begin{proof}
Recall from the last proof that $u\in\fL_{(0)}$ is such that $[e_0, u]\notin \fL_{(0)}$. Moreover, $u$ is a $\sigma$-eigenvector corresponding to the eigenvalue $\xi^0=1$, i.e. $\sigma(u)=u$. We want to show that $ku$ is a $1$-dimensional torus in $\fL_{(0)}$. We first show that $u$ is not nilpotent. Note that 
$u\notin \fL_{(1)}$. Indeed, if $u\in \fL_{(1)}$, then \thref{e_0info}(iii) implies that $[e_0, u]\in [\fL_{(-1)}, \fL_{(1)}]\subseteq \fL_{(0)}$. But this contradicts that $[e_0, u]\notin \fL_{(0)}$. Hence $u\in \fL_{(0)}$ is such that $u\notin \fL_{(1)}$. If $u$ is nilpotent, then $\fL_{(0)}$ would be $p$-nilpotent. But this contradicts \eqref{eip}. Hence $u$ is not nilpotent. Next we show that $u^{p}$ is a multiple of $u$. Since $u\in\fL_{(0)}$ and $\fL_{(0)}$ is restricted, we have that $u^p\in \fL_{(0)}$. Since $\sigma$ is an automorphism of the restricted Lie algebra $\fL_{(0)}$, we have that $\sigma(u^p)=\sigma(u)^p=u^p$. So $u^p$ is $\sigma$-fixed. By \thref{l:l18}, the eigenvalues of $\sigma$ on $\fL_{(0)}$ are $\xi^0=1, \xi, \dots, \xi^{p^{n}-3}$ and $\xi^{p^{n}-2}=\xi^{-1}$, and they all have multiplicity $1$. Since both $u$ and $u^p$ are $\sigma$-fixed, it follows that $u^p$ is a multiple of $u$. Therefore, $ku$ is a $1$-dimensional torus in $\fL_{(0)}$. This completes the proof.
\end{proof}

Let us summarize what we know so far.

\begin{rmk}\thlabel{rkV}
It follows from \thref{l:l18} that
\begin{enumerate}[\upshape(i)]
\item the eigenvalues of $\sigma$ on $\fL=\spn\{\del, x\del, x^{(2)}\del, \dots, x^{(p^n-2)}\del, x^{(p^n-1)}\del\}$ are $\xi^{-1}, \\\xi^0=1, \xi, \dots, \xi^{p^{n}-3}$ and $\xi^{p^{n}-2}=\xi^{-1}$. All have multiplicity $1$ except $\xi^{-1}$ which has multiplicity $2$;
\item the eigenvalues of $\sigma$ on $\fL_{(0)}=\spn\{x\del, x^{(2)}\del, \dots, x^{(p^n-2)}\del, x^{(p^n-1)}\del\}$ are $\xi^0=1, \xi, \dots, \xi^{p^{n}-3}$ and $\xi^{p^{n}-2}=\xi^{-1}$. All have multiplicity $1$;
\item the eigenspace $\fL[i]:=\{\ccD\in \fL\,|\,\sigma(\ccD)=\xi^{i}\ccD\}$ corresponding to the eigenvalue $\xi^{i}$, where $0\leq i\leq p^{n}-3$, has dimension $1$. In particular, the eigenspace $\fL[0]=ku$ is a $1$-dimensional torus in $\fL_{(0)}$; see \thref{utoral}. Since any torus in a restricted Lie algebra has a basis consisting of toral elements (see \thref{toralbasis}), we may assume that $u$ is toral, i.e. $u^p=u$;
\item the eigenspace $\fL[-1]=\text{$\spn \{e_0, v\,|\, v\in \fL_{(p^n-2)}\}$}$ and it has dimension $2$; see \thref{e_0info}(i) and (iii).
\end{enumerate}
\end{rmk}

We are now ready to prove \thref{l:l16} which states that \\$\cN_{\text{sing}}:=\cN\setminus \cN_\text{reg}=\big \{\ccD \in \cN \,|\, \ccD^{p^{n-1}} \in \fL_{(0)} \big\}$ has $\dim \cN_{\text{sing}}<\dim \cN$. 

\begin{proof}[Proof of \thref{l:l16}]
Recall the $n$-dimensional torus $T=\sum_{i=0}^{n-1}ke_0^{p^{i}}$ in $\fL_p$ and the $1$-dimensional torus $ku$ in $\fL_{(0)}$; see \thref{Ttorus} and \thref{utoral}. Put $V:=T\oplus ku$. We want to show that $V\cap \cN_{\text{sing}}=\{0\}$. Then the result follows from \thref{ADthm}. Suppose for contradiction that $V\cap \cN_{\text{sing}}\neq\{0\}$. Then take a nonzero element $y$ in $V\cap \cN_{\text{sing}}$, we can write
\begin{align*}
y=\sum_{i=0}^{n-1}\lambda_{i}e_{0}^{p^{i}}+\mu u
\end{align*}
for some $\lambda_i, \mu\in k$ with at least one $\lambda_i\neq 0$. 

\textbf{\textit{Case 1}}. Suppose $\lambda_0\neq 0$. We want to show this implies that $y^{p^{n-1}}\notin\fL_{(0)}$. But $y \in \cN_{\text{sing}}$ by our assumption, this contradicts the definition of $\cN_{\text{sing}}$. 

By \thref{generalJacobF} and the facts that $e_0^{p^{n}}=e_0$ and $u^p=u$ (see \thref{e_0info}(ii) and \thref{rkV}(iii)), we have that
\begin{equation}\label{lambdaneq0ypn-1}
\begin{aligned}
y^{p^{n-1}}=&\bigg(\sum_{i=0}^{n-1}\lambda_{i}e_{0}^{p^{i}}+\mu u\bigg)^{p^{n-1}}\\
=&\lambda_{0}^{p^{n-1}}e_0^{p^{n-1}}+\lambda_{n-1}^{p^{n-1}}e_0^{p^{n-2}}+\lambda_{n-2}^{p^{n-1}}e_0^{p^{n-3}}+\dots+\lambda_2^{p^{n-1}}e_0^{p}+\lambda_1^{p^{n-1}}e_0\\
&+\mu^{p^{n-1}} u+\sum_{r=0}^{n-2}u_r^{p^{r}},
\end{aligned}
\end{equation}
where $u_r$ is a linear combination of commutators in $u$ and $e_0^{p^{i}}, 0\leq i\leq n-1$. We show that $u_r\in \fL$ for all $0\leq r\leq n-2$. By \thref{utoral}, $u\in \fL_{(0)}\subset \fL$. By \thref{e_0info}(iii), for any $0\leq i\leq n-1$, we may assume that $e_0^{p^{i}}=\del^{p^{i}}+\sum_{j=0}^{i-1}\alpha_{i, j}\del^{p^{j}}+z_i$ for some $\alpha_{i, j}\in k$ and $0\neq z_i\in \fL_{(0)}$. So $e_0\in \fL$ is such that $e_0\notin \fL_{(0)}$ and for $1\leq i\leq n-1$, $e_0^{p^{i}}\in \fL_p$ is such that $e_0^{p^{i}} \notin \fL$. Since $\fL$ is an ideal in $\fL_p$ and the $\del^{p^{i}}$, $0\leq i\leq n-1$, commute amongst each other, it follows immediately that $u_r\in \fL$ for all $0\leq r\leq n-2$. 

Next we show that $\sum_{r=0}^{n-2}u_r^{p^{r}}\in k\del^{p^{n-2}}+k\del^{p^{n-3}}+\dots+k\del^{p}+\fL$. By \thref{wittthm}, $\fL=\spn\{x^{(\gamma)}\del\,|\, 0\leq \gamma\leq p^n-1\}$ is a subalgebra of the restricted Lie algebra $\Der\cO(1;n)$. By \thref{Spdefn}(ii), $\fL^{p^{r}}:=\{\ccD^{p^{r}}\,|\, \ccD \in \fL\}$. Since $u_r\in \fL$ for all $0\leq r\leq n-2$, we have that $u_r^{p^{r}}\in \fL^{p^{r}}$. Hence $\sum_{r=0}^{n-2}u_r^{p^{r}}\in \sum_{r=0}^{n-2}\fL^{p^{r}}$. By \thref{computeSp} and the fact that $\fL_{(0)}\subset \fL$ is a restricted subalgebra of $\Der\cO(1;n)$, we get
\[
\sum_{r=0}^{n-2}\fL^{p^{r}}=k\del^{p^{n-2}}+k\del^{p^{n-3}}+\dots+k\del^{p}+\fL.
\]
Hence $\sum_{r=0}^{n-2}u_r^{p^{r}}\in k\del^{p^{n-2}}+k\del^{p^{n-3}}+\dots+k\del^{p}+\fL$ as desired.

Now look at $\mu^{p^{n-1}} u$ in \eqref{lambdaneq0ypn-1}. Since $u\in \fL_{(0)}$, it follows from the above that we can write
\[
\mu^{p^{n-1}} u+\sum_{r=0}^{n-2}u_r^{p^{r}}=\sum_{i=0}^{n-2}\lambda_i'\del^{p^{i}}+z
\]
for some $\lambda_i'\in k$ and $z\in \fL_{(0)}$. Substituting this and $e_0^{p^{i}}=\del^{p^{i}}+\sum_{j=0}^{i-1}\alpha_{i, j}\del^{p^{j}}+z_i$ ($0\leq i\leq n-1$) into \eqref{lambdaneq0ypn-1}, we get 
\begin{equation}\label{case1ypn-1prop3.2.2}
y^{p^{n-1}}=\lambda_{0}^{p^{n-1}}\del^{p^{n-1}}+\sum_{i=0}^{n-2}\lambda_i''\del^{p^{i}}+\tilde{z}
\end{equation}
for some $\lambda_i''\in k$ and $\tilde{z}\in \fL_{(0)}$. Since $\lambda_0\neq 0$, this shows that $y^{p^{n-1}}\notin \fL_{(0)}$. But $y\in \cN_{\text{sing}}=\big \{\ccD \in \cN \,|\, \ccD^{p^{n-1}} \in \fL_{(0)} \big\}$, this contradicts the definition of $\cN_{\text{sing}}$.

\textbf{\textit{Case 2}}. Suppose now $\lambda_0=0$ and let $1\leq s\leq n-1$ be the largest index such that $\lambda_s\neq 0$. Then
\[
y=\sum_{i=1}^{s}\lambda_{i}e_{0}^{p^{i}}+\mu u.
\]
We want to show this implies that $y^{p^{2n-s-1}}\neq 0$. But $y\in \cN_{\text{sing}}$, in particular, $y$ is nilpotent with $y^{p^{n}}=0$. As $2n-s-1\geq n$, this contradicts that $y$ is nilpotent. 

By \thref{generalJacobF} and the facts that $e_0^{p^{n}}=e_0$ and $u^p=u$ (see \thref{e_0info}(ii) and \thref{rkV}(iii)), we have that
\begin{align*}
y^{p^{n-s}}=\lambda_{s}^{p^{n-s}}e_{0}+\lambda_{s-1}^{p^{n-s}}e_{0}^{p^{n-1}}+\lambda_{s-2}^{p^{n-s}}e_0^{p^{n-2}}+\dots+\lambda_{1}^{p^{n-s}}e_{0}^{p^{n-s+1}}+\mu^{p^{n-s}}u+\sum_{l=0}^{n-s-1}v_l^{p^{l}},
\end{align*}
where $v_l$ is a linear combination of commutators in $u$ and $e_0^{p^{i}}, 1\leq i\leq s$. By Jacobi identity, we can rearrange each $v_l$ so that $v_l$ is in the span of $[w_t,[w_{t-1},[\dots,[w_1,u]\dots]$, where $t=p^{n-s-l}-1$ and each $w_{\nu}, 1\leq\nu\leq t$, is equal to $u$ or to some $e_0^{p^{i}}, 1\leq i\leq s$. Arguing similarly as in case 1, one can show that $[w_t,[w_{t-1},[\dots,[w_1,u]\dots]\in\fL$. More precisely, we show that $[w_t,[w_{t-1},[\dots,[w_1,u]\dots]\in\fL_{(0)}$. Since $\sigma(u)=u$ and $\sigma(e_0^{p^{i}})=\xi^{-p^{i}}e_0^{p^{i}}$ for $1\leq i\leq s$ (see \thref{e_0info}(i) and \thref{rkV}(iii)), the $\sigma$-eigenvalue of each such iterated commutator is $\xi^{-a}$, where
\[
1<p\leq a\leq tp^s=p^{n-l}-p^s\leq p^n-p<p^n-2.
\]
So this eigenvalue is not equal to $\xi^{-1}$. Then $[w_t,[w_{t-1},[\dots,[w_1,u]\dots]\in\fL_{(0)}\setminus \fL_{(p^n-2)}\subset \fL_{(0)}$; see \thref{rkV}. Hence $v_l \in \fL_{(0)}$. Since $\fL_{(0)}$ is restricted, we have that $v_l^{p^{l}}\in \fL_{(0)}$ and so $\sum_{l=0}^{n-s-1}v_l^{p^{l}}\in \fL_{(0)}$. Therefore,
\begin{equation*}\label{ypn-s}
y^{p^{n-s}}=\lambda_{s}^{p^{n-s}}e_{0}+\gamma_{n-1}e_{0}^{p^{n-1}}+\gamma_{n-2}e_0^{p^{n-2}}+\dots+\gamma_{n-s+1}e_{0}^{p^{n-s+1}}+f_1
\end{equation*}
for some $\gamma_{i}\in k, \lambda_s\neq 0$ and $f_1\in \fL_{(0)}$. Note that $y^{p^{n-s}}$ has a similar expression to $y$ in case 1. By \thref{generalJacobF} and a similar argument as in case 1, one can show that 
\begin{equation*}\label{case2bypn-sprop3.2.2}
(y^{p^{n-s}})^{p^{n-1}}=y^{p^{2n-s-1}}=\lambda_s^{p^{2n-s-1}}\del^{p^{n-1}}+\sum_{i=0}^{n-2}\gamma_i'\del^{p^{i}}+\tilde{f} 
\end{equation*}
for some $\gamma_i'\in k$ and $\tilde{f}\in \fL_{(0)}$; see \eqref{case1ypn-1prop3.2.2}. Since $\lambda_s\neq 0$, this shows that $y^{p^{2n-s-1}}\neq 0$. But $y\in \cN_{\text{sing}}$, in particular, $y$ is nilpotent and $y^{p^{n}}=0$ by \eqref{nilpotentconditioninL}. Since $2n-s-1\geq n$, the above contradicts that $y$ is nilpotent. Therefore, we proved by contradiction that $V\cap \cN_{\text{sing}}=\{0\}$. Then the result follows from \thref{ADthm}. This completes the proof.
\end{proof}

\begin{thm}\thlabel{main}
The variety $\cN$ coincides with the Zariski closure of
\[
\cN_\text{reg}=G.(\del+k\del^p+\dots+k\del^{p^{n-1}})
\] and hence is irreducible.
\end{thm}
\begin{proof}
By \thref{nvarietythm}, we know that the variety $\cN$ is equidimensional of dimension $p^n-1$. The ideal defining $\cN$ is homogeneous, hence any irreducible component of $\cN$ contains $0$. It follows from \thref{l:l15} that the Zariski closure of $\cN_\text{reg}$ is an irreducible component of $\cN$. Let $Z_1, \dots ,Z_t$ be pairwise distinct irreducible components of $\cN$, and set $Z_1=\overline{\cN_\text{reg}}$. Suppose $t\geq 2$. Then $Z_2\setminus Z_1$ is contained in $\cN_{\text{sing}}$, which is Zariski closed in $\cN$ with $\dim \cN_{\text{sing}}<\dim \cN$  by \thref{l:l16}. Since $Z_2\setminus Z_1=Z_2 \setminus (Z_1\cap Z_2)$, this set is Zariski dense in $Z_2$. Then its closure $Z_2$ is also contained in $\cN_{\text{sing}}$, i.e. $\dim Z_2=\dim \cN \leq \dim \cN_{\text{sing}}$. This is a contradiction. Hence $t=1$ and the variety $\cN$ is irreducible. This completes the proof.
\end{proof}



\end{document}